\documentclass[11pt,letterpaper]{amsart}
\usepackage{mathtools}

\usepackage[T1]{fontenc}
\usepackage[latin9]{inputenc}
\usepackage{geometry}
\usepackage{wrapfig}
\usepackage{multicol}
\usepackage{graphicx}
\usepackage{soul}
\usepackage{xcolor}
\usepackage{amssymb}

\newtheorem{theorem}{Theorem}[section]
\newtheorem{lemma}[theorem]{Lemma}
\newtheorem{proposition}[theorem]{Proposition}

{ \theoremstyle{definition}
\newtheorem{definition}[theorem]{Definition}}
{ \theoremstyle{remark}
\newtheorem{remark}[theorem]{Remark}}
\usepackage{placeins}
\usepackage{cite}
\usepackage{caption}
\usepackage{enumerate}
\usepackage{afterpage}
\usepackage{enumitem}
\usepackage{bmpsize}
\usepackage{mathtools}

\newcommand\eqd{\stackrel{\mathclap{\normalfont d}}{=}}

\usepackage{hyperref}
\usepackage{tabu}
\usepackage{enumitem}   
\numberwithin{equation}{section}
\usepackage{stmaryrd}

\usepackage{tikz}
\usetikzlibrary{matrix,graphs,arrows,positioning,calc,decorations.markings,shapes.symbols}

\title{Two-point convergence of the stochastic six-vertex model to the Airy process }
\date{\today}
\author{Evgeni Dimitrov}

\begin{document}

\maketitle 

\vspace{-3mm}
\begin{abstract}
In this paper we consider the stochastic six-vertex model in the quadrant started with step initial data. After a long time $T$, it is known that the one-point height function fluctuations are of order $T^{1/3}$ and governed by the Tracy-Widom distribution. We prove that the two-point distribution of the height function, rescaled horizontally by $T^{2/3}$ and vertically by $T^{1/3}$, converges to the two-point distribution of the Airy process. The starting point of this result is a recent connection discovered by Borodin-Bufetov-Wheeler between the stochastic six-vertex model and the ascending Hall-Littlewood process (a certain measure on plane partitions). Using the Macdonald difference operators, we obtain formulas for two-point observables for the ascending Hall-Littlewood process, which for the six-vertex model give access to the joint cumulative distribution function for its height function. A careful asymptotic analysis of these observables gives the two-point convergence result under certain restrictions on the parameters of the model. 
\end{abstract}

\tableofcontents

%
\section{Introduction and main results}\label{Section1}

%
\subsection{Preface}\label{Section1.1} 

More than thirty years ago Kardar, Parisi and Zhang \cite{KPZ} studied the time evolution of random growing interfaces and proposed the following stochastic partial differential equation (called the {\em KPZ equation}) for a height function $\mathcal{H}(T,X) \in \mathbb{R}$
\begin{equation}\label{KPZEq}
\partial_TH(X,T) = \frac{1}{2} \partial^2_X \mathcal{H}(X,T) + \frac{1}{2} [ \partial_X \mathcal{H}(X,T)]^2 +  \xi(X,T).
\end{equation}
In (\ref{KPZEq}) the letters $T$ and $X$ denote time and space, and $\xi$ is space-time Gaussian white noise, so that formally $\mathbb{E} \left[ \xi(X,T) \xi(Y,S)\right] = \delta(T-S)\delta(X-Y).$ Drawing upon the earlier work of Forster, Nelson and Stephen \cite{FNS}, KPZ predicted that for large time $T$, the height function $\mathcal{H}(X,T)$ exhibits fluctuations of order $T^{1/3}$ and has spatial correlation length of order $T^{2/3}$. The critical exponents $1/3$ and $2/3$ are believed to be universal for a large class of growth models, which is now called the KPZ universality class. For more on the KPZ universality class we refer to the surveys and books \cite{CU2, QS, HT} and the references therein.

The $3$ : $2$ : $1$ scaling of time : space : fluctuation (now known as KPZ scaling) is believed to take any model from the KPZ universality class to a universal fluctuating field, which does not depend on the particular model, but does depend on its initial data class. The conjectural space-time limiting field that attracts all models in the class is now called the {\em KPZ fixed point}, which can be described either through its Markov transition kernel \cite{MQR} or through a variational formula involving a multiparameter scaling limit known as the {\em Airy sheet} or {\em directed landscape} \cite{DOV18}. There are various natural choices for initial conditions that have been considered, see \cite[Figure 4]{CU2}, but in the present paper we will focus our attention on the {\em narrow wedge} initial data, which for discrete models is typically called {\em step} initial data. For models in the KPZ universalty class, started from step initial conditions, it is believed that under the KPZ scaling the height function converges to the {\em Airy process}. The Airy process was first introduced in \cite{Spohn} and  is a stationary continuous process, whose one point distribution is given by the celebrated Tracy-Widom distribution \cite{TWAiry}.

While the $1/3$ and $2/3$ critical exponents have been established in greater generality, demonstrating the convergence to the Airy process (in a rigorous mathematical sense) has only been done for special {\em integrable} or {\em exactly solvable} models. For the asymmetric simple exclusion process (ASEP) the one-point convergence to the Airy process, i.e. the convergence to the Tracy-Widom distribution, was first established in \cite{TWASEP}, see also \cite{TWASEP2}. Analogous asymptotic results are proved for the stochastic six-vertex model in \cite{Bor16, BCG14}, the KPZ equation in \cite{CorQ, BCF2}, the semi-discrete directed polymer in \cite{BorCor}, the log-gamma polymer in \cite{BorCorRem,KQ}, and the $q$-TASEP in \cite{Bar15,FerVet}. 

Some of the aforementioned models have special cases, sometimes referred to as a {\em free fermion point} or {\em zero temperature limit} in the case of polymer models, where the multi-point (and not just single-point) convergence  to the Airy process is known. Possibly the most famous example comes from the totally asymmetric simple exclusion process (TASEP), which is obtained from the ASEP by sending the asymmetry parameter to zero. In \cite{Joh05} it was shown that when started from step initial conditions the (discrete time) TASEP, or geometric last passage percolation, converges to the Airy process in the sense of finite dimensional distributions. Analogous results exist for the polynuclear growth model \cite{Spohn, JDPNG} and random lozenge tilings \cite{FSpohn, P1}.

The common feature shared by all models for which finite-dimensional convergence to the Airy process is known is that they have the structure of determinantal processes. The determinantal structure is especially useful for proving convergence to the Airy process, since the latter is itself a determinantal process with a correlation kernel given by the {\em extended Airy kernel} \cite{Spohn}. In particular, when proving that a determinantal process converges to the Airy process in the finite dimensional sense it is sufficient to show that the kernel converges (in a sufficiently strong sense) to the extended Airy kernel, and this is indeed how convergence has been proved in the past. The problem with generalizing this approach to models like the ASEP is that once one moves away from the free fermion point, the determinantal structure is lost. The way the one-point convergence to the Airy process has been established for positive temperature integrable models (including the ASEP, stochastic six-vertex model, KPZ equation) is to utilize one of two main algebraic structures -- Macdonald processes and quantum integrable systems. In fact, there are now bridges between these structures indicating that they might be eventually joined together. Both of these structures produce moment formulas, which formally should completely characterize the distribution. However, despite the existence of multi-point exact formulas for various positive temperature integrable models \cite{BorCor, BCS, BP16, NZ} the convergence to the Airy process has proved elusive so far and only rigorously worked out in the one-point case.\\

{\bf The goal of the present paper is to prove that the height function of the stochastic six-vertex model at two points jointly converges to the Airy process, which is the first rigorous multi-point convergence result for a positive temperature integrable model in the KPZ universality class.} The starting point of our analysis is a remarkable distributional equality between the stochastic six-vertex model and the {\em ascending Hall-Littlewood process} (a special case of the Macdonald processes from \cite{BorCor}), which was established in \cite{BBW}. This identification allows us to recast the problem into the framework of Macdonald processes, where we use the method of the Macdonald difference operators from \cite{BorCor} to derive two-point observables that are suitable for asymptotic analysis. Our current framework suffers from two types of limitations: (1) we only derive formulas and perform the asymptotics for two points and (2) we can only carry out the framework for a small (but non-trivial) range of parameters. In this sense, the present paper is a proof of concept -- that one can use the method of the Macdonald difference operators to obtain multi-point convergence to the Airy process. In the future we hope to extend our framework to arbitrary parameters and number of points. In addition, we hope that the approach we develop can be extended to other integrable models in the KPZ universality class (for example by utilizing the limit of the stochastic six-vertex model to the ASEP, which was probably already known to \cite{Gwa}, was observed in \cite{BCG14} and proved in significant generality in \cite{Agg16}). 

Despite our result being the first of its kind for the stochastic six-vertex model, there exist previous (conditional/non-rigorous) works showing the two-point convergence of the KPZ equation started from narrow wedge initial data \cite{Dot13, Dot14, PSpohn, ISS13} and the log-gamma polymer \cite{NZ} to the Airy process. In Section \ref{Section9} we give a more detailed account of these previous works, and formally compare the techniques of the present paper with those in \cite{Dot13,PSpohn,ISS13} and \cite{NZ}. Specifically, in Section \ref{Section9.1} we give a formal explanation of how to obtain a prelimit formula using the Maconald difference operators and compare our result to the formulas obtained for the KPZ equation in \cite{Dot13,PSpohn,ISS13}. In Section \ref{Section9.2} we discuss some of the convergence issues of our formulas, and give a concise explanation of how we do the asymptotic analysis, comparing the approach to the one for the log-gamma polymer in \cite{NZ}. 

Here we mention that the works addressing the KPZ equation come from the physics literature, are based on the Bethe ansatz replica technique, and at their core involve a non-rigorous moment expansion formula. As we explain in Section \ref{Section9.1}, the moment expansion formulas in \cite{Dot13,PSpohn,ISS13} can be seen as shadows of the rigorous $t$-moment formulas from the present paper; however, we are presently unable to conceptually match the Bethe ansatz framework in those papers with our difference operators approach. In \cite{NZ} the authors derive formulas for the joint Laplace transform of the partition function of the log-gamma polymer model at several locations using the geometric Robinson-Schensted-Knuth correspondence. Afterwards, using a result from \cite{BCR}, the authors are able to rewrite their formulas as a ``Fredholm determinant''-like series. These series term-wise converge to a corresponding Fredholm determinant expansion for the Airy process, and the essential ingredient missing and making the proof conditional is an estimate on the growth of the terms in the series that would allow one to exchange the order of the sum and the limit.

The reason the authors of \cite{NZ} were unable to obtain a suitable bound on their Fredholm-like series comes from the presence of certain ``cross terms'' in the formulas that in a sense reflect the correlation of the log-gamma polymer partition functions at two locations. Part of the progress made in our paper is the ability to control similar cross terms (at least for some small range of parameters) and obtain the necessary bounds on the analogous Fredholm-like series that we derive for our model, see also Section \ref{Section9.2.2}. Nevertheless, we want to emphasize the importance of \cite{NZ} from which the present paper has greatly benefited. Indeed, as mentioned earlier for the log-gamma polymer (as is the case for all known positive temperature integrable models) the determinantal structure is lost and despite having multi-point observables it was a significant challenge to obtain any formula that would converge to the joint cdf for the Airy process. One of the many remarkable contributions of \cite{NZ} is finding a way to rewrite the joint Laplace transform for the log-gamma polymer through ingenious and highly non-trivial manipulations in a form that mimics the Fredholm determinant structure of the limit. In developing the results of the present paper, we have frequently drawn inspiration from \cite{NZ} and many of our own formulas can be seen as discrete analogues of those in \cite{NZ}.

Since this paper was completed there have been two important developments on the problem of multi-point convergence for positive temperature models in the KPZ universality class. The first is the paper \cite{QuaSar}, which establishes finite dimensional convergence of the ASEP height function to the Airy process, and the second is the paper \cite{virag2020heat}, which proposes a framework for proving finite dimensional convergence for directed polymer models. Surprisingly enough, these two and the present paper all develop very different approaches to the study of multi-point limits and apply to different classes of models. Namely, the work \cite{QuaSar} works well for exclusion processes that can be appropriately coupled to TASEP, the work \cite{virag2020heat} is suitable for directed polymer models and our present work focuses on vertex models and Macdonald processes. Consequently, despite all three works addressing the question of multi-point convergence for positive temperature models in the KPZ universality class, they have little in common in terms of scope of results and methodology.

The remainder of the introduction is structured as follows. In Section \ref{Section1.2} we define the stochastic six-vertex model on a quadrant and present our main result as Theorem \ref{thmMain}. In Section \ref{Section1.3} we give an outline of the paper and our approach. We also eagerly recommend Section \ref{Section9} to readers who are interested in a more accessible general exposition of our arguments. 

%
\subsection{Main result}\label{Section1.2} In this section we give the definition of the homogeneous stochastic six-vertex model in a quadrant, considered in \cite{Gwa,BCG14,BP16}, and state the main result we prove about it. There are several (equivalent) ways to define the model and we follow \cite[Section 1.1.2]{Agg16B}. 

A {\em six-vertex directed path ensemble} is a family of up-right directed paths drawn in the first quadrant $\mathbb{Z}^2_{\geq 1}$ of the square lattice, such that all the paths start from a left-to-right arrow entering each of the points $\{(1,m): m \geq 1 \}$ on the left boundary (no path enters from the bottom boundary) and no two paths share any horizontal or vertical edge (but common vertices are allowed); see Figure \ref{S1_1}. In particular, each vertex has six possible {\em arrow configurations}, presented in Figure \ref{S1_2}. 
\begin{figure}[h]
\centering
\scalebox{0.6}{\includegraphics{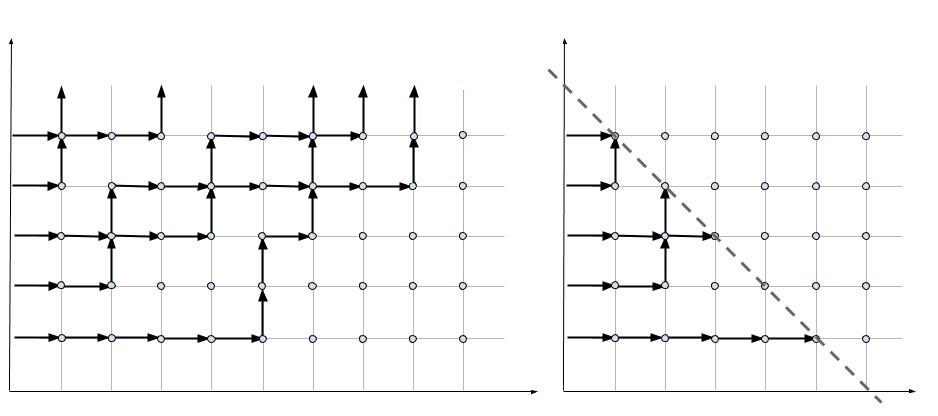}}
\caption{The left picture shows an example of a six-vertex directed path ensemble. The right picture shows an element in $P_n$ for $n = 6$. The vertices on the dashed line belong to $D_n$ and are given half of an arrow configuration if a directed path ensemble from $P_n$ is drawn.   }
\label{S1_1}
\end{figure}

\begin{figure}[h]
\centering
\scalebox{0.5}{\includegraphics{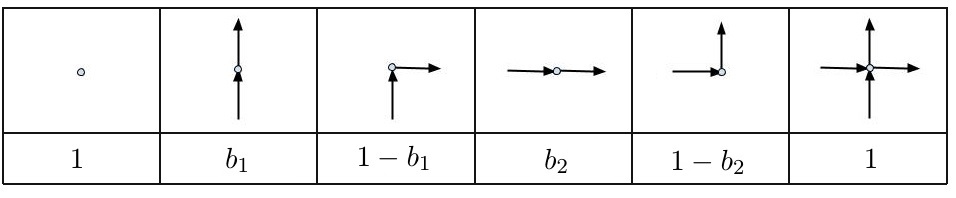}}
\caption{The top row shows the six possible arrow configurations at a vertex. The bottom row shows the probabilities of top-right completion, given the bottom-left half of a configuration. }
\label{S1_2}
\end{figure}

The stochastic six-vertex model is a probability distribution $\mathcal{P}(b_1,b_2)$ on six-vertex directed path ensembles, which depends on two parameters $b_1, b_2$ such that
$0 < b_1,  b_2 < 1$. It is defined as the $n \rightarrow \infty$ limit of a sequence of probability measures $\mathcal{P}_n$, which are constructed as follows.

For $n \geq 1$ we consider the triangular regions $T_n = \{ (x,y) \in \mathbb{Z}_{\geq 1}^2: x + y \leq n\}$ and let $P_n$ denote the set of six-vertex directed path ensembles whose vertices are all contained in $T_n$. By convention, the set $P_1$ consists of a single empty ensemble. We construct a consistent family of probability distributions $\mathcal{P}_n$ on $P_n$ (in the sense that the restriction of a random element sampled from $\mathcal{P}_{n+1}$ to $T_{n}$ has law $\mathcal{P}_n$) by induction on $n$, starting from $\mathcal{P}_1$, which is just the delta mass at the single element in $P_1$. 

For any $n \in \mathbb{N}$ we define $\mathcal{P}_{n+1}$ from $\mathcal{P}_n$ in the following Markovian way. Start by sampling a directed path ensemble $\mathcal{E}_n$ on $T_n$ according to $\mathcal{P}_n$. This gives arrow configurations (as in Figure \ref{S1_2}) to all vertices in $T_{n-1}$. In addition, each vertex in $D_n = \{ (x,y) \in \mathbb{Z}^2_{\geq 1}: x+ y = n\}$  is given ``half'' of an arrow configuration, meaning that the arrows entering the vertex from the bottom or left are specified, but not those leaving from the top or right; see the right part of Figure \ref{S1_1}. 

To extend $\mathcal{E}_n$ to a path ensemble on $T_{n+1}$, we must ``complete'' the configurations, i.e. specify the top and right arrows, for the vertices on $D_n$. Any half-configuration at a vertex $(x,y)$ can be completed in at most two ways; selecting between these completions is done independently for each vertex in $D_n$ at random according to the probabilities given in the second row of Figure \ref{S1_2}. In this way we obtain a random ensemble $\mathcal{E}_{n+1}$ in $P_{n+1}$ and we denote its law by $\mathcal{P}_{n+1}$. One readily verifies that the distributions $\mathcal{P}_n$ are consistent and then we define $\mathcal{P} = \lim_{n \rightarrow \infty} \mathcal{P}_n$. \\

Given a six-vertex directed path ensemble on $\mathbb{Z}_{\geq 1}^2$, we define the {\em height function } $h(x,y)$ as the number of up-right paths, which intersect the 
horizontal line through $y$ at or to the right of $x$. Our main result states that under suitable rescaling the two-point distribution of the random height function $h$ converges to the two-point distribution of the {\em Airy process}, and in order to state it we need to define the latter object.

The Airy process $A(t)$ is a continuous process on $\mathbb{R}$, which was introduced in \cite{Spohn}. We define it here by its fiite-dimensional distribution functions. Given $\xi_1, \dots, \xi_m \in \mathbb{R}$ and $\tau_1 < \cdots < \tau_m$ in $\mathbb{R}$ we define $f$ on $\{\tau_1, \dots, \tau_m\} \times \mathbb{R}$ through
$$f(\tau_j, x) = {\bf 1}_{(\xi_j, \infty)}(x) \mbox{ for $j = 1,\dots, m$}.$$
Then 
\begin{equation}\label{FiniteDimAiry}
\mathbb{P} \left( A(\tau_1) \leq \xi_1, \dots, A(\tau_m) \leq \xi_m \right) = \det \left(I - fA f \right)_{L^2(\{\tau_1, \dots, \tau_m\} \times \mathbb{R})},
\end{equation}
where $A$ is the {\em extended Airy kernel}
\begin{equation}\label{AiryKernel}
A(\tau, \xi; \tau', \xi') = \begin{cases}\int_0^\infty e^{-\lambda (\tau - \tau') }Ai(\xi + \lambda) Ai(\xi' + \lambda) d\lambda \hspace{5mm} &\mbox{ if $\tau \geq \tau'$},\\ -\int_{-\infty}^0 e^{-\lambda (\tau - \tau') }Ai(\xi + \lambda) Ai(\xi' + \lambda) d\lambda \hspace{5mm} &\mbox{ if $\tau < \tau'$}.\end{cases}
\end{equation}
and $Ai(\cdot)$ is the Airy function. When $\tau = \tau'$ the extended Airy kernel reduces to the usual Airy kernel from \cite{TWAiry}. In (\ref{FiniteDimAiry}) the $L^2$ space is defined with respect to the product measure on $\{\tau_1, \dots, \tau_m\} \times \mathbb{R}$ coming from the counting measure on $\{\tau_1, \dots, \tau_m\} $ and the usual Lebesgue measure on $\mathbb{R}$. In \cite{JDPNG} it was shown that $fA f $ is a trace class operator on $L^2(\{\tau_1, \dots, \tau_m\} \times \mathbb{R})$, so that the determinant in (\ref{FiniteDimAiry}) is the usual Fredholm determinant of trace class operators, see \cite{Simon}. Numerically, the Fredholm determinant in (\ref{FiniteDimAiry}) is equal to
\begin{equation}\label{FDE}
 1+ \sum_{n = 1}^\infty \frac{(-1)^n}{n!} \sum_{i_1, \dots, i_n = 1}^m  \int_{\xi_{i_1}}^\infty \cdots \int_{\xi_{i_n}}^\infty \det \left[ A(\tau_{i_k}, x_k; \tau_{i_l}, x_l)  \right]_{k,l= 1}^n dx_n \cdots dx_1,
\end{equation}
where the latter sum converges absolutely as the Fredholm series expansion of a trace class operator.

The main result of the paper is as follows.
\begin{theorem}\label{thmMain}Let $0 < b_1 < b_2 < 1$ and put $a = \sqrt{\frac{1- b_2}{1-b_1}}$, $t = \frac{b_1}{b_2}$. There exist $a^*, t^* \in (0,1)$ such that that the following holds for any $0 < b_1 < b_2 < 1$ that satisfy $a \in (0, a^*]$, $t \in (0,t^*]$. Let $s_1, s_2 \in \mathbb{R}$ be such that $s_1 > s_2$. For $M \in \mathbb{N}$ sufficiently large so that $M + s_2 M^{2/3} \geq 1$ we define $n_1(M), n_2(M) \in \mathbb{N}$ through
\begin{equation}\label{S1ScaleN}
n_1 =  \lfloor M + s_1 M^{2/3} \rfloor \mbox{ and } n_2 = \lfloor M + s_2 M^{2/3} \rfloor.
\end{equation}
Let $h(x,y)$ be the height function of the stochastic six-vertex model distributed according to $\mathcal{P}(b_1,b_2)$. Define the random variables
$$\tilde{X}_M = \sigma_a^{-1}M^{-1/3}( - h(n_1+1, M)- (f_1-1)M - f_1' [M - n_1] - (1/2)f''_1 s_1^2 M^{1/3} )$$
$$\tilde{Y}_M= \sigma_a^{-1}M^{-1/3}( - h(n_2 + 1, M) - (f_1-1)M - f_1' [M - n_2] - (1/2)f''_1 s_2^2 M^{1/3} ),$$
where
\begin{equation}\label{S1eqnConst}
\sigma_a = \frac{a^{1/3} \left(1 - a \right)^{1/3} }{1 + a }, \hspace{3mm} f_1=   \frac{2a}{1 + a}, \hspace{3mm} f_1'=  \frac{a}{1+a}, \hspace{3mm}  f_1'' = \frac{-a}{2  (1 - a^2)}.
\end{equation}
Then we have that for any $x_1, x_2 \in \mathbb{R}$
\begin{equation}\label{limitSSV}
\lim_{M \rightarrow \infty} \mathbb{P} \left(\tilde{X}_M \leq x_1, \tilde{Y}_M \leq x_2 \right) = \mathbb{P} \left(A(\tau_1) \leq x_1, A(\tau_2) \leq x_2 \right),
\end{equation}
where $A(\cdot)$ is the Airy process from (\ref{FiniteDimAiry}) and $\tau_1 = \frac{s_1a^{1/3}}{2(1-a)^{2/3}}$, $\tau_2 = \frac{s_2a^{1/3}}{2(1-a)^{2/3}}$.
\end{theorem}
\begin{remark}\label{S1R1} The condition $b_1 < b_2$ is necessary to obtain the Airy process limit and if $b_1 \geq b_2$ a different asymptotic behavior is expected, cf. \cite{BCG14}. In the coordinates $(b_1, b_2)$, the range $0 < b_1 < b_2 < 1$ defines an open right triangle in $\mathbb{R}^2$ with vertices at $(0,0)$, $(0,1)$ and $(1,1)$. In simple words, Theorem \ref{thmMain} states that if $(b_1, b_2)$ is close to the point $(0,1)$ then the height function of the stochastic six-vertex model evaluated at two points near the $x = y$ line, properly shifted and scaled, converges jointly to the two-point distribution of the Airy process.  The assumption that $(b_1, b_2)$ is sufficiently close to $(0,1)$ is technical,  and we believe that the theorem should hold even if we remove it. We discuss this parameter limitation later in Remarks \ref{RemS2} and \ref{RemarkRest}. The proof of the Theorem \ref{thmMain} can be found in Section \ref{Section2.2}.

We choose to evaluate the height function at $n_i +1$ in the definition of $\tilde{X}_M,\tilde{Y}_M$ (and not $n_i$) in order to obtain slightly simpler formulas later in the text. Since $|h(n_i+1) - h(n_i)| \leq 1$ we see that shifting the argument of $h$ by finite quantities does not affect the limit because of the $M^{1/3}$ scaling. 
\end{remark}
\begin{remark} It is worth pointing out that removing the restriction on $a$ and $t$ is important for future applications. Specifically, there is a limit transition that takes the stochastic six-vertex model to the KPZ equation, we discuss this further in Section \ref{Section9}, and this limit involves sending $t \rightarrow 1-$. If one wants to extend the methods of this paper to the KPZ equation being able to handle any $t\in(0,1)$ in the discrete setting is a natural first step to overcome. 
\end{remark}

%
\subsection{Outline}\label{Section1.3} In this section we give a brief outline of the general approach we take to prove Theorem \ref{thmMain}. The discussion below will involve certain expressions that will be properly introduced in the main text, and which should be treated as black boxes for the purposes of the outline. 

The starting point of the proof of Theorem \ref{thmMain} is a remarkable distributional equality between the stochastic six-vertex model and the ascending Hall-Littlewood process, which is a measure on sequences of partitions $\Lambda = (\lambda(1), \dots, \lambda(N))$ that depends on two sets of parameters $X = (x_1, \dots, x_N)$ and $Y = (y_1, \dots, y_M)$ and a number $t\in [0,1)$. The probability of $\Lambda$ is given by
$$\mathbb{P}_{X,Y}(\Lambda) = \prod_{i = 1}^N \prod_{j = 1}^M \frac{1 - x_iy_j}{1 - t x_i y_j} \times \prod_{i = 1}^N P_{\lambda(i) / \lambda(i-1)} (x_i) \times Q_{\lambda(N)}(Y),$$
where $P_{\lambda/ \mu}$ and $Q_{\lambda}$ denote the (skew) Hall-Littlewood polynomials with parameter $t$, see \cite[Chapter 3]{Mac}. The ascending Hall-Littlewood process is defined in Section \ref{Section2.2}.

As a special case of \cite[Theorem 4.1]{BBW}, we have that if $x_i = y_j = a$ for all $i = 1, \dots, N$ and $j = 1, \dots, M$ then the following distributional equality holds
\begin{equation}\label{S1DEQ}
 \left(M - \lambda_1'(0),\dots,M - \lambda'_1(N) \right) \eqd \left(h(1,M), \dots, h(N+1,M) \right),
\end{equation}
where $\lambda'_1$ denotes the largest column of $\lambda$, $\lambda_1'(0) \equiv 0$ and $h$ is the height function of the stochastic six-vertex model distributed according to $\mathcal{P}(b_1, b_2)$ with $b_1/b_2 = t$ and $(1-b_2)/(1-b_1) = a^2$.
 
In view of (\ref{S1DEQ}) we see that Theorem \ref{thmMain} can be rephrased in terms of the ascending Hall-Littlewood process by simply replacing $h(n_i+1)$ with $M - \lambda_1'(n_i)$ everywhere (this is why we shifted the argument in $h$ by $1$, see Remark \ref{S1R1}). This restatement can be found as Theorem \ref{thmHLMain} in the main text. The benefit of recasting the problem in the setup of the Hall-Littlewood process is that we can apply the Macdonald difference operators to obtain joint observables for $\lambda_1'(n_1)$ and $\lambda_1'(n_2)$. We recall these operators in Section \ref{Section2.3} and eventually work with their affine shifted version that we denote by $\mathcal{D}_n$ (here $n$ means that the operator acts on the variables $x_1, \dots, x_n$). 

The operator $\mathcal{D}_n$ is an eigenoperator for the Hall-Littlewood polynomials $P_\lambda(x_1, \dots, x_n)$ with eigenvalue $t^{-\lambda_1'}$ and by utilizing this fact alone one can obtain the formula
$$\mathbb{E}_{X,Y}\left[ t^{-k \lambda_1'(n_1)} \right] =  \frac{\mathcal{D}^k_{n_1} \Pi(X;Y) }{\Pi(X;Y)} , \mbox{ where }\Pi(X;Y) =\prod_{i = 1}^N \prod_{j = 1}^M \frac{1 - t x_i y_j}{1 - x_iy_j}, $$
and we have written $\mathbb{E}_{X,Y}$ for the expectation with respect to $\mathbb{P}_{X,Y}$. The expression on the right side of the above $k$-th moment formula can be written as a $k$-fold {\em nested contour integral}, where the contours become larger as $k$ increases. In order to handle this problem of growing contours, we deform all of these contours to the same one. As one deforms all of these contours, certain poles are crossed that diminish the dimension of the integral, and the nested contour integral can be rewritten as a sum over residue subspaces of integrals over the same contour. This is an instance of the {\em nested contour integral ansatz}, which was investigated in great detail in \cite{BBC}. The result of applying this ansatz method is that
\begin{equation}\label{S1Moment1}
\mathbb{E}_{X,Y}\left[ t^{-k \lambda_1'(n_1)} \right]  = \sum_{\lambda \vdash k} \frac{1}{(2\pi \iota)^{\ell(\lambda)}}\int_{\gamma^{\ell(\lambda)}} \det \left[\frac{1}{z_j t^{-\lambda_i} - z_i} \right]_{i,j = 1}^{\ell(\lambda)} \prod_{i = 1}^{\ell(\lambda)} F_{n_1}(z_i,\lambda_i; X,Y) d\vec{z} .
\end{equation}
In (\ref{S1Moment1}) the sum is over partitions of $k$, which are denoted by $\lambda$ -- these are the labels of the residue subspaces that come from the contour integral ansatz. The contour $\gamma$ is a positively oriented, zero-centered circle that contains $x_1, \dots, x_{n_1}$ and excludes $y_1^{-1}, \dots, y_M^{-1}$, $\ell(\lambda)$ is the number of parts of $\lambda$, $\iota = \sqrt{-1}$, and the function $F_{n}(z,m; X,Y) $ is given by
$$ F_{n}(z,m; X,Y) = \prod_{j =1}^M \frac{1 - zy_j}{1 - zt^{m} y_j} \cdot \prod_{j = 1}^n \frac{1 - z^{-1} t^{-m} x_j}{1 - z^{-1} x_j}.$$

Starting from (\ref{S1Moment1}) one can use the generating series of the $t$-exponential function
\begin{equation}\label{S1tExpon}
e_t(u) = \frac{1}{((1-t)u;t)_\infty} = \sum_{k = 0}^\infty \frac{u^k(1-t)^{-k}}{k_t!},
\end{equation}
where $(a;t)_\infty = \prod_{m = 0}^\infty (1 -at^m)$ is the $t$-Pochhammer symbol and $k_t! = \frac{(1 - t)(1-t^2) \cdots (1-t^k)}{(1-t)^k}$, to obtain
\begin{equation}\label{S1Laplace1}
\begin{split}
&\mathbb{E}_{X,Y}\left[ \frac{1}{(u_1t^{-\lambda_1'(n_1)};t)_{\infty}} \right] = \sum_{k = 0}^\infty \frac{u_1^k(1-t)^{-k}}{k_t!} \mathbb{E}_{X,Y}\left[ t^{-k \lambda_1'(n_1)} \right]  = \\
& \sum_{k = 0}^\infty \sum_{\lambda \vdash k} \frac{u_1^k(1-t)^{-k}}{k_t!} \frac{1}{(2\pi \iota)^{\ell(\lambda)}}\int_{\gamma^{\ell(\lambda)}} \det \left[\frac{1}{z_j t^{-\lambda_i} - z_i} \right]_{i,j = 1}^{\ell(\lambda)} \prod_{i = 1}^{\ell(\lambda)} F_{n_1}(z_i,\lambda_i; X,Y) d\vec{z} .
\end{split}
\end{equation}
In the top row of (\ref{S1Laplace1}) the expectation is a certain discrete analogue of the Laplace transform, called the $t$-Laplace transform (the name comes from the connection between the observable and the $t$-exponential function). In the second line of (\ref{S1Laplace1}) one proceeds to symmetrize the expression in $\lambda_1, \dots, \lambda_{\ell(\lambda)}$ and rewrite the sums over $\lambda$ as contour integrals, by using the formal identity
$$\sum_{n = 1}^\infty u^n g(t^n) = \frac{1}{2\pi \iota} \int_{-\iota \infty + 1/2}^{\iota \infty + 1/2} \frac{\pi}{\sin(-\pi s)} (-u)^s g(t^s)ds,$$
which essentially follows from $Res_{s = n} \frac{\pi}{\sin(-\pi s)} = (-1)^{n+1} .$ The result of this operation is that 
\begin{equation}\label{S1Laplace2}
\begin{split}
&\mathbb{E}_{X,Y}\left[ \frac{1}{(u_1t^{-\lambda_1'(n_1)};t)_{\infty}} \right] = \sum_{N_1 = 0}^\infty \frac{1}{N_1!}  \frac{1}{(2\pi \iota)^{2N_1}}\int_{\gamma_1^{N_1}}\int_{\gamma_2^{N_1}} D(\vec{z}, \vec{w}) \cdot {G}_{n_1}(\vec{z}, \vec{w}; u_1, N_1) d\vec{w} d\vec{z},
\end{split}
\end{equation}
where $D(\vec{z}, \vec{w}) = \det \left[\frac{1}{z_i - w_j} \right]_{i,j = 1}^{N_1}$ is the Cauchy determinant and $\gamma_1, \gamma_2$ are two zero-centered circles with radii $r_1 > r_2$ respectively. We forgo stating what ${G}_{n_1}(\vec{z}, \vec{w}; u_1, N_1)$ is here, as it is a bit involved, but refer the interested reader to (\ref{Bu}) where the full formula is written.

We remark that in order to carry out the manipulations in (\ref{S1Laplace1}) and (\ref{S1Laplace2}) one needs to first restrict the parameters $X,Y,u_1$ to small neighborhoods of $0$ but then both sides in (\ref{S1Laplace2}) can be extended analytically to a general set of parameters. The above framework of deriving (\ref{S1Laplace2}) was carried out by the author in \cite{ED} and the formula is recalled in the main text as Lemma \ref{0thmoment}. We also mention that the right side of (\ref{S1Laplace2}) is in fact a Fredholm determinant, and equation (\ref{S1Laplace2}) was the starting point in \cite{ED}  for proving that $\lambda'_1(n_1)$ converges to the Tracy-Widom distribution.\\

In Section \ref{Section3} of the present paper we start from (\ref{S1Laplace2}) and essentially repeat the same steps above but with $\mathcal{D}_{n_2}$ instead of $\mathcal{D}_{n_1}$. There are some technical subtleties in carrying out the same framework, which will be discussed in Section \ref{Section3}, but the resulting formula has the form
 \begin{equation}\label{S1JointLaplace}
\begin{split}
&\mathbb{E}_{X,Y}\left[ \frac{1}{(u_1t^{-\lambda_1'(n_1)};t)_{\infty}} \frac{1}{(u_2t^{-\lambda_1'(n_2)};t)_{\infty}} \right] = \sum_{N_1 = 0}^\infty\sum_{N_2 = 0}^\infty \frac{1}{N_1! N_2!}  \frac{1}{(2\pi \iota)^{2N_1 + 2N_2}}\int_{\gamma_1^{N_1}}\int_{\gamma_2^{N_1}} \\
&\int_{\gamma_3^{N_2}}\int_{\gamma_4^{N_2}} D(\vec{z}, \vec{w})  {G}_{n_1}(\vec{z}, \vec{w}; u_1, N_1) \cdot D(\vec{\hat{z}}, \vec{\hat{w}})  {G}_{n_2}(\vec{\hat{z}}, \vec{\hat{w}}; u_2, N_2)  \cdot CT( \vec{z}, \vec{w}; \vec{\hat{z}}, \vec{\hat{w}}) d\vec{\hat{w}} d\vec{\hat{z}}d\vec{w} d\vec{z},
\end{split}
\end{equation}
where $\gamma_i$ are zero-centered circles with radii $r_1 > r_2 > r_3 > r_4$. As before $D(\vec{u}, \vec{v})$ stands for the Cauchy determinant and we forgo stating the exact formula for the $G$ functions, see Theorem \ref{PrelimitT} for an explicit expression of the right side in (\ref{S1JointLaplace}). 

Comparing (\ref{S1JointLaplace}) with (\ref{S1Laplace2}) we see that the integrand in (\ref{S1JointLaplace}) is a product of three terms, one corresponding to $\lambda_1'(n_1)$, one corresponding to $\lambda_1'(n_2)$, but also there is now a {\em cross term} $CT$, which explicitly is given by
\begin{equation}\label{S1CT}
CT = \prod_{i = 1}^{N_1} \prod_{j = 1}^{N_2} \frac{(\hat{z}_j z_i^{-1}; t)_\infty }{(\hat{w}_j z_i^{-1} ; t)_\infty }\frac{(\hat{w}_j w_i^{-1} ; t)_\infty }{(\hat{z}_j w_i^{-1}; t)_\infty },
\end{equation}
and in a sense reflects the correlation between $\lambda_1'(n_1)$ and $\lambda_1'(n_2)$. This cross term is pointwise of order $e^{c_t N_1 N_2}$, which makes the $N_1! N_2!$ in (\ref{S1JointLaplace}) insufficient to ensure the summability of the terms. Part of the technical work behind deriving (\ref{S1JointLaplace}) in Section \ref{Section3} is showing that the Cauchy determinants in (\ref{S1JointLaplace}) provide {\em some} decay which can offset the contribution of this cross term, but we can only accomplish this if the parameters $X,Y,t$ are in a small enough neighborhood of zero, see Remark \ref{RemarkRest} for details. This is one of the sources of the parameter restriction in Theorem \ref{thmMain}. We also refer the interested reader to Section \ref{Section9.2.2} for a concise discussion on how we handle the cross term in our analysis.

Once formula (\ref{S1JointLaplace}) is established, we set all $X$ and $Y$ parameters to be equal to the same number $a$ and take the limit as $M \rightarrow \infty$. Showing that each summand on the right of (\ref{S1JointLaplace}) converges as $M\rightarrow \infty$ is an essentially straightforward application of the steepest descent argument. The exact limit statement is given as Proposition \ref{PropTermConv} and proved in Section \ref{Section5} in the main text. In order to show that the limit of the sum in (\ref{S1JointLaplace}) is equal to the sum of the limits, we require uniform in $M$ estimates on the growth (in terms of $N_1, N_2$) of the summands in (\ref{S1JointLaplace}). The order of growth we can establish is given as Proposition \ref{PropTermBound} and proved in Section \ref{Section6}. We remark that we can only find a uniform in $M$  bound on the growth of the summands if $a$ and $t$ are sufficiently close to zero -- this is the other source of the parameter restriction in Theorem \ref{thmMain}. Once Propositions \ref{PropTermConv} and Proposition \ref{PropTermBound} are proved, we know that the right side (\ref{S1JointLaplace}) converges to a certain double infinite series, which in Proposition \ref{PropTermLimit} is identified with the Fredholm determinant expansion of the two-point joint cdf of the Airy process. The weak joint convergence of $\lambda_1'(n_1), \lambda_1'(n_2)$ to the Airy process is an easy consequence of Propositions \ref{PropTermConv}, \ref{PropTermBound} and \ref{PropTermLimit}, and the argument is the content of Section \ref{Section4.2}.

%
\subsection*{Acknowledgments}\label{Section1.5} The author would like to thank Alexei Borodin, Guillaume Barraquand and Ivan Corwin for useful comments on earlier drafts of this paper as well as Amol Aggarwal for stimulating conversations about computing $L^2$ norms of Cauchy determinants. The author is partially supported by the Minerva Foundation Fellowship.

%
%
\section{The ascending Hall-Littlewood process}\label{Section2}
In Section \ref{Section2.1} we introduce some terminology related to partitions, plane partitions and Hall-Littlewood symmetric functions. In Section \ref{Section2.2} we define the ascending Hall-Littlewood process, state the main result we prove about it as Theorem \ref{thmHLMain} and use the latter to prove Theorem \ref{thmMain}. In Section \ref{Section2.3} we introduce the Hall-Littlewood difference operators, which are the main algebraic tool in our arguments, and derive some of their properties.

%
%
\subsection{Definitions and notation}\label{Section2.1}

We start by fixing terminology and notation following \cite{Mac}. A {\em partition} is a sequence $\lambda = (\lambda_1, \lambda_2,\cdots)$ of non-negative integers such that $\lambda_1 \geq \lambda_2 \geq \cdots$ and all but finitely many elements are zero. We denote the set of all partitions by $\mathbb{Y}$. The {\em length} $\ell (\lambda)$ is the number of non-zero $\lambda_i$ and the {\em weight} is given by $|\lambda| = \lambda_1 + \lambda_2 + \cdots$ . If $|\lambda| = n$ we say that $\lambda$ {\em partitions} $n$, also denoted by $\lambda \vdash n$. There is a single partition of $0$, which we denote by $\varnothing$. An alternative representation is given by $\lambda = 1^{m_1}2^{m_2}\cdots$, where $m_j(\lambda) = |\{i \in \mathbb{N}: \lambda_i = j\}|$ is called the {\em multiplicity} of $j$ in the partition $\lambda$. There is a natural ordering on the space of partitions, called the {\em reverse lexicographic order}, which is given by
$$\lambda > \mu \iff \exists k \in \mathbb{N} \mbox{ such that } \lambda_i = \mu_i \mbox{, whenever }i < k \mbox{ and } \lambda_k > \mu_k.$$
A {Young diagram} is a graphical representation of a partition $\lambda$, with $\lambda_1$ left justified boxes in the top row, $\lambda_2$ in the second row and so on. In general, we do not distinguish between a partition $\lambda$ and the Young diagram representing it. The {\em conjugate} of a partition $\lambda$ is the partition $\lambda'$ whose Young diagram is the transpose of the diagram $\lambda$. In particular, we have the formula $\lambda_i' = |\{j \in \mathbb{N}: \lambda_j \geq i\}|$.

Given two diagrams $\lambda$ and $\mu$ such that $\mu \subset \lambda$ (as a collection of boxes), we call the difference $\theta = \lambda - \mu$ a {\em skew Young diagram}. A skew Young diagram $\theta$ is a {\em horizontal $m$-strip} if $\theta$ contains $m$ boxes and no two lie in the same column. If $\lambda - \mu$ is a horizontal strip we write $\lambda \succeq \mu$. Some of these concepts are illustrated in Figure \ref{S2_1}.
\begin{figure}[h]
\centering
\scalebox{0.45}{\includegraphics{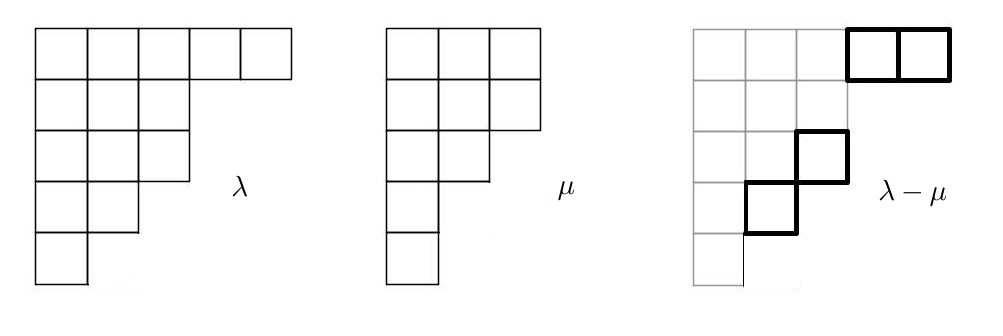}}
\caption{The Young diagram $\lambda = (5,3,3,2,1)$ and its transpose (not shown) $\lambda' = (5,4,3,1,1)$. The length $\ell(\lambda) = 5$ and weight $|\lambda| = 14$. The Young diagram $\mu = (3,3,2,1,1)$ is such that $\mu \subset \lambda$. The skew Young diagram $\lambda - \mu$ is shown in {\em black bold lines} and is a horizontal $4$-strip.}
\label{S2_1}
\end{figure}

A {\em plane partition} is a two-dimensional array of non-negative integers
$$\pi = (\pi_{i,j}), \hspace{3mm} i,j = 0,1,2,\dots,$$
such that $\pi_{i,j} \geq \max (\pi_{i,j+1}, \pi_{i+1,j})$ for all $i,j \geq 0$ and the {\em volume} $|\pi| = \sum_{i,j \geq 0} \pi_{i,j}$ is finite. Alternatively, a plane partition is a Young diagram filled with positive integers that form non-increasing rows and columns. A graphical representation of a plane partition $\pi$ is given by a {\em $3$-dimensional Young diagram}, which can be viewed as the plot of the function 
$$(x,y) \rightarrow \pi_{\lfloor x \rfloor, \lfloor y \rfloor} \hspace{3mm} x,y > 0.$$
Given a plane partition $\pi$ we consider its diagonal slices $\lambda^t$ for $t\in \mathbb{Z}$, i.e. the sequences
$$\lambda^t = (\pi_{i, i + t}) \hspace{3mm} \mbox{ for } i \geq \max(0, -t).$$
One readily observes that $\lambda^t$ are partitions and satisfy the following interlacing property
$$\cdots \prec \lambda^{-2} \prec \lambda^{-1} \prec \lambda^0 \succ \lambda^1 \succ \lambda^2 \succ \cdots.$$
Conversely, any (terminating) sequence of partitions $\lambda^{t}$, satisfying the interlacing property, defines a partition $\pi$ in the obvious way. Concepts related to plane partitions are illustrated in Figure \ref{S2_2}.\\
\begin{figure}[h]
\centering
\scalebox{0.5}{\includegraphics{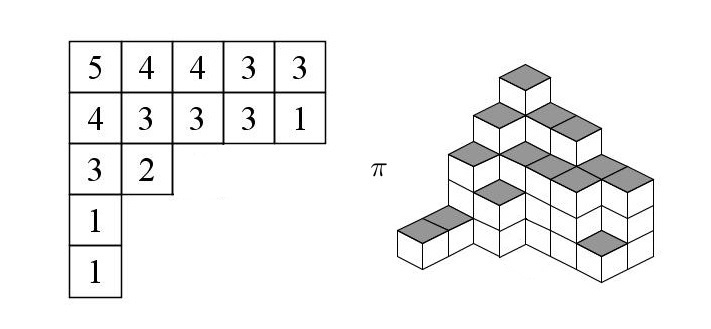}}
\caption{The plane partition $\pi = \varnothing \prec (1) \prec (1) \prec (3) \prec (4,2) \prec (5,3) \succ (4,3) \succ (4,3) \succ (3,1) \succ(3) \succ \varnothing$ . The volume $|\pi| = 40$}
\label{S2_2}
\end{figure}

We let $\Lambda_X$ denote the $\mathbb{Z}_{\geq 0}$ graded algebra over $\mathbb{C}$ of symmetric functions in variables $X = (x_1,x_2,\dots)$, which can be viewed as the algebra of symmetric polynomials in infinitely many variables with bounded degree, see e.g. \cite[Chapter I]{Mac} for general information on $\Lambda_X$. One way to view $\Lambda_X$ is as an algebra of polynomials in Newton power sums
$$p_k(X) = \sum_{i = 1}^{\infty} x_i^k , \hspace{3mm} \mbox{     for          } k\geq 1.$$
For any partition $\lambda$ we define
$$p_\lambda(X) = \prod_{i = 1}^{\ell(\lambda)}p_{\lambda_i}(X),$$
and note that $p_\lambda(X)$, $\lambda \in \mathbb{Y}$ form a linear basis in $\Lambda_X$.

In what follows we fix a parameter $t \in [0,1)$ and introduce the Hall-Littlewood symmetric functions $P_\lambda(X;t)$ with parameter $t$. Unless the dependence on $t$ is important we will suppress it from our notation, similarly for the variable set $X$.

One way to define the Hall-Littlewood symmetric functions is in terms of the following scalar product $\langle \cdot, \cdot \rangle$ on $\Lambda$ (see \cite[Chapter III.4]{Mac})
\begin{equation}\label{MSP}
\langle p_\lambda, p_\mu \rangle = \delta_{\lambda, \mu}\prod_{i = 1}^{\ell(\lambda)} (1 - t^{\lambda_i})^{-1}\prod_{i  =1}^{\lambda_1} i^{m_i(\lambda)}m_i(\lambda)!.
\end{equation}
\begin{definition}
 The Hall-Littlewood symmetric functions $P_\lambda$, $\lambda \in \mathbb{Y}$, are the unique linear basis of $\Lambda$ such that
\begin{enumerate}[label = \arabic{enumi}., leftmargin=1.5cm]
\item $\langle P_\lambda, P_\mu \rangle = 0$ unless $\lambda = \mu$.
\item The leading (with respect to reverse lexicographic order) monomial in $P_\lambda$ is $\prod_{i=1}^{\ell(\lambda)}x_i^{\lambda_i}.$
\end{enumerate}
\end{definition}

\begin{remark}
 $P_\lambda$ is a homogeneous symmetric function of degree $|\lambda|$.  
\end{remark}

\begin{remark} If we set $x_{N+1} = x_{N+2} = \cdots = 0$ in $P_\lambda(X;t)$, then we obtain the symmetric polynomials $P_\lambda(x_1,\dots,x_N;t)$ in $N$ variables, which are called the Hall-Littlewood polynomials.
\end{remark}

There is a second family of Hall-Littlewood symmetric functions $Q_\lambda$, $\lambda \in \mathbb{Y}$, which are dual to $P_\lambda$ with respect to the above scalar product:
$$Q_\lambda = \langle P_\lambda, P_\lambda \rangle^{-1}P_{\lambda}, \hspace{4mm} \langle P_\lambda,Q_\mu \rangle = \delta_{\lambda,\mu}, \hspace{2mm} \lambda, \mu \in \mathbb{Y}.$$

We next proceed to define the skew Hall-Littlewood functions (see \cite[Chapter III.5]{Mac} for details). Take two sets of variables $X = (x_1,x_2, \dots)$ and $Y = (y_1,y_2,\dots)$ and a symmetric function $f \in \Lambda$. Let $(X,Y)$ denote the union of sets of variables $X$ and $Y$. Then we can view $f(X,Y) \in \Lambda_{(X,Y)}$ as a symmetric function in $x_i$ and $y_i$ together. More precisely, let
$$f = \sum_{\lambda \in \mathbb{Y}} C_{\lambda} p_\lambda = \sum_{\lambda \in \mathbb{Y}} C_{\lambda}\prod_{i = 1}^{\ell(\lambda)}p_{\lambda_i},$$
be the expansion of $f$ into the basis $p_\lambda$ of power symmetric functions (in the above sum $C_\lambda = 0$ for all but finitely many $\lambda$). Then we have
$$f(X,Y) = \sum_{\lambda \in \mathbb{Y}} C_{\lambda} \prod_{i = 1}^{\ell(\lambda)}(p_{\lambda_i}(X) + p_{\lambda_i}(Y)).$$
In particular, we see that $f(X,Y)$ is the sum of products of symmetric functions of $x_i$ and symmetric functions of $y_i$. The skew Hall-Littlewood functions $P_{\lambda / \mu}$, $Q_{\lambda/\mu}$ are defined as the coefficients in the expansion
\begin{equation}
P_{\lambda}(X, Y) = \sum_{\mu \in \mathbb{Y}}P_{\mu}(X) P_{\lambda/\mu}(Y) \hspace{2mm} \mbox{ and } \hspace{2mm} Q_{\lambda}(X, Y) = \sum_{\mu \in \mathbb{Y}}Q_{\mu}(X) Q_{\lambda/\mu}(Y) 
\end{equation}

\begin{remark}
The skew Hall-Littlewood function $P_{\lambda/ \mu}$ is $0$ unless $\mu \subset \lambda$, in which case it is homogeneous of degree $|\lambda| - |\mu|$.  
\end{remark}

\begin{remark}
 When $\lambda = \mu$, $P_{\lambda/ \mu} = 1$ and if $\mu = \varnothing$ (the unique partition of $0$), then $P_{\lambda / \mu} = P_\lambda.$
\end{remark}

We mention here an important special case of the skew Hall-Littlewood symmetric function. Suppose $0 = x_2 = x_3 = \cdots $. Then we have
\begin{equation}\label{skewOne}
P_{\lambda / \mu}(x_1;t) = \psi_{\lambda/\mu}(t) \cdot x_1^{|\lambda| - |\mu|} \hspace{2mm} \mbox{ and }Q_{\lambda / \mu}(x_1;t) = \phi_{\lambda/\mu}(t) \cdot x_1^{|\lambda| - |\mu|}.
\end{equation}
The coefficients $\phi_{\lambda / \mu}$ and $\psi_{\lambda /\mu}$ have exact formulas as is shown in \cite[Chapter III, (5.8) and (5.8')]{Mac}:
\begin{equation}\label{phipsiForm}
\phi_{\lambda / \mu}(t) = {\bf 1}_{\lambda \succeq \mu} \cdot  \prod_{i \in I}(1 - t^{m_i(\lambda)}) \mbox{, and } \psi_{\lambda / \mu}(t) = {\bf 1}_{\lambda \succeq \mu} \cdot  \prod_{j \in J} (1 - t^{m_j(\mu)}),
\end{equation}
where if we set $\theta = \lambda - \mu$ we have that $I$ denotes the set of integers $i \geq 1$ such that $\theta_i'  > \theta_{i+1}' $; while $J$ is the set of integers $j \geq 1$ such that $\theta_j'  < \theta_{j+1}' $.

Fix $N, M \in \mathbb{N}$ and two sets of variables $X = (x_1, x_2,\dots, x_N)$ and $Y = (y_1, y_2,\dots, y_M)$. Using (\ref{skewOne}) we have
\begin{equation}\label{HLSkewExpand}
P_\lambda(X;t) = \hspace{-3mm} \sum_{\varnothing = \lambda^0 \preceq \lambda^1 \preceq \cdots \preceq \lambda^N} \prod_{i = 1}^N P_{\lambda^i/ \lambda^{i-1}}(x_i;t)  \mbox{ and } Q_\lambda(Y;t) =\hspace{-3mm} \sum_{\varnothing = \lambda^0 \preceq \lambda^1 \preceq \cdots \preceq \lambda^M} \prod_{i = 1}^M Q_{\lambda^i/ \lambda^{i-1}}(y_i;t).
\end{equation}

From \cite[Chapter III.4, (4.4)]{Mac} we have
\begin{equation}\label{CauchyConv}
\prod_{i = 1}^N \prod_{j = 1}^M\frac{1 -t x_i y_j}{1 - x_i y_j} =:\Pi(X;Y) = \sum_{\lambda \in \mathbb{Y}} P_{\lambda}(X) Q_{\lambda}(Y),
\end{equation}
where the equality a priori holds as an identity of formal power series in the variables $X$, $Y$ and it is known as the Cauchy identity. If $x_i \in [0, 1)$ for $i = 1, \dots, N$ and $y_j \in [0,1)$ for $j = 1,\dots, M$ then as shown in e.g. \cite[Chapter 2]{BorCor} we have that the right side of (\ref{CauchyConv}) converges absolutely and (\ref{CauchyConv}) is a numeric identity. Using (\ref{HLSkewExpand}) and the fact that $t \in [0,1)$ we have that 
$$|P_{\lambda}(x_1, x_2, \dots, x_N;t)| = P_{\lambda}(|x_1|, \dots, |x_N|) \mbox{ and }  |Q_{\lambda}(y_1, y_2, \dots, y_M;t)| = Q_{\lambda}(|y_1|, \dots, |y_M|),$$
from which we see that (\ref{CauchyConv}) is a numeric identity whenever $x_i, y_j \in \mathbb{D}$ (the unit disc in $\mathbb{C}$) for all $i =1, \dots ,N$ and $j = 1, \dots, M$.

%
%
\subsection{The ascending Hall-Littlewood process}\label{Section2.2} In this section we define the ascending Hall-Littlewood process and explain how it arises in a certain random plane partition model, first studied in \cite{Vul,Vul2}. Afterwards we state the main asymptotic result we prove about the ascending Hall-Littlewood process as Theorem \ref{thmHLMain} and use it to prove Theorem \ref{thmMain} from Section \ref{Section1.2}.

\begin{definition}\label{AHLP}
Let $N, M \in \mathbb{N}$ and $t \in [0,1)$ be given. Suppose that $X = (x_1, \dots, x_N)$ and $Y = (y_1, \dots, y_M)$ are sets of variables such that $x_i \in [0,1)$ for $i = 1, \dots, N$ and $y_j \in [0,1)$ for $j = 1, \dots, M$. The {\em ascending Hall-Littlewood process} is a probability measure on the the collection of lists of $N$ Young diagrams $\Lambda = (\lambda(1), \dots, \lambda(N))$ such that $\varnothing = \lambda(0) \preceq \lambda(1) \cdots \preceq \lambda(N)$. The probability of a given list is given by
\begin{equation}\label{ProbDef}
\mathbb{P}_{X,Y} (\Lambda) = \prod_{i = 1}^N \prod_{j = 1}^M \frac{1 - x_iy_j}{1 - t x_i y_j} \times \prod_{i = 1}^N P_{\lambda(i) / \lambda(i-1)} (x_i) \times Q_{\lambda(N)}(Y),
\end{equation}
where $P_{\lambda / \mu}$ and $Q_\lambda$ are the (skew) Hall-Littlewood polynomials from Section \ref{Section2.1}. We will write $\mathbb{E}_{X,Y}$ for the expectation with respect to $\mathbb{P}_{X,Y}$.
\end{definition}
\begin{remark}
The non-negativity of $\mathbb{P}_{X,Y}(\Lambda)$ follows from the non-negativity of $x_i, y_j$, the fact that $t \in [0,1)$ using (\ref{skewOne}) and (\ref{HLSkewExpand}). The fact that the sum over $\Lambda$ of $\mathbb{P}_{X,Y}(\Lambda)$ equals $1$ follows from (\ref{CauchyConv}) and the fact that $x_i, y_j \in [0,1)$. Thus $\mathbb{P}_{X,Y}$ is a well-defined probability measure.
\end{remark}
\begin{remark} The measure $\mathbb{P}_{X,Y}$ is a special case of the {\em ascending Macdonald process} from \cite{BorCor}, and corresponds to setting $q = 0$ in the Macdonald symmetric functions. 
\end{remark}
The ascending Hall-Littlewood process enjoys the property that the induced distribution of $(\lambda(1), \dots, \lambda(k))$ under $\mathbb{P}_{X,Y}$ for any $1 \leq k \leq N$ is also an ascending Hall-Littlewood process. The following lemma gives the precise statement.
\begin{lemma}\label{LprojectAHLP} Let $N, M \in \mathbb{N}$ and $t \in [0,1)$ be given. Suppose that $X = (x_1, \dots, x_N)$ and $Y = (y_1, \dots, y_M)$ are sets of variables such that $x_i \in [0,1)$ for $i = 1, \dots, N$ and $y_j \in [0,1)$ for $j = 1, \dots, M$. Let $\mathbb{P}_{X,Y}$ be as in Definition \ref{AHLP}. Then for any $1 \leq k \leq N$ and $\mu(1), \dots, \mu(k) \in \mathbb{Y}$
\begin{equation}\label{projectAHLP}
\mathbb{P}_{X,Y} (\lambda(1) = \mu(1), \dots, \lambda(k) = \mu(k)) = \mathbb{P}_{\tilde{X}, Y}  (\lambda(1) = \mu(1), \dots, \lambda(k) = \mu(k)),
\end{equation}
where $\tilde{X} = (x_1, \dots, x_k)$. 
\end{lemma}
\begin{proof} This is a general fact for Macdonald processes, of which $\mathbb{P}_{X,Y}$ is a special case -- see \cite[Proposition 3.6]{BorGorShak} and \cite[Section 2.2.2]{BorCor}.
\end{proof}

Our study of $\mathbb{P}_{X,Y}$ goes through finding a set of observables for this measure for arbitrary values of $X, Y$ as in Definition \ref{AHLP} and then obtaining the limit of those observables as $N, M$ tend to infinity. When we go to the asymptotic analysis we specialize all the $X$ and $Y$ variables to be equal to the same number $a \in [0,1)$, and we refer to the latter measure as the {\em homogeneous ascending Hall-Littlewood process} (HAHP). We denote the corresponding measure by $\mathbb{P}^{N,M}_{a,t}$ and write $\mathbb{E}_{a,t}^{N,M}$ for the expectation with respect to this measure. Our next task is to explain how the measure $\mathbb{P}^{N,M}_{a,t}$ arises in a certain random plane partition model.\\

\begin{figure}[h]
\centering
\scalebox{0.6}{\includegraphics{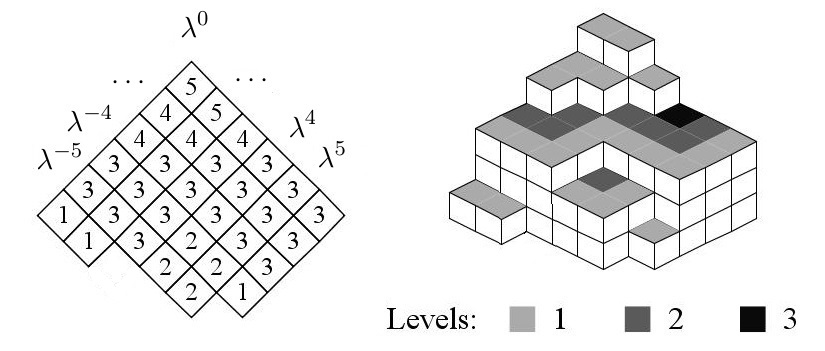}}
\caption{The figure represents a plane partition $\pi$. The left part indicates the corresponding sequence of interlacing partitions $\lambda^{-5} = (3)$, $\lambda^{-4} = (3,1)$, $\lambda^{-3}= (3,3)$ and so on. In this example $N = M = 6$. The right part of the figure shows the corresponding 3d Young diagram. \\
For the above diagram we have $diag(\pi) = 5 + 4 + 3 +2 + 2 = 16$.\\
To find $A_{\pi}(t)$ we do the coloring in the right part of the figure. Each cell gets a level, which measures the distance of the cell to the boundary of the terrace on which it lies. We consider connected components (formed by cells of the same level that share a side) and for each one we have a factor $(1 - t^i)$, where $i$ is the level of the cells in the component. The product of all these factors is $A_{\pi}(t)$. For the example above we have $7$ components of level $1$, $3$ of level $2$ and one of level $3$ -- thus $A_{\pi}(t)= (1-t)^7 (1-t^2)^3(1-t^3)$.
}
\label{S2_3}
\end{figure}
The model we describe next is a probability distribution on plane partitions $\pi$, which depends on parameters $a, t \in [0,1)$ and $N, M \in \mathbb{N}$. Given a plane partition $\pi$, we define its weight by
\begin{equation}\label{Apit}
W(\pi) = \begin{cases} A_{\pi}(t) \times a^{2 \cdot diag(\pi)} & \mbox{ if } \pi_{N, 0} = \pi_{0, M} =0 \\ 0 & \mbox{ otherwise}. \end{cases}
\end{equation}
In particular, the weight of a plane partition is zero unless its base is contained in the $N \times M$ rectangle. In view of the notation from Section \ref{Section2.1} any such plane partition $\pi$ can be expressed as a sequence of $N+M -1$ interlacing partitions
$$\varnothing \preceq \lambda^{-N+1} \preceq \cdots \preceq \lambda^0 \succeq  \cdots \succeq \lambda^{M-1} \succeq \varnothing.$$
In (\ref{Apit}) the notation $diag(\pi)$ denotes the sum of the entries on the main diagonal of $\pi$ (alternatively this is the sum of the parts of $\lambda^0$ or the number of cubes on the diagonal $x= y$ in the 3d Young diagram). Then $a^{2 \cdot diag(\pi)}$ in (\ref{Apit}) is a volume term, which penalizes partitions that are big. The function $ A_{\pi}(t)$ is a simple polynomial in $t$ and depends on the geometry of $\pi$. It is described in the caption to Figure \ref{S2_3} (see also \cite[Section 1]{ED} for a more detailed explanation). With the above notation, we have that the probability $\mathbb{P}(\pi)$ of a plane partition is given by the weight $W(\pi)$, divided by the sum of the weights of all plane partitions. 

Let us denote $\lambda(i) = \lambda^{i - N}$ for $i = 1,\dots,N$. Then the probability distribution induced from the weights (\ref{Apit}) and projected to the first $N$ terms $\varnothing \preceq \lambda(1) \preceq \lambda(2) \preceq \cdots \prec \lambda(N)$ is precisely the HAHP, i.e. the measure $\mathbb{P}^{N,M}_{a,t}$. For a brief proof of this fact we refer the reader to \cite[Section 2.1]{CD}. 

The above geometric interpretation of $\mathbb{P}^{N,M}_{a,t}$ dates back to \cite{Vul,Vul2} and is one of the initial motivations for studying this measure. We next state our main asymptotic result for $\mathbb{P}^{N,M}_{a,t}$.
\begin{theorem}\label{thmHLMain} There exist $a^*, t^* \in (0,1)$ such that that the following holds for any $a \in (0, a^*]$, $t \in (0,t^*]$. Let $s_1, s_2 \in \mathbb{R}$ be such that $s_1 > s_2$. For $M \in \mathbb{N}$ sufficiently large so that $M + s_2 M^{2/3} \geq 1$ we define $n_1(M), n_2(M) \in \mathbb{N}$ through
\begin{equation}\label{ScaleN}
n_1 = \lfloor M + s_1 M^{2/3} \rfloor \mbox{ and } n_2 = \lfloor M + s_2 M^{2/3} \rfloor,
\end{equation}
and suppose that $N(M) \geq n_1(M)$ is also an integer. Let $\mathbb{P}^{N,M}_{a,t}$ be the ascending Hall-Littlewood process from Definition \ref{AHLP} with all $x,y$ parameters equal to $a$. Define the random variables
$$X_M = \sigma_a^{-1}M^{-1/3}(\lambda_1'( n_1) - f_1M - f_1' [M - n_1] - (1/2)f''_1 s_1^2 M^{1/3} )$$
$$Y_M= \sigma_a^{-1}M^{-1/3}(\lambda_1'( n_2) - f_1M - f_1' [M - n_2] - (1/2)f''_1 s_2^2 M^{1/3} ),$$
where $(\lambda(1), \dots, \lambda(N))$ are $\mathbb{P}^{N,M}_{a,t}$-distributed and
\begin{equation}\label{S2eqnConst}
\sigma_a = \frac{a^{1/3} \left(1 - a \right)^{1/3} }{1 + a }, \hspace{3mm} f_1=   \frac{2a}{1 + a}, \hspace{3mm} f_1'=  \frac{a}{1+a}, \hspace{3mm}  f_1'' = \frac{-a}{2  (1 - a^2)}.
\end{equation}
Then we have that for any $x_1, x_2 \in \mathbb{R}$
\begin{equation}\label{limitAHLP}
\lim_{M \rightarrow \infty} \mathbb{P}^{N,M}_{a,t} \left(X_M \leq x_1, Y_M \leq x_2 \right) = \mathbb{P} \left(A(\tau_1) \leq x_1, A(\tau_2) \leq x_2 \right),
\end{equation}
where $A(\cdot)$ is the Airy process from (\ref{FiniteDimAiry}) and $\tau_1 = \frac{s_1a^{1/3}}{2(1-a)^{2/3}}$, $\tau_2 = \frac{s_2a^{1/3}}{2(1-a)^{2/3}}$.
\end{theorem}
\begin{remark}\label{RemS2} In terms of the plane partition model the sequence $\lambda_1'(1), \dots, \lambda'_1(N)$ traces out the base of the left part of the plane partition, which is a certain random curve of local slope $-1$ or $0$. The above theorem states that as long as $a$ and $t$ are small enough and positive, the two-point distribution of this random curve is asymptotically governed by the two-point distribution of the Airy process. The restriction on the parameters $a,t$ will be further discussed in Section \ref{Section3}, see Remark \ref{RemarkRest}. Here we mention that this restriction is two-fold. Even before taking the $M \rightarrow \infty$ limit, we derive a formula for a certain approximation for the probability $\mathbb{P}^{N,M}_{a,t} \left(X_M \leq x_1, Y_M \leq x_2 \right)$, which makes sense only for {\em some} choices of $a$ and $t$, in particular if they are small enough. This formula is the content of Theorem \ref{PrelimitT} in Section \ref{Section3}. Once we go to the $M \rightarrow \infty$ limit, some of our arguments require that $a$ and $t$ are also sufficiently small. We will discuss the latter point further in Section \ref{Section6}. While the condition that $a$ and $t$ are small enough appears to be technical, we do not know how to remove it at this time. 
\end{remark}

In the remainder of this section we explain the connection between the HAHP and the stochastic six-vertex model from Section \ref{Section1.2}, and deduce Theorem \ref{thmMain} from Theorem \ref{thmHLMain} above.

\begin{proof} (Theorem \ref{thmMain}) The key ingredient, which enables the reduction of Theorem \ref{thmMain} to Theorem \ref{thmHLMain} is the following distributional equality, from \cite{BBW}. Let us fix $a,t \in (0,1)$ and let $h(x,y)$ be the random height function of the stochastic six-vertex model $\mathcal{P}(b_1, b_2)$ from Section \ref{Section1.2} with parameters
\begin{equation}\label{S2ProbSV}
b_1= \frac{t (1- a^2)}{1 - a^2 t} \hspace{5mm} b_2= \frac{1- a^2}{1 - a^2 t}.
\end{equation}
Notice that the above equations imply that $b_1/b_2 = t$ and $(1-b_2)/(1-b_1) = a^2$.

In addition, suppose that $(\lambda_1(0), \dots, \lambda_1(N))$ is distributed according to $\mathbb{P}^{N,M}_{a,t}$. Then, as a special case of \cite[Theorem 4.1]{BBW}, we have the following distributional equality
$$ \left(M - \lambda_1'(0),\dots,M - \lambda'_1(N) \right) \eqd \left(h(1,M), \dots, h(N+1,M) \right), \mbox{ where by convention $\lambda_1'(0) = 0$.}$$
In particular, the random vector $(\tilde{X}_M, \tilde{Y}_M)$ from Theorem \ref{thmMain} has the same distribution as $(X_M, Y_M)$ from Theorem \ref{thmHLMain}. Thus Theorem \ref{thmHLMain} implies Theorem \ref{thmMain} with the same choice for $a^*, t^*$.
\end{proof}

%
%
\subsection{Hall-Littlewood difference operators}\label{Section2.3} In this section we recall the Hall-Littlewood difference operators, which are a special case of the Macdonald difference operators \cite[Chapter VI]{Mac}. Our discussion will follow \cite[Section 3]{ED}; however, we remark that many of the arguments and statements we write below can be traced back to \cite[Section 2.2.3]{BorCor}. A much more general framework of what we do below can be found in \cite{BorGorShak}. 

In what follows fix a natural number $N$ and $t \in (0,1)$ and consider the space of functions in $N$ variables $X = (x_1,\dots,x_N)$. Inside this space lies the space of symmetric polynomials $\Lambda^N_X$ in $N$ variables. It will be convenient to assume that $(x_1, \dots, x_N) \in [0,1)^N \cap \mathcal{W}_N$, where $\mathcal{W}_N = \{ (x_1, \dots, x_N) \in \mathbb{R}^N : x_i \neq x_j$ for $i,j  = 1, \dots, N$ and $i \neq j$ $\}$.

For $N \geq n \geq 1$ we let $D_n^1$ be the operator that acts on functions of $n$ variables $(x_1, \dots, x_n)$ as
$$D_n^1 :=\sum_{i = 1}^n \prod_{j \neq i}\frac{tx_i - x_j}{x_i - x_j} T_{0,x_i}, \mbox{ where } (T_{0,x_i}F)(x_1,\dots,x_n) = F(x_1,\dots,x_{i-1}, 0 , x_{i+1},\dots,x_n).$$
\begin{remark} \label{remarkMDO} $D_n^1$ is the $q \rightarrow 0$ limit of the first Macdonald difference operator, of the same letter (see Chapter VI of \cite{Mac}).
In particular, the $q \rightarrow 0$ limit of (4.15) in Chapter VI of \cite{Mac} shows that $D_n^1 P_\lambda(x_1, \dots, x_n;t) = \frac{1 - t^{n - \lambda_1'}}{1 - t} P_\lambda(x_1, \dots, x_n;t)$ and so $D_n^1$ is diagonalized by $P_\lambda(x_1, \dots, x_n;t)$.
\end{remark}

Set $\mathcal{D}_n :=  \frac{(t-1)D_n^1 + 1}{t^{n}}$ and observe that $\mathcal{D}_n$ satisfies the following properties:
\begin{enumerate}[label = \arabic{enumi}., leftmargin=1.5cm]
\item $\mathcal{D}_n$ is linear; 
\item   If $F_k$ converge pointwise to a function $F$ in $n$ variables, then $\mathcal{D}_nF_n$ converge pointwise to $\mathcal{D}_nF$ away from the set $\{ (x_1,\dots,x_n): x_i = x_j$ for some $i \neq j\}$;
\item $\mathcal{D}_nP_\lambda(x_1,\dots,x_n;t) = t^{-\lambda_1'}P_\lambda(x_1,\dots,x_n;t)$ (see Remark \ref{remarkMDO}).
\end{enumerate}
\begin{remark}
Since ultimately we will let $n \rightarrow \infty$, it is desirable to work with operators, whose eigenvalues do not depend on $n$. This explains our preference to work with $\mathcal{D}_n$ and not $D_n^1$.
\end{remark}

In the remainder of this section we summarize the results we will need about the operators $\mathcal{D}_n$ starting with the following proposition.
\begin{proposition}\label{babyContProp} \cite[Proposition 3.4]{Dim16}
Assume that $F(u_1,\dots,u_n) = f(u_1)\cdots f(u_n)$ with $f(0) = 1$. Take $(x_1, \dots, x_n) \in (0, \infty)^n \cap \mathcal{W}_n$ and assume $f(u)$ is holomorphic and non-zero in a complex neighborhood $D$ of an interval in $\mathbb{R}$ that contains $x_1,\dots,x_n$ and $0$. Then for any $k \geq 1$
\begin{equation}\label{toddlerContour}
 (\mathcal{D}_n^k F)(x_1,\dots,x_n) = \frac{F(x_1,\dots,x_n)}{(2\pi \iota )^k} \hspace{-1mm}\int_{C_{0,1}} \hspace{-4mm}\cdots \int_{C_{0,k}} \hspace{-1mm}\prod_{1 \leq a < b \leq k}\hspace{-1mm}  \frac{z_a - z_b}{z_a - z_bt^{-1}}  \prod_{i =1}^k \hspace{-1mm}\left[ \prod_{ j = 1}^n\hspace{-1mm} \frac{z_i - x_jt^{-1}}{z_i - x_j} \right]\hspace{-2mm} \frac{dz_i}{f(z_i) z_i} ,
\end{equation}
where $C_{0,a}$ are positively oriented simple contours encircling $x_1,\dots,x_n$ and $0$ and no zeros of $f(z)$. In addition, $C_{0,a}$ contains $t^{-1}C_{0,b}$ for $a < b$ and the region enclosed by $C_{0,1}$ is contained in $D$.
\end{proposition} 

\begin{proposition}\label{milestone1}
Fix integers $k \geq 0$, $M \geq 1$, and  $n$, $N$ with $N\geq n \geq 1$ and a parameter $t\in(0,1)$. Let $X = (x_1,\dots ,x_N)$ and $Y = (y_1,\dots ,y_M)$, with $ y_i \in [0,1)$ for $i = 1, \dots , M$ and $(x_1, \dots, x_N) \in [0,1)^N \cap \mathcal{W}_N$. We also let $\mathbb{P}_{X,Y}$ be as in Definition \ref{AHLP}. Suppose further that $f : \mathbb{Z}^n_{\geq 0}  \times  \mathbb{Z}^{n+1}_{\geq 0} \times \cdots \times \mathbb{Z}^{N}_{\geq 0} \rightarrow \mathbb{C}$ is a bounded function. Then if $\Pi(X;Y)$ is as in (\ref{CauchyConv}) we have that
\begin{equation}\label{MomentPreDeform}
 \Pi(X;Y)^{-1} \cdot \mathcal{D}^k_n \cdot \Pi(X;Y) \cdot \mathbb{E}_{X,Y}\left[ f(\lambda(n), \dots, \lambda(N)) \right] =  \mathbb{E}_{X,Y}\left[ t^{-k \lambda'_1(n)}f(\lambda(n), \dots, \lambda(N)) \right].
\end{equation}
\end{proposition}
\begin{proof}
Multiplying both sides by $\Pi(X;Y)$ and using the branching relations for Hall-Littlewood symmetric functions, see (\ref{HLSkewExpand}), we see that (\ref{MomentPreDeform}) is equivalent to
\begin{equation}\label{MomentPD1}
\begin{split}
&\mathcal{D}^k_n  \left[ \sum_{\lambda(n) \preceq \lambda(n+1) \preceq \cdots \preceq \lambda(N) } \hspace{-12mm} f(\lambda(n), \dots, \lambda(N))  P_{\lambda(n)} (x_1, \dots, x_n; t) \prod_{i = n}^{N-1}  \hspace{-1.5mm} P_{\lambda(i+1) / \lambda(i)}(x_i; t) Q_{\lambda(N)} (Y;t)  \right] = \\
&\sum_{\lambda(n) \preceq \lambda(n+1) \preceq \cdots \preceq \lambda(N) } \hspace{-12mm} t^{-k \lambda_1'(n)} f(\lambda(n), \dots, \lambda(N))  P_{\lambda(n)} (x_1, \dots, x_n; t) \prod_{i = n}^{N-1} \hspace{-1.5mm} P_{\lambda(i+1) / \lambda(i)}(x_i; t) Q_{\lambda(N)} (Y;t).
\end{split}
\end{equation}
We remark that both of the above sums are absolutely convergent as a consequence of the absolute convergence of (\ref{CauchyConv}), the assumed boundedness of $f$ and the fact that $t^{-k\lambda_1'(n)} \leq t^{-kN}$.
We prove (\ref{MomentPD1}) by induction on $k$ with base case $k = 0$ being tautologically true.

Suppose that (\ref{MomentPD1}) holds for $k$ and apply $\mathcal{D}_n$ to both sides to obtain
\begin{equation*}
\begin{split}
&\mathcal{D}^{k+1}_n  \left[ \sum_{\lambda(n) \preceq \lambda(n+1) \preceq \cdots \preceq \lambda(N) } \hspace{-12mm} f(\lambda(n), \dots, \lambda(N))  P_{\lambda(n)} (x_1, \dots, x_n; t) \prod_{i = n}^{N-1} P_{\lambda(i+1) / \lambda(i)}(x_i) Q_{\lambda(N)} (Y;t)  \right] = \\
&\mathcal{D}^{1}_n \left[\sum_{\lambda(n) \preceq \lambda(n+1) \preceq \cdots \preceq \lambda(N) } \hspace{-12mm} t^{-k \lambda_1'(n)} f(\lambda(n), \dots, \lambda(N))  P_{\lambda(n)} (x_1, \dots, x_n; t) \prod_{i = n}^{N-1} P_{\lambda(i+1) / \lambda(i)}(x_i) Q_{\lambda(N)} (Y;t).\right]
\end{split}
\end{equation*}
The absolute convergence of the second line above together with Properties $1.$ and $2.$ of $\mathcal{D}_n$ implies that we can exchange the order of $\mathcal{D}_n$ and the sum to get
\begin{equation*}
\begin{split}
&\mathcal{D}^{k+1}_n  \left[ \sum_{\lambda(n) \preceq \lambda(n+1) \preceq \cdots \preceq \lambda(N) } \hspace{-12mm} f(\lambda(n), \dots, \lambda(N))  P_{\lambda(n)} (x_1, \dots, x_n; t) \prod_{i = n}^{N-1} P_{\lambda(i+1) / \lambda(i)}(x_i) Q_{\lambda(N)} (Y;t)  \right] = \\
& \left[\sum_{\lambda(n) \preceq \lambda(n+1) \preceq \cdots \preceq \lambda(N) } \hspace{-12mm} t^{-k \lambda_1'(n)} f(\lambda(n), \dots, \lambda(N)) \mathcal{D}^{1}_n P_{\lambda(n)} (x_1, \dots, x_n; t) \prod_{i = n}^{N-1} P_{\lambda(i+1) / \lambda(i)}(x_i) Q_{\lambda(N)} (Y;t).\right].
\end{split}
\end{equation*}
Using Property $3.$ we have that $\mathcal{D}^{1}_n P_{\lambda(n)} (x_1, \dots, x_n; t) = t^{-\lambda_1'(n)} P_{\lambda(n)} (x_1, \dots, x_n; t) $ and putting this above we arrive at (\ref{MomentPD1}) for $k +1$. The general result now follows by induction.
\end{proof}

%
%
\section{Prelimit formula}\label{Section3} The purpose of this section is to derive a prelimit formula for the joint $t$-Laplace transform of the ascending Hall-Littlewood process. This formula is the starting point of our asymptotic analysis in Section \ref{Section4} and is given in Theorem \ref{PrelimitT} below. Theorem \ref{PrelimitT} is stated in Section \ref{Section3.1} and is proved in Section \ref{Section3.3} after some preliminary results are presented in Section \ref{Section3.2}.

%
%
\subsection{Result formulation}\label{Section3.1} The main result of this section is Theorem \ref{PrelimitT}. In order to formulate it we will require the definition of certain functions $S(w, z; u,t) $ (see Definition \ref{DefFunS}) and $H(u_1,u_2)$ (see Lemma \ref{S3BigSumAnal}). In addition, the assumptions of Theorem \ref{PrelimitT} require that a certain pair of parameters $a,t$ satisfy a technical assumption, detailed in Definition \ref{DefNiceRange}. We proceed to gradually define all the elements needed for formulating Theorem \ref{PrelimitT}.\\

We begin by introducing some useful notation. If $t \in (0,1)$ and $a \in \mathbb{C}$ we define the $t$-Pochhammer symbol $(a;t)_\infty := \prod_{n = 0}^\infty (1 - at^n).$
Observe that the above product converges for any $a \in \mathbb{C}$, is continuous in $a$ and non-zero if $a \neq t^{-n}$ for $n \in \mathbb{Z}, n \geq 0$. We continue to denote by $\iota$ the number $\sqrt{-1}$, as in Section \ref{Section1}. We will also use the notation $k_t! = \frac{(1 - t)(1-t^2) \cdots (1-t^k)}{(1-t)^k}$ for any $t \neq 1$. The following definition introduces the function $S(w, z; u,t)$ that appears many times in our formulas.
\begin{definition}\label{DefFunS}
Fix $t \in (0,1)$. Let $w,z,u \in \mathbb{C}$ be such that $zw\neq 0$, $|z| \neq t^n |w|$ for any $n \in \mathbb{Z}$ and $u \not \in [0, \infty)$. For such a set of parameters we define the function
\begin{equation}\label{SpiralDef}
S(w, z; u,t) =   \sum_{m \in \mathbb{Z}} \frac{\pi \cdot [ - u ]^{[\log w - \log z] [\log t]^{-1} -  2m \pi  \iota [\log t ]^{-1}}}{\sin(-\pi [[\log w - \log z] [\log t]^{-1} -  2m \pi  \iota [\log t ]^{-1}])},
\end{equation}
where everywhere we take the principal branch of the logarthm, i.e. if $v = re^{\iota \theta}$ with $r > 0$ and $\theta \in (-\pi, \pi]$ we set $\log v = \log r + \iota \theta$. Observe that 
$$Re \left[ -\pi [[\log w - \log z] [\log t]^{-1} -  2m \pi  \iota [\log t ]^{-1}]  \right] = -\pi \cdot \frac{\log|w| - \log |z|}{\log t} \not \in \pi \cdot \mathbb{Z},$$
which implies that each of the summands in (\ref{SpiralDef}) is well-defined and finite.

Observe that given $c > 0$ we can find $c' > 0$ such that if $x,y \in \mathbb{R}$ and $d(x, \mathbb{Z}) \geq c$ then 
\begin{equation}\label{BoundSine}
\frac{1}{| \sin (\pi x + \iota \pi y)|} \leq c' e^{-\pi |y|}.
\end{equation}
We can use the latter statement to show that the sum in (\ref{SpiralDef}) is absolutely convergent as we explain here. Put $A = [\log w - \log z] [\log t]^{-1}$, $B = - 2\pi [\log t]^{-1}$ and $-u = R e^{\iota \phi}$ with $\phi \in (-\pi, \pi)$. From our assumption that $|z| \neq t^n |w|$ for any $n \in \mathbb{Z}$ and (\ref{BoundSine}) we conclude that for any $m \in \mathbb{Z}$ 
$$\left| \frac{\pi \cdot [ - u ]^{[\log w - \log z] [\log t]^{-1} -  2m \pi  \iota [\log t ]^{-1}}}{\sin(-\pi [[\log w - \log z] [\log t]^{-1} -  2m \pi  \iota [\log t ]^{-1}])} \right| \leq  \pi c' \cdot e^{ |\phi| | A| + |\log R| |A| } e^{(|\phi| - \pi) B |m| }.  $$
The above shows that the sum in (\ref{SpiralDef}) is absolutely convergent by comparison with the geometric series $e^{(|\phi| - \pi) B |m|}.$
\end{definition}
\begin{remark}
In equation (\ref{SpiralDef}) we chose the principal branch of the logarithm for expressing $\log w$ and $\log z$. However, we could have chosen different branches for $\log w$ and $\log z$, and note that then $[\log w - \log z] [\log t]^{-1}$ would shift by $2k \pi \iota [\log t ]^{-1}$ for some $k \in \mathbb{Z}$. Since the sum in the definition of $S(w, z; u,t)$ is over $\mathbb{Z}$ we see that such a shift does not change the value of $S(w, z; u,t)$. So even though the logarithm is a multi-valued function for fixed $u,t$ the function $S(w, z; u,t) $ as a function of $z,w$ is single valued and well-defined as long as $|z| \neq t^{n} |w|$ for some $n \in \mathbb{Z}$.
\end{remark}

We make the following technical definition about a pair of parameters $a,t \in (0,1)$.
\begin{definition}\label{DefNiceRange} We say that the pair of parameters $a,t \in (0,1)$ is {\em good} if there exist constants $\rho \in (0,1)$ and $r_1, r_2, r_3, r_4 \in (a, a^{-1})$ such that $r_1 > r_2 > r_3 > r_4 > tr_1 $, and
\begin{equation}\label{RadiiCond}
\begin{split}
&  \max \left(\frac{\sqrt{r_2/r_1}}{(1 - r_2/r_1)}, \frac{\sqrt{r_4/r_3}}{(1 - r_4/r_3)} \right) \cdot \frac{(-r_3/r_1;t)_\infty (-r_4/r_2;t)_\infty}{(r_4/r_1;t)_\infty (r_3/r_2;t)_\infty} \leq \rho.
\end{split}
\end{equation}
\end{definition}
Before we go to the main result of the section we also state the following lemma, whose proof is given in Section \ref{Section7.1}.
\begin{lemma}\label{S3BigSumAnal}
Let $N, M, n \in \mathbb{N}$ be given with $N \geq n$. Suppose that $a,t \in (0,1)$ are good in the sense of Definition \ref{DefNiceRange} and let $r_1,r_2,r_3,r_4,\rho$ be as in that definition. Assume that $X = (x_1, \dots, x_N)$, $Y = (y_1, \dots, y_M)$ with $x_i, y_j \in \mathbb{C}$, $|x_i| \leq a$ and $|y_j| \leq a$ for $i = 1, \dots, N$ and $j = 1, \dots, M$.
Finally, let $K \subset \mathbb{C} \setminus [0, \infty)$ be a compact set. Then we can find a constant $C$ depending on $r_1, r_2, r_3, r_4,K,a, t, M, N, n$ such that
\begin{equation}\label{S3HFun}
|H(\vec{z}, \vec{w}, \vec{\hat{z}}, \vec{\hat{w}}; N_1, N_2; u_1, u_2)| \leq C^{N_1 + N_2} \cdot \left( \frac{1 + \rho}{2} \right)^{N_1^2 + N_2^2},
\end{equation}
where
\begin{equation}\label{S3HFun2}
\begin{split}
&H(\vec{z}, \vec{w}, \vec{\hat{z}}, \vec{\hat{w}}; N_1, N_2; u_1, u_2)=   \det \left[\frac{1}{z_i - w_j} \right]_{i,j = 1}^{N_1}  \det \left[ \frac{1}{\hat{z}_i - \hat{w}_j}\right]_{i,j = 1}^{N_2}  \\
&\prod_{i = 1}^{N_1}\frac{ S(w_i, z_i; u_1,t)}{[-\log t] \cdot w_i} \cdot \prod_{i = 1}^{N_2}\frac{ S(\hat{w}_i, \hat{z}_i; u_2,t)}{[-\log t] \cdot \hat{w}_i} \cdot \prod_{i = 1}^{N_1}\prod_{j = 1}^M \frac{1 - y_j z_i}{1 -y_j w_i} 
\prod_{i = 1}^{N_1}\prod_{j = 1}^N \frac{1 - w_i^{-1}x_j}{1 -  z_i^{-1}x_j}\prod_{i = 1}^{N_2} \prod_{j = 1}^M \frac{1 - \hat{z}_i y_j}{1 - \hat{w}_i y_j}  \times \\
&\prod_{i = 1}^{N_2}  \prod_{ j = 1}^{n} \frac{1 - \hat{w}^{-1}_ix_j}{1 - \hat{z}^{-1}_i x_j} \cdot  \prod_{i = 1}^{N_1} \prod_{j = 1}^{N_2} \frac{(\hat{z}_j z_i^{-1}; t)_\infty }{(\hat{w}_j z_i^{-1} ; t)_\infty }\frac{(\hat{w}_j w_i^{-1} ; t)_\infty }{(\hat{z}_j w_i^{-1}; t)_\infty }.
\end{split}
\end{equation}
In (\ref{S3HFun}) we have that $u_1, u_2 \in K$, and $|z_i| = r_1$, $|w_i| = r_2$, $|\hat{z}_j| = r_3$ and $|\hat{w}_j| = r_4$ for $i = 1, \dots, N_1$ and $j = 1, \dots, N_2$. In (\ref{S3HFun2}) the function $S$ is as in Definition \ref{DefFunS}.

 Moreover, if we fix $u_1 \in \mathbb{C} \setminus [0, \infty)$ then the function
\begin{equation}\label{S3HFunB}
H(u_1,u_2) = \hspace{-3mm}\sum_{N_1, N_2 = 0}^\infty \hspace{0mm} \int_{\gamma_1} \int_{\gamma_2} \int_{\gamma_3} \int_{\gamma_4}\hspace{-1mm} \frac{ H(\vec{z}, \vec{w}, \vec{\hat{z}}, \vec{\hat{w}}; N_1, N_2; u_1, u_2)  }{N_1! N_2!} \prod_{i =1}^{N_2} \frac{d\hat{w}_i}{2\pi \iota} \prod_{i = 1}^{N_2} \frac{d \hat{z}_i}{2\pi \iota}\prod_{i = 1}^{N_1}  \frac{dw_i}{2\pi \iota} \prod_{i =1}^{N_1}  \frac{d z_i}{2\pi \iota},
\end{equation}
is well-defined and analytic in $u_2 \in \mathbb{C} \setminus [0, \infty)$. In (\ref{S3HFunB}) we have that $\gamma_i$ is a positively oriented circle of radius $r_i$ for $i =1,2,3,4$ and if $N_1 = N_2 = 0$ then the summand equals $1$ by convention. 
\end{lemma}

We now turn to the main result of the section.
\begin{theorem}\label{PrelimitT} Make the same assumptions as in Lemma \ref{S3BigSumAnal}. Let $\mathbb{P}_{X,Y}$ be as in Definition \ref{AHLP}. If $u_1, u_2 \in \mathbb{C} \setminus [0, \infty)$ then we have 
\begin{equation}\label{PrelimitEq}
\mathbb{E}_{X,Y} \left[ \frac{1}{(u_1 t^{-\lambda'_1(N)};t)_{\infty}} \cdot \frac{1}{(u_2 t^{-\lambda'_1(n)};t)_{\infty}}  \right] = H(u_1,u_2),
\end{equation}
where $H(u_1,u_2)$ is as in (\ref{S3HFunB}).
\end{theorem}
\begin{remark}\label{tLaplace}
For a random variable $Z$ supported on $\mathbb{Z}$ and $u \in \mathbb{C} \setminus [0, \infty)$ the expression
$$\mathbb{E} \left[\frac{1}{(ut^{Z}; t)_{\infty}} \right]$$
 is sometimes called the {\em $t$-Laplace transform}, see \cite{BorCor}. Notice that by the definition of the $t$-Pochhammer symbol the expression $(ut^n; t)_{\infty}$ is uniformly bounded away from $0$ if $u \in \mathbb{C} \setminus [0, \infty)$ and $n \in \mathbb{Z}$ and so the expectation in (\ref{PrelimitEq}) is well-defined. We see that equation (\ref{PrelimitEq}) is providing us with a formula for the joint $t$-Laplace transform of $\lambda_1'(n)$ and $\lambda_1'(N)$. This formula will play the role of a starting point for proving our two-point asymptotic result from Theorem \ref{thmHLMain}.
\end{remark}

\begin{remark}\label{RemarkRest} Let us discuss the assumptions of Theorem \ref{PrelimitT}. The assumption $x_i, y_j \in [0,a]$ for some $a \in (0,1)$ is necessary to ensure that the measure $\mathbb{P}_{X,Y}$ is well-defined. In the process of deriving (\ref{PrelimitEq}) we will see that the critical properties we demand from the contours $\gamma_i$ for $i = 1, \dots, 4$ are:
\begin{enumerate}
\item $\gamma_1, \dots, \gamma_4$ are nested, i.e. $\gamma_{i+1}$ is contained in the region enclosed by $\gamma_i$ for $i = 1,2,3$;
\item $\gamma_3$ encircles the points $x_1, \dots, x_N$ and $0$;
\item $\gamma_2$ excludes the points $y_1^{-1}, \dots, y_M^{-1}$;
\item $\gamma_4$ encircles the point $0$;
\item $t \cdot \gamma_1$ is contained in the region enclosed by $\gamma_4$.
\end{enumerate} 
Having the above five conditions satisfied dictates our choice of $r_1, \dots, r_4$ satisfying $a^{-1} > r_1 > r_2 > r_3 > r_4 > a$ and $r_4/r_1 > t$. Such a choice of radii is possible for any $a,t \in (0,1)$ by simply picking the radii to all be very close to $1$. What this in particular implies is that each summand in (\ref{S3HFunB}) makes sense for any $a,t \in (0,1)$ and not just when this pair of parameters are good in the sense of Definition \ref{DefNiceRange}.

Let us explain why we demand that $a,t$ be good in Theorem \ref{PrelimitT}. The existence of $\rho \in (0,1)$ satisfying (\ref{RadiiCond}) is technical and has to do with the convergence, and hence well-posedness, of the function $H(u_1,u_2)$ in (\ref{S3HFunB}). Specifically, the mixed product
$$ \prod_{i = 1}^{N_1} \prod_{j = 1}^{N_2} \frac{(\hat{z}_j z_i^{-1}; t)_\infty }{(\hat{w}_j z_i^{-1} ; t)_\infty }\frac{(\hat{w}_j w_i^{-1} ; t)_\infty }{(\hat{z}_j w_i^{-1}; t)_\infty }$$
that appears in the last line of (\ref{S3HFun2}) is pointwise of order $e^{c_t N_1 N_2}$. This makes the $N_1! N_2!$ in (\ref{S3HFunB}) insufficient to ensure the summability of the terms. Part of our proof of Lemma \ref{S3BigSumAnal} is to utilize the fact that the Cauchy determinants, that appear in the first line of (\ref{S3HFun2}), provide {\em some} decay which can offset the contribution of this mixed product, but only if (\ref{RadiiCond}) holds, which is why we require it. In simple words, while each summand in (\ref{S3HFunB}) is well-defined for any $a,t \in (0,1)$, we can prove that the sum is convergent only when $a,t$ are good. 

A simple condition that ensures that $a,t$ are good is if these parameters are close enough to $0$.

We emphasize that while our requirement that $a,t$ are good is technical it appears to be crucial. For values of $a,t$ that are close to $1$ at this time we have no way to handle the contribution of the mixed product and make sense of the sum in (\ref{S3HFunB}). 
\end{remark}

We end this section by summarizing several statements about the function $S(w, z; u,t)$ from Definition \ref{DefFunS} in the following lemmas. The proof of these lemmas is given in Section \ref{Section7.1}.

\begin{lemma}\label{S3BoundOnS} Fix $t \in (0,1)$ and compact sets $K_1 \subset (t,1)$ and $K_2 \in \mathbb{C} \setminus [0, \infty)$. Then there exists a constant $M_0 > 0$ depending on $K_1, K_2, t$ such that if $z, w \in \mathbb{C}$ satisfy $|w| = r, |z| = R$ with $R > r > tR > 0$ and $r/R \in K_1$, and $u \in K_2$ then
\begin{equation}
\left| S(w, z; u,t) \right| \leq M_0, 
\end{equation}
where $S(w, z; u,t)$ is as in Definition \ref{DefFunS}.
\end{lemma}

\begin{lemma}\label{S3Analyticity} Fix $t \in (0,1)$ and $R, r \in (0, \infty)$ such that $R > r > tR$. Denote by $A(r,R) \subset \mathbb{C}$ the annulus of inner radius $r$ and outer radius $R$ that has been centered at the origin. Then the function $S(w, z; u,t)$ from Definition \ref{DefFunS} is well-defined for $(w,z,u) \in Y = \{ (x_1, x_2, x_3) \in  A(r,R) \times A(r,R) \times (\mathbb{C} \setminus [0, \infty)) : |x_1|< |x_2| \}$ and is jointly continuous in those variables (for fixed $t$) over $Y$. If we fix $u \in \mathbb{C} \setminus [0, \infty)$ and $w \in A(r,R)$ then as a function of $z$, $S(w, z; u,t)$ is analytic on $\{ \zeta \in A(r,R)  : |\zeta| > |w|\} $; analogously, if we fix $u \in \mathbb{C} \setminus [0, \infty)$ and $z \in A(r,R)$ then as a function of $w$, $S(w, z; u,t)$ is analytic on $\{ \zeta \in A(r,R)  : |\zeta| < |z|\} $. Finally, if we fix $w,z \in A(r,R)$ with $|z| > |w|$ then $S(w, z; u,t)$ is analytic in $\mathbb{C} \setminus [0, \infty)$ as a function of $u$.
\end{lemma}

%
%
\subsection{$k$-th moment formula}\label{Section3.2} The proof of  Theorem \ref{PrelimitT} is given in Section \ref{Section3.3}. In this section we establish the following important ingredient we require for the proof.
\begin{lemma}\label{kthmoment}
Let $M, N \in \mathbb{N}$, $a,t \in (0,1)$ and fix $x_1, \dots, x_N \in [0,a]$ and $y_1, \dots, y_M \in [0,a]$. Let $\mathbb{P}_{X,Y}$ be as in Definition \ref{AHLP} and $u_1 \in \mathbb{C} \setminus [0, \infty)$. Then for any $n, k \in \mathbb{N}$ with $ 1 \leq n \leq N$ one has
\begin{equation}\label{kthmomenteq}
\mathbb{E}_{X,Y}\left[ t^{-k\lambda_1'(n)}  \cdot \frac{1}{(u_1 t^{-\lambda'_1(N)};t)_{\infty}}   \right] = \sum_{N_1 = 0}^\infty  \sum_{\lambda \vdash k} \frac{(t^{-1} - 1)^k k_t!}{m_1(\lambda)!m_2(\lambda)!\cdots} \frac{B_\lambda (N_1; X,Y, n; u_1,t)}{N_1!},
\end{equation}
where 
\begin{equation}\label{Bu2}
\begin{split}
&B_\lambda(N_1; X,Y,n; u_1,t) =  \int_{\gamma_1^{N_1}} \int_{\gamma_2^{N_1}} \int_{\gamma_{3}^{\ell(\lambda)}}  \det \left[\frac{1}{z_i - w_j} \right]_{i,j = 1}^{N_1}  \det \left[ \frac{1}{\hat{z}_jt^{-\lambda_i} - \hat{z}_i}\right]_{i,j = 1}^{\ell(\lambda)}  \\
&\prod_{i = 1}^{N_1}\frac{ S(w_i, z_i; u_1,t)}{[-\log t] \cdot w_i} \cdot \prod_{i = 1}^{N_1}\prod_{j = 1}^M \frac{1 - y_j z_i}{1 -y_j w_i} 
\prod_{i = 1}^{N_1}\prod_{j = 1}^N \frac{1 - x_jw_i^{-1}}{1 - x_j z_i^{-1}}\prod_{i = 1}^{\ell(\lambda)} \prod_{j = 1}^M \frac{1 - \hat{z}_i y_j}{1 - \hat{z}_it^{\lambda_i} y_j}  \times \\
&\prod_{i = 1}^{\ell(\lambda)}  \prod_{ j = 1}^{n} \frac{1 - \hat{z}^{-1}_ix_jt^{-\lambda_i}}{1 - \hat{z}^{-1}_i x_j} \cdot \prod_{i = 1}^{N_1} \prod_{j = 1}^{\ell(\lambda)}   \frac{(\hat{z}_j z_i^{-1}; t)_\infty }{(\hat{z}_j z_i^{-1} t^{\lambda_j}; t)_\infty }\frac{(\hat{z}_j w_i^{-1} t^{\lambda_j}; t)_\infty }{(\hat{z}_j w_i^{-1}; t)_\infty } \prod_{i = 1}^{\ell(\lambda)} \frac{d\hat{z}_i}{2\pi \iota} \prod_{i = 1}^{N_1} \frac{dw_i}{2\pi \iota} \prod_{i = 1}^{N_1} \frac{dz_i}{2\pi \iota} .
\end{split}
\end{equation}
In (\ref{Bu2}) the $\gamma_1, \gamma_2, \gamma_3$ are positively oriented zero-centered circles of radii $r_1, r_2, r_3$ respectively with $a^{-1} > r_1 > r_2 > r_3 > a$, $r_3/r_1 > t$, and $S(w,z; u_1, t)$ is as in Definition \ref{DefFunS}. 
Moreover, if $I_\lambda(N_1)$ denotes the integrand in (\ref{Bu2}) we  have the following upper bound
\begin{equation}\label{IntegrandUB}
|I_\lambda(N_1)| \leq C^{N_1} \cdot N_1^{N_1/2},
\end{equation}
where the constant $C > 0 $ depends on $ N, M, a,t, r_1, r_2, r_3,k$  and a compact set $K \subset \mathbb{C} \setminus [0, \infty)$. The inequality (\ref{IntegrandUB}) holds whenever $\lambda \vdash k$, $x_i, y_j \in \mathbb{C}$ satisfy $|x_i| \leq (a + r_3)/2$, $|y_j| \leq a$ for $i = 1,\dots, N$, $j = 1, \dots, M$ and $u_1 \in {K}$ uniformly on $\gamma_1^{N_1} \times \gamma_2^N \times \gamma_3^{\ell(\lambda)}$. 
\end{lemma}
\begin{remark} Notice that (\ref{IntegrandUB}) and the compactness of the integration contours ensures that the $N_1!$ in the denominator on the right side of (\ref{kthmomenteq}) is enough to make the sum absolutely convergent. In addition, as explained in Remark \ref{tLaplace}, the expression 
$$\frac{1}{(u_1 t^{-\lambda'_1(N)};t)_{\infty}}  $$ 
is uniformly bounded, while by definition we have that $\mathbb{P}_{X,Y}$-almost surely $0 \leq \lambda_1'(n)  \leq n$ and so the expectation on the left side of (\ref{kthmomenteq}) is well-defined and finite.
\end{remark}

To prove Lemma \ref{kthmoment} we need several results that we list here. First of all, we require the following result, whose proof can be found in Section \ref{Section7.1}. 
\begin{lemma}\label{DetBounds} Let $N \in \mathbb{N}$.
\begin{enumerate}
\item Hadamard's inequality: If $A$ is an $N \times N$ matrix and $v_1, \dots ,v_N$ denote the column vectors of $A$ then $|\det A| \leq \prod_{i = 1}^N \|v_i\|$ where $\|x\| = (x_1^2 + \cdots + x_N^2)^{1/2}$ for $x = (x_1, \dots, x_N)$. 
\item Fix $r, R \in (0,\infty)$ with $R > r$. Let $z_i, w_i \in \mathbb{C}$ be such that $|z_i| = R$ and $|w_i|  \leq r$ for $i = 1, \dots, N$. Then 
\begin{equation}\label{CDetGood}
\left|\det \left[ \frac{1}{z_i - w_j}\right]_{i,j = 1}^N \right| \leq R^{-N} \cdot \frac{N^N \cdot (r/R)^{\binom{N}{2}}}{(1-r/R)^{N^2}}.
\end{equation}
\end{enumerate}
\end{lemma}

We also require the following contour integral identity, which can be found as \cite[Lemma 3.1]{ED}, and whose proof dates back to \cite[Proposition 7.2]{BBC}. 
\begin{lemma}\label{nestedContours}
Fix $k \geq 1$ and $q \in (1,\infty)$. Assume that we are given a set of positively oriented closed contours $\gamma_1, \dots,\gamma_k$, containing $0$, and a function $F(z_1,\dots,z_k)$, satisfying the following properties:
\begin{enumerate}[label = \arabic{enumi}., leftmargin=1.5cm]
\item $F(z_1,\dots,z_k) = \prod_{i = 1}^k f(z_i)$;
\item For all $1 \leq A < B \leq k$, the interior of $\gamma_A$ contains the image of $\gamma_B$ multiplied by $q$;
\item For all $1 \leq j \leq k$ there exists a deformation $D_j$ of $\gamma_j$ to $\gamma_k$ so that for all $z_1, \dots,z_{j-1},z_j,\dots,z_k$ with $z_i \in \gamma_i$ for $1 \leq i < j$ and $z_i \in \gamma_k$ for $j < i \leq k$, the function $z_j \rightarrow F(z_1,\dots,z_j,\dots,z_k)$ is analytic in a neighborhood of the area swept out by the deformation $D_j$. 
\end{enumerate}
Then we have the following residue expansion identity:
\begin{equation}\label{nestedEq}
\begin{split}
&\int_{\gamma_1} \cdots \int_{\gamma_k} \prod_{1 \leq A < B \leq k} \frac{z_A - z_B}{z_A - qz_B}F(z_1, \cdots, z_k) \prod_{i = 1}^k\frac{dz_i}{ 2\pi \iota z_i} =   \sum_{\lambda \vdash k} \frac{(1-q)^k(-1)^kq^{\frac{-k(k-1)}{2}}k_{q }!}{m_1(\lambda)!m_2(\lambda)!\cdots} \\& \int_{\gamma_k}\cdots \int_{\gamma_k} 
\prod_{j = 1}^{\ell(\lambda)}f(w_j)f(w_jq)\cdots f(w_jq^{ \lambda_j - 1}) \det \left[ \frac{1}{w_iq^{\lambda_i} - w_j}\right]_{i,j = 1}^{\ell(\lambda)}\prod_{j = 1}^{\ell(\lambda)} \frac{dw_i}{2\pi \iota},\end{split}
\end{equation}
where we recall that $k_t! = \frac{(1 - t)(1-t^2) \cdots (1-t^k)}{(1-t)^k}$ and $m_1(\lambda), m_2(\lambda), \dots$ are such that $\lambda = 1^{m_1}2^{m_2}\cdots$.
\end{lemma}

Finally, we need the following result, which is essentially  \cite[Proposition 3.5]{ED} upon a change of variables. As the change of variables is rather non-trivial we also supply the proof.
\begin{lemma}\label{0thmoment}
Let $M, N \in \mathbb{N}$, $a,t \in (0,1)$ and fix $x_1, \dots, x_N \in [0,a]$ and $y_1, \dots, y_M \in [0,a]$. Let $\mathbb{P}_{X,Y}$ be as in Definition \ref{AHLP} and $u_1 \in \mathbb{C} \setminus [0, \infty)$. Then we have
\begin{equation}\label{0thmomenteq}
\mathbb{E}_{X,Y} \left[\frac{1}{(u_1 t^{-\lambda'_1(N)};t)_{\infty}}   \right] = 1 + \sum_{N_1 = 1}^\infty \frac{B(N_1; X,Y; u_1,t)}{N_1!},
\end{equation}
where 
\begin{equation}\label{Bu}
\begin{split}
&B(N_1; X,Y; u_1,t) =  \int_{C_1^{N_1}} \int_{C_2^{N_1}} \det \left[\frac{1}{z_i - w_j} \right]_{i,j = 1}^{N_1}  \\
&  \prod_{i = 1}^{N_1}\prod_{j = 1}^N \frac{1 - x_jw_i^{-1}}{1 - x_j z_i^{-1}} \cdot \prod_{i = 1}^{N_1}\prod_{j = 1}^M \frac{1 - y_j z_i}{1 -y_j w_i} \cdot \prod_{i = 1}^{N_1}\frac{ S(w_i, z_i; u_1,t)}{[-\log t] \cdot w_i}  \prod_{i = 1}^{N_1} \frac{dw_i}{2\pi \iota}  \prod_{i = 1}^{N_1}  \frac{dz_i}{2\pi \iota}.
\end{split}
\end{equation}
Here $C_1, C_2$ are positively oriented circles of radii $r_1,r_2$ respectively such that $a^{-1} > r_2 > 0$, $r_1 > a$ and $1 > r_2/r_1 > t$, and $S(w,z; u_1, t)$ is as in Definition \ref{DefFunS}. Moreover, we have 
\begin{equation}\label{BoundBu}
|B(N_1; X,Y; u_1,t)| \leq C^{N_1} \cdot N_1^{N_1/2},
\end{equation}
where the constant $C$ depends on $N,M, a,t, r_1, r_2$ and a compact set $K \subset \mathbb{C} \setminus [0, \infty)$ and (\ref{BoundBu}) holds whenever $x_i, y_j \in \mathbb{C}$ satisfy $|x_i|\leq a, |y_j| \leq a$ for $i = 1, \dots, N$, $j = 1, \dots, M$ and $u_1 \in K$.
\end{lemma}
\begin{proof} For clarity we split the proof into two steps.\\

{\raggedleft \bf Step 1.} In this step we establish (\ref{BoundBu}). Observe that for $|z| = r_1$ and $|w| = r_2$ we have that 
$$\left| \frac{1}{z-w} \right| \leq \frac{1}{r_1 - r_2}.$$
Combining the latter with Hadamard's inequality, see Lemma \ref{DetBounds}, we conclude that 
$$\left| \det \left[\frac{1}{z_i - w_j} \right]_{i,j = 1}^{N_1} \right| \leq \left( \frac{1}{r_1 - r_2} \right)^{N_1} \cdot N_1^{N_1/2},$$
for all $z_i \in C_1$ and $w_i \in C_2$ for $i =1,\dots, N_1$.
Next by Lemma \ref{S3BoundOnS} we see that we can find a constant $A_1 > 0$ depending on $K, r_1, r_2$ and $t$ such that if $u_1 \in K$ we have
$$ \left|\prod_{i = 1}^{N_1}\frac{ S(w_i, z_i; u_1,t)}{[-\log t] \cdot w_i}   \right| \leq A_1^{N_1},$$ 
for all $z_i \in C_1$ and $w_i \in C_2$ for $i =1,\dots, N_1$.
Also we can find a constant $A_2 > 0$ depending on $M, N, a, t, r_1, r_2$ such that if $|x_i| \leq a$ and $|y_j| \leq a$ for $i =1, \dots, N$ and $j = 1,\dots, M$ then
$$ \left| \prod_{i = 1}^{N_1}\prod_{j = 1}^N \frac{1 - x_jw_i^{-1}}{1 - x_j z_i^{-1}} \cdot \prod_{i = 1}^{N_1}\prod_{j = 1}^M \frac{1 - y_j z_i}{1 -y_j w_i} \right| \leq A_2^{N_1}.$$
for all $w_i \in C_1$ and $z_i \in C_2$ for $i =1,\dots, N_1$. Combining the last three inequalities with the compactness of $C_1, C_2$ we conclude (\ref{BoundBu}).\\

{\raggedleft \bf Step 2.} In this step we prove (\ref{0thmomenteq}). Observe that for $u \not \in [0,\infty)$ we have that 
$$\frac{1}{(u t^{-r};t)_{\infty}} $$
 is bounded for $r \in \mathbb{Z}, r\geq 0$ and so the right side of (\ref{0thmomenteq}) is well-defined and finite. Furthermore, by (\ref{BoundBu}), we see that the $N_1!$ in the denominator on the right side of (\ref{0thmomenteq}) ensures that the series is convergent. Thus both sides of (\ref{0thmomenteq}) are well-defined and we need to show that they are equal.

From \cite[Proposition 3.5]{ED} we have 
\begin{equation}\label{RecallFD}
\mathbb{E}_{X,Y} \left[ \frac{1}{(u (1-t)t^{-\lambda'_1(N)};t)_{\infty}}   \right] = 1 + \sum_{N_1 = 1}^\infty \frac{1}{N_1!} \int_{C_0} \cdots \int_{C_0} \det \left[ K_u(w_i, w_j)\right]_{i,j = 1}^{N_1} \prod_{i = 1}^{N_1} \frac{dw_i}{2\pi \iota},
\end{equation}
where $C_0$ is the positively oriented circle of radius $t^{-1}$ and $K_u$ is defined through 
$$K_u(w,w') = \frac{1}{2\pi \iota} \int_{1/2 - \iota \infty}^{1/2 + \iota \infty} ds \Gamma(-s) \Gamma(1+s) ( - u(t^{-1} - 1))^s g_{w,w'}(t^s),$$
with
$$g_{w,w'}(t^s) = \frac{1}{w t^{-s} - w'} \prod_{j = 1}^N \frac{1 - x_j (wt)^{-1}}{1 - x_j (wt)^{-1}t^s} \cdot \prod_{j = 1}^M \frac{1 - y_j (wt)t^{-s}}{1 - y_j(wt)}.$$
In the definition of $K_u$ we have that the contour is a vertical line passing through $1/2$, which is oriented to have increasing imaginary part. We remark that \cite[Proposition 3.5]{ED} was formulated for the case when $M = N$ and to get the above equality we only need to apply that proposition to $N = \max(M,N)$ and set $y_i = 0$ for $\max(N,M) \geq i \geq M+1$ and $x_i = 0$ for $\max(N,M) \geq i \geq N+1$. We further remark that in the proof of \cite[Lemma 3.2]{ED} it was shown that for some positive constants $C, c > 0$ we have that 
\begin{equation}\label{gwwbound}
|g_{w,w'}(t^s) (-\zeta)^s| \leq C e^{-cx} e^{|y \theta|},
\end{equation}
where $\zeta = u (t^{-1} - 1)$, $- \zeta = re^{\iota \theta}$ with $|\theta| < \pi$ and $s = x + \iota y$. Furthermore from (\ref{BoundSine}) we have
\begin{equation}\label{gwwbound2}
\left| \Gamma(-s) \Gamma(1+s) \right| =\left| \frac{\pi}{\sin(-\pi s)} \right| \leq \pi \cdot c' e^{-\pi|Im(s)|},
\end{equation}
for some constant $c' > 0$, where the first equality follows from the Euler's Gamma reflection formula, cf. \cite[Chapter 6, Theorem 1.4]{Stein}. Equations (\ref{gwwbound}) and (\ref{gwwbound2}) together imply that if $u$ varies over compact subsets of $\mathbb{C} \setminus [0,\infty)$ we can find a constant $C> 0$ such that for all $w,w' \in C_0$
$$|K_u(w,w') | \leq C,$$
which together with Hadamard's inequality, see Lemma \ref{DetBounds}, and the compactness of $C_0$ implies that the sum in (\ref{RecallFD}) is absolutely convergent. \\

Using (\ref{gwwbound}) and (\ref{gwwbound2}) we have for any $\phi \in \mathbb{R}$
$$K_u(w,w') \hspace{-1mm}= \hspace{-1mm}\sum_{ m \in \mathbb{Z}} \int_{1/2 + 2\pi \iota m/ [-\log t] + \iota \phi}^{1/2 + 2\pi \iota (m+1) /[-\log t]+ \iota \phi} \hspace{-1mm} \frac{ds(2\pi \iota)^{-1}}{w t^{-s} - w'} \frac{\pi (-\zeta)^s}{\sin (- \pi s)}   \prod_{j = 1}^N \frac{1 - x_j (wt)^{-1}}{1 - x_j (wt)^{-1}t^s} \hspace{-1mm} \prod_{j = 1}^M \hspace{-1mm} \frac{1 - y_j (wt)t^{-s}}{1 - y_j(wt)} =$$
$$ \int_{1/2+ \iota \phi}^{1/2 + 2\pi \iota [-\log t] + \iota \phi}\hspace{-1mm} \frac{ds(2\pi \iota)^{-1}}{w t^{-s} - w'} \prod_{j = 1}^N \frac{1 - x_j (wt)^{-1}}{1 - x_j (wt)^{-1}t^s}   \hspace{-1mm}\prod_{j = 1}^M \frac{1 - y_j (wt)t^{-s}}{1 - y_j(wt)}   \hspace{-1mm} \sum_{m \in \mathbb{Z}}\frac{\pi (-\zeta)^{s + 2\pi \iota m /[-\log t]}}{\sin (- \pi (s + 2\pi \iota m /[-\log t]))},$$
where we used that $t^{2\pi \iota m /[-\log t]} = 1.$

We proceed to change variables: $t\zeta = u(1- t) = u_1$, $z = w t^{1-s}$, and pick $\phi $ such that $w e^{-i \phi} = t^{-1}$ and expand the determinant in (\ref{RecallFD}) to get
\begin{equation}\label{kthmoment0case}
\begin{split}
&\mathbb{E}_{X,Y} \left[ \frac{1}{(u_1t^{-\lambda'_1(N)};t)_{\infty}}   \right] = 1 + \sum_{N_1 = 1}^\infty \frac{1}{N_1!} \sum_{\sigma \in S_{N_1}} (-1)^{\sigma} \int_{C_0^{N_1}} \int_{C_1^{N_1}}   \prod_{i =1}^{N_1} \frac{t  }{z_i - tw_{\sigma(i)}} \cdot  \\
& \prod_{i =1}^{N_1} \prod_{j = 1}^N \frac{1 - x_j(w_it)^{-1}}{1 - x_j z_i^{-1}} \cdot \prod_{i =1}^{N_1} \prod_{j = 1}^M \frac{1 - y_j z_i}{1 -y_j (w_it)} \cdot \prod_{i = 1}^{N_1}\frac{ S(tw_i, z_i; u_1,t)}{[-\log t] \cdot (w_i t)} \prod_{i = 1}^{N_1} \frac{dz_i}{2\pi \iota} \prod_{i = 1}^{N_1} \frac{dw_i}{2\pi \iota}     ,
\end{split}
\end{equation}
where $C_1$ is the positively oriented circle of radius $t^{-1/2}$ centered at the origin and $S(w, z; u,t)$ is as in Definition \ref{DefFunS}. Changing variables $w_i \rightarrow t^{-1}w_i$ and using the formula
$$\sum_{ \sigma \in S_{N_1}} (-1)^{\sigma} \prod_{i = 1}^{N_1} \frac{1}{z_i - w_{\sigma(i)}} = \det \left[ \frac{1}{z_i - w_j} \right]_{i,j = 1}^{N_1},$$
we see that (\ref{kthmoment0case}) implies (\ref{0thmomenteq}) except that $C_1,C_2$ have radii $1$ and $t^{-1/2}$ respectively. However, by Lemma \ref{S3Analyticity}, we may deform $C_1$ and $C_2$ to the circles of radii $r_1, r_2$ without changing the result by Cauchy's theorem.\\
\end{proof}

We now turn to the proof of Lemma \ref{kthmoment}.
\begin{proof} (Lemma \ref{kthmoment}) For clarity we split the proof into three steps. In Step 1 we prove (\ref{IntegrandUB}). In Step 2 we assume that (\ref{kthmomenteq}) holds when $x_1, \dots, x_N$ are very small and distinct and deduce the validity of (\ref{kthmomenteq}) for all $(x_1, \dots, x_N) \in [0,a]^N$ by showing that both sides of (\ref{kthmomenteq}) can be analytically continued to a region containing $[0,a]^N$. In Step 3 we prove that (\ref{kthmomenteq}) holds when $x_1, \dots, x_N$ are very small and distinct  by utilizing the difference operators of Section \ref{Section2.3} and Lemma \ref{0thmoment}.\\

{\bf \raggedleft Step 1.} In this step we establish (\ref{IntegrandUB}). From Hadamard's inequality, see Lemma \ref{DetBounds}, we know that there exist constants $A_1, A_2 > 0$ depending on $r_1, r_2, r_3, k $ and $t$ such that for all $\lambda \vdash k$ we have
$$ \left| \det \left[\frac{1}{z_i - w_j} \right]_{i,j = 1}^{N_1} \right| \leq A_1^{N_1} N_1^{N_1/2} \hspace{2mm} \mbox{, and } \hspace{2mm}  \left| \det \left[ \frac{1}{\hat{z}_jt^{-\lambda_i} - \hat{z}_i}\right]_{i,j = 1}^{\ell(\lambda)}  \right| \leq A_2$$
for all $z_i \in \gamma_1$, $w_i \in \gamma_2$, $\hat{z}_j \in \gamma_3$ for $i = 1, \dots, N_1$ and $j = 1, \dots, \ell(\lambda)$. 
In addition, by Lemma \ref{S3BoundOnS} we can find a constant $A_3 > 0$ depending on $K, t, r_1, r_2, r_3$ (we use that $r_3/r_1 > t$) such that 
$$ \left|\prod_{i = 1}^{N_1}\frac{ S(w_i, z_i; u_1,t)}{[-\log t] \cdot w_i}  \right| \leq A_3^{N_1},$$
for all $z_i \in \gamma_1$, $w_i \in \gamma_2$ for $i = 1, \dots, N_1$.
By our choice of $a^{-1} > r_1 > r_2 > r_3 > a$ we can find a constant $A_4 > 0$ depending on $a,t, r_1, r_2, N,M$ such that 
$$ \left| \prod_{i = 1}^{N_1}\prod_{j = 1}^M \frac{1 - y_j z_i}{1 -y_j w_i} \prod_{i = 1}^{N_1}\prod_{j = 1}^N \frac{1 - x_jw_i^{-1}}{1 - x_j z_i^{-1}} \right| \leq A_4^{N_1},$$
for all $z_i \in \gamma_1$, $w_i \in \gamma_2$ for $i = 1, \dots, N_1$. Next we can find a constant $A_5 > 0$ depending on $ a, t, r_3, N,M,k$ such that for all $\lambda \vdash k$ we have
$$ \left|\prod_{i = 1}^{\ell(\lambda)} \prod_{j = 1}^M \frac{1 - \hat{z}_i y_j}{1 - \hat{z}_it^{\lambda_i} y_j}\prod_{i = 1}^{\ell(\lambda)}  \prod_{ j = 1}^{n} \frac{1 - \hat{z}^{-1}_ix_jt^{-\lambda_i}}{1 - \hat{z}^{-1}_i x_j}   \right| \leq A_5,$$
for all $\hat{z}_j \in \gamma_3$ for $j = 1, \dots, \ell(\lambda)$. Finally, by our choice of $a^{-1} > r_1 > r_2 > r_3 > a$ and $r_3/r_1 > t$ we can find a constant $A_6$ depending on $k, t, r_1, r_2, r_3$ such that for all $\lambda \vdash k$ we have
$$ \left| \prod_{i = 1}^{N_1}\prod_{j = 1}^{\ell(\lambda)}  \frac{(\hat{z}_jz_i^{-1}; t)_\infty }{(\hat{z}_j z_i^{-1} t^{\lambda_j}; t)_\infty }\frac{(\hat{z}_j w_i^{-1} t^{\lambda_j}; t)_\infty }{(\hat{z}_j w_i^{-1}; t)_\infty } \right| \leq A_6^{N_1},$$
for all $z_i \in \gamma_1$, $w_i \in \gamma_2$, $\hat{z}_j \in \gamma_3$ for $i = 1, \dots, N_1$ and $j = 1, \dots, \ell(\lambda)$. 
Combining the six inequalities above we conclude (\ref{IntegrandUB}).\\

{\bf \raggedleft Step 2.} In this step we fix $y_1, \dots, y_M \in [0, a]$ and assume that (\ref{kthmomenteq}) holds for $(x_1, \dots, x_N) \in (0, at^{k + 1})^N \cap \mathcal{W}_N$, where we recall that $\mathcal{W}_N = \{ (a_1, \dots ,a_N) \in \mathbb{R}^N: a_1 < a_2 < \cdots < a_N\}$. Under this assumption we want to show that (\ref{kthmomenteq}) holds for all $(x_1, \dots, x_N) \in [0, a]^N.$ For convenience, we set $r = (a + r_3)/2$.

From Lemma \ref{S3Analyticity} we know that $I_\lambda(N_1)$ (the integrand in \ref{Bu2}) is jointly continuous in $w_i, z_i, \hat{z}_j$ for $i =1 ,\dots, N_1$ and $j= 1, \dots, \ell(\lambda)$ as well as $(x_1, \dots, x_N)$. Moreover, for fixed $z_i, w_i, \hat{z}_j$ the integrand is analytic in $x_i \in \mathbb{D}_{r}$ (the zero-centered disc of radius $r = (a + r_3)/2$) for $i = 1, \dots, N$. By \cite[Theorem 5.4]{Stein} we conclude that $B_\lambda(N_1; X,Y,n;u_1, t)$ is analytic in $x_i\in \mathbb{D}_{r}$ for $i = 1, \dots, N$ (this statement should be read as the function being analytic in $x_i\in \mathbb{D}_{r}$ if all the other $x_j$'s are fixed in $ \mathbb{D}_{r}$ for each $i = 1, \dots, N$). From (\ref{IntegrandUB}), which we established in Step 1, we see that the sum on the right side of (\ref{kthmoment}) is absolutely convergent and so by \cite[Theorem 5.2]{Stein} we conclude that the right side of (\ref{kthmoment}) is analytic in $x_i\in \mathbb{D}_{r} $ for $i = 1, \dots, N$. We next show that the same is true for the left side.

Recall, from Section \ref{Section2} that
\begin{equation}\label{momentExpandCauchy}
\begin{split}
&\mathbb{E}_{X,Y} \left[t^{-k\lambda_1'(n)}  \cdot \frac{1}{(u_1 t^{-\lambda'_1(N)};t)_{\infty}}  \right] =  \prod_{i = 1}^N \prod_{j = 1}^M \frac{1- x_iy_j}{1 -tx_iy_j}\cdot \sum_{\lambda(n) \preceq \cdots \preceq \lambda(N) }\\
& t^{-k\lambda_1'(n)}  \cdot \frac{1}{(u_1 t^{-\lambda'_1(N)};t)_{\infty}} \cdot P_{\lambda(n)} (x_1, \dots, x_n; t) \prod_{i = n+1}^N  \hspace{-3mm} P_{\lambda(i+1) / \lambda(i)}(x_i; t) Q_{\lambda(N)} (Y;t).
\end{split}
\end{equation}
Notice that we can find a constant $A_7 > 0$ depending on $u_1 ,t, N, n, k$ such that 
$$\left| t^{-k\lambda_1'(n)} \cdot \frac{1}{(u_1 t^{-\lambda'_1(N)};t)_{\infty}} \right| \leq A_7$$
where we used that $u_1 \not \in [0, \infty)$ and the fact that $0 \leq \lambda_1'(n) \leq n$. Using (\ref{skewOne}) and (\ref{HLSkewExpand}) we have for $x_i \in \mathbb{D}_{r}$ for $i = 1, \dots, N$ that
$$\left| P_{\lambda(n)} (x_1, \dots, x_n; t) \prod_{i = n+1}^N  \hspace{-3mm} P_{\lambda(i+1) / \lambda(i)}(x_i; t)\right| \leq  P_{\lambda(n)} (r, \dots, r; t) \prod_{i = n+1}^N  \hspace{-3mm} P_{\lambda(i+1) / \lambda(i)}(r; t).$$
Combining the last two observations we get
\begin{equation*}
\begin{split}
&\sum_{\lambda(n) \preceq \cdots \preceq \lambda(N) } \left|t^{-k\lambda_1'(n)}  \cdot \frac{1}{(u_1 t^{-\lambda'_1(N)};t)_{\infty}} \cdot P_{\lambda(n)} (x_1, \dots, x_n; t) \prod_{i = n+1}^N  \hspace{-3mm} P_{\lambda(i+1) / \lambda(i)}(x_i; t) Q_{\lambda(N)} (Y;t)\right| \leq \\
&A_7 \cdot \sum_{\lambda(n) \preceq \cdots \preceq \lambda(N) } P_{\lambda(n)} (r , \dots, r; t) \prod_{i = n+1}^N  P_{\lambda(i+1) / \lambda(i)}(r; t) Q_{\lambda(N)} (Y;t)= A_7 \prod_{j = 1}^M  \left(  \frac{1 -try_j}{1- ry_j}\right)^N,
\end{split}
\end{equation*}
where in the last equality we used (\ref{HLSkewExpand}) and (\ref{CauchyConv}). The work above shows the sum on the right side of (\ref{momentExpandCauchy}) is absolutely convergent for $x_i \in \mathbb{D}_{r}$ for $i = 1,\dots, N$. Since each summand is a polynomial in $x_i$, we conclude by \cite[Theorem 5.2]{Stein} that the left side of (\ref{momentExpandCauchy}) is an analytic function of $x_i\in \mathbb{D}_{r}$ for $i = 1, \dots, N$. \\

We now know that both the left and right side of (\ref{kthmomenteq}) are given by functions that are analytic in each $x_i\in \mathbb{D}_{r}$ for $i = 1, \dots, N$. We proceed to inductively on $m$ show that if $(x_1, \dots, x_{N-m}) \in (0,at^{k+1})^{N-m} \cap \mathcal{W}_{N-m}$ and $x_i \in \mathbb{D}_{r}$ for $i \geq N- m + 1$ then (\ref{kthmomenteq}) holds. The base case $m = 0$ is the statement we assumed in the beginning of the step. Assuming the result for $N-1 \geq m \geq 0$ we fix $(x_1, \dots, x_{N- m -1}) \in (0, at^{k + 1})^{N-m-1} \cap \mathcal{W}_{N-m - 1}$ and $x_{i} \in \mathbb{D}_{r}$ for $i = N-m+1, \dots, N$. Then as a function of $x_{N-m}$ both  sides of (\ref{kthmomenteq}) are analytic in $x_{N-m} \in \mathbb{D}_{r}$ and are equal for $x_{N-m} \in (x_{N-m-1}, at^{k+1})$ by induction hypothesis. As the latter segment has a limit point in $\mathbb{D}_{r}$ by our choice of $r$ we conclude by \cite[Corollary 2.4.9]{Stein} that (\ref{kthmomenteq}) holds whenever $(x_1, \dots, x_{N- m -1}) \in (0, at^{k + 1})^{N-m-1} \cap \mathcal{W}_{N-m - 1}$ and $x_{i} \in \mathbb{D}_{r}$ for $i = N-m, \dots, N$. This proves the induction step and iterating the above statement for $m = 0, \dots, N-1$ we finally arrive at the fact that (\ref{kthmomenteq}) holds for all $(x_1, \dots, x_N) \in \mathbb{D}_{r}^N$, which in particular shows the equality when $(x_1, \dots, x_N) \in [0, a]^N.$\\

{\bf \raggedleft Step 3.} In Steps 1 and 2 we reduced the proof of the lemma to establishing (\ref{kthmomenteq}) when $(x_1, \dots, x_N) \in (0, at^{k + 1})^N \cap \mathcal{W}_N$. We prove this statement here and what we will see is that in this case (\ref{kthmomenteq}) follows from applying $\Pi(X;Y)^{-1} \cdot \mathcal{D}^k_{n} \cdot \Pi(X;Y)$ to both sides of (\ref{0thmomenteq}), where $\Pi(X;Y)$ is as in (\ref{CauchyConv}). \\

By Lemma \ref{milestone1} applied to the function 
$$f(\lambda(n), \dots, \lambda(N)) = \frac{1}{(u_1 t^{-\lambda_1(N)};t)_\infty} ,$$
we conclude that the left side of (\ref{kthmomenteq}) equals
$$\Pi(X;Y)^{-1} \cdot \mathcal{D}^k_{n} \left[ \Pi(X;Y) + \sum_{N_1 = 1}^\infty \Pi(X;Y) \frac{B(N_1; X,Y; u_1,t)}{N_1!} \right].$$
In view of (\ref{BoundBu}) we know that the sum above converges absolutely and by Properties 1 and 2 of $\mathcal{D}_n$, see Section \ref{Section2.3}, we see that we can exchange the order of the sum and operator. In particular, to prove (\ref{kthmomenteq}) it suffices to show that 
\begin{equation}\label{RedBu2}
\sum_{\lambda \vdash k }\frac{(t^{-1} - 1)^k k_t!}{m_1(\lambda)!m_2(\lambda)!\cdots} B_\lambda (N_1; X,Y, n; u_1,t) =  \Pi(X;Y)^{-1} \cdot \mathcal{D}^k_{n} \cdot \Pi(X;Y) \cdot B(N_1; X,Y; u_1,t),
\end{equation}
where we have adopted the convention $B(0; X,Y; u_1,t) = 1$. From (\ref{Bu}) and the definition of $\Pi(X;Y)$ in (\ref{CauchyConv}) we know that the right side of (\ref{RedBu2}) is equal to
\begin{equation*}
\begin{split}
& \prod_{i = 1}^N \prod_{j = 1}^M \frac{1 -x_i y_j}{1 - tx_i y_j} \cdot \mathcal{D}_n^k \Bigg{[}  \int_{\gamma_1^{N_1}} \int_{\gamma_2^{N_1}} \det \left[\frac{1}{z_i - w_j} \right]_{i,j = 1}^{N_1}\prod_{i = 1}^{N_1} \frac{ S(w_i, z_i; u_1,t)}{[-\log t] \cdot w_i} \cdot \prod_{i = 1}^{N_1}\prod_{j = 1}^M \frac{1 - y_j z_i}{1 -y_j w_i} \cdot  \\
& \prod_{i = 1}^N \prod_{j = 1}^M \frac{1 -tx_i y_j}{1 - x_i y_j} \prod_{i = 1}^{N_1}\prod_{j = 1}^N \frac{1 - x_jw_i^{-1}}{1 - x_j z_i^{-1}} \prod_{i = 1}^{N_1} \frac{dw_i}{2\pi \iota} \prod_{i = 1}^{N_1} \frac{dz_i}{2\pi \iota}   \Bigg{]},
\end{split}
\end{equation*}
where $\gamma_1, \gamma_2$ are as in the statement of the lemma. We may now take a sequence of Riemann sums converging to the above integrals and use Properties 1 and 2 of $\mathcal{D}_n$ to exchange the order of the integral and operator. Consequently, the right side of (\ref{RedBu2}) is equal to
\begin{equation}\label{Red2Bu2}
\begin{split}
& \int_{\gamma_1^{N_1}} \int_{\gamma_2^{N_1}} \det \left[\frac{1}{z_i - w_j} \right]_{i,j = 1}^{N_1}\prod_{i = 1}^{N_1} \frac{ S(w_i, z_i; u_1,t)}{[-\log t] \cdot w_i} \cdot \prod_{i = 1}^{N_1}\prod_{j = 1}^M \frac{1 - y_j z_i}{1 -y_j w_i}  \prod_{i = 1}^N \prod_{j = 1}^M \frac{1 -x_i y_j}{1 - t x_i y_j} \\
& \prod_{i = n+1}^N \prod_{j = 1}^M \frac{1 - tx_iy_j}{1 - x_i y_j} \cdot \prod_{j = n+1}^N \prod_{i = 1}^{N_1}\frac{1 - x_jw_i^{-1}}{1 - x_j z_i^{-1}} \mathcal{D}^k_n \left[ \prod_{i = 1}^n f(x_i)\right] \prod_{i = 1}^{N_1} \frac{dw_i}{2\pi \iota} \prod_{i = 1}^{N_1} \frac{dz_i}{2\pi \iota} ,
\end{split}
\end{equation}
where 
$$f(x) = \prod_{j = 1}^M \frac{1 -tx y_j}{1 - x y_j} \prod_{i = 1}^{N_1} \frac{1 - xw_i^{-1}}{1 - x z_i^{-1}}.$$
Observe that $f(z)$ is non-zero and holomorphic in the unit disc, centered at $0$ and that $f(0) = 1$. It follows by Proposition \ref{babyContProp} that 
\begin{equation}\label{ApplyProp1}
\mathcal{D}^k_n \left[ \prod_{i = 1}^n f(x_i)\right] =  \frac{\prod_{i = 1}^n f(x_i)}{(2\pi \iota )^k} \int_{C_{0,1}} \hspace{-2mm}\cdots \int_{C_{0,k}} \hspace{-1mm}\prod_{1 \leq a < b \leq k}  \frac{\zeta_a - \zeta_b}{\zeta_a - \zeta_bt^{-1}}  \prod_{i =1}^k \hspace{-1mm}\left[ \prod_{ j = 1}^n\hspace{-1mm} \frac{\zeta_i - x_jt^{-1}}{\zeta_i - x_j} \right] \frac{1}{f(\zeta_i)} \frac{d\zeta_i}{\zeta_i},
\end{equation}
where $C_{0,p}$ are positively oriented simple contours encircling $x_1, \dots ,x_n$ and $0$ and no zeros of $f(z)$. In addition, $C_{0,p}$ contains $t^{-1}C_{0,q}$ for $p <q$ and $C_{0,1} \subset D$ -- a complex neighborhood containing $0, x_1, \dots, x_n$ where $f(x)$ is holomorphic and non-vanishing. Since by assumption $|x_i| < at^{k+1}$ we see that we can take $C_{0,i}$ above to be the positively oriented, zero centered circles of radius $at^{(i-1) + i/k}$ -- we fix this choice of contours in the sequel. 

Combining (\ref{Red2Bu2}) and (\ref{ApplyProp1}) and performing a bit of cancellations we see that the right side of (\ref{RedBu2}) is equal to
\begin{equation}\label{Red2Bu3}
\begin{split}
& \int_{\gamma_1^{N_1}} \int_{\gamma_2^{N_1}} \det \left[\frac{1}{z_i - w_j} \right]_{i,j = 1}^{N_1} \prod_{i = 1}^{N_1}\frac{ S(w_i, z_i; u_1,t)}{[-\log t] \cdot w_i} \cdot \prod_{i = 1}^{N_1}\prod_{j = 1}^M \frac{1 - y_j z_i}{1 -y_j w_i} \cdot\prod_{i = 1}^{N_1}\prod_{j = 1}^N \frac{1 - x_jw_i^{-1}}{1 - x_j z_i^{-1}}\cdot \\
&\int_{C_{0,1}} \cdots \int_{C_{0,k}} \prod_{1 \leq a < b \leq k}  \frac{\zeta_a - \zeta_b}{\zeta_a - \zeta_bt^{-1}}  \prod_{i =1}^k \left[ \prod_{ j = 1}^n \frac{\zeta_i - x_jt^{-1}}{\zeta_i - x_j} \cdot \frac{1}{f(\zeta_i)} \right]  \prod_{i = 1}^k \frac{d\zeta_i}{2 \pi \iota \zeta_i}  \prod_{i = 1}^{N_1} \frac{dw_i}{2\pi \iota} \prod_{i = 1}^{N_1} \frac{dz_i}{2\pi \iota} .
\end{split}
\end{equation}

We next perform the change of variables $\omega_i = - \zeta^{-1}_{k- i +1}$ for $i =1 , \dots, k$. This allows us to rewrite
\begin{equation}\label{R1}
\begin{split}
&\int_{C_{0,1}} \cdots \int_{C_{0,k}} \prod_{1 \leq a < b \leq k}  \frac{\zeta_a - \zeta_b}{\zeta_a - \zeta_bt^{-1}}  \prod_{i =1}^k \left[ \prod_{ j = 1}^n \frac{\zeta_i - x_jt^{-1}}{\zeta_i - x_j} \cdot \frac{1}{f(\zeta_i)} \right]  \prod_{i = 1}^k \frac{d\zeta_i}{2 \pi \iota \zeta_i} = \\
&\int_{\tilde{\gamma}_1} \cdots \int_{\tilde{\gamma}_k} \prod_{1 \leq a < b \leq k}  \frac{\omega_a - \omega_b}{\omega_a - \omega_bt^{-1}}  \prod_{i =1}^k \left[ \prod_{ j = 1}^n \frac{\omega_i^{-1} + x_jt^{-1}}{\omega_i^{-1} + x_j} \cdot \frac{1}{f(-\omega^{-1}_i)} \right]  \prod_{i = 1}^k \frac{d\omega_i}{2 \pi \iota \omega_i},
\end{split}
\end{equation}
where $\tilde{\gamma}_i$ is the positively oriented zero-centered circle of radius $a^{-1}t^{-(k-i) - (k-i+1)/k}.$

We may now apply Lemma \ref{nestedContours} for the case $q = t^{-1}$, $\gamma_i = \tilde{\gamma}_i$ as above and 
$$f(z) = \prod_{j = 1}^n \frac{z^{-1} + x_jt^{-1}}{z^{-1} +x_j} \cdot \prod_{j = 1}^M \frac{1 + z^{-1} y_j}{1 + tz^{-1} y_j} \cdot \prod_{i = 1}^{N_1} \frac{1 + z^{-1} z_i^{-1}}{1 + z^{-1}w_i^{-1}}.$$
We then obtain
\begin{equation}\label{R2}
\begin{split}
&\int_{\tilde{\gamma}_1} \cdots \int_{\tilde{\gamma}_k} \prod_{1 \leq a < b \leq k}  \frac{\omega_a - \omega_b}{\omega_a - \omega_bt^{-1}}  \prod_{i =1}^k \left[ \prod_{ j = 1}^n \frac{\omega_i^{-1} + x_jt^{-1}}{\omega_i^{-1} + x_j} \cdot \frac{1}{f(-\omega^{-1}_i)} \right]  \prod_{i = 1}^k \frac{d\omega_i}{2 \pi \iota \omega_i}  = \\
&  \sum_{\lambda \vdash k} \frac{(1-t^{-1})^k(-1)^kt^{\frac{k(k-1)}{2}}k_{t^{-1} }!}{m_1(\lambda)!m_2(\lambda)!\cdots} \int_{\tilde{\gamma}_k^{\ell(\lambda)}}
 \det \left[ \frac{1}{v_it^{-\lambda_i} - v_j}\right]_{i,j = 1}^{\ell(\lambda)} \prod_{i = 1}^{\ell(\lambda)}H(v_i, \lambda_i)\prod_{i = 1}^{\ell(\lambda)} \frac{dv_i}{2\pi \iota},
\end{split}
\end{equation}
where 
\begin{equation}\label{R3}
\begin{split}
&H(v, \lambda) = \prod_{s= 0}^{\lambda - 1}  \prod_{j = 1}^n \frac{v^{-1} t^{s} + x_jt^{-1}}{v^{-1} t^{s}+ x_j} \cdot\prod_{s= 0}^{\lambda - 1}  \prod_{j = 1}^M \frac{1 + v^{-1} t^{s} y_j}{1 + t v^{-1} t^{s} y_j} \cdot\prod_{s= 0}^{\lambda - 1}  \prod_{i = 1}^{N_1} \frac{1 + v^{-1} t^{s}z_i^{-1}}{1 + v^{-1} t^{s} w_i^{-1}}.
\end{split}
\end{equation}
We next note that
\begin{equation}\label{R4}
\begin{split}
&  \prod_{s = 0}^{\lambda-1} \prod_{j = 1}^M \frac{1 + v^{-1}t^s y_j}{1 + t  v^{-1}t^s y_j} =   \prod_{j = 1}^M \frac{1 + v^{-1} y_j}{1 + v^{-1}t^{\lambda} y_j}, \\
&\prod_{s = 0}^{\lambda-1}\prod_{ j = 1}^{n} \frac{v^{-1}t^{s} + x_jt^{-1}}{v^{-1}t^{s} + x_j}  = \prod_{s = 0}^{\lambda-1}\prod_{ j = 1}^{n} \frac{1 + vx_jt^{-s-1}}{1 + vx_jt^{-s}} =  \prod_{ j = 1}^{n} \frac{1 + vx_jt^{-\lambda}}{1 + v x_j},\\
& \prod_{s = 0}^{\lambda-1} \prod_{j = 1}^{N_1}  \frac{1 + v^{-1}t^{s}z_j^{-1}}{1 +v^{-1}t^{s}w_j^{-1}} =  \prod_{j = 1}^{N_1}  \frac{(-v^{-1} z_j^{-1}; t)_\infty }{(-v^{-1} z_j^{-1} t^{\lambda}; t)_\infty }\frac{(-v^{-1} w_j^{-1} t^{\lambda}; t)_\infty }{(-v^{-1} w_j^{-1}; t)_\infty }.
\end{split}
\end{equation}

Combining (\ref{R1}), (\ref{R2}), (\ref{R3}) and (\ref{R4}) with the identity $(-1)^k t^{k(k-1)/2} k_{t^{-1}}! (1 - t^{-1})^k = (t^{-1} - 1)^k k_t!$ we get
\begin{equation}\label{R5}
\begin{split}
&\int_{C_{0,1}} \cdots \int_{C_{0,k}}  \prod_{1 \leq a < b \leq k}  \frac{\zeta_a - \zeta_b}{\zeta_a - \zeta_bt^{-1}}  \prod_{i =1}^k \left[ \prod_{ j = 1}^n \frac{\zeta_i - x_jt^{-1}}{\zeta_i - x_j} \cdot \frac{1}{f(\zeta_i)} \right]  \prod_{i = 1}^k \frac{d\zeta_i}{2 \pi \iota \zeta_i} =  \\
& \sum_{\lambda \vdash k} \frac{(t^{-1} - 1)^k k_t!}{m_1(\lambda)!m_2(\lambda)!\cdots} \int_{\tilde{\gamma}_k^{\ell(\lambda)}}
 \det \left[ \frac{1}{v_it^{-\lambda_i} - v_j}\right]_{i,j = 1}^{\ell(\lambda)} \prod_{i = 1}^{\ell(\lambda)} \prod_{j = 1}^M \frac{1 + v_i^{-1} y_j}{1 + v_i^{-1}t^{\lambda_i} y_j}  \times \\
&\prod_{i = 1}^{\ell(\lambda)}  \prod_{ j = 1}^{n} \frac{1 + v_ix_jt^{-\lambda}}{1 + v_i x_j} \cdot \prod_{i = 1}^{\ell(\lambda)} \prod_{j = 1}^{N_1}  \frac{(-v_i^{-1} z_j^{-1}; t)_\infty }{(-v_i^{-1} z_j^{-1} t^{\lambda_i}; t)_\infty }\frac{(-v_i^{-1} w_j^{-1} t^{\lambda_i}; t)_\infty }{(-v_i^{-1} w_j^{-1}; t)_\infty } \prod_{i = 1}^{\ell(\lambda)} \frac{dv_i}{2\pi \iota}.
\end{split}
\end{equation}
Combining (\ref{R5}) with (\ref{Red2Bu2}) and (\ref{ApplyProp1}) and performing the change of variables $\hat{z}_i = -v_i^{-1}$ for $i = 1, \dots, \ell(\lambda)$ we see that the right side of (\ref{RedBu2}) is equal to 
\begin{equation}\label{Red3Bu2}
\begin{split}
& \sum_{\lambda \vdash k} \frac{(t^{-1} - 1)^k k_t!}{m_1(\lambda)!m_2(\lambda)!\cdots}  \int_{\gamma_1^{N_1}} \int_{\gamma_2^{N_1}} \int_{C_{0,1}^{\ell(\lambda)}}  \det \left[\frac{1}{z_i - w_j} \right]_{i,j = 1}^{N_1}  \det \left[ \frac{1}{\hat{z}_jt^{-\lambda_i} - \hat{z}_i}\right]_{i,j = 1}^{\ell(\lambda)}  \\
&\prod_{i = 1}^{N_1}\frac{ S(w_i, z_i; u_1,t)}{[-\log t] \cdot w_i} \cdot \prod_{i = 1}^{N_1}\prod_{j = 1}^M \frac{1 - y_j z_i}{1 -y_j w_i} 
\prod_{i = 1}^{N_1}\prod_{j = 1}^N \frac{1 - x_jw_i^{-1}}{1 - x_j z_i^{-1}}\prod_{i = 1}^{\ell(\lambda)} \prod_{j = 1}^M \frac{1 - \hat{z}_i y_j}{1 - \hat{z}_it^{\lambda_i} y_j}  \times \\
&\prod_{i = 1}^{\ell(\lambda)}  \prod_{ j = 1}^{n} \frac{1 - \hat{z}^{-1}_ix_jt^{-\lambda_i}}{1 - \hat{z}^{-1}_i x_j} \cdot \prod_{i = 1}^{N_1} \prod_{j = 1}^{\ell(\lambda)}   \frac{(\hat{z}_j z_i^{-1}; t)_\infty }{(\hat{z}_jz_i^{-1} t^{\lambda_j}; t)_\infty }\frac{(\hat{z}_j w_i^{-1} t^{\lambda_j}; t)_\infty }{(\hat{z}_j w_i^{-1}; t)_\infty } \prod_{i = 1}^{\ell(\lambda)} \frac{d\hat{z}_i}{2\pi \iota}  \prod_{i = 1}^{N_1} \frac{dw_i}{2\pi \iota}\prod_{i = 1}^{N_1}\frac{dz_i}{2\pi \iota} .
\end{split}
\end{equation}
The above now is equal to the left side of (\ref{RedBu2}) since we may deform $C_{0,1}$ to $\gamma_3$, without affecting the value of the integral by Cauchy's theorem. In the last deformation we implicitly used Lemma \ref{S3BigSumAnal}. This proves (\ref{RedBu2}), which concludes the proof of the lemma.
\end{proof}

%
%
\subsection{Proof of Theorem \ref{PrelimitT}}\label{Section3.3} In this section we give the proof of Theorem \ref{PrelimitT}. We require the following lemma, whose proof is deferred to Section \ref{Section7.1}.

\begin{lemma}\label{S3Expansion} Let $ M ,n \in \mathbb{N}$, $N \in \mathbb{Z}_{\geq 0}$, $t \in (0,1)$ and $R > 0$. Suppose that $a, z\in \mathbb{C}$ saitsfy $0 < |a|=|z|  \leq 1$.  Assume further that $z_i, w_i , x_j, y_k \in \mathbb{C}$ satisfy $|z_i| \geq |z|$, $|w_i| \geq |z|$, $|x_j| \leq R$, $|y_k| \leq 1$ for $i = 1, \dots, N$, $j = 1, \dots, n$ and $k = 1, \dots, M$. Then if $u \in \mathbb{C}$ we have for any $c \in \mathbb{N}$ 
\begin{equation}\label{S3TermCBound}
\left| \frac{ u^c }{a - z t^{c} }    \prod_{k = 1}^M \frac{1}{1 - z y_k t^{c}}  \prod_{ j = 1}^{n} (1 - x_j z^{-1}t^{-c})  \prod_{i = 1}^{N}  \frac{ (z w_i^{-1} t^{c}; t)_\infty }{(zz_i^{-1} t^{c}; t)_\infty } \right| \leq  \frac{(-t; t)^N_{\infty}(1 + R/|a|)^n |u|^c t^{-cn} }{|a| (1 - t)^{M+1}(t; t)^N_{\infty}} .
\end{equation} 
In particular, if $|u| < t^{n}$ the function
\begin{equation}\label{S3HF}
H(u) = \sum_{c = 1}^\infty  \frac{ u^c }{a - z t^{c} }    \prod_{k = 1}^M \frac{1}{1 - z y_k t^{c}}  \prod_{ j = 1}^{n} (1 - x_j z^{-1}t^{-c})  \prod_{i = 1}^{N}  \frac{ (z w_i^{-1} t^{c}; t)_\infty }{(zz_i^{-1} t^{c}; t)_\infty } ,
\end{equation} 
is well-defined and finite. If $|u| < t^{n}$ and $u \in \mathbb{C} \setminus [0, \infty)$ we also have that
\begin{equation}\label{S3gResidues}
H(u) = \frac{1}{2\pi \iota} \int_{C}  \frac{ S(z,w; u,t) }{w(a - w) [-\log t]}    \prod_{k = 1}^M \frac{1}{1 - w y_k }  \prod_{ j = 1}^{n} (1 - x_j w^{-1})  \prod_{i = 1}^{N}  \frac{ (w w_i^{-1} ; t)_\infty }{(wz_i^{-1} ; t)_\infty } dw,
\end{equation}
where $C$ is a positively oriented circle of radius $ r \in (t |z|, |z|)$ that is centered at the origin and $S(z,w;u,t)$ is as in Definition \ref{DefFunS}. 
\end{lemma}

\begin{proof} (Theorem \ref{PrelimitT}) For clarity we split the proof into several steps. In the first step we explain by an analytic continuation argument that it suffices to prove the theorem when $|u_2|$ is small. In Step 2 we apply Lemma \ref{kthmoment} to obtain a sequence of equalities, indexed by $L$ -- these are given in (\ref{P1}). The left side of these equalities converges to the left side of (\ref{PrelimitEq}) -- this is proved in Step 3. The right side of these equalities converges to the right side of (\ref{PrelimitEq}) -- this is proved in Steps 4-6. \\

{\bf \raggedleft Step 1.} In this step we assume that we have proved (\ref{PrelimitEq})  when $u_1 , u_2 \in \mathbb{C} \setminus [0, \infty)$ and $|u_2| < t^n (1-t)$. Under this assumption we show that (\ref{PrelimitEq})  holds for all $u_1 , u_2 \in \mathbb{C} \setminus [0, \infty)$. We fix $ u_1  \in \mathbb{C} \setminus [0, \infty)$ and note that by Lemma \ref{S3BigSumAnal} the function $H(u_1,u_2)$ is analytic in $u_2 \in \mathbb{C} \setminus [0, \infty)$. By definition
\begin{equation*}
\begin{split}
& \mathbb{E}_{X,Y} \left[\frac{1}{(u_1 t^{-\lambda'_1(N)};t)_{\infty}} \cdot \frac{1}{(u_2 t^{-\lambda'_1(n)};t)_{\infty}}  \right] = \sum_{a = 0}^N \sum_{b = 0}^n \frac{ \mathbb{P}_{X,Y} (\lambda_1'(n) = a, \lambda_1'(N) = b) }{(u_1 t^{-a};t)_{\infty}(u_2 t^{-b};t)_{\infty}} .
\end{split}
\end{equation*}
As each summand is analytic in $u_2 \in \mathbb{C} \setminus [0, \infty)$ and the sum is finite, we conclude that the above is analytic in $u_2  \in \mathbb{C} \setminus [0, \infty)$.
Thus both sides of (\ref{PrelimitEq}) define analytic functions in $u_2 \in \mathbb{C} \setminus [0, \infty)$  and since by our assumption they are equal in a neighborhood in this set, we have by \cite[Corollary 2.4.9]{Stein} that they are equal for all $u_2  \in \mathbb{C} \setminus [0, \infty)$.\\

{\bf \raggedleft Step 2.} In the remainder of the proof we fix $u_1 , u_2 \in \mathbb{C} \setminus [0, \infty)$ with $|u_2| < t^n (1-t)$ and proceed to prove (\ref{PrelimitEq}) in this case. From Lemmas \ref{kthmoment} and \ref{0thmoment} we have that for any $L \in \mathbb{N}$
\begin{equation}\label{P1}
\begin{split}
\mathbb{E}_{X,Y} \left[ \left(\sum_{k = 0}^L \frac{u_2^k (1-t)^{-k} t^{-k\lambda_1'(n)}}{k_t!} \right) \frac{1}{(u_1 t^{-\lambda'_1(N)};t)_{\infty}}  \hspace{-1mm} \right] = \sum_{N_1 = 0}^\infty \sum_{N_2 = 0}^\infty  \frac{H_L(u_2, N_1,N_2)}{N_1! N_2!},
\end{split}
\end{equation}
where 
\begin{equation}\label{P2}
\begin{split}
&H_L(u_2, N_1, N_2) = \int_{\gamma_1^{N_1}} \int_{\gamma_2^{N_1}} H^1(\vec{z}, \vec{w},N_1)\cdot     H^2_L(\vec{z}, \vec{w}, N_1,N_2)  \prod_{i = 1}^{N_1} \frac{dw_i}{2\pi \iota}  \prod_{i = 1}^{N_1}\frac{dz_i}{2\pi \iota}  
\end{split}
\end{equation}
and 
\begin{equation}\label{P3}
\begin{split}
&H^1(\vec{z}, \vec{w},N_1) = \det \left[\frac{1}{z_i - w_j} \right]_{i,j = 1}^{N_1} \cdot \prod_{i = 1}^{N_1}\frac{ S(w_i, z_i; u_1,t)}{[-\log t] \cdot w_i} \prod_{i = 1}^{N_1}\prod_{j = 1}^M \frac{1 - y_j z_i}{1 -y_j w_i} 
\prod_{i = 1}^{N_1}\prod_{j = 1}^N \frac{1 - x_jw_i^{-1}}{1 - x_j z_i^{-1}}   \\
& H^2_L(\vec{z}, \vec{w}, N_1, N_2) = \sum_{\lambda \in \mathbb{Y}: \ell(\lambda) = N_2} {\bf 1}\{ |\lambda| \leq L\} \int_{\gamma_{3}^{N_2}} \frac{[u_2t^{-1}]^{|\lambda|} N_2!}{\prod_{i = 1}^\infty m_i(\lambda)!} \det \left[ \frac{1}{\hat{z}_jt^{-\lambda_i} - \hat{z}_i}\right]_{i,j = 1}^{N_2}   \cdot    \\
&\prod_{i = 1}^{N_2} \prod_{j = 1}^M \frac{1 - \hat{z}_i y_j}{1 - \hat{z}_it^{\lambda_i} y_j} \prod_{i = 1}^{N_2}  \prod_{ j = 1}^{n} \frac{1 - \hat{z}^{-1}_ix_jt^{-\lambda_i}}{1 - \hat{z}^{-1}_i x_j} \cdot \prod_{i = 1}^{N_1} \prod_{j = 1}^{N_2}  \frac{(\hat{z}_j z_i^{-1}; t)_\infty }{(\hat{z}_j z_i^{-1} t^{\lambda_j}; t)_\infty }\frac{(\hat{z}_j w_i^{-1} t^{\lambda_j}; t)_\infty }{(\hat{z}_j w_i^{-1}; t)_\infty }  \prod_{i = 1}^{N_2} \frac{d\hat{z}_i}{2\pi \iota}.
\end{split}
\end{equation}
In the above equations $\gamma_1, \gamma_2, \gamma_3$ are as in the statement of the theorem. Also, we have used the convention $H^2_L(\vec{z}, \vec{w}, N_1, 0) = 1$ and $H_L(u_2, 0,0) = 1$.

We claim that 
\begin{equation}\label{R1E1}
\begin{split}
&\mathbb{E}_{X,Y} \left[ \frac{1}{(u_1 t^{-\lambda'_1(N)};t)_{\infty}} \frac{1}{(u_2 t^{-\lambda'_1(n)};t)_{\infty}} \right]  = \\
& \lim_{L \rightarrow \infty} \mathbb{E}_{X,Y} \left[ \left(\sum_{k = 0}^L \frac{u_2^k(1-t)^{-k}t^{-k\lambda_1'(n)}}{k_t!} \right) \cdot\frac{1}{(u_1 t^{-\lambda'_1(N)};t)_{\infty}}  \right].
\end{split}
\end{equation}
and
\begin{equation}\label{R1E2}
\lim_{L \rightarrow \infty} \sum_{N_1 = 0}^\infty \sum_{N_2 = 0}^\infty  \frac{H_L(u_2, N_1,N_2)}{N_1! N_2!} = H(u_1,u_2).
\end{equation}
Clearly (\ref{P1}), (\ref{R1E1}) and (\ref{R1E2}) together imply (\ref{PrelimitEq}) and so it suffices to prove (\ref{R1E1}) and (\ref{R1E2}). We do this in the steps below.\\

{\bf \raggedleft Step 3.} In this step we prove (\ref{R1E1}). We have that 
\begin{equation*}
\begin{split}
&\lim_{L \rightarrow \infty} \mathbb{E}_{X,Y} \left[ \left(\sum_{k = 0}^L \frac{u_2^k(1-t)^{-k}t^{-k\lambda_1'(n)}}{k_t!} \right) \cdot \frac{1}{(u_1 t^{-\lambda'_1(N)};t)_{\infty}}  \hspace{-1mm} \right] = \sum_{a = 0}^N \sum_{b = 0}^{n} \frac{\mathbb{P}_{X,Y} (\lambda_1'(N) = a, \lambda_1'(n) = b)}{(u_1 t^{-a};t)_{\infty}}  \times  \\
&  \left(\lim_{L \rightarrow \infty} \sum_{k = 0}^L \frac{u_2^k(1-t)^{-k}t^{-k b}}{k_t!} \right) =\sum_{a = 0}^N \sum_{b = 0}^{n} \frac{ \mathbb{P}_{X,Y} (\lambda_1'(N) = a, \lambda_1'(n) = b) }{(u_1 t^{-a};t)_{\infty}(u_2 t^{-b};t)_{\infty}},
\end{split}
\end{equation*}
where in the last equality we used \cite[Corollary 10.2.2a]{Andrews}, which ensures that for $|x| < 1$ 
$$\lim_{L \rightarrow \infty} \sum_{k = 0}^L \frac{x^k}{k_t!} = \frac{1}{((1-t)x; t)_{\infty}}.$$
Here we implicitly used that $|u_2| \leq (1-t)t^n$ so that $|u_2(1-t)^{-1}t^{- b}| < 1 $ for all $0 \leq b \leq n$. The above equation proves (\ref{R1E1}). \\

{\bf \raggedleft Step 4.} In the remaining steps we prove (\ref{R1E2}). We first notice that the integrals in the definition of $H^2_L(\vec{z}, \vec{w}, N_1, N_2) $ are invariant upon permuting the parts of $\lambda$. Since the number of ways to permute the latter is precisely $\frac{ N_2!}{\prod_{i = 1}^\infty m_i(\lambda)!}$ we conclude 
\begin{equation}\label{P4}
\begin{split}
&H^2_L(\vec{z}, \vec{w}, N_1, N_2) =  \int_{\gamma_{3}^{N_2}} \sum_{c_1 = 1}^\infty \cdots \sum_{c_{N_2} = 1}^\infty  {\bf 1}\{ c_1 + \cdots + c_{N_2} \leq L\} [u_2]^{\sum_{i = 1}^{N_2} c_i} \\
&\det \left[ \frac{1}{\hat{z}_j - \hat{z}_it^{c_i}}\right]_{i,j = 1}^{N_2}   \cdot   \prod_{i = 1}^{N_2} \prod_{j = 1}^M \frac{1 - \hat{z}_i y_j}{1 - \hat{z}_it^{c_i} y_j} \prod_{i = 1}^{N_2}  \prod_{ j = 1}^{n} \frac{1 - \hat{z}^{-1}_ix_jt^{-c_i}}{1 - \hat{z}^{-1}_i x_j} \cdot\\
&  \prod_{i = 1}^{N_1} \prod_{j = 1}^{N_2} \frac{(\hat{z}_j z_i^{-1}; t)_\infty }{(\hat{z}_j z_i^{-1} t^{c_j}; t)_\infty }\frac{(\hat{z}_j w_i^{-1} t^{c_j}; t)_\infty }{(\hat{z}_j w_i^{-1}; t)_\infty }  \prod_{i = 1}^{N_2} \frac{d\hat{z}_i}{2\pi \iota}.
\end{split}
\end{equation}
In this step we find the limit as $L \rightarrow \infty$ in (\ref{P4}) and obtain estimates for $H^1$ and $H^2_L$. \\

We first derive estimates for the functions in (\ref{P4}) that hold whenever $z_i \in \gamma_1$, $w_i \in \gamma_2$ for $i =1 , \dots, N_1$ and $\hat{z}_j \in \gamma_3$, $c_j \geq 1 $ for $j =1, \dots, N_2$. Observe that if $|\zeta| \leq \alpha < 1$ we have that 
\begin{equation}\label{S3PochBound}
(\alpha;t)_\infty = \prod_{n = 0}^\infty( 1 - \alpha t^n) \leq \prod_{n = 0}^\infty | 1 - \zeta t^n| = | (\zeta;t)_\infty | \leq \prod_{n = 0}^\infty( 1 + \alpha t^n)  = (-\alpha; t)_\infty.
\end{equation}
Combining (\ref{S3PochBound}) wtih $r_1 > r_2 > r_3 > r_4 > tr_1 $ we get
\begin{equation}\label{ET1}
 \left| \prod_{i = 1}^{N_2} \prod_{j = 1}^{N_1}  \frac{(\hat{z}_i z_j^{-1}; t)_\infty }{(\hat{z}_i z_j^{-1} t^{c_i}; t)_\infty }\frac{(\hat{z}_i w_j^{-1} t^{c_i}; t)_\infty }{(\hat{z}_i w_j^{-1}; t)_\infty }  \right| \leq \left(\frac{(-r_3/r_1; t)_{\infty} (-r_4/r_2; t)_{\infty} }{(r_4/r_1; t)_\infty (r_3/r_2; t)_{\infty}} \right)^{N_1N_2}.
\end{equation}
By Lemma \ref{DetBounds} we also have
$$\left|\det \left[\frac{1}{z_i - w_j} \right]_{i,j = 1}^{N_1} \right| \leq r_1^{-N_1} \cdot \frac{N_1^{N_1} (r_2/r_1)^{\binom{N_1}{2}}}{(1 - r_2/r_1)^{N_1^2}} \mbox{ and } \left| \det \left[ \frac{1}{\hat{z}_j- \hat{z}_it^{c_i}}\right]_{i,j = 1}^{N_2}   \right| \leq  r_3^{-N_2} \cdot \frac{N_2^{N_2} (r_4/r_3)^{\binom{N_2}{2} }}{(1 - r_4/r_3)^{N_2^2}} .$$
By Lemma \ref{S3BoundOnS} we can find a constant $A_1 > 0$ depending on $u_1, t, r_1, r_2$ such that 
$$\left|\prod_{i = 1}^{N_1}\frac{ S(w_i, z_i; u_1,t)}{[-\log t] \cdot w_i}\right| \leq  A_1^{N_1 }.$$
Finally, we can find constants $A_2, A_3 > 0$ depending on $X, Y, N,n,M$ such that 
$$ \left| \prod_{i = 1}^{N_2} \prod_{j = 1}^M \frac{1 - \hat{z}_i y_j}{1 - \hat{z}_it^{c_i} y_j} \prod_{i = 1}^{N_2}  \prod_{ j = 1}^{n} \frac{1 - \hat{z}^{-1}_ix_jt^{-c_i}}{1 - \hat{z}^{-1}_i x_j} \right| \leq A_2^{N_2} t^{- \sum_{i = 1}^{N_2} n c_i},$$
$$ \left|  \prod_{i = 1}^{N_1}\prod_{j = 1}^M \frac{1 - y_j z_i}{1 -y_j w_i} 
\prod_{i = 1}^{N_1}\prod_{j = 1}^N \frac{1 - x_jw_i^{-1}}{1 - x_j z_i^{-1}} \right| \leq A_3^{N_1} .$$
Combining the above estimates we conclude that 
\begin{equation}\label{P5}
\begin{split}
& |H^1(\vec{z}, \vec{w}, N_1)| \leq  r_1^{-N_1} \cdot \frac{N_1^{N_1} (r_2/r_1)^{\binom{N_1}{2} }}{(1 - r_2/r_1)^{N_1^2}} \cdot (A_1 A_3)^{N_1} \mbox{ and }\\
&\sum_{c_1 = 1}^\infty \cdots \sum_{c_{N_2} = 1}^\infty \Bigg|  [u_2]^{\sum_{i = 1}^{N_2} c_i} \det \left[ \frac{1}{\hat{z}_j - \hat{z}_i t^{c_i}}\right]_{i,j = 1}^{N_2}   \cdot   \prod_{i = 1}^{N_2} \prod_{j = 1}^M \frac{1 - \hat{z}_i y_j}{1 - \hat{z}_it^{c_i} y_j} \\
&\prod_{i = 1}^{N_2}  \prod_{ j = 1}^{n} \frac{1 - \hat{z}^{-1}_ix_jt^{-c_i}}{1 - \hat{z}^{-1}_i x_j} \cdot  \prod_{i = 1}^{N_1}  \prod_{j = 1}^{N_2}\frac{(\hat{z}_j z_i^{-1}; t)_\infty }{(\hat{z}_j z_i^{-1} t^{c_j}; t)_\infty }\frac{(\hat{z}_j w_i^{-1} t^{c_j}; t)_\infty }{(\hat{z}_j w_i^{-1}; t)_\infty } \Bigg| \leq \\
&r_3^{-N_2}\frac{ N_2^{N_2} \cdot (r_4/r_3)^{\binom{N_2}{2}}}{(1-r_4/r_3)^{N_2^2}} \cdot A_2^{N_2} \cdot  \left(\frac{(-r_3/r_1; t)_{\infty} (-r_4/r_2; t)_{\infty} }{(r_4/r_1; t)_\infty (r_3/r_2; t)_{\infty}} \right)^{N_1N_2} \sum_{c_1 = 1}^\infty \cdots \sum_{c_{N_2} = 1}^\infty  \prod_{i = 1}^{N_2} \epsilon^{ c_i} t^{-c_i n}   ,
\end{split}
\end{equation}
where $\epsilon = |u_2|$. Since $|u_2| <(1-t) t^n$ we see that the last sum in (\ref{P5}) converges as a geometric series. Consequently, the integrand in (\ref{P4}) is uniformly bounded in $L$ and converges to
\begin{equation*}
\begin{split}
&\sum_{c_1 = 1}^\infty \cdots \sum_{c_{N_2} = 1}^\infty   [u_2]^{\sum_{i = 1}^{N_2} c_i} \det \left[ \frac{1}{\hat{z}_j - \hat{z}_it^{c_i}}\right]_{i,j = 1}^{N_2}   \cdot   \prod_{i = 1}^{N_2} \prod_{j = 1}^M \frac{1 - \hat{z}_i y_j}{1 - \hat{z}_it^{c_i} y_j} \\
&\prod_{i = 1}^{N_2}  \prod_{ j = 1}^{n} \frac{1 - \hat{z}^{-1}_ix_jt^{-c_i}}{1 - \hat{z}^{-1}_i x_j} \cdot  \prod_{i = 1}^{N_1} \prod_{j = 1}^{N_2}  \frac{(\hat{z}_j z_i^{-1}; t)_\infty }{(\hat{z}_j z_i^{-1} t^{c_j}; t)_\infty }\frac{(\hat{z}_jw_i^{-1} t^{c_j}; t)_\infty }{(\hat{z}_j w_i^{-1}; t)_\infty }
\end{split}
\end{equation*}
as $L \rightarrow \infty$. From the bounded convergence theorem we conclude that 
\begin{equation}\label{P6}
\begin{split}
&\lim_{L \rightarrow \infty}H^2_L(\vec{z}, \vec{w}, N_1, N_2) =  \int_{\gamma_{3}^{N_2}} \sum_{c_1 = 1}^\infty \cdots \sum_{c_{N_2} = 1}^\infty   [u_2]^{\sum_{i = 1}^{N_2} c_i} \det \left[ \frac{1}{\hat{z}_j- \hat{z}_i t^{c_i}}\right]_{i,j = 1}^{N_2}   \cdot  \\
&  \prod_{i = 1}^{N_2} \prod_{j = 1}^M \frac{1 - \hat{z}_i y_j}{1 - \hat{z}_it^{c_i} y_j} \prod_{i = 1}^{N_2}   \prod_{ j = 1}^{n} \frac{1 - \hat{z}^{-1}_ix_jt^{-c_i}}{1 - \hat{z}^{-1}_i x_j}  \prod_{i= 1}^{N_1} \prod_{j = 1}^{N_2} \frac{(\hat{z}_jz_i^{-1}; t)_\infty }{(\hat{z}_j z_i^{-1} t^{c_j}; t)_\infty }\frac{(\hat{z}_j w_i^{-1} t^{c_j}; t)_\infty }{(\hat{z}_j w_i^{-1}; t)_\infty }  \prod_{i = 1}^{N_2} \frac{d\hat{z}_i}{2\pi \iota},
\end{split}
\end{equation}
and moreover, from (\ref{P5}), (\ref{RadiiCond}) and the compactness of $\gamma_3$ we have for any $L \in \mathbb{N}$
\begin{equation}\label{P7}
\begin{split}
&|H^1(\vec{z}, \vec{w}, N_1) H^2_L(\vec{z}, \vec{w}, N_1, N_2)| \leq  C_{\epsilon}^{N_1 + N_2} N_1^{N_1} N_2^{N_2} \cdot \rho^{N_1^2 + N_2^2}  ,
\end{split}
\end{equation}
where $C_{\epsilon}$ depends on $t,n, r_1, r_2, r_3, r_4, N, M, X, Y$ and $\epsilon = |u_2|$. \\

{\bf \raggedleft Step 5.} Let us denote the limit in (\ref{P6}) by $H^2_\infty(\vec{z}, \vec{w}, N_1, N_2)$. Using (\ref{P6}), (\ref{P7}) and a second application of the bounded convergence theorem we conclude that for each $N_1, N_2 \in \mathbb{Z}_{\geq 0}$ we have
\begin{equation}\label{P8}
\begin{split}
&\lim_{L \rightarrow \infty}H_L(u, N_1, N_2)   = \int_{\gamma_1^{N_1}} \int_{\gamma_2^{N_1}} H^1(\vec{z}, \vec{w},N_1)    \cdot  H^2_\infty(\vec{z}, \vec{w}, N_1,N_2)    \prod_{i = 1}^{N_1}\frac{dz_i}{2\pi \iota}  \prod_{i = 1}^{N_1} \frac{dw_i}{2\pi \iota}.
\end{split}
\end{equation}
Moreover, by the compactness of $\gamma_1, \gamma_2$ and (\ref{P7}) we conclude that 
\begin{equation}\label{P10}
\begin{split}
\frac{|H_L(u, N_1, N_2)| }{N_1! N_2!} \leq  C_{\epsilon}^{N_1 + N_2} \left(\frac{1 + \rho}{2} \right)^{N_1^2 + N_2^2} ,
\end{split}
\end{equation}
for some possibly bigger $C_\epsilon$ than before. Using (\ref{P10}) and the dominated convergence theorem we conclude that the limit on the left side of (\ref{R1E2}) exists and equals to
\begin{equation}\label{P11}
\begin{split}
\sum_{N_1 = 0}^\infty \sum_{N_2 = 0}^\infty \frac{1}{N_1! N_2!} \int_{\gamma_1^{N_1}} \int_{\gamma_2^{N_1}} H^1(\vec{z}, \vec{w},N_1)    \cdot  H^2_\infty(\vec{z}, \vec{w}, N_1,N_2)    \prod_{i = 1}^{N_1}\frac{dz_i}{2\pi \iota}  \prod_{i = 1}^{N_1} \frac{dw_i}{2\pi \iota},
\end{split}
\end{equation}
where we recall that $H^2_\infty(\vec{z}, \vec{w}, N_1, N_2)$ is the right side in (\ref{P6}). Comparing (\ref{P11}) with the right side of (\ref{R1E2}) (see also (\ref{S3HFunB})) we see that to show that they are equal it suffices to show that for each $N_1 \in \mathbb{Z}_{\geq 0}$ and $ N_2 \in \mathbb{N}$ we have 
\begin{equation}\label{RedFDL}
\begin{split}
&\sum_{c_1 = 1}^\infty \cdots \sum_{c_{N_2} = 1}^\infty   [u_2]^{\sum_{i = 1}^{N_2} c_i} \det \left[ \frac{1}{\hat{z}_j- \hat{z}_i t^{c_i}}\right]_{i,j = 1}^{N_2}  \prod_{i = 1}^{N_2} \prod_{j = 1}^M \frac{1}{1 - \hat{z}_it^{c_i} y_j} \prod_{i = 1}^{N_2}  \prod_{ j = 1}^{n} (1 - \hat{z}^{-1}_ix_jt^{-c_i}) \\
&\prod_{i= 1}^{N_1} \prod_{j= 1}^{N_2}   \frac{(\hat{z}_jw_i^{-1} t^{c_j}; t)_\infty}{(\hat{z}_j z_i^{-1} t^{c_j}; t)_\infty }= \int_{\gamma_4^{N_2}}  \det \left[ \frac{1}{\hat{z}_i - \hat{w}_j}\right]_{i,j = 1}^{N_2} \prod_{i = 1}^{N_2}\frac{ S(\hat{w}_i, \hat{z}_i; u_2,t)}{[-\log t] \cdot \hat{w}_i} \prod_{i = 1}^{N_2} \prod_{j = 1}^M \frac{1}{1 - \hat{w}_i y_j}  \\
&\prod_{i = 1}^{N_2}  \prod_{ j = 1}^{n} (1 - \hat{w}^{-1}_ix_j)\cdot \prod_{i = 1}^{N_1}  \prod_{j = 1}^{N_2} \frac{(\hat{w}_jw_i^{-1} ; t)_\infty  }{(\hat{w}_jz_i^{-1} ; t)_\infty }\prod_{i = 1}^{N_2} \frac{d\hat{w}_i}{2\pi \iota}.
\end{split}
\end{equation}

{\bf \raggedleft Step 6.} In Step 5 we reduced the proof of (\ref{R1E2}), which was the remaining statement we needed to show to prove the theorem, to establishing (\ref{RedFDL}). We prove (\ref{RedFDL}) in this final step by appealing to Lemma \ref{S3Expansion}.\\

Expanding the Cauchy determinants on both sides of (\ref{RedFDL}) we see that it suffices to show that for each $\sigma \in S_{N_2}$ we have that
\begin{equation}\label{RedFDL2}
\begin{split}
&\sum_{c_1 = 1}^\infty \cdots \sum_{c_{N_2} = 1}^\infty   [u_2]^{\sum_{i = 1}^{N_2} c_i} \prod_{i = 1}^{N_2} \frac{1}{\hat{z}_{\sigma(i)} - \hat{z}_i t^{c_i}}  \prod_{i = 1}^{N_2} \prod_{j = 1}^M \frac{1}{1 - \hat{z}_it^{c_i} y_j} \prod_{i = 1}^{N_2}  \prod_{ j = 1}^{n} (1 - \hat{z}^{-1}_ix_jt^{-c_i}) \\
&\prod_{i = 1}^{N_1} \prod_{j = 1}^{N_2}   \frac{(\hat{z}_j w_i^{-1} t^{c_j}; t)_\infty}{(\hat{z}_j z_i^{-1} t^{c_j}; t)_\infty }= \int_{\gamma_4^{N_2}} \prod_{i = 1}^{N_2} \frac{1}{\hat{z}_{\sigma(i)} - \hat{w}_i } \prod_{i = 1}^{N_2}\frac{ S(\hat{w}_i, \hat{z}_i; u_2,t)}{[-\log t] \cdot \hat{w}_i} \prod_{i = 1}^{N_2} \prod_{j = 1}^M \frac{1}{1 - \hat{w}_i y_j}  \\
&\prod_{i = 1}^{N_2}  \prod_{ j = 1}^{n} (1 - \hat{w}^{-1}_ix_j)\cdot \prod_{i= 1}^{N_1} \prod_{j = 1}^{N_2}   \frac{(\hat{w}_j w_i^{-1} ; t)_\infty  }{(\hat{w}_j z_i^{-1} ; t)_\infty }\prod_{i = 1}^{N_2} \frac{d\hat{w}_i}{2\pi \iota}.
\end{split}
\end{equation}
Notice that the left side of (\ref{RedFDL2}) is nothing but 
$$\prod_{j =1}^{N_2} \left[\sum_{c = 1}^\infty     \frac{[u_2]^{c}}{\hat{z}_{\sigma(j)} - \hat{z}_j t^{c}}   \prod_{i = 1}^M \frac{1}{1 - \hat{z}_jt^{c} y_i}   \prod_{ i = 1}^{n} (1 - \hat{z}^{-1}_jx_it^{-c}) \prod_{i = 1}^{N_1}  \frac{(\hat{z}_j w_i^{-1} t^{c}; t)_\infty}{(\hat{z}_j z_i^{-1} t^{c}; t)_\infty } \right],$$
while the right side is 
$$ \prod_{j = 1}^{N_2} \left[  \int_{\gamma_4}\frac{1}{\hat{z}_{\sigma(j)} - \hat{w}} \frac{ S( \hat{w}, \hat{z}_j; u_2,t)}{[-\log t] \cdot  \hat{w}}  \prod_{i = 1}^M \frac{1}{1 -  \hat{w}y_i}    \prod_{ i= 1}^{n} (1 - \hat{w}^{-1}x_i)\cdot \prod_{i = 1}^{N_1}  \frac{( \hat{w} w_i^{-1} ; t)_\infty  }{( \hat{w} z_i^{-1} ; t)_\infty } \frac{d \hat{w}}{2\pi \iota}\right].$$
Since $|u_2| < t^n$ and $u_2 \not \in [0, \infty)$ by assumption we see that Lemma \ref{S3Expansion} is applicable and from it we conclude that the two products in the last two equations agree term-wise. This proves the equality in (\ref{RedFDL2}), and hence the proof of the theorem is complete.
\end{proof}

%
\section{Two-point convergence}\label{Section4} In this section we prove Theorem \ref{thmHLMain}. In Section \ref{Section4.1} we summarize some notation and present several results that will be used in Section \ref{Section4.2} where the proof is given.

%
\subsection{Scaling regime}\label{Section4.1} We begin by summarizing some relevant notation in the following two definitions. The first definition details how we scale the parameters in Theorem \ref{thmHLMain} and the second definition gives the formula for the $(N_1, N_2)$ term in (\ref{S3HFun2}) for these scaled parameters.
\begin{definition}\label{constants} For $a \in (0,1)$ we recall from (\ref{S2eqnConst}) the constants
\begin{equation}\label{eqnConst}
\sigma_a = \frac{a^{1/3} \left(1 - a \right)^{1/3} }{1 + a }, \hspace{3mm} f_1=   \frac{2a}{1 + a}, \hspace{3mm} f_1'=  \frac{a}{1+a}, \hspace{3mm}  f_1'' = \frac{-a}{2  (1 - a^2)}.
\end{equation}
Let us fix $s_1, s_2, x_1, x_2 \in \mathbb{R}$ such that $s_1 > s_2$. For $M \in \mathbb{N}$ sufficiently large so that $M + s_2 M^{2/3} \geq 1$ we define $n_1(M), n_2(M) \in \mathbb{N}$ through
\begin{equation}\label{ScaleN}
n_1 = \lfloor M + s_1 M^{2/3} \rfloor \mbox{ and } n_2 = \lfloor M + s_2 M^{2/3} \rfloor.
\end{equation}
For $t \in (0,1)$ and $M$ as above we define the real numbers $u_1(t,M), u_2(t,M) $ through
\begin{equation}\label{ScaleU}
\begin{split}
&u_1 = - \exp \left( \log t \cdot M \cdot f_1 + \log t \cdot  M^{1/3}  \sigma_a  x_1 + \log t  \cdot  f_1'   [M - n_1] + \log t\cdot (1/2) f''_1 s_1^2 M^{1/3} \right),  \\
&u_2 = - \exp \left( \log t \cdot M \cdot f_1 + \log t \cdot  M^{1/3}  \sigma_a  x_2 + \log t  \cdot f_1'  [M - n_2] + \log t \cdot (1/2) f''_1 s_2^2 M^{1/3}   \right).
\end{split}
\end{equation}
\end{definition}

\begin{definition}\label{functions} Let $a, t \in (0,1)$ and $s_1, s_2, x_1, x_2 \in \mathbb{R}$ be such that $s_1 > s_2$. For $M \in \mathbb{N}$ sufficiently large so that $M + s_2 M^{2/3} \geq 1$ let $n_1, n_2, u_1, u_2$ be as in Definition \ref{constants}. Suppose that $r_1, \dots, r_4 \in (a, a^{-1})$ are such that $r_1 > r_2 > r_3 > r_4 > tr_1$. With this choice of parameters we define for any $N_1, N_2 \in \mathbb{Z}_{\geq 0}$ the numbers
\begin{equation}\label{B2}
\begin{split}
&I_M(N_1, N_2) = \frac{1}{N_1! N_2!}\int_{\gamma_1^{N_1}} \int_{\gamma_2^{N_1}}  \int_{\gamma_3^{N_2}}\int_{\gamma_4^{N_2}} T(\vec{w}, \vec{z}, n_1)  T(\vec{\hat{w}}, \vec{\hat{z}},n_2)     \\
& D(\vec{w}, \vec{z}) D(\vec{\hat{w}}, \vec{\hat{z}})    \cdot B(\vec{w}, \vec{z}, u_1)  B(\vec{\hat{w}}, \vec{\hat{z}}, u_2) \cdot  G(\vec{w}, \vec{z}; \vec{\hat{w}}, \vec{\hat{z}}) \prod_{i = 1}^{N_2}\frac{d\hat{w}_i}{2\pi \iota}\prod_{i = 1}^{N_2}\frac{d\hat{z}_i}{2\pi \iota}\prod_{i = 1}^{N_1}\frac{dw_i}{2\pi \iota}\prod_{i = 1}^{N_1}\frac{dz_i}{2\pi \iota},
\end{split}
\end{equation}
where $\gamma_i$ is a positively oriented circle of radius $r_i$ for $i =1 ,\dots, 4$ and 
\begin{equation}\label{MPLTGv2}
\begin{split}
& D(\vec{w}, \vec{z}) =   \det \left[ \frac{1}{z_i - w_j}\right]_{i,j = 1}^{N_1} \hspace{3mm} B(\vec{w}, \vec{z}, u) =\prod_{i = 1}^{N_1}\frac{ S(w_i, z_i; u,t )}{ -\log t \cdot w_i}  \\
&T(\vec{w}, \vec{z}, n) =  \prod_{j = 1}^{N_1} \left( \frac{1 - az_j}{1 - aw_j} \right)^M \prod_{j = 1}^{N_1} \left( \frac{1 - a/ w_j}{1 - a/ z_j}\right)^n \\
&G(\vec{w}, \vec{z}; \vec{\hat{w}}, \vec{\hat{z}}) =  \prod_{i = 1}^{N_1}\prod_{j = 1}^{N_2}  \frac{(\hat{z}_j z_i^{-1}; t)_\infty }{(\hat{w}_j z_i^{-1}; t)_\infty }\frac{(\hat{w}_j w_i^{-1} ; t)_\infty }{(\hat{z}_j w_i^{-1}; t)_\infty }.
\end{split}
\end{equation}
If $N_1 = N_2 = 0$ then we adopt the convention $I_M(0,0) = 1$.
\end{definition}
\begin{remark} For any $a,t \in (0,1)$ one can find radii $r_1, \dots, r_4$ as in Definition \ref{functions} by picking these to be very close to $1$. We also mention here that the integrand in (\ref{B2}) is continuous by Lemma \ref{S3Analyticity} and as the contours $\gamma_1, \dots, \gamma_4$ are compact we have that $I_M(N_1, N_2) $ are well-defined and finite for all $N_1, N_2, M \in \mathbb{N}$ and $a,t \in (0,1)$.  
\end{remark}
\begin{remark}\label{RemINM} The numbers $I_M(N_1, N_2)$ are nothing but the $(N_1, N_2)$ term in (\ref{S3HFun2}) when all the $x,y$ variables are set to $a$.  
\end{remark}

We next turn to formulating three key propositions and a useful probability lemma.

\begin{proposition}\label{PropTermLimit} Let $x_1, x_2, \tau_1, \tau_2 \in \mathbb{R}$ be given such that $\tau_1 > \tau_2$. Let $\Gamma_1, \Gamma_2, \Gamma_3, \Gamma_4$ be vertical contours in $\mathbb{C}$ that pass through the points $c_1, c_2, c_3, c_4$ respectively with $ c_1 > 0$, $c_3 > 0$, $0 > c_2$, $0 > c_4$, $c_2 + \tau_1  > c_3 + \tau_2$, that are oriented in the direction of increasing imaginary parts. For $N_1, N_2 \in \mathbb{Z}_{\geq 0}$ define 
\begin{equation}\label{ST0}
\begin{split}
& K(N_1, N_2) = \frac{(-1)^{N_1 + N_2}}{N_1! N_2!}\int_{\Gamma_1^{N_1}} \int_{\Gamma_2^{N_1}}  \int_{\Gamma_3^{N_2}}\int_{\Gamma_4^{N_2}} \prod_{i = 1}^{N_1} \exp (S_1(z_i) - S_1(w_i)) \cdot     \\
&\prod_{i = 1}^{N_2} \exp (S_2(\hat{z}_i) - S_2(\hat{w}_i))  \cdot \det \hat{D}   \cdot \prod_{i = 1}^{N_1} \frac{1}{z_i - w_i} \cdot \prod_{i = 1}^{N_2} \frac{1}{\hat{z}_i - \hat{w}_i} \cdot \prod_{i = 1}^{N_2}\frac{d\hat{w}_i}{2\pi \iota}\prod_{i = 1}^{N_2}\frac{d\hat{z}_i}{2\pi \iota}\prod_{i = 1}^{N_1}\frac{dw_i}{2\pi \iota}\prod_{i = 1}^{N_1}\frac{dz_i}{2\pi \iota},
\end{split}
\end{equation}
where
\begin{equation}\label{ST1}
S_1(z) = \frac{z^3}{3} - x_1 z \mbox{ and } S_2(z) = \frac{z^3}{3} - x_2 z,
\end{equation}
and $\hat{D}$ is a $(N_1 +N_2) \times (N_1 +N_2)$ matrix that has the block form $\hat{D} = \begin{bmatrix} \hat{D}_{11} & \hat{D}_{12} \\ \hat{D}_{21} & \hat{D}_{22} \end{bmatrix},$ with
\begin{equation}\label{HatBlockMatrixD}
\begin{split}
&\hat{D}_{11} = \left[ \frac{1}{z_i- w_j} \right]{\substack{i = 1, \dots, N_1  \\ j = 1, \dots, N_1}}, \hspace{2mm}  \hat{D}_{12} = \left[ \frac{1}{z_i - \hat{w}_j + \tau_1 - \tau_2}\right]{\substack{i = 1, \dots, N_1  \\ j = 1, \dots, N_2}}, \\
&  \hat{D}_{21} = \left[ \frac{1}{\hat{z}_i - w_j - \tau_1 + \tau_2}\right]{\substack{i = 1, \dots, N_2  \\ j = 1, \dots, N_1}}, \hspace{2mm} \hat{D}_{22} = \left[ \frac{1}{\hat{z}_i - \hat{w}_j} \right]{\substack{i = 1, \dots, N_2  \\ j = 1, \dots, N_2}.}
\end{split}
\end{equation}
If $N_1 = N_2 = 0$ we use the convention $K(N_1, N_2)  = 1$. Then the integrand in (\ref{ST0}) is absolutely integrable so that $K(N_1, N_2)$ is well-defined and moreover, there exists a constant $C > 0$ depending on $x_1, x_2, \tau_1, \tau_2$ such that for all $N_1, N_2 \geq 0$
\begin{equation}\label{ST2}
\begin{split}
&\left| K(N_1, N_2) \right| \leq C^{N_1 + N_2} \cdot (N_1 +N_2)^{-(N_1 + N_2)/2}, 
\end{split}
\end{equation}
with the convention $0^0 = 1$. The series 
\begin{equation}\label{ST3}
\begin{split}
Q:=  \sum_{N_1, N_2 = 0}^\infty  K(N_1, N_2) 
\end{split}
\end{equation}
is absolutely convergent and satisfies the equality $Q = \det \left(I - fA f \right)_{L^2(\{\tau_1, \tau_2\} \times \mathbb{R})}$, where the latter Fredholm determinant is as in (\ref{FiniteDimAiry}).
\end{proposition}
\begin{remark} Proposition \ref{PropTermLimit} gives an alternative representation of $\det \left(I - fA f \right)_{L^2(\{\tau_1, \tau_2\} \times \mathbb{R})}$, which will be more useful in our asymptotic analysis. The proposition is proved in Section \ref{Section7.2}.
\end{remark}

\begin{proposition}\label{PropTermConv} Let $a, t \in (0,1)$ and $s_1, s_2, x_1, x_2 \in \mathbb{R}$ be such that $s_1 > s_2$. For $M \in \mathbb{N}$ sufficiently large so that $M + s_2 M^{2/3} \geq 1$ and $N_1, N_2 \in \mathbb{Z}_{\geq 0}$ let $I_M(N_1, N_2)$ be as in Definition \ref{functions}. Then 
\begin{equation}\label{TermConvEq}
\lim_{M \rightarrow \infty}I_M(N_1, N_2) = K(N_1, N_2) ,
\end{equation}
where $K(N_1, N_2)$ is as in Proposition \ref{PropTermLimit} for $x_1, x_2$ as above and $\tau_1 = \frac{s_1a^{1/3}}{2(1-a)^{2/3}}$, $\tau_2 = \frac{s_2a^{1/3}}{2(1-a)^{2/3}}$.
\end{proposition}

\begin{proposition}\label{PropTermBound} There exist $a_2, t_2  \in (0,1)$ such that the following holds. For any $a \in (0, a_2]$, $t \in (0,t_2]$ and $s_1, s_2, x_1, x_2 \in \mathbb{R}$ such that $s_1 > s_2$ there exist constants $ \tilde{C}_2 > 0$ and $M_2 \in \mathbb{N}$ (depending on $a,t,x_1, x_2, s_1, s_2$) such that for any $N_1, N_2 \in \mathbb{Z}_{\geq 0}$ and $M \geq M_2$ we have
\begin{equation}\label{TermBoundEq}
\left|I_M(N_1, N_2)\right| \leq \tilde{C}_2^{N_1 +N_2} \cdot\exp \left( - \frac{1}{4} (N_1 +N_2) \log (N_1 + N_2) \right) ,
\end{equation}
where $I_M(N_1, N_2)$ is as in Definition \ref{functions}.
\end{proposition}

As mentioned in Remark \ref{RemINM}, $I_M(N_1, N_2)$ is the $(N_1, N_2)$ term in the double sum (\ref{S3HFun2}) and Proposition \ref{PropTermConv} states that termwise  (\ref{S3HFun2}) converges to (\ref{ST3}). Proposition \ref{PropTermBound} gives uniform bounds on the summands in (\ref{S3HFun2}) that, after an application of the dominated convergence theorem, improves the termwise convergence to convergence of the full sums. Propositions \ref{PropTermConv} and \ref{PropTermBound} form the heart of our asymptotic analysis and they are proved in Sections \ref{Section5} and \ref{Section6} respectively. 

We end this section with the following probability lemma, which is a two-point analogue of \cite[Lemma 4.39]{BorCor}. This lemma appeared without proof in \cite[Lemma 4.2]{NZ} and for the sake of completeness we give its proof in Section \ref{Section7.2}. We mention that analogues of the below lemma have been known for a while in the physics literature, see e.g. \cite[Equation (14)]{CDR10}.
\begin{lemma}\label{ProbLemma}
Suppose that $f_n: \mathbb{R} \rightarrow [0,1]$ is a sequence of functions. Assume that for each $\delta > 0$ one has on $\mathbb{R}\backslash [-\delta,\delta]$, $f_n \rightarrow {\bf 1}_{\{y < 0\}}$ uniformly. Let $(X_n,Y_n)$ be a sequence of random vectors such that for each $x,y  \in \mathbb{R}$ 
$$\mathbb{E}[f_n(X_n - x)f_n(Y_n - y)] \rightarrow p(x,y),$$
and assume that $p(x,y)$ is a continuous probability distribution function on $\mathbb{R}^2$. Then $(X_n,Y_n)$ converges in distribution to a random vector $(X,Y)$, such that $\mathbb{P}(X < x, Y < y) = p(x,y)$.
\end{lemma}

%
\subsection{Proof of Theorem \ref{thmHLMain}}\label{Section4.2} For clarity we split the proof into three steps. In the first step we specify the choice of parameters $a^*$ and $t^*$ in the statement of the theorem. In the second step we apply Lemma \ref{ProbLemma} and Proposition \ref{PropTermLimit} to reduce the problem to showing a certain limiting identity, which is proved in the third step using Theorem \ref{PrelimitT} and Propositions \ref{PropTermConv} and \ref{PropTermBound} . \\

{\bf \raggedleft Step 1.} For $R > 1$ put $r_1 = R^{3}, r_2 = R, r_3 = R^{-1}, r_4 = R^{-3}$ and observe that 
$$\max \left(\frac{\sqrt{r_2/r_1}}{(1 - r_2/r_1)}, \frac{\sqrt{r_4/r_3}}{(1 - r_4/r_3)} \right)  \frac{(-r_3/r_1;1/2)_\infty (-r_4/r_2;1/2)_\infty}{(r_4/r_1;1/2)_\infty (r_3/r_2;1/2)_\infty}  =  \frac{(R-R^{-1})^{-1}(-R^{-4};1/2)^2_\infty }{(R^{-6};1/2)_\infty (R^{-2};1/2)_\infty}.$$
By making $R$ large enough we can ensure that the above is less than $1/2$. We now pick $a_1 \in (0,1)$ sufficiently small so that $a_1 < R^{-3}$ and $t_1$ small enough so that $ t_1 < R^{-6}$ and $t_1 \leq 1/2$. With this choice we observe that any $a \in (0, a_1]$ and $t \in (0,t_1]$ is good in the sense of Definition \ref{DefNiceRange} for $r_i$ as above and $\rho = 1/2$. 

If $a_2, t_2$ are as in Proposition \ref{PropTermBound} we let $a^* = \min(a_1, a_2)$ and $t^* = \min(t_1, t_2)$. This fixes our choice of parameters.\\

{\bf \raggedleft Step 2.} In the rest of the proof we fix $a \in (0, a^*]$ and $t \in (0, t^*]$ where $a^*, t^*$ are as in Step 1. From \cite[Lemma 5.1]{FerVet} we have that 
\begin{equation}\label{seqfeq}
h_q(y) := \frac{1}{(- t^{-qy}; t)_\infty} = \prod_{k = 1}^\infty \frac{1}{1 + t^{-qy + k}}
\end{equation}
is strictly decreasing for all $q > 0$. Moreover, for each $\delta > 0$ one has $h_q(y) \rightarrow {1}_{\{y < 0\}}$ uniformly on $\mathbb{R} \backslash [-\delta, \delta]$ as $q \rightarrow \infty$. For $x \in \mathbb{R}$ define 
$$f_M(x) =\frac{1}{(- t^{-\sigma_aM^{1/3}y}; t)_\infty}$$
and observe that by our result for $h_q$ we know that $f_M$ satisfy the conditions of Lemma \ref{ProbLemma}. Consequently, from that lemma we see that it suffices to prove for $X_M, Y_M$ as in Theorem \ref{thmHLMain} that
\begin{equation}\label{S4R1}
\lim_{M \rightarrow \infty} \mathbb{E}_{a,t}^{N,M} \left[ f_M(X_M - x_1) f_M(Y_M - x_2)  \right] = \mathbb{P} \left(A (\tau_1) \leq x_1, A(\tau_2) \leq x_2 \right).
\end{equation}
By definition of $u_1, u_2$ in (\ref{ScaleU}) we have that 
$$\mathbb{E}_{a,t}^{N,M} \left[ f_M(X_M - x_1) f_M(Y_M - x_2)  \right]  = \mathbb{E}_{a,t}^{N,M} \left[ \frac{1}{(u_1 t^{-\lambda'_1(n_1)};t)_{\infty}}\frac{1}{(u_2 t^{-\lambda'_1(n_2)};t)_{\infty}}\right].$$
Combining the last equality and Proposition \ref{PropTermLimit} we see that to prove (\ref{S4R1}) it suffices to show that 
\begin{equation}\label{S4R2}
\lim_{M \rightarrow \infty} \mathbb{E}_{a,t}^{N,M} \left[ \frac{1}{(u_1 t^{-\lambda'_1(n_1)};t)_{\infty}} \cdot \frac{1}{(u_2 t^{-\lambda'_1(n_2)};t)_{\infty}}   \right] = \sum_{N_1, N_2 = 0}^\infty  K(N_1,N_2).
\end{equation}

{\bf \raggedleft Step 3.} In this step we prove (\ref{S4R2}). From Lemma \ref{LprojectAHLP} we know that 
$$ \mathbb{E}_{a,t}^{N,M} \left[  \frac{1}{(u_1 t^{-\lambda'_1(n_1)};t)_{\infty}} \cdot \frac{1}{(u_2 t^{-\lambda'_1(n_2)};t)_{\infty}} \right]  =   \mathbb{E}_{a,t}^{n_1,M} \left[ \frac{1}{(u_1 t^{-\lambda'_1(n_1)};t)_{\infty}} \cdot \frac{1}{(u_2 t^{-\lambda'_1(n_2)};t)_{\infty}} \right].$$
By our choice of $a \in (0, a^*]$ and $t \in (0, t^*]$ we know that $a,t$ are good in the sense of Definition \ref{DefNiceRange}. Thus by Theorem \ref{PrelimitT}  applied to $N = n_1, n = n_2$, $x_i = a$ for $i = 1, \dots, n_1$ and $y_i = a$ for $i = 1, \dots, M$ we have that 
\begin{equation}\label{S4R3}
\mathbb{E}_{a,t}^{N,M} \left[\frac{1}{(u_1 t^{-\lambda'_1(n_1)};t)_{\infty}} \cdot \frac{1}{(u_2 t^{-\lambda'_1(n_2)};t)_{\infty}} \right] = \sum_{N_1, N_2 =0}^{\infty}I_M(N_1, N_2) ,
\end{equation}
where $I_M(N_1, N_2)$ are as in (\ref{B2}). By Proposition \ref{PropTermConv} we know that 
$$\lim_{M \rightarrow \infty}I_M(N_1, N_2) = K(N_1, N_2) ,$$
and by Proposition \ref{PropTermBound} and our choice of $a \in (0,a^*]$, $t \in (0, t^*]$ we know that for all large $M$
$$|I_{M}(N_1, N_2)| \leq \tilde{C}_2^{N_1 +N_2} \cdot\exp \left( - \frac{1}{4} (N_1 +N_2) \log (N_1 + N_2) \right) ,$$
where $ \tilde{C}_2$ is as in the statement of the proposiiton. We may thus take he limit as $M \rightarrow \infty$ in (\ref{S4R3}) and by the dominated convergence theorem with dominating series 
$$\tilde{C}_2^{N_1 +N_2} \cdot\exp \left( - \frac{1}{4} (N_1 +N_2) \log (N_1 + N_2) \right) $$
we conclude (\ref{S4R2}). This suffices for the proof.

%

\section{Asymptotic analysis: Part I}\label{Section5} The purpose of the present section is to prove Proposition \ref{PropTermConv}. In Section \ref{Section5.1} we summarize various results that will be used in the proof, which is given in Section \ref{Section5.2}. In what follows we continue with the same notation as in Section \ref{Section4.1}.

%
\subsection{Preliminary results}\label{Section5.1} In this section we introduce some notation and various results, which will be used in the proof of Proposition \ref{PropTermConv} in Section \ref{Section5.2}. After each result we give a brief explanation as well as a reference of where in the paper it is proved.

For $a \in (0,1)$ we define the functions
\begin{equation}\label{FunExp}
\begin{split}
S_a(z) = \log (1 + ae^z) - \log (1 + ae^{-z}) - \frac{2az}{1+a};  \hspace{2mm} R_a(z) = \log (1 + ae^{-z}) - \log (1 + a) + \frac{az}{1 + a}.
\end{split}
\end{equation}

\begin{lemma}\label{descentLemma} Let $a \in (0,1)$ be given. Then there exists $A_0 \in (0,1)$ depending on $a$ such that $S_a(z)$ and $R_a(z)$ are analytic in the vertical strip $\{z \in \mathbb{C}: |Re(z)| < \pi A_0\}$ and for any $A \in (0, A_0)$ and $\epsilon \in \{-1, 1\}$ we have that
\begin{equation}\label{K1}
\frac{d}{dy}Re [S_{a}(Ay + \epsilon \iota y)] \leq 0 , \hspace{2mm}\mbox{ and } \frac{d}{dy}Re[S_{a}(-Ay + \epsilon \iota y)] \geq 0 \mbox{ for all } y\in \left[0,\pi \right].
\end{equation}
Furthermore, we have
\begin{equation}\label{K3}
\frac{d}{dy}Re [R_{a}( \epsilon \iota y)] \leq 0 \mbox{ for $y \in [0, \pi]$ and $\epsilon \in \{-1, 1\}$ }.
\end{equation}
\end{lemma}
\begin{remark}\label{S5R1} In plain words, the lemma states that the real part of $S_a$ decreases or increases along straght segments started from the origin, that are close enough (depending on $a$) to being vertical. It also says that the real part of $R_a$ decreases as one moves away from the origin vertically up or down. The proof of Lemma \ref{descentLemma} can be found in Section \ref{Section8} where it is recalled as Lemma \ref{S8descentLemma}.
\end{remark}

We next define a collection of contours that depend on $a, t \in (0,1)$. 
\begin{definition}\label{Defcontours} Let $a, t \in (0,1)$ be given. Then we can find a small $A = A(a,t) > 0$ such that 
\begin{enumerate}
\item $A < A_0$ as in Lemma \ref{descentLemma} for the given $a$;
\item $\tan^{-1} (A^{-1}) > \pi/3$;
\item $ A\pi < - \log t/20$.
\end{enumerate}
With this choice of $A$ we define the contours 
$$\gamma_W= \{ -A|y| + \iota y: y \in I\} \mbox{, } \gamma_{mid} = \{ \iota y: y \in I\}, \mbox{ and } \gamma_Z = \{A|y| + \iota y: y \in I\}, \mbox{ where }I = \left[ -\pi,\pi\right].$$
The orientation is determined from $y$ increasing in $I$. Furthermore, if $M \in \mathbb{N}$ is sufficiently large so that $M^{-1/3} \leq A \pi $ we define the contour 
$\gamma^M_Z$ to be the contour consisting of three straight segments connecting $A\pi - \iota \pi$ to $M^{-1/3} - \iota \pi M^{-1/3} A^{-1}$; $M^{-1/3} - \iota \pi M^{-1/3} A^{-1}$ to $M^{-1/3}+\iota \pi M^{-1/3} A^{-1}$ and $M^{-1/3} + \iota \pi M^{-1/3} A^{-1}$ to $A\pi + \iota \pi$ and is oriented to have increasing imaginary part. 

We also define the contours $C_{in}, C^M_{out}, C_{mid}$ as the positively oriented contours obtained from $\gamma_W$, $\gamma_Z^M$, and $\gamma_{mid}$ under the map $x \rightarrow e^x$. Observe that $C_{in}, C^M_{out}, C_{mid}$  are piecewise smooth contours, that enclose $0$; $C_{in}$ and $C_{mid}$ are both contained in the interior of $C^M_{out}$. All three contours are contained in the zero-centered annulus of inner radius $t^{1/20}$ and outer radius $t^{-1/20}$. See Figure \ref{S5_1}. 

Finally, we let $\gamma_w, \gamma_m, \gamma_z$ be the contours that are obtained from $\gamma_W, \gamma_{Mid}$ and $\gamma_Z^M$ by multiplication by $M^{1/3}$ and have been extended linearly outside of the disc of radius $\pi \cdot \sqrt{1 + A^2} \cdot M^{1/3}$.  Observe that the latter contours no longer depend on $M$.
\end{definition}

\begin{figure}[h]
\centering
\scalebox{0.6}{\includegraphics{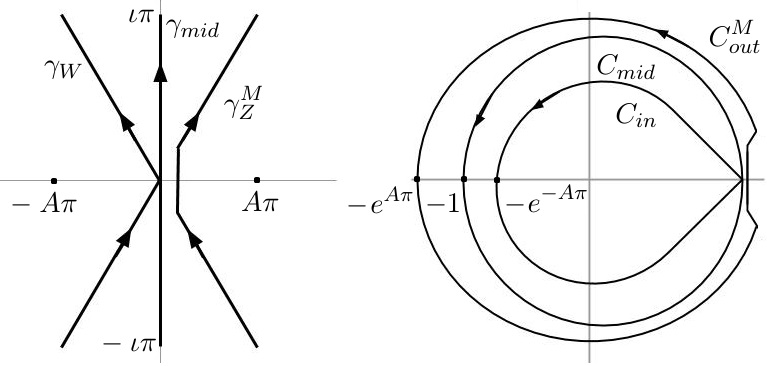}}
\caption{The figure depicts the contours $\gamma_W, \gamma_Z^M$ and $\gamma_{mid}$ on the left and $C_{in}, C^M_{out}$ and $C_{mid}$ on the right.
}
\label{S5_1}
\end{figure}

\begin{lemma}\label{LemmaTaylor} Let $a, t \in (0,1)$ be given and suppose that $A(a,t)$ is as in Definition \ref{Defcontours}. There exist $\epsilon_0 \in (0,1)$ and $C_0 > 0$ such that for all $|z| \leq \epsilon_0$ we have
\begin{equation}\label{taylorE1}
\begin{split}
&\left|S_a(z) -\frac{a(1-a)}{3(1 + a)^3} \cdot z^3 \right| \leq C_0 |z|^5 \mbox{, and } \left| R_a(z) - \frac{a}{2(1+a)^2} \cdot  z^2  \right| \leq C_0 |z|^3.
\end{split}
\end{equation}
Furthermore, there is a constant $\epsilon_1 > 0$ such that 
\begin{equation}\label{taylorE2}
\begin{split}
&Re[S_a(z)] \geq \epsilon_1 |z|^3  \mbox{ if } z \in \gamma_W, \hspace{3mm} Re[S_a(z)] \leq - \epsilon_1 |z|^3 \mbox{ if } z \in \gamma_Z \mbox{, }\\
& \mbox{ and } Re[R_a(z)] \leq -\epsilon_1 |z|^2 \mbox{ if $z \in \gamma_{mid}$},
\end{split}
\end{equation}
where $\gamma_Z, \gamma_W, \gamma_{mid}$ are as in Definition \ref{Defcontours}.
\end{lemma}
\begin{remark}\label{S5R2} In plain words, the lemma states that in a neighborhood of the origin $S_a(z)$ looks like $\frac{a(1-a)}{3(1 + a)^3}\cdot z^3$ and $R_a(z)$ looks like $ \frac{a}{2(1+a)^2} \cdot z^2$, which one observes by Taylor expanding the two functions. The second part of the lemma shows that $\gamma_Z$ and $\gamma_W$ are descent contours for the functions $S_a(z)$ and $\gamma_{mid}$ is a descent contour for $R_a(z)$ and estimates the speed of the decay of these functions along these contours. The proof of Lemma \ref{LemmaTaylor} can be found in Section \ref{Section8} where it is recalled as Lemma \ref{S8LemmaTaylor}.
\end{remark}

\begin{lemma}\label{techies}
Let $t, u, U \in (0,1)$ be given such that $0 < u < U < \min ( 1 , -\log t/10)$. Suppose that $z,w \in \mathbb{C}$ are such that $Re(w) \in [- U, 0]$, $Re(z) \in [u, U]$. Then there exists a constant $C > 0$, depending on $t$ such that the following hold
\begin{equation}\label{yellow1}
\left| \frac{1}{e^z - e^w}\right| \leq Cu^{-1} \mbox{ and } \sum_{k \in \mathbb Z} \left| \frac{1}{\sin(-\pi [ [w- z][\log t]^{-1} -  2\pi\iota k [\log t]^{-1}] )}\right| \leq Cu^{-1}.
\end{equation}
\end{lemma}
\begin{remark}\label{S5R3} In plain words, the lemma provides estimates for two functions when $z,w$ are complex numbers in a vertical strip of width $2U$ centered at the imaginary axis, and the real parts of $z,w$ are bounded away from each other by a small parameter $u > 0$. A proof of the lemma can be found in \cite[Lemma 4.5]{ED}.
\end{remark}

In the proof of Proposition \ref{PropTermConv} in the next section we will need to deform the contours $\gamma_1, \dots, \gamma_4$ in the definition of $I_M(N_1,N_2)$ (see (\ref{B2})) to the descent contours $C_{in}$ and $C^M_{out}$ from Definition \ref{Defcontours}. In the process of this deformation one needs to deform the $\gamma_3$ contours past the $\gamma_2$ contours, which introduces residues in the formula for $I_M(N_1,N_2)$. The following lemma details the residue structure that follows from this contour deformation.
\begin{lemma}\label{LemmaSwap1} Suppose that $N_1, N_2 \in \mathbb{N}$. Suppose that $A(r,R)$ is an annulus with inner radius $r$ and outer radius $R$, which is centered at the origin. Suppose that $C_i$ are positively oriented circles with radii $R_i$ for $i = 1, \dots, 4$ such that $R > R_1 > \cdots > R_4 > r$. Furthermore let $G(z,w), G_1(z,w)$ and $G_2(z,w)$ be functions that are jointly continuous in $A(r,R)$ and for a fixed $w \in A(r,R)$ are analytic in $z$ in $A(r,R)$ and for a fixed $z \in A(r,R)$ are analytic in $w \in A(r,R)$. Moreover, $G(z,w)$ is non-vanishing as $z,w$ vary over $A(r,R)$. Define with the above data
\begin{equation}
\begin{split}
& F(N_1, N_2) := \frac{1}{N_1! N_2!}  \oint_{C_{1}^{N_1}} \oint_{C^{N_1}_2}  \oint_{C_3^{N_2}}\oint_{C^{N_2}_{4}} \det D\cdot \prod_{i = 1}^{N_1} \prod_{j = 1}^{N_2}  \frac{G(\hat{z}_j, z_i) G(\hat{w}_j, w_i)}{G(\hat{w}_j, z_i) G(\hat{z}_j, w_i)} \cdot    \\
&  \prod_{i =1}^{N_2} G_1(\hat{z}_i,\hat{w}_i) \prod_{i =1}^{N_1} G_2(z_i,w_i) \cdot \prod_{i = 1}^{N_2}\frac{d\hat{w}_i}{2\pi \iota}\prod_{i = 1}^{N_2}\frac{d\hat{z}_i}{2\pi \iota}\prod_{i = 1}^{N_1}\frac{dw_i}{2\pi \iota}\prod_{i = 1}^{N_1}\frac{dz_i}{2\pi \iota},
\end{split}
\end{equation}
where $D$ is a $(N_1 +N_2) \times (N_1 +N_2)$ matrix that has the block form $D = \begin{bmatrix} {D}_{11} & {D}_{12} \\ {D}_{21} & {D}_{22} \end{bmatrix},$ with
\begin{equation}\label{BlockMatrixD}
\begin{split}
&{D}_{11} = \left[ \frac{1}{z_i- w_j} \right]{\substack{i = 1, \dots, N_1  \\ j = 1, \dots, N_1}}, \hspace{2mm}  {D}_{12} = \left[ \frac{1}{z_i - \hat{w}_j }\right]{\substack{i = 1, \dots, N_1  \\ j = 1, \dots, N_2}}, \\
&  {D}_{21} = \left[ \frac{1}{\hat{z}_i - w_j}\right]{\substack{i = 1, \dots, N_2  \\ j = 1, \dots, N_1}}, \hspace{2mm} {D}_{22} = \left[ \frac{1}{\hat{z}_i - \hat{w}_j} \right]{\substack{i = 1, \dots, N_2  \\ j = 1, \dots, N_2}.}
\end{split}
\end{equation}
Then we have $ F(N_1, N_2) = \sum_{k = 0}^{\min(N_1, N_2)} F(N_1, N_2,k)$, where 
\begin{equation}\label{FN1P1}
\begin{split}
&F(N_1, N_2,k) = \frac{1}{k! (N_1- k)! (N_2- k)!}  \oint_{C_{1}^{N_1-k}} \oint_{C^{N_1-k}_3}  \oint_{C_2^{N_2-k}}\oint_{C^{N_2-k}_{4}} \oint_{C^{k}_1}\oint_{C^{k}_3} \oint_{C^k_4}\\
& \det B \cdot  F_1 F_2  \prod_{i = 1}^{k}\frac{d\hat{w}^2_i}{2\pi \iota}\prod_{i = 1}^{k}\frac{d{w}^2_i}{2\pi \iota}\prod_{i = 1}^{k}\frac{dz^2_i}{2\pi \iota}\prod_{i = 1}^{N_2 - k}\frac{d\hat{w}^1_i}{2\pi \iota} \prod_{i = 1}^{N_2 - k} \frac{d\hat{z}^1_i}{2\pi \iota} \prod_{i = 1}^{N_1 - k}\frac{dw^1_i}{2\pi \iota}\prod_{i = 1}^{N_1- k}\frac{dz^1_i}{2\pi \iota}.
\end{split}
\end{equation}
In the above formula we have that $B$ is a $(N_1 +N_2 - k) \times (N_1 + N_2 - k)$ matrix that has the block form $B = \begin{bmatrix} B_{11} & B_{12} & B_{13} \\ B_{21} & B_{22} & B_{23} \\ B_{31} & B_{32} & B_{33} \end{bmatrix}$ with blocks given by
\begin{equation}\label{BlockMatrix}
\begin{split}
&B_{11} = \left[ \frac{1}{z_i^1 - w_j^1} \right]{\substack{i = 1, \dots, N_1 - k \\ j = 1, \dots, N_1- k}}, \hspace{2mm} B_{12} = \left[ \frac{1}{z^1_i - \hat{w}_j^1}\right]{\substack{i = 1, \dots, N_1 - k \\ j = 1, \dots, N_2- k}} , \hspace{2mm} B_{13} =\left[ \frac{1}{z^1_i - \hat{w}_j^2}\right]{\substack{i = 1, \dots, N_1 - k \\ j = 1, \dots,k}} \\
&B_{21} = \left[ \frac{1}{\hat{z}^1_i - w_j^1}\right]{\substack{i = 1, \dots, N_2 - k \\ j = 1, \dots, N_1- k}},  \hspace{2mm}  B_{22} = \left[ \frac{1}{\hat{z}^1_i - \hat{w}^1_j} \right]{\substack{i = 1, \dots, N_2 - k \\ j = 1, \dots, N_2- k}}, \hspace{2mm} B_{23} = \left[ \frac{1}{\hat{z}^1_i - \hat{w}^2_j} \right]{\substack{i = 1, \dots, N_2 - k\\ j = 1, \dots, k}} \\
&  B_{31} =  \left[ \frac{1}{z^2_i - w_j^1}\right]{\substack{i = 1, \dots, k \\ j = 1, \dots, N_1- k}},  \hspace{2mm} B_{32} = \left[ \frac{1}{z^2_i - \hat{w}^1_j} \right]{\substack{i = 1, \dots, k\\ j = 1, \dots, N_2 - k}},  \hspace{2mm} B_{33} = \left[ \frac{1}{z_i^2 - \hat{w}_j^2}\right]{\substack{i = 1, \dots, k\\ j = 1, \dots, k}}.
\end{split}
\end{equation}
The functions $F_1, F_2$ are given by
\begin{equation}\label{FN1P3}
\begin{split}
F_1 = &\prod_{i= 1}^{N_1-k} \prod_{j= 1}^{N_2-k} \frac{G(\hat{z}^1_j, z^1_i) G(\hat{w}^1_j, w^1_i)}{G(\hat{w}^1_j, z^1_i) G(\hat{z}^1_j, w^1_i)} \cdot \prod_{i= 1}^{k} \prod_{j = 1}^{N_2-k}  \frac{G(\hat{z}^1_j, z^2_i) G(\hat{w}^1_j, w^2_i)}{G(\hat{w}^1_j, z^2_i) G(\hat{z}^1_j, w^2_i)} \cdot \\
&  \prod_{i = 1}^{N_1-k}  \prod_{j = 1}^{k} \frac{G(w^2_j, z^1_i) G(\hat{w}^2_j, w^1_i)}{G(\hat{w}^2_j, z^1_i) G(w^2_j, w^1_i)} \cdot \prod_{i = 1}^{k} \prod_{j= 1}^{k} \frac{G(w^2_j, z^2_i) G(\hat{w}^2_j, w^2_i)}{G(\hat{w}^2_j, z^2_i) G(w^2_j, w^2_i)};
\end{split}
\end{equation}
\begin{equation}\label{FN1P4}
\begin{split}
&F_2 = \prod_{i =1}^{N_2-k} G_1(\hat{z}^1_i,\hat{w}^1_i) \cdot \prod_{i =1}^{k} G_1(w^2_i,\hat{w}^2_i)  \cdot \prod_{i =1}^{N_1-k} G_2(z^1_i,w^1_i) \cdot \prod_{i =1}^{k} G_2(z^2_i,w^2_i) .
\end{split}
\end{equation}
\end{lemma}
\begin{remark}The proof of Lemma \ref{LemmaSwap1} can be found in Section \ref{Section8}, see Lemma \ref{S8LemmaSwap1}.
\end{remark}

The last lemma in this section details what happens when we perform the analogous contour deformation in Lemma \ref{LemmaSwap1} to the contours $\Gamma_1, \Gamma_2, \Gamma_3, \Gamma_4$ in the definition of $K(N_1, N_2)$ in (\ref{ST0}), i.e. when we deform $\Gamma_3$ past $\Gamma_2$. 
\begin{lemma}\label{LemmaSwap2} Let $x_1, x_2, \tau_1, \tau_2 \in \mathbb{R}$ be given such that $\tau_1 > \tau_2$. Let $\Gamma_1, \Gamma_2, \Gamma_3, \Gamma_4$ and $K(N_1, N_2)$ be as in Lemma \ref{PropTermLimit}. Then we have $ K(N_1, N_2) = \sum_{k = 0}^{\min(N_1, N_2)} K(N_1, N_2,k)$, where 
\begin{equation}\label{KN1P1}
\begin{split}
&K(N_1, N_2,k) = \frac{(-1)^{N_1 + N_2}}{k! (N_1- k)! (N_2- k)!}   \int_{\Gamma_{1}^{N_1-k}} \int_{(\Gamma_3 + \tau_2 - \tau_1)^{N_1-k}}  \int_{(\Gamma_2 + \tau_1 -\tau_2)^{N_2-k}}   \\
& \int_{\Gamma^{N_2-k}_{4}} \int_{\Gamma^{k}_1}\int_{(\iota \mathbb{R})^k} \int_{\Gamma^k_4}   \det \hat{B} \prod_{i =1}^{N_2-k} \frac{\exp (S_2(\hat{z}^1_i) - S_2(\hat{w}^1_i))}{ \hat{z}_i^1 - \hat{w}^1_i } \cdot \prod_{i =1}^{k} \frac{\exp (S_1({z}^2_i) - S_2(\hat{w}^2_i) + S_3(w_i^2) )}{ (z_i^2 - w_i^2 + \tau_1)(w_i^2 - \hat{w}_i^2 - \tau_2) }        \\
&\prod_{i =1}^{N_1-k} \frac{\exp (S_1({z}^1_i) - S_1({w}^1_i))}{ z_i^1 - w_i^1 }    \prod_{i = 1}^{k}\frac{d\hat{w}^2_i}{2\pi \iota}\prod_{i = 1}^{k}\frac{d{w}^2_i}{2\pi \iota}\prod_{i = 1}^{k}\frac{dz^2_i}{2\pi \iota}\prod_{i = 1}^{N_2 - k}\frac{d\hat{w}^1_i}{2\pi \iota} \prod_{i = 1}^{N_2 - k} \frac{d\hat{z}_i}{2\pi \iota} \prod_{i = 1}^{N_1 - k}\frac{dw^1_i}{2\pi \iota}\prod_{i = 1}^{N_1- k}\frac{dz^1_i}{2\pi \iota},
\end{split}
\end{equation}
where $S_1, S_2$ are as in (\ref{ST1}), $S_3$ is given by
\begin{equation}\label{DefS3}
S_3(w) = [\tau_1 - \tau_2] w^2 + [x_1 - x_2 - \tau_1^2 + \tau_2^2]w
\end{equation}
and $\hat{B}$ is a $(N_1 +N_2 - k) \times (N_1 + N_2 - k)$ matrix that has the block form $\hat{B} = \begin{bmatrix} \hat{B}_{11} & \hat{B}_{12} & \hat{B}_{13} \\ \hat{B}_{21} & \hat{B}_{22} & \hat{B}_{23} \\ \hat{B}_{31} & \hat{B}_{32} & \hat{B}_{33} \end{bmatrix}$ with blocks given by
\begin{equation}\label{HatBlockMatrix}
\begin{split}
&\hat{B}_{11} = \left[ \frac{1}{z_i^1 - w_j^1} \right]{\substack{i = 1, \dots, N_1 - k \\ j = 1, \dots, N_1- k}}, \hspace{2mm} \hat{B}_{12} = \left[ \frac{1}{z^1_i - \hat{w}_j^1 + \tau_1 - \tau_2}\right]{\substack{i = 1, \dots, N_1 - k \\ j = 1, \dots, N_2- k}} , \hspace{2mm} \\
& \hat{B}_{13} =\left[ \frac{1}{z^1_i - \hat{w}_j^2 + \tau_1 - \tau_2}\right]{\substack{i = 1, \dots, N_1 - k \\ j = 1, \dots,k}}, \hspace{2mm} \hat{B}_{21} = \left[ \frac{1}{\hat{z}^1_i - w_j^1 - \tau_1 + \tau_2}\right]{\substack{i = 1, \dots, N_2 - k \\ j = 1, \dots, N_1- k}},  \hspace{2mm} \\
&  \hat{B}_{22} = \left[ \frac{1}{\hat{z}^1_i - \hat{w}^1_j} \right]{\substack{i = 1, \dots, N_2 - k \\ j = 1, \dots, N_2- k}}, \hspace{2mm} \hat{B}_{23} = \left[ \frac{1}{\hat{z}^1_i - \hat{w}^2_j } \right]{\substack{i = 1, \dots, N_2 - k\\ j = 1, \dots, k}}, \hspace{2mm} \hat{B}_{31} =  \left[ \frac{1}{z^2_i - w_j^1}\right]{\substack{i = 1, \dots, k \\ j = 1, \dots, N_1- k}}, \\
& \hat{B}_{32} = \left[ \frac{1}{z^2_i - \hat{w}^1_j + \tau_1 - \tau_2} \right]{\substack{i = 1, \dots, k\\ j = 1, \dots, N_2 - k}},  \hspace{2mm} \hat{B}_{33} = \left[ \frac{1}{z_i^2 - \hat{w}_j^2 + \tau_1 - \tau_2}\right]{\substack{i = 1, \dots, k\\ j = 1, \dots, k}}.
\end{split}
\end{equation}
\end{lemma}
\begin{remark}The proof of Lemma \ref{LemmaSwap2} can be found in Section \ref{Section8}, see Lemma \ref{S8LemmaSwap2}.
\end{remark}

%
\subsection{ Proof of Proposition \ref{PropTermConv}}\label{Section5.2} We continue with the same notation as in Section \ref{Section5.1} above as well as Section \ref{Section4.1}. For clarity we split the proof into several steps. In Step 1 we use Lemmas \ref{LemmaSwap1}  and \ref{LemmaSwap2} to rewrite $I_M(N_1, N_2) = \sum_{k = 0}^{\min(N_1, N_2)} I_M(N_1, N_2, k)$ and $K(N_1, N_2) = \sum_{k = 0}^{\min(N_1, N_2)} K(N_1, N_2, k)$ and reduce the proof of the proposition to showing that for each $k$ we have $\lim_{M \rightarrow \infty} I_M(N_1, N_2, k) = K(N_1, N_2, k)$. In Step 2 we find a formula for $ I_M(N_1, N_2,k)$ that is suitable for taking the $M\rightarrow \infty$ limit and we compute this limit in Step 3. In Step 4 we identify our formula for $\lim_{M \rightarrow \infty} I_M(N_1, N_2, k) $ with $ K(N_1, N_2, k)$, which concludes the proof of the proposition.  \\

In all that follows we fix $a, t \in (0,1)$  and $A(a,t)$ as in Definition \ref{Defcontours}. 

{\bf \raggedleft Step 1.} In this step we rewrite $I_M(N_1,N_2)$ using Lemma \ref{LemmaSwap1}. Recall the Cauchy determinant formula, see e.g. \cite[1.3]{Prasolov},
$$\det \left[ \frac{1}{x_i - y_j}\right]_{i,j = 1}^N  = \frac{\prod_{1 \leq i < j \leq N} (x_i - x_j) (y_j - y_i)}{\prod_{i,j = 1}^N (x_i - y_j)}.$$
From the Cauchy determinant formula and the fact that $(b;t)_\infty = (1-b)(bt;t)_\infty$ we have that 
$$D(\vec{w}, \vec{z}) D(\vec{\hat{w}}, \vec{\hat{z}})    G(\vec{w}, \vec{z}; \vec{\hat{w}}, \vec{\hat{z}}) = \det D \cdot \prod_{i = 1}^{N_1} \prod_{j = 1}^{N_2}  \frac{G(\hat{z}_j, z_i) G(\hat{w}_j, w_i)}{G(\hat{w}_j, z_i) G(\hat{z}_j, w_i)},$$
where $G(z,w) = (t z w^{-1}; t)_\infty$ and we recall that $D(\vec{w}, \vec{z}),   G(\vec{w}, \vec{z}; \vec{\hat{w}}, \vec{\hat{z}}) $ were defined in (\ref{MPLTGv2}). In particular, we see that $I_M(N_1,N_2)$ has the form of $F(N_1,N_2)$ from Lemma \ref{LemmaSwap1} for the functions
\begin{equation}\label{B3}
G(z,w) = (t z w^{-1}; t)_\infty \mbox{ and } G_i(z,w) = \frac{S(w,z; u_i, t)}{- \log t \cdot w} \cdot \left( \frac{1 - az}{1 - a w}\right)^M\cdot \left(\frac{1 - a /w}{1 - a/z} \right)^{n_i},
\end{equation}
for $i = 1,2$ where $R = t^{-\epsilon} r_1$, $r =r_4 t^{\epsilon}$, and $C_i = \gamma_i$ for $i = 1, \dots, 4$ as in (\ref{B2}). Here $\epsilon > 0$ is chosen sufficiently small so that $t^{\epsilon} r_4 > a$ and $t^{-\epsilon}r_1 < a^{-1}$ while $t^{2\epsilon}r_4/r_1 > t$. Notice that by our choice of contours and  Lemma \ref{S3Analyticity} we have that $G,G_1, G_2$ satisfy all the conditions of Lemma \ref{LemmaSwap1}. Thus we may apply Lemma \ref{LemmaSwap1} to $I_N(N_1, N_2)$ and conclude
\begin{equation}\label{DefIN}
\begin{split}
& I_M(N_1,N_2) = \sum_{k = 0}^{\min(N_1,N_2)} I_N(N_1, N_2, k) \mbox{, where }\\
& I_M(N_1, N_2, k) =  \frac{1}{k! (N_1- k)! (N_2- k)!}  \int_{\gamma_{1}^{N_1-k}} \int_{\gamma^{N_1-k}_3}  \int_{\gamma_2^{N_2-k}}\int_{\gamma^{N_2-k}_{4}} \int_{\gamma^{k}_1}\int_{\gamma^{k}_3} \int_{\gamma^k_4} \det{B}\\
&\prod_{i = 1}^{N_1-k}  \prod_{j = 1}^{N_2-k}  \frac{(t\hat{z}^1_j/ z^1_i; t)_\infty (t\hat{w}^1_j / w^1_i; t)_\infty}{(t\hat{w}^1_j/ z^1_i;t)_\infty (t\hat{z}^1_j/ w^1_i;t)_\infty} \cdot  \prod_{i = 1}^{k} \prod_{j = 1}^{N_2-k} \frac{(t\hat{z}^1_j / z^2_i;t)_\infty (t\hat{w}^1_j/ w^2_i;t)_\infty}{(t\hat{w}^1_j/ z^2_i;t)_\infty (t\hat{z}^1_j/ w^2_i; t)_\infty} \cdot \\
&  \prod_{i = 1}^{N_1-k}  \prod_{j = 1}^{k} \frac{(tw^2_j/ z^1_i;t)_\infty (t\hat{w}^2_j/ w^1_i;t)_\infty}{(t\hat{w}^2_j/ z^1_i;t)_\infty (tw^2_j/ w^1_i; t)_\infty} \cdot \prod_{i = 1}^{k} \prod_{j = 1}^{k} \frac{(tw^2_j/ z^2_i;t)_\infty (t\hat{w}^2_j/ w^2_i;t)_\infty}{(t\hat{w}^2_j/ z^2_i;t)_\infty (tw^2_j/ w^2_i;t)_\infty} \cdot \\
& \prod_{i =1}^{N_2-k} \frac{S(\hat{w}_i^1,\hat{z}^1_i; u_2, t)}{- \log t \cdot \hat{w}_i^1} \cdot \prod_{i =1}^{k} \frac{S(\hat{w}_i^2,w^2_i; u_2, t)}{- \log t \cdot \hat{w}_i^2}    \cdot \prod_{i =1}^{N_1-k} \frac{S(w^1_i,z_i^1; u_1, t)}{- \log t \cdot w_i^1} \cdot \prod_{i =1}^{k} \frac{S(w_i^2,z_i^2; u_1, t)}{- \log t \cdot w_i^2}\cdot \\
& \prod_{i =1}^{N_2-k}  \left( \frac{1 - a\hat{z}^1_i}{1 - a \hat{w}_i^1}\right)^M \cdot \left(\frac{1 - a / \hat{w}^1_i}{1 - a / \hat{z}_i^1} \right)^{n_2} \cdot \prod_{i =1}^{k} \left( \frac{1 - aw^2_i}{1 - a \hat{w}_i^2}\right)^M \cdot \left(\frac{1 - a / \hat{w}^2_i}{1 - a / w_i^2} \right)^{n_2}   \cdot  \\
&\prod_{i =1}^{N_1-k}  \left( \frac{1 - az_i^1}{1 - a w_i^1}\right)^M \cdot \left(\frac{1 - a/ w_i^1}{1 - a/ z_i^1} \right)^{n_1} \cdot \prod_{i =1}^{k}  \left( \frac{1 - az_i^2}{1 - a w_i^2}\right)^M \cdot \left(\frac{1 - a/ w_i^2}{1 - a/ z_i^2} \right)^{n_1} \cdot \\
& \prod_{i = 1}^{k}\frac{d\hat{w}^2_i}{2\pi \iota}\prod_{i = 1}^{k}\frac{d{w}^2_i}{2\pi \iota}\prod_{i = 1}^{k}\frac{dz^2_i}{2\pi \iota}\prod_{i = 1}^{N_2 - k}\frac{d\hat{w}^1_i}{2\pi \iota} \prod_{i = 1}^{N_2 - k} \frac{d\hat{z}_i}{2\pi \iota} \prod_{i = 1}^{N_1 - k}\frac{dw^1_i}{2\pi \iota}\prod_{i = 1}^{N_1- k}\frac{dz^1_i}{2\pi \iota},
\end{split}
\end{equation}
where we recall that $B$ is as in (\ref{BlockMatrix}). 

After performing a bit of cancellation and changing all variables to minus their value, we get
\begin{equation}\label{INK}
\begin{split}
&I_M(N_1, N_2, k) = \frac{(-1)^{N_1 + N_2}}{k! (N_1- k)! (N_2- k)!}  \int_{\gamma_{1}^{N_1-k}} \int_{\gamma^{N_1-k}_3}  \int_{\gamma_2^{N_2-k}}\int_{\gamma^{N_2-k}_{4}} \int_{\gamma^{k}_1}\int_{\gamma^{k}_3} \int_{\gamma^k_4} \\
& \det B \cdot  \prod_{i = 1}^{N_1-k} \prod_{j = 1}^{N_2-k}\frac{(t\hat{z}^1_j/ z^1_i; t)_\infty (t\hat{w}^1_j / w^1_i; t)_\infty}{(t\hat{w}^1_j/ z^1_i;t)_\infty (t\hat{z}^1_j/ w^1_i;t)_\infty} \cdot  \prod_{i = 1}^{k} \prod_{j = 1}^{N_2-k} \frac{(t\hat{z}^1_j / z^2_i;t)_\infty (t\hat{w}^1_j/ w^2_i;t)_\infty}{(t\hat{w}^1_j/ z^2_i;t)_\infty (t\hat{z}^1_j/ w^2_i; t)_\infty} \cdot \\
&   \prod_{i = 1}^{N_1-k}\prod_{j = 1}^{k} \frac{(tw^2_j/ z^1_i;t)_\infty (t\hat{w}^2_j/ w^1_i;t)_\infty}{(t\hat{w}^2_j/ z^1_i;t)_\infty (tw^2_j/ w^1_i; t)_\infty} \cdot \prod_{i = 1}^{k} \prod_{j = 1}^{k} \frac{(tw^2_j/ z^2_i;t)_\infty (t\hat{w}^2_j/ w^2_i;t)_\infty}{(t\hat{w}^2_j/ z^2_i;t)_\infty (tw^2_j/ w^2_i;t)_\infty} \cdot \\
& \prod_{i =1}^{N_2-k} \frac{S(\hat{w}_i^1,\hat{z}^1_i; u_2, t)}{ \log t \cdot \hat{w}_i^1} \cdot \prod_{i =1}^{k} \frac{S(\hat{w}_i^2,w^2_i; u_2, t)}{ \log t \cdot \hat{w}_i^2}    \cdot \prod_{i =1}^{N_1-k} \frac{S(w^1_i,z_i^1; u_1, t)}{\log t \cdot w_i^1} \cdot \prod_{i =1}^{k} \frac{S(w_i^2,z_i^2; u_1, t)}{ \log t \cdot w_i^2}\cdot \\
& \prod_{i =1}^{N_2-k}  \left( \frac{1 + a\hat{z}^1_i}{1 + a \hat{w}_i^1}\right)^M \cdot \left(\frac{1 + a / \hat{w}^1_i}{1 + a / \hat{z}_i^1} \right)^{n_2} \cdot \prod_{i =1}^{k} \left( \frac{1 + az_i^2}{1 + a \hat{w}_i^2}\right)^M \cdot \left(\frac{1 + a / \hat{w}^2_i}{1 + a/ z_i^2} \right)^{n_2} \cdot  \\
&\prod_{i =1}^{N_1-k}  \left( \frac{1 + az_i^1}{1 + a w_i^1}\right)^M \cdot \left(\frac{1 + a/ w_i^1}{1 + a/ z_i^1} \right)^{n_1}  \cdot \prod_{i = 1}^k \left(\frac{1 + a/ z_i^2}{1 + a / w_i^2} \right)^{n_2- n_1}    \\
&\prod_{i = 1}^{k}\frac{d\hat{w}^2_i}{2\pi \iota}\prod_{i = 1}^{k}\frac{d{w}^2_i}{2\pi \iota}\prod_{i = 1}^{k}\frac{dz^2_i}{2\pi \iota}\prod_{i = 1}^{N_2 - k}\frac{d\hat{w}^1_i}{2\pi \iota} \prod_{i = 1}^{N_2 - k} \frac{d\hat{z}_i}{2\pi \iota} \prod_{i = 1}^{N_1 - k}\frac{dw^1_i}{2\pi \iota}\prod_{i = 1}^{N_1- k}\frac{dz^1_i}{2\pi \iota} ,
\end{split}
\end{equation}
where we remark that we absorbed $(-1)^{N_1 + N_2 - k}$ into $\det B$. In view of the top line of (\ref{DefIN}) and Lemma \ref{LemmaSwap2} we see that it suffices to show that for each $k \in \{0 , \dots, \min(N_1, N_2)\}$
\begin{equation}\label{S5R1}
\lim_{M \rightarrow \infty} I_M(N_1, N_2, k)  = K(N_1, N_2, k),
\end{equation}
where $K(N_1, N_2, k)$ is as in (\ref{KN1P1}) for $x_1, x_2, \tau_1, \tau_2$ as in the statement of the proposition and $I_M(N_1, N_2, k) $ as in (\ref{INK}).\\

{\bf \raggedleft Step 2.} In this step we fix $k \in \{0 , \dots, \min(N_1, N_2)\}$ and find a suitable expression for $I_M(N_1, N_2, k)$ for taking the $M \rightarrow \infty$ limit. We first observe that by Cauchy's theorem and Lemma \ref{S3Analyticity} we can deform the $z_i^1$, $z_i^2$ and $\hat{z}_i$ contours to $C^M_{out}$, the $w^1_i$, $\hat{w}^1_i$ and $\hat{w}^2_i$ contours to $C_{in}$ and the $w^2_i$-contours to $C_{mid}$ without affecting the value of the integral. Here $C_{in}, C_{out}^M, C_{mid}$ are as in Definition \ref{Defcontours} and we assumed that $M$ is sufficiently large so that $M^{-1/3} \leq A \pi $. We next proceed to change variables
$$\hat{z}^1_{i} \rightarrow e^{ M^{-1/3}\hat{z}^1_i}, \hat{w}_i^1 \rightarrow e^{M^{-1/3}\hat{w}^1_i}, \hat{w}^2_i \rightarrow  e^{ M^{-1/3}\hat{W}_i^2}, z_i^1 \rightarrow e^{ M^{-1/3} z_i^1},$$
$$ z_i^2 \rightarrow e^{M^{-1/3} z^2_i}, w_i^1 \rightarrow  e^{M^{-1/3}w_i^1} \mbox{ and }w_i^2 \rightarrow e^{M^{-1/3}w_i^2}.$$ 
After applying this change of variables and utilizing the definition of $u_1, u_2$ from (\ref{ScaleU}). the definition of $S_a, R_a$ from (\ref{FunExp}) and the definition of $S(w, z; u,t)$ from Definition \ref{DefFunS} we conclude
\begin{equation}\label{INK3}
\begin{split}
&I_M(N_1, N_2, k) = \frac{(-1)^{N_1 + N_2}}{k! (N_1- k)! (N_2- k)!}  \int_{\gamma_z^{N_1-k}}  \int_{\gamma_z^{N_2-k}} \int_{\gamma_z^{k}} \int_{\gamma_w^{N_1-k}}\int_{\gamma_w^{N_2-k}} \int_{\gamma_w^k}\int_{\gamma_m^{k}}      \\
&H^M_1  H^M_2 H^M_3  H^M_4  \prod_{i = 1}^{k} \frac{{\bf 1}\{ |w_i^2| \leq \pi L\} d{w}^2_i}{2\pi \iota  L} \prod_{i = 1}^{k} \frac{{\bf 1}\{ |w_i^2| \leq \pi \rho_A  L \} d\hat{w}^2_i }{2\pi \iota  L} \prod_{i = 1}^{N_2 - k} \frac{{\bf 1}\{ |\hat{w}_i^1| \leq \pi \rho_A  L\} d\hat{w}^1_i}{2\pi \iota L} \\
& \prod_{i = 1}^{N_1 - k}  \frac{{\bf 1}\{ |w_i^1| \leq \pi \rho_A  L \} dw^1_i}{2\pi \iota  L}  \prod_{i = 1}^{k}   \frac{{\bf 1}\{ |z_i^2| \leq \pi \rho_A  L \} e^{{z}^2_i/ L}dz^2_i}{2\pi \iota  L}  \\
&\prod_{i = 1}^{N_2 - k}\frac{{\bf 1}\{ |\hat{z}_i^1| \leq \pi \rho_A  L \} e^{\hat{z}^1_i/L}d\hat{z}_i}{2\pi \iota L} \prod_{i = 1}^{N_1- k}\frac{{\bf 1}\{ |z_i^1| \leq \pi \rho_A   L \}  e^{{z}^1_i/L}dz^1_i}{2\pi \iota  L} ,
\end{split}
\end{equation}
where $\gamma_w, \gamma_m, \gamma_z$ are as in Definition \ref{Defcontours}, $L = M^{1/3}$ and $\rho_A = \sqrt{1 + A^2}$. Below we explain what are the functions $H_1^M, H_2^M,  H_3^M$ and $H_4^M$ that appear in (\ref{INK3}) continuing to use $L$ for $M^{1/3}$ to ease the notation wherever appropriate. 

The function $H^M_1$ equals $\det \tilde{B}$ with $\tilde{B}$ being a $(N_1 +N_2 - k) \times (N_1 +N_2 - k)$ matrix that has a block form $\tilde{B} = \begin{bmatrix} \tilde{B}_{11} & \tilde{B}_{12} & \tilde{B}_{13} \\ \tilde{B}_{21} & \tilde{B}_{22} & \tilde{B}_{23} \\ \tilde{B}_{31} & \tilde{B}_{32} & \tilde{B}_{33} \end{bmatrix}$ with blocks given by
\begin{equation}\label{TildeBlockMatrix}
\begin{split}
&\tilde{B}_{11} = \left[ \frac{1}{e^{z_i^1/L} - e^{w_j^1/L}} \right]{\substack{i = 1, \dots, N_1 - k \\ j = 1, \dots, N_1- k}}, \hspace{2mm} \tilde{B}_{12} = \left[ \frac{1}{e^{ z^1_i/L} - e^{\hat{w}_j^1/L}}\right]{\substack{i = 1, \dots, N_1 - k \\ j = 1, \dots, N_2- k}} , \hspace{2mm} \\
& \tilde{B}_{13} =\left[ \frac{1}{e^{z^1_i/L} - e^{\hat{w}_j^2/L}}\right]{\substack{i = 1, \dots, N_1 - k \\ j = 1, \dots,k}}, \hspace{2mm} \tilde{B}_{21} = \left[ \frac{1}{e^{ \hat{z}^1_i/L} -e^{ w_j^1/L}}\right]{\substack{i = 1, \dots, N_2 - k \\ j = 1, \dots, N_1- k}},  \hspace{2mm} \\
&  \tilde{B}_{22} = \left[ \frac{1}{e^{\hat{z}^1_i/L} - e^{\hat{w}^1_j/L}} \right]{\substack{i = 1, \dots, N_2 - k \\ j = 1, \dots, N_2- k}}, \hspace{2mm} \tilde{B}_{23} = \left[ \frac{1}{e^{\hat{z}^1_i/L} - e^{\hat{w}^2_j/L} } \right]{\substack{i = 1, \dots, N_2 - k\\ j = 1, \dots, k}}, \hspace{2mm} \\
& \tilde{B}_{31} =  \left[ \frac{1}{e^{z^2_i/L} - e^{w_j^1/L}}\right]{\substack{i = 1, \dots, k \\ j = 1, \dots, N_1- k}},  \tilde{B}_{32} = \left[ \frac{1}{e^{z^2_i/L} - e^{\hat{w}^1_j/L }} \right]{\substack{i = 1, \dots, k\\ j = 1, \dots, N_2 - k}},  \hspace{2mm} \\
& \tilde{B}_{33} = \left[ \frac{1}{e^{z_i^2/L} - e^{\hat{w}_j^2/L}}\right]{\substack{i = 1, \dots, k\\ j = 1, \dots, k}}.
\end{split}
\end{equation}
The function $H^M_2$ is given by

\begin{equation}\label{DefH2V2}
\begin{split}
& \prod_{i = 1}^{N_1-k}\prod_{j= 1}^{N_2-k} \frac{(te^{(\hat{z}^1_j - z^1_i)/L}; t)_\infty (te^{(\hat{w}^1_j - w^1_i)/L}; t)_\infty}{(te^{(\hat{w}^1_j - z^1_i)/L};t)_\infty (te^{(\hat{z}^1_j- w^1_i)/L};t)_\infty}  \cdot \prod_{i = 1}^{k}  \prod_{j = 1}^{N_2-k} \frac{(te^{(\hat{z}^1_j - z^2_i)/L};t)_\infty (te^{(\hat{w}^1_j-  w^2_i)/L};t)_\infty}{(te^{(\hat{w}^1_j -  z^2_i)/L};t)_\infty (te^{(\hat{z}^1_j-  w^2_i)/L}; t)_\infty} \cdot \\
&  \prod_{i = 1}^{N_1-k}  \prod_{j = 1}^{k}\frac{(te^{(w^2_j- z^1_i)/L};t)_\infty (te^{(\hat{w}^2_j - w^1_i)/L};t)_\infty}{(te^{(\hat{w}^2_j- z^1_i)/L};t)_\infty (te^{(w^2_j-  w^1_i)/L}; t)_\infty} \cdot \prod_{i = 1}^{k} \prod_{j = 1}^{k} \frac{(te^{(w^2_j-  z^2_i)/L};t)_\infty (te^{(\hat{w}^2_j - w^2_i)/L};t)_\infty}{(te^{(\hat{w}^2_j -  z^2_i)/L};t)_\infty (te^{(w^2_j-  w^2_i)/L};t)_\infty}.
\end{split}
\end{equation}
The function $H^M_3$ is given by $H^M_3 =  \prod_{i = 1}^7 H^M_{3,i}$ where
\begin{equation}\label{DefH3V2}
\begin{split}
& H^M_{3,1}=  \prod_{i =1}^{N_2-k}e^{ F^M_{2}(\hat{z}_i^1/L)}, \hspace{2mm}  H^M_{3,2} = \prod_{i =1}^{N_2-k}e^{- F^M_2(\hat{w}_i^1/L)},\hspace{2mm} H^M_{3,3} =\prod_{i =1}^{N_1-k} e^{F^M_1({z}_i^1/L)}, \\
& H^M_{3,4} = \prod_{i =1}^{N_1-k} e^{-F_1^M(w_i^1/L)}, H^M_{3,5}= \prod_{i =1}^{k} e^{F_1^M(z_i^2/L)},  H^M_{3,6}= \prod_{i =1}^{k} e^{-F_2^M(\hat{w}_i^2/L)}, H^M_{3,7} = \prod_{i =1}^{k} e^{F_3^M(w_i^2/L)},
\end{split}
\end{equation}
where $F_1^M, F_2^M, F_3^M$ are given by
\begin{equation}\label{DefH3V22}
\begin{split}
&F_1^M(Z) = \exp\left({M S_a(Z) -  M^{1/3}\sigma_a x_1 Z + (M-n_1)R_a(Z) - (Z/2) f''_1 s_1^2 M^{1/3}}\right), \\
& F_2^M(Z) = \exp \left({MS_a(Z) - M^{1/3} \sigma_a x_2 Z} + (M-n_2)R_a(Z) - (Z/2) f''_1 s_2^2 M^{1/3} \right)  \\
& F_3^M(W) = \exp \left(  M^{1/3} \sigma_a (x_1 - x_2)W  + (n_1-n_2)R_a(W) + (W/2)M^{1/3} f''_1 [s_1^2 - s_2^2]  \right).
\end{split}
\end{equation}
The function $H_4^M$ is given by
\begin{equation}\label{DefH4V2}
\begin{split}
&  \prod_{i =1}^{N_2-k} \frac{\tilde{S}(\hat{w}_i^1/L,\hat{z}^1_i/L; u_2, t)}{ \log t } \cdot \prod_{i =1}^{k} \frac{\tilde{S}(\hat{w}_i^2/L,w^2_i/L; u_2, t)}{ \log t }  \frac{\tilde{S}(w_i^2/L,Z_i^2; u_1, t)}{ \log t } \cdot \\
&   \prod_{j =1}^{N_1-k} \frac{\tilde{S}(w^1_j/L,z_j^1/L; u_1, t)}{\log t } , \mbox{ where } \\
&  \tilde{S}(W,Z;u,t) =  \sum_{m \in \mathbb{Z}} \frac{\pi \cdot [ - u ]^{-  2m \pi  \iota [\log t ]^{-1}}}{\sin(-\pi [[W - Z] [\log t]^{-1} -  2m \pi  \iota [\log t ]^{-1}])}.
\end{split}
\end{equation}
Equation (\ref{INK3}) is the one that is suitable for taking the $M \rightarrow \infty$ limit. \\

{\bf \raggedleft Step 3.} In this step we prove that 
\begin{equation}\label{INKLimit}
\begin{split}
&\lim_{M \rightarrow \infty} I_M(N_1, N_2, k) = \frac{(-1)^{N_1 + N_2}}{k! (N_1- k)! (N_2- k)!}  \int_{\gamma_z^{N_1-k}}  \int_{\gamma_z^{N_2-k}} \int_{\gamma_z^{k}} \int_{\gamma_w^{N_1-k}} \int_{\gamma_w^{N_2-k}} \int_{\gamma_w^k}\int_{\gamma_m^{k}}  \\
& \det B \cdot  \prod_{i =1}^{N_2-k} \frac{\exp (\hat{S}_2(\hat{z}^1_i) - \hat{S}_2(\hat{w}^1_i))}{ \hat{z}_i^1 - \hat{w}^1_i } \cdot \prod_{i =1}^{k} \frac{\exp (\hat{S}_2(w^2_i) - \hat{S}_2(\hat{w}^2_i))}{ w_i^2 - \hat{w}_i^2}  \frac{\exp (\hat{S}_1({z}^2_i) - \hat{S}_1({w}^2_i))}{ z_i^2 - w_i^2 } \\
& \prod_{i =1}^{N_1-k} \frac{\exp (\hat{S}_1({z}^1_i) - \hat{S}_1({w}^1_i))}{ z_i^1 - w_i^1 }  \prod_{i = 1}^{k} \frac{ d{w}^2_i}{2\pi \iota } \prod_{i = 1}^{k} \frac{ d\hat{w}^2_i }{2\pi \iota} \prod_{i = 1}^{N_2 - k} \frac{d\hat{w}^1_i}{2\pi \iota }  \prod_{i = 1}^{N_1 - k}  \frac{ dw^1_i}{2\pi \iota }  \prod_{i = 1}^{k}   \frac{ dz^2_i}{2\pi \iota }  \prod_{i = 1}^{N_2 - k}\frac{ d\hat{z}_i}{2\pi \iota } \prod_{i = 1}^{N_1- k}\frac{ dz^1_i}{2\pi \iota },
\end{split}
\end{equation}
where 
\begin{equation}\label{DefH3V24}
\begin{split}
&\hat{S}_1(z) = \exp \left(\frac{a(1-a)}{3 (1+ a)^3} \cdot z^3 -  [\sigma_a x_1 + (1/2)f''_1 s_1^2] z  -   \frac{s_1 a}{2(1+a)^2} \cdot z^2    \right), \\
& \hat{S}_2(z) =\exp \left(\frac{a(1-a)}{3 (1+ a)^3} \cdot z^3 -  [\sigma_a x_2 + (1/2)f''_1 s_2^2] z  -   \frac{s_2 a}{2(1+a)^2} \cdot z^2   \right).
\end{split}
\end{equation}
In the sequel we denote by $E$ the set of points $(\vec{z}^1, \vec{\hat{z}}^1, \vec{z}^2, \vec{w}^1, \vec{\hat{w}}^1, \vec{\hat{w}}^2, \vec{w}^2) \in \mathbb{C}^{2N_1 + 2N_2 - k}$ that satisfy the inequalities in the indicator functions in (\ref{INK3}). Below we study the pointwise limit of the integrand (\ref{INK3}) and obtain estimates for it on the set $E$.\\

Notice that by (\ref{S3PochBound}) and (\ref{DefH2V2}) we have that 
 \begin{equation}\label{BoundH2}
\begin{split}
&\left|{\bf 1}_{E} H^M_2\right|  \leq  \left( \frac{(-t^{1/2}; t)_\infty}{(t^{1/2} ;t)_{\infty}} \right)^{2N_1 N_2},
\end{split}
\end{equation}
where we used property (3) in Definition \ref{Defcontours}. Furthermore we have that pointwise
 \begin{equation}\label{LimitH2}
\begin{split}
&\lim_{N \rightarrow \infty} {\bf 1}_{E} H^M_2  = 1.
\end{split}
\end{equation}

In addition, by Lemma \ref{techies} and property (3) in Definition \ref{Defcontours} we have
 \begin{equation}\label{BoundH1H4}
\begin{split}
&|{\bf 1}_{E} H^M_1|  \leq   C^{N_1+N_2} (N_1 + N_2-k)^{(N_1 + N_2-k)/2} L^{N_1 + N_2-k}, \mbox{ and } \\
&|{\bf 1}_{E} H_4^M| \leq C^{N_1 + N_2} L^{N_1 + N_2},
\end{split}
\end{equation}
where $C$ is sufficiently large depending on $a, t$ alone. In deriving the first inequality we also used Hadamard's inequality, see Lemma \ref{DetBounds}. In deriving the second inequality we also used that $-u_1, -u_2 \in (0, \infty)$ so that $\left| [-u_i]^{-2\pi \iota [\log t]^{-1}}\right| = 1$. 

We also have the following pointwise limits
 \begin{equation}\label{LimitH1H4}
\begin{split}
&\lim_{M \rightarrow \infty} L^{-N_1 + N_2 - k} H^M_1  = \det B, \mbox{ and } \\
& \lim_{M \rightarrow \infty}  L^{-N_1 + N_2}H_4^M =  \prod_{i =1}^{N_2-k} \frac{1}{ \hat{z}_i^1 - \hat{w}^1_i } \cdot \prod_{i =1}^{k} \frac{1}{ w_i^2 - \hat{w}_i^2}  \frac{1}{ z_i^2 - w_i^2 }   \cdot \prod_{i =1}^{N_1-k} \frac{1}{ z_i^1 - w_i^1 } ,
\end{split}
\end{equation}
where $B$ is the $(N_1 + N_2 -k ) \times (N_1 + N_2 -k )$ matrix from (\ref{BlockMatrix}).

We finally turn our attention to $H_3^N$. From Lemma \ref{LemmaTaylor} there are constants $C, C_1, C_2 > 0$ depending on $a,t, x_1, x_2, s_1, s_2$ such that
\begin{equation}\label{BoundH3}
\begin{split}
&\left|e^{F^M_1(z/L)} {\bf 1}\{ |z| \leq \pi \rho_A L \}  \right| \leq C \cdot e^{-\epsilon_1 |z|^3 + C_1 |z| + C_2|z|^2} \mbox{ if $z \in \gamma_{z}$}; \\
&\left|e^{F^M_2(w/L)}{\bf 1}\{ |w| \leq \pi \rho_A L \} \right| \leq C \cdot e^{-\epsilon_1 |w|^3 + C_1 |w| + C_2|w|^2}  \mbox{ if $w \in \gamma_{w}$}; \\
&\left|e^{F^M_3(w/L)} {\bf 1}\{ |w| \leq \pi L \} \right| \leq  C e^{ - \epsilon_1 (s_1 - s_2) |w|^2 + C_1 |w|} \mbox{ if $w \in \gamma_{m}$}. \\
\end{split}
\end{equation}
From the same lemma we also obtain the pointwise limit
\begin{equation}\label{LimitH3}
\begin{split}
&\lim_{M \rightarrow \infty} H^M_3 =  \prod_{i =1}^{N_2-k} \exp (\hat{S}_2(\hat{z}^1_i) - \hat{S}_2(\hat{w}^1_i)) \cdot  \\
&\prod_{i =1}^{k} \exp (\hat{S}_2(w^2_i) - \hat{S}_2(\hat{w}^2_i)) \exp (\hat{S}_1({z}^2_i) - \hat{S}_1({w}^2_i)) \cdot \prod_{i =1}^{N_1-k} \exp (\hat{S}_1({z}^1_i) - \hat{S}_1({w}^1_i)).
\end{split}
\end{equation}

From equations (\ref{LimitH2}), (\ref{LimitH1H4}), (\ref{LimitH3}) we conclude that the integrand in (\ref{INK3}) converges pointwise to the integrand on the right side of (\ref{INKLimit}) and by (\ref{BoundH2}), (\ref{BoundH1H4}) and (\ref{BoundH3}) we may apply the Dominated convergence theorem to conclude (\ref{INKLimit}). \\

{\bf \raggedleft Step 4.} In this step we prove (\ref{S5R1}). Starting from (\ref{INKLimit}) we perform a change of variables 
$$\hat{z}^1_{i} \rightarrow \sigma_a^{-1} \hat{z}^1_{i} +  \sigma_a^{-1}\tau_2, \hat{w}_i^1 \rightarrow  \sigma_a^{-1} \hat{w}_i^1 +  \sigma_a^{-1}\tau_2 , \hat{w}^2_i \rightarrow \sigma_a^{-1}\hat{w}^2_i +  \sigma_a^{-1}\tau_2 $$
$$ z_i^1  \rightarrow \sigma_a^{-1} z_i^1 +  \sigma_a^{-1}\tau_1 , z_i^2  \rightarrow  \sigma_a^{-1} z_i^2 +  \sigma_a^{-1}\tau_1, w_i^1 \rightarrow   \sigma_a^{-1}w_i^1 +  \sigma_a^{-1}\tau_1 \mbox{ and }w_i^2 \rightarrow \sigma_a^{-1} w_i^2.$$ 
This allows us to rewrite (\ref{INKLimit}) as 
\begin{equation}\label{INKLimit2}
\begin{split}
&\lim_{M \rightarrow \infty} I_M(N_1, N_2, k) = \frac{(-1)^{N_1 + N_2}}{k! (N_1- k)! (N_2- k)!}  \int_{\gamma_1^{N_1-k}}  \int_{\gamma_2^{N_2-k}} \int_{\gamma_3^{k}} \int_{\gamma_4^{N_1-k}} \int_{\gamma_5^{N_2-k}} \int_{\gamma_6^k}\int_{\gamma_7^{k}}  \\
& \det \hat{B}    \prod_{i =1}^{N_2-k} \frac{\exp (S_2(\hat{z}^1_i) - S_2(\hat{w}^1_i))}{ \hat{z}_i^1 - \hat{w}^1_i }  \prod_{i =1}^{k} \frac{\exp (S_1({z}^2_i) - S_2(\hat{w}^2_i) + S_3(w_i^2))}{ (z_i^2 - w_i^2 + \tau_1)(w_i^2 - \hat{w}_i^2 - \tau_2) }        \\
&\prod_{i =1}^{N_1-k} \frac{\exp (S_1({z}^1_i) - S_1({w}^1_i))}{ z_i^1 - w_i^1 }   \prod_{i = 1}^{k} \frac{ d{w}^2_i}{2\pi \iota } \prod_{i = 1}^{k} \frac{ d\hat{w}^2_i }{2\pi \iota} \prod_{i = 1}^{N_2 - k} \frac{d\hat{w}^1_i}{2\pi \iota }  \prod_{i = 1}^{N_1 - k}  \frac{ dw^1_i}{2\pi \iota }  \prod_{i = 1}^{k}   \frac{ dz^2_i}{2\pi \iota }  \prod_{i = 1}^{N_2 - k}\frac{ d\hat{z}_i}{2\pi \iota } \prod_{i = 1}^{N_1- k}\frac{ dz^1_i}{2\pi \iota },
\end{split}
\end{equation}
where $\hat{B}$ and $S_1, S_2, S_3$ are as in (\ref{ST1}), (\ref{DefS3}) and (\ref{HatBlockMatrix}). In addition, the contours are given by
$$\gamma_1 = \gamma_3 = \sigma_a \cdot \gamma_z - \tau_1, \gamma_2 = \sigma_a \cdot \gamma_z -  \tau_2, \gamma_4 =  \sigma_a \cdot \gamma_w - \tau_1, \gamma_5 = \gamma_6 =  \sigma_a \cdot \gamma_w - \tau_2, \gamma_7 = \sigma_a\cdot \gamma_m. $$

Recall that $\Gamma_i = c_i + \iota \mathbb{R}$ for $i =1,2,3,4$ as in Proposition \ref{PropTermLimit} -- we will use these constants below. We may deform the contours by translating $\gamma_1$ and $\gamma_3$ until they become $\gamma_z + c_1$, $\gamma_2$ until it becomes $\gamma_z + c_2 + \tau_1 - \tau_2$, $\gamma_4$ until it becomes $\gamma_w+ c_3 +\tau_2 - \tau_1$, $\gamma_5$ and $\gamma_6$ until they become $\gamma_w + c_4$. Observe that in the process of deformation we do not pass any poles and so by Cauchy's theorem the integral remains unchanged. The deformation near infinity is justified by the cubic term in $S_1, S_2$.

By definition we have that $\gamma_z + c_1$ passes through the point $c_1$ and is contained in the sector $V_1 = \{c_1 + r e^{\iota \phi}: r \geq 0, \phi \in [\pi/3, \pi/2] \cup [ -\pi/2, -\pi/3]\}$ where we used condition (2) n Definition \ref{Defcontours}. For $z \in V_1$ such that $z = c_1 + re^{\iota \phi}$ we have that 
$$Re[S_1(z)] = Re\left[(1/3) (c_1 + r \cos \phi + \iota r \sin (\phi))^3 -x_1 (c_1 + r \cos \phi + \iota r \sin (\phi)) \right] = $$
$$ =\frac{r^3 \cos(3\phi)}{3} + c_1r^2\cos(2\phi) + O(r) \leq \frac{-c_1 r^2}{2} + O(r),$$
where the constant in the big $O$ notation depends on $x_1, c_1$. In particular, the above shows that we have a quadratic exponential decay that allows us to deform $\gamma_1$ and $\gamma_3$ to $\Gamma_1$ without affecting the value of the integral by Cauchy's theorem. Analogous arguments show that we may deform $\gamma_2 $ to $\Gamma_2 + \tau_1 - \tau_2$, $\gamma_4$ to $\Gamma_3 + \tau_2 - \tau_1$ and $\gamma_5, \gamma_6$ to $\Gamma_4$. From here we have that 
\begin{equation*}
\begin{split}
&\lim_{M \rightarrow \infty} I_M(N_1, N_2, k) = \frac{(-1)^{N_1 + N_2}}{k! (N_1- k)! (N_2- k)!}  \int_{\Gamma_1^{N_1-k}}  \int_{(\Gamma_2 + \tau_1 - \tau_2)^{N_2-k}} \int_{\Gamma_1^{k}} \int_{(\Gamma_3 + \tau_2 - \tau_1)^{N_1-k}} \int_{\Gamma_4^{N_2-k}} \int_{\Gamma_4^k}\int_{\gamma_m^{k}}  \\
& \int_{\Gamma_4^{N_2-k}} \int_{\Gamma_4^k}\int_{\gamma_m^{k}} \det \hat{B}    \prod_{i =1}^{N_2-k} \frac{\exp (S_2(\hat{z}^1_i) - S_2(\hat{w}^1_i))}{ \hat{z}_i^1 - \hat{w}^1_i }  \prod_{i =1}^{k} \frac{\exp (S_1({z}^2_i) - S_2(\hat{w}^2_i) + S_3(w_i^2))}{ (z_i^2 - w_i^2 + \tau_1)(w_i^2 - \hat{w}_i^2 - \tau_2) }        \\
&\prod_{i =1}^{N_1-k} \frac{\exp (S_1({z}^1_i) - S_1({w}^1_i))}{ z_i^1 - w_i^1 }   \prod_{i = 1}^{k} \frac{ d{w}^2_i}{2\pi \iota } \prod_{i = 1}^{k} \frac{ d\hat{w}^2_i }{2\pi \iota} \prod_{i = 1}^{N_2 - k} \frac{d\hat{w}^1_i}{2\pi \iota }  \prod_{i = 1}^{N_1 - k}  \frac{ dw^1_i}{2\pi \iota }  \prod_{i = 1}^{k}   \frac{ dz^2_i}{2\pi \iota }  \prod_{i = 1}^{N_2 - k}\frac{ d\hat{z}^1_i}{2\pi \iota } \prod_{i = 1}^{N_1- k}\frac{ dz^1_i}{2\pi \iota },
\end{split}
\end{equation*}
where we recall that $\gamma_m$ is the $y$-axis in $\mathbb{C}$ with vertical orientation. The latter formula now agrees with the formula for $K(N_1, N_2, k)$ from (\ref{KN1P1}),which proves (\ref{S5R1}) and hence the proposition.

%
\section{Asymptotic analysis: Part II}\label{Section6}  The purpose of the present section is to prove Proposition \ref{PropTermBound} . In Section \ref{Section6.1} we summarize various results that will be used in the proof, which is given in Section \ref{Section6.2}.

%
\subsection{Technical lemmas}\label{Section6.1} In this section we state several lemmas that go into the proof of Proposition \ref{PropTermBound}. The proofs of these lemmas can be found in Section \ref{Section8.2}. 

The first result we require is as follows.
\begin{lemma}\label{LCDetBound} Suppose that $0 < r < R$ are real numbers and $\alpha \in [0, (r/R)^{1/2})$. Then for any $N \in \mathbb{N}$ and $z_i, w_i \in \mathbb{C}$ with $|w_i| = r$ and $|z_i| = R$ for $i =1, \dots, N$ we have
\begin{equation}\label{CDetBound}
\left|\det \left[ \frac{1}{z_i - w_j}\right]_{i,j=1}^N \right| \leq \frac{N^{N/2}}{(R\alpha - r\alpha^{-1})^N} \cdot \left(\frac{R \alpha + r\alpha^{-1} }{r + R} \right)^{N^2}.
\end{equation}
\end{lemma}

The second result we require is as follows.
\begin{lemma}\label{S5techies}
Let $t \in (0, 1)$ and suppose that $U \in \mathbb{R}$ satisfies $0 < U \leq [-\log t]/4$. Suppose further that $z \in \mathbb{C}$ is such that $Re(z) \in [U, [-\log t]/2]$. Then there exists a constant $C_t^1 > 0$, depending on $t$ alone, such that the following holds
\begin{equation}\label{S5yellow1}
 \sum_{k \in \mathbb Z} \left| \frac{1}{\sin(-\pi [z + 2\pi\iota k]/ [-\log t])}\right| \leq \frac{C_t^1}{U}.
\end{equation}
\end{lemma}

The third result we require is as follows.
\begin{lemma}\label{descentLemma2} Let $S_a(z)$ be as in (\ref{FunExp}). There exist universal constants $C_1, C_2> 0$ such that the following holds. Let $a \in (0,e^{-2\pi}]$ and $\delta \in [0, \pi]$ be given. Then for $y\in [-\pi, \pi]$ we have
\begin{equation}\label{SaBoundVert}
\pm Re \left[ S_a(\pm \delta + \iota y  )  \right] \leq - a \cdot \delta \cdot \left[ C_1 \cdot y^2 - C_2 \cdot \delta^2   \right] . 
\end{equation}
\end{lemma}

The fourth result we require is as follows.
\begin{lemma}\label{descentLemma3} Let $R_a(z)$ be as in (\ref{FunExp}). Suppose that $a \in (0, e^{-2\pi}]$ and $z = x + \iota y$ with $x, y \in [-\pi, \pi]$. There is a universal constant $C_3 > 0$ such that
\begin{equation}\label{RBound2}
-C_3 \cdot a \cdot|z|^2 \leq Re[R_a(z)]  \leq  C_3 \cdot a \cdot  |z|^2.
\end{equation}
\end{lemma}

The fifth result we require is as follows.
\begin{lemma}\label{descentLemma4} There is a universal constant $C_4 > 0$ such that for all $\delta \in [-\pi, \pi]$ we have 
\begin{equation}\label{HypSineBound}
\frac{e^{\delta/2} + e^{-\delta/2} }{ e^{\delta} + e^{-\delta} } \leq e^{-C_4 \delta^2}. 
\end{equation}
\end{lemma}

The sixth result we require is as follows.
\begin{lemma}\label{descentLemma5} For any $x \in [0, 1/2]$ we have
\begin{equation}\label{RatFracIneq}
\frac{1 + x}{1 - x} \leq e^{3x}.
\end{equation}
\end{lemma}

The seventh result we require is as follows.
\begin{lemma}\label{LemmaCDet2} There exists a function $f: (0,\infty) \rightarrow (0,\infty)$ such that the following holds. Let $\delta \in (0, 1],$ $c > 0$ and $N \in \mathbb{N}$. Suppose that $x_i, y_i \in \mathbb{R}$ for $i = 1, \dots ,N$  are such that 
\begin{equation}\label{squaresBound}
\sum_{i = 1}^N (x_i^2 +  y_i^2) \leq N c^2 \delta^2 .
\end{equation}
Then 
\begin{equation}\label{CDetSquare}
\left| \det \left[ \frac{1}{e^{\delta} e^{\iota x_i} - e^{-\delta} e^{\iota y_j}} \right]_{i,j = 1}^N\right| \leq  \frac{e^{-f(c) N^2} N^{N/2}}{(e^{\delta/2} - e^{-\delta/2})^N}.
\end{equation}
\end{lemma}

The eighth result we require is as follows.
\begin{lemma}\label{BoundMixedS} There exists a universal constant $C_5 > 0$ such that the following holds. For any $t \in (0, e^{-2\pi}]$ we have that 
\begin{equation}
\frac{(-te^2;t)_\infty}{(t e^2;t)_\infty} \leq \exp \left( tC_5 \right).
\end{equation}
\end{lemma}

The ninth result we require is as follows.
\begin{lemma}\label{LBoundQ} Let $z = x+ iy$ with $|x| \leq 2\pi$, $|y| \leq 2\pi$ and $t \in (0, e^{-4\pi}]$. For such a choice of $z$ and $t$ define the function
\begin{equation}\label{DefQ}
Q_t(z) = \log (te^z;t)_\infty - \log(t;t)_\infty + z \cdot \sum_{n = 1}^\infty \frac{t^n}{1 - t^n}.
\end{equation}
Then we have
\begin{equation}\label{QBound}
-2 t|z|^2 \leq Re[Q_t(z)]  \leq 2t |z|^2.
\end{equation}
\end{lemma}

%
\subsection{Proof of Proposition }\label{Section6.2} In this section we give the proof of Proposition \ref{PropTermBound}. We will follow the notation from Section \ref{Section4.1} and Section \ref{Section6.1} above. For clarity we split the proof into several steps. In the first step we specify $a_2,t_2$, fix the parameters $a, t, x_1, x_2, s_1, s_2$ and specify $M_2$ as in as in the statement of the proposition. In the second step we formulate a certain inequality for $|I_M(N_1, N_2,k)|$ as (\ref{BoundEachTerm}), where we recall that $I_M(N_1, N_2,k)$ were defined in equation (\ref{INK}). Assuming the inequality (\ref{BoundEachTerm}) we deduce the statement of the proposition. The next three steps establish (\ref{BoundEachTerm})  by considering the cases when $N$ is bigger, roughly of the same size as or much smaller than $\Delta = N_1 + N_2$. \\

{\bf \raggedleft Step 1.} In this step we formulate our choice of $a_2$, $t_2$ and $M_2$ as in the statement of the proposition. Let us denote for simplicity
\begin{equation}\label{defp}
p =\frac{2^{1/2} + 2^{-1/2}}{5/2},
\end{equation}
which is a fixed constant in $(0,1)$ that will be used later in the arguments. We set $c =  \sqrt{\frac{8(C_2+C_5)}{C_1}}$, where $C_1, C_2$ are as in Lemma \ref{descentLemma2} and $C_5$ is as in Lemma \ref{BoundMixedS}. We also pick $R > 1$ sufficiently large so that 
\begin{equation}\label{defR}
p^{-1/8} \geq e^{18/R} \mbox{ and } \frac{2C_2 + 2C_5}{R} \leq \frac{1}{8} \cdot f\left( c\right),
\end{equation}
where the function $f$ is the one afforded to us by Lemma \ref{LemmaCDet2}. 

We fix $a_2, t_2  \in (0,1)$ sufficiently small so that the following conditions hold:
\begin{equation}\label{Cond1}
t_2 \leq e^{-4\pi} \mbox{ and }a_2 \leq e^{-2\pi};
\end{equation}
\begin{equation}\label{Cond2}
 \left(\frac{(-4t_2;t_2)_\infty}{(4t_2;t_2)_\infty} \right) \leq p^{-1/8};
\end{equation}
\begin{equation}\label{Cond3}
t_2 \leq R^{-1}.
\end{equation}

The above fixes our definition of $a_2$ and $t_2$ and we fix $a \in (0, a_2]$ and $t \in (0, t_2]$ as well as $x_1, x_2  \in \mathbb{R}$ and $s_1, s_2 \in \mathbb{R}$ such that $s_1 > s_2$. For this choice of parameters we let $n_1, n_2, u_1, u_2$ be the $M$-dependent functions from Definition \ref{constants}. We proceed to specify $M_2$ as in the statement of the proposition. 
We pick $M_2$ sufficiently large depending on $x_1, x_2, s_1, s_2, a ,t$ so that the following inequalities all hold for $M \geq M_2$:
\begin{equation}\label{M1}
\begin{split}
M \cdot f_1 +   M^{1/3}  \sigma_a  x_i + f_1'   [M - n_i] + (1/2) f''_1 s_i^2 M^{1/3} \geq 0 \mbox{ for $i = 1,2$};
\end{split}
\end{equation}
\begin{equation}\label{M3}
\begin{split}
2M  \geq n_i  \mbox{ for $i = 1,2$};
\end{split}
\end{equation}
\begin{equation}\label{M4}
\begin{split}
M^{-1/6} \leq (s_1-s_2)(\epsilon_1/2);
\end{split}
\end{equation}
\begin{equation}\label{M2}
\begin{split}
& C_1 (Ra)^{-1/3} M^{5/6} - 2 C_3 |M-n_i| \geq 0, \mbox{ for $i = 1,2$}
\end{split}
\end{equation}
where $\epsilon_1$ is as in Lemma \ref{LemmaTaylor}, $C_2$ and $C_3$ are the universal positive constants from Lemmas \ref{descentLemma2} and \ref{descentLemma3} respectively and $R$ is as in (\ref{defR}). We recall that $f_1, f_1', f_1'', \sigma_a$ were all defined in (\ref{eqnConst}) and that $M-n_i = O(M^{2/3})$, which is why the choice of $M_2$ is possible. This fixes our choice of $M_2$.

For future use, we let $C_\alpha$ be sufficiently large depending on $x_1, x_2, s_1, s_2, a$ so that 
\begin{equation}\label{Calpha}
2\pi \cdot [ (|x_1| + |x_2| )  \cdot \sigma_a + |f_1''| (s_1^2 + s_2^2)]\leq C_\alpha.
\end{equation}

{\bf \raggedleft Step 2.} Recall from (\ref{DefIN}) that 
\begin{equation}\label{S6INKD}
I_M(N_1,N_2) = \sum_{k = 0}^{\min(N_1,N_2)} I_M(N_1, N_2, k), 
\end{equation}
where $I_M(N_1,N_2,k)$ are as in (\ref{INK}) -- we recall this formula here for the reader's convenience in a form that will be easier to work with below.
\begin{equation}\label{S6INK}
\begin{split}
&I_M(N_1, N_2, k) = \frac{(-1)^{N_1 + N_2}}{k! (N_1- k)! (N_2- k)!}  \int_{\gamma_{1}^{N_1-k}} \int_{\gamma^{N_1-k}_3}  \int_{\gamma_2^{N_2-k}}\int_{\gamma^{N_2-k}_{4}} \int_{\gamma^{k}_1}\int_{\gamma^{k}_3} \int_{\gamma^k_4} \\
& A_1^M  A_2^M A_3^M A_4^M\prod_{i = 1}^{k}\frac{d\hat{w}^2_i}{2\pi \iota}\prod_{i = 1}^{k}\frac{d{w}^2_i}{2\pi \iota}\prod_{i = 1}^{k}\frac{dz^2_i}{2\pi \iota}\prod_{i = 1}^{N_2 - k}\frac{d\hat{w}^1_i}{2\pi \iota} \prod_{i = 1}^{N_2 - k} \frac{d\hat{z}_i}{2\pi \iota} \prod_{i = 1}^{N_1 - k}\frac{dw^1_i}{2\pi \iota}\prod_{i = 1}^{N_1- k}\frac{dz^1_i}{2\pi \iota}.
\end{split}
\end{equation}
In the above formula we have that $A_1^M = \det B$ with $B$ being the matrix from (\ref{BlockMatrix}). Also we have 
\begin{equation}\label{C1DefA2}
\begin{split}
&A^M_2 = \prod_{i = 1}^{N_1-k} \prod_{j = 1}^{N_2-k}\frac{(t\hat{z}^1_j/ z^1_i; t)_\infty (t\hat{w}^1_j / w^1_i; t)_\infty}{(t\hat{w}^1_j/ z^1_i;t)_\infty (t\hat{z}^1_j/ w^1_i;t)_\infty} \cdot  \prod_{i = 1}^{k} \prod_{j = 1}^{N_2-k} \frac{(t\hat{z}^1_j / z^2_i;t)_\infty (t\hat{w}^1_j/ w^2_i;t)_\infty}{(t\hat{w}^1_j/ z^2_i;t)_\infty (t\hat{z}^1_j/ w^2_i; t)_\infty} \cdot \\
&   \prod_{i = 1}^{N_1-k}\prod_{j = 1}^{k} \frac{(tw^2_j/ z^1_i;t)_\infty (t\hat{w}^2_j/ w^1_i;t)_\infty}{(t\hat{w}^2_j/ z^1_i;t)_\infty (tw^2_j/ w^1_i; t)_\infty} \cdot \prod_{i = 1}^{k} \prod_{j = 1}^{k} \frac{(tw^2_j/ z^2_i)_\infty (t\hat{w}^2_j/ w^2_i;t)_\infty}{(t\hat{w}^2_j/ z^2_i;t)_\infty (tw^2_j/ w^2_i;t)_\infty};
\end{split}
\end{equation}
\begin{equation}\label{C1DefA3}
\begin{split}
A^M_3 = &\prod_{i =1}^{N_2-k}  [-u_2]^{[\log(\hat{w}_i^1) - \log(\hat{z}_i^1)][\log t]^{-1} }\left( \frac{1 + a\hat{z}^1_i}{1 + a \hat{w}_i^1}\right)^M \cdot \left(\frac{1 + a/\hat{w}^1_i}{1 + a/\hat{z}_i^1} \right)^{n_2} \cdot \\
& \prod_{i =1}^{N_1-k} [-u_1]^{[\log(w_i^1) - \log(z_i^1)][\log t]^{-1} }  \left( \frac{1 + az_i^1}{1 + a w_i^1}\right)^M  \left(\frac{1 + a/w_i^1}{1 + a/z_i^1} \right)^{n_1} \cdot \\
& \prod_{i =1}^{k}[-u_2]^{[\log(\hat{w}_i^2) - \log(w_i^2)][\log t]^{-1} } [-u_1]^{[\log(w_i^2) - \log(z_i^2)][\log t]^{-1}}  \cdot  \\
& \left( \frac{1 + az_i^2}{1 + a \hat{w}_i^2}\right)^M  \left(\frac{1 + a/\hat{w}^2_i}{1 + a/ z_i^2} \right)^{n_2} \left(\frac{1 + a/z_i^2}{1 + a/w_i^2} \right)^{n_2- n_1};   \\
\end{split}
\end{equation}
\begin{equation}\label{C1DefA4}
\begin{split}
&A^M_4 =  \prod_{i =1}^{N_2-k} \frac{\hat{S}(\hat{w}_i^1,\hat{z}^1_i; u_2, t)}{ \log t \cdot \hat{w}_i^1 } \cdot \prod_{i =1}^{k} \frac{\hat{S}(\hat{w}_i^2,w^2_i; u_2, t)}{ \log t \cdot \hat{w}_i^2}  \frac{\hat{S}(w_i^2,z_i^2; u_1, t)}{ \log t \cdot w_i^2}   \cdot \prod_{i =1}^{N_1-k} \frac{\hat{S}(w^1_i,z_i^1; u_1, t)}{\log t \cdot w_i^1 } , \\
& \mbox{ where } \hat{S}(w,z;u,t) =  \sum_{m \in \mathbb{Z}} \frac{\pi \cdot [ - u ]^{-  2m \pi  \iota [\log t ]^{-1}}}{\sin(-\pi [[\log w - \log z] [\log t]^{-1} -  2m \pi  \iota [\log t ]^{-1}])}.
\end{split}
\end{equation}

We claim that there is a constant $\tilde{C}_2> 0$ that depends on $a,t, x_1, x_2, s_1,s_2$  such that for each $M, N_1, N_2 \in \mathbb{N}$ with $M \geq M_2$ as in Step 1 and $k \in \{1 , \dots, \min(N_1, N_2)\}$ we have 
\begin{equation}\label{BoundEachTerm}
|I_M(N_1, N_2, k)| \leq \tilde{C}_2^{\Delta} \cdot  \Delta^{-\Delta/4}, \mbox{ with $\Delta = N_1 + N_2$. }
\end{equation}
From (\ref{S6INKD}) and (\ref{BoundEachTerm}) we conclude that $|I_M(N_1, N_2) | \leq \Delta \cdot \tilde{C}_2^{\Delta} \cdot  \Delta^{-\Delta/4}$, which clearly implies the statement of the proposition. 

In the remainder of the proof we prove (\ref{BoundEachTerm}) by considering the three cases: (1) $\Delta \geq RaM$, (2) $\Delta  \in [M^{1/2}, RaM]$ and (3) $\Delta  \in [1, M^{1/2}],$ where $R$ is the universal constant from (\ref{defR}).\\

{\bf \raggedleft Step 3.}  In this and all the steps below we will denote by $C$ a generic constant that depends on $a$ and $t$ alone, whose value may change from line to line. The purpose of this step is to prove (\ref{BoundEachTerm}) if $\Delta \geq RaM$.\\ 

Let $C_+, C_-$ and $C_m$ be the positively oriented circles, centered at the origin, of radius $1/2$, $2$ and $1$ respectively.
Observe that by Lemma \ref{S3Analyticity} and Cauchy's theorem we may deform the $z_i^1$, $\hat{z}_i^1$ and $z_i^2$ contours in  (\ref{S6INK}) to $C_+$, the $w_i^1$, $\hat{w}_i^1$ and $\hat{w}_i^2$ contours to $C_-$ and the $w_i^2$ contours to $C_m$ without affecting the value of the integral. Here we used $a \in (0, a_2]$ with $a_2 \leq e^{-2\pi}$ -- see (\ref{Cond1}) -- which ensures we do not cross any poles in the process of the deformation. We now proceed to find appropriate upper bounds for each of the four terms $A^M_1, A^M_2, A^M_3$ and $A^M_4$ along these contours.\\

Let us first analyze $A^M_4$. Notice that if $u \in (-\infty, 0)$ we have that 
$$\left| [ - u ]^{-  2m \pi  \iota [\log t ]^{-1}}\right| = 1$$
for all $m \in \mathbb{Z}$. Consequently, if $1/2= |w|$ and $|z| = 2$ we have from Lemma \ref{S5techies} applied to $U = 2 \log 2$ 
$$|\hat{S}(w,z;u_i,t)|  \leq \pi \cdot \sum_{k \in \mathbb Z} \left| \frac{1}{\sin(-\pi[\log z - \log w + 2\pi\iota k]/[-\log t])}\right| \leq  \frac{\pi C_t^1}{2 \log 2} \mbox{ for $i =1 ,2$},$$
where $C_t^1$ is as in Lemma \ref{S5techies} and depends on $t$ alone. In applying the above lemma we used that $t \leq t_2 \leq e^{-4\pi}$ -- see (\ref{Cond1}) -- this ensures that $U = 2 \log 2 \leq [-\log t]/4$. 
The above inequality (using our definition of $C_+$ and $C_-$) implies that 
$$\left| \prod_{i =1}^{N_2-k} \frac{\hat{S}(\hat{w}_i^1,\hat{z}^1_i; u_2, t)}{ \log t \cdot \hat{w}_i^1} \cdot \prod_{i =1}^{N_1-k} \frac{\hat{S}(w^1_i,z_i^1; u_1, t)}{\log t \cdot w^1_i} \right| \leq C^{N_1 +N_2 - 2k},$$
where we recall that $C$ stood for a generic constant that depends on $a$ and $t$ alone, whose value changes from line to line. We will not mention this further. We may again apply Lemma \ref{S5techies} for $U = \log 2$ to conclude that if $1/2 = |w| < |z| = 1$ or $1 = |w| < |z| = 2$ we have
$$|\hat{S}(w,z;u_i,t)|  \leq   \frac{\pi C_t^1}{\log 2} \mbox{ for $i =1 ,2$}.$$
The above inequality (using our definition of $C_{+}, C_-$ and $C_m$) implies that 
$$\left|  \prod_{i =1}^{k} \frac{\hat{S}(\hat{w}_i^2,w^2_i; u_2, t)}{ \log t \cdot \hat{w}_i^2 }  \frac{\hat{S}(w_i^2,z_i^2; u_1, t)}{ \log t \cdot w_i^2}  \right| \leq C^{2k}.$$
Combining the above two inequalities we conclude that 
\begin{equation}\label{C1BoundH4}
\begin{split}
&|A^M_4|  \leq C^{\Delta}.
\end{split}
\end{equation}

We next analyze $A^M_2$. Recall from (\ref{S3PochBound}) that if $|\zeta| \leq \alpha < 1$ we have that 
\begin{equation*}
(\alpha;t)_\infty = \prod_{n = 0}^\infty( 1 - \alpha t^n) \leq \prod_{n = 0}^\infty | 1 - \zeta t^n| = | (\zeta;t)_\infty | \leq \prod_{n = 0}^\infty( 1 + \alpha t^n)  = (-\alpha; t)_\infty.
\end{equation*}
Using the latter inequalities and our definition of $C_-, C_+, C_m$ we conclude that 
\begin{equation}\label{C1BoundH2}
\begin{split}
&|A^M_2|  \leq \left( \frac{(-4t; t)_\infty}{(4t; t)_\infty}\right)^{2N_1N_2} \leq \left( \frac{(-4t; t)_\infty}{(4t; t)_\infty}\right)^{\Delta^2/2}.
\end{split}
\end{equation}
In deriving the above inequality we used that $t \leq t_2 \leq e^{-4\pi}$ -- see (\ref{Cond1}) -- which ensures that $4t <1$. \\

We continue with the analysis of $A^M_3$. Using the definition of $u_1$ and $u_2$ from (\ref{ScaleU}) and the definition of the contours $C_-, C_+$ and $C_m$ we have for each $i = 1, \dots, N_2 -k$
$$\left|[-u_2]^{[\log(\hat{w}_i^1) - \log(\hat{z}_i^1)][\log t]^{-1} }\hspace{-0.5mm} \right| \hspace{-0.5mm}=\hspace{-0.5mm} \exp \left( \hspace{-0.5mm} -2 \log 2  \hspace{-0.5mm} \left( \hspace{-0.5mm} M  f_1 +   M^{1/3}  \sigma_a  x_2 + f_1'   [M - n_2] \hspace{-0.5mm} +\hspace{-0.5mm}  (1/2) f''_1 s_2^2 M^{1/3}\hspace{-0.5mm} \right) \hspace{-1mm} \right) \hspace{-0.5mm} \leq 1;$$ 
for each $i = 1, \dots, N_1 - k$
$$ \left| [-u_1]^{[\log(w_i^1) - \log(z_i^1)][\log t]^{-1} }\hspace{-0.5mm} \right| \hspace{-0.5mm} =\hspace{-0.5mm} \exp \hspace{-0.5mm}\left(\hspace{-0.5mm}  - 2\log 2  \left(  M f_1 +   M^{1/3}  \sigma_a  x_1\hspace{-0.5mm} +\hspace{-0.5mm}  f_1'   [M - n_1] \hspace{-0.5mm} +\hspace{-0.5mm}  (1/2) f''_1 s_1^2 M^{1/3}\hspace{-0.5mm} \right)\hspace{-1mm} \right) \hspace{-0.5mm} \leq 1; $$
for each $i = 1, \dots , k$
$$\left| [-u_2]^{[\log(\hat{w}_i^2) - \log(w_i^2)][\log t]^{-1} } \hspace{-0.5mm}\right| \hspace{-0.5mm}= \hspace{-0.5mm}\exp\hspace{-0.5mm} \left( \hspace{-0.5mm} - \log 2  \left( \hspace{-0.5mm}  M  f_1 +   M^{1/3}  \sigma_a  x_2 + f_1'   [M - n_2] \hspace{-0.5mm} +\hspace{-0.5mm}  (1/2) f''_1 s_2^2 M^{1/3} \hspace{-0.5mm} \right) \hspace{-1mm} \right) \hspace{-0.5mm} \leq 1;$$
$$\left| [-u_1]^{[\log(w_i^2) - \log(z_i^2)][\log t]^{-1}}\hspace{-0.5mm} \right| \hspace{-0.5mm} = \hspace{-0.5mm} \exp \hspace{-0.5mm}\left( \hspace{-0.5mm} - \log 2   \left( M  f_1 +   M^{1/3}  \sigma_a  x_1 + f_1'   [M - n_1] + (1/2) f''_1 s_1^2 M^{1/3}\hspace{-0.5mm} \right) \hspace{-1mm} \right) \hspace{-0.5mm} \leq 1, $$
where in all of the above inequalities we used that $M \geq M_2$ and $M_2$ was chosen sufficiently large so that the exponents above are negative -- see (\ref{M1}). Furthermore, by the triangle inequality 
$$|1 + z| \leq 1 + |z| \mbox{ and } |(1- z)^{-1}| \leq (1- |z|)^{-1}$$
whenever $|z| < 1.$ Utilizing the latter and the inequality $n_i \leq 2 M$, which follows from (\ref{M3}), we see that for each $i = 1, \dots, N_2 - k$
$$\left| \left( \frac{1 + a\hat{z}^1_i}{1 + a \hat{w}_i^1}\right)^M \cdot \left(\frac{1 + a/\hat{w}^1_i}{1 + a/\hat{z}_i^1} \right)^{n_2}\right| \leq \left( \frac{1 + 2a}{1- 2a}\right)^{3M}.$$
Similarly, we have for each $i = 1, \dots, N_1 -k$ that
$$ \left| \left( \frac{1 + az_i^1}{1 + a w_i^1}\right)^M  \left(\frac{1 + a/w_i^1}{1 + a/z_i^1} \right)^{n_1} \right| \leq \left( \frac{1 + 2a}{1- 2a}\right)^{3M},$$
and for each $i = 1, \dots, k$ that
$$ \left|\left( \frac{1 + az_i^2}{1 + a \hat{w}_i^2}\right)^M  \left(\frac{1 + a/\hat{w}^2_i}{1 + a/ z_i^2} \right)^{n_2} \left(\frac{1 + a/z_i^2}{1 + a/w_i^2} \right)^{n_2- n_1} \right| \leq \left( \frac{1 + 2a}{1-2 a}\right)^{3M}.$$
Note that $2a \leq 1/2$ by our choice of $a \leq a_2\leq e^{-2\pi}$, see (\ref{Cond1}), and so by Lemma \ref{descentLemma5} we conclude 
$$\frac{1 + 2a}{1- 2a} \leq \exp \left( 6 a \right).$$
Summarizing all of the above inequalities we conclude that for $M \geq M_2$ we have
\begin{equation}\label{C1BoundH3}
\begin{split}
&|A^M_3| \leq \exp \left(9 M \Delta a  \right).
\end{split}
\end{equation}

We finally analyze $A^M_1$. Observe that by our definition of $C_-, C_+, C_m$ and Lemma \ref{LCDetBound} applied to $R =2$, $r = 1/2$ and $\alpha = 2^{-1/2}$ we have 
\begin{equation}\label{C1BoundH1}
\begin{split}
&|A^M_1|  \leq \frac{(N_1 +N_2 - k)^{(N_1 + N_2 -k)/2}}{(2^{1/2} - 2^{-1/2})^{N_1+ N_2- k}} \cdot \left(\frac{2^{1/2} + 2^{-1/2}}{5/2}\right)^{(N_1 + N_2 - k )^2} \leq C^{\Delta} \Delta^{\Delta/2}p^{\Delta^2/4} ,
\end{split}
\end{equation}
where in the last inequality we used that $\Delta \geq N_1 +N_2 - k \geq 2^{-1}\Delta$ as $k \leq \min (N_1, N_2)$. We also recall that $p$ was defined in (\ref{defp}).\\

Combining (\ref{C1BoundH4}), (\ref{C1BoundH2}), (\ref{C1BoundH3}) and (\ref{C1BoundH1}) we conclude that 
\begin{equation*}
|A^M_1 A^M_2 A^M_3 A^M_4| \leq C^{\Delta} \cdot \Delta^{\Delta/2} \cdot \left(p^{1/2} \cdot \frac{(-4t;t)_\infty}{(4t;t)_\infty}\right)^{\Delta^2/2} \cdot \exp (9M \Delta a).
\end{equation*}
Finally, using that $\Delta \geq RaM$ we conclude that 
\begin{equation}\label{C1BoundAll}
|A^M_1 A^M_2 A^M_3A^M_4| \leq  C^{\Delta} \cdot \Delta^{\Delta/2} \cdot \left(p^{1/2} \cdot e^{18/R} \cdot  \frac{(-4t;t)_\infty}{(4t;t)_\infty}\right)^{\Delta^2/2} \leq C^{\Delta} \cdot \Delta^{\Delta/2} \cdot p^{\Delta^2/8},
\end{equation}
where in the last inequality we used (\ref{Cond2}) and (\ref{defR}). It is now clear that (\ref{S6INK}) and (\ref{C1BoundAll}) together imply (\ref{BoundEachTerm}) since $p \in (0,1)$. This concludes Step 3.\\

{\bf \raggedleft Step 4.} The purpose of this step is to prove (\ref{BoundEachTerm}) if $\Delta \in [M^{1/2}, RaM]$. We put $\delta = \left( \Delta / RaM \right)^{1/3}$ and note that $\delta \in [M^{-1/6}(Ra)^{-1/3}, 1]$. We start with the same argument as in Step 3 except that the contours $C_+$ and $C_-$ that we deform to have radii $e^{\delta}$ and $e^{-\delta}$ respectively. Notice that our assumption that $a \leq a_2 \leq e^{-2\pi}$, see (\ref{Cond1}), implies that in the process of deforming to these contours we do not cross any poles and so the value of the integral does not change by Cauchy's theorem. As in Step 3 we proceed to derive estimates for $A^M_1, A^M_2, A^M_3, A^M_4$.

We first analyze $A^M_4$. Observe that $\delta \leq 1 \leq \frac{-\log t}{8}$ by our assumption that $t \leq t_2 \leq e^{-4\pi}$ from Step 1, see (\ref{Cond1}). Similarly to our work in Step 3, we may apply Lemma \ref{S5techies} with $U = 2\delta$ to get that for all $i = 1, \dots, N_2 - k$ we have
$$\left| \frac{\hat{S}(\hat{w}_i^1,\hat{z}^1_i; u_2, t)}{ \log t\cdot \hat{w}_i^1 } \right| \leq \frac{\pi \cdot e \cdot C_t^1}{2\delta [-\log t]}.$$
Analogously for all $i = 1, \dots, N_1 - k$ we have
$$\left|\frac{\hat{S}(w^1_i,z_i^1; u_1, t)}{\log t \cdot w^1_i } \right| \leq \frac{\pi \cdot e \cdot C_t^1}{2\delta[-\log t] }.$$
Also we may apply Lemma \ref{S5techies} with $U = \delta$ to get that for all $i = 1, \dots,  k$ we have
$$\left| \frac{\hat{S}(\hat{w}_i^2,w^2_i; u_2, t)}{ \log t \cdot \hat{w}_i^2}  \frac{\hat{S}(w_i^2,z_i^2; u_1, t)}{ \log t \cdot w_i^2} \right| \leq  \left( \frac{\pi \cdot e \cdot C_t^1}{\delta  [-\log t] } \right)^2.$$
Combining all of the above inequalities we arrive at
\begin{equation}\label{C2BoundH4}
\begin{split}
&|A^N_4|  \leq \delta^{-\Delta} C^{\Delta} \leq C^{\Delta} \exp \left( \frac{ \Delta \log \Delta }{3} \right),
\end{split}
\end{equation}
where in the last inequality we used the definition of $\delta$ and the fact that $M \leq \Delta^2$. \\

Next we analyze $A^M_2$. Notice that by (\ref{S3PochBound}) and (\ref{C1DefA2}) we have that 
\begin{equation}\label{C2BoundH2}
\begin{split}
&|A^M_2| \leq \left( \frac{(-te^{2\delta};t)_\infty}{(t e^{2\delta}; t)_\infty } \right)^{2N_1N_2} \leq \left( \frac{(-te^{2};t)_\infty}{(t e^{2}; t)_\infty } \right)^{2N_1N_2} \leq \exp \left( tC_5 \Delta^2/2 \right),
\end{split}
\end{equation}
where in the first inequality we used that $\delta \in (0,1]$ and that $t \leq t_2 \leq e^{-4\pi}$, see (\ref{Cond1}). In the last inequality we used Lemma \ref{BoundMixedS} and the constant $C_5$ is as in the statement of that lemma. \\

We next turn our attention to $A^M_3$ and it will be convenient to consider the following change of variables. Let us put 
$$\log (\hat{z}^1_i) = \hat{Z}^1_i = \delta + \iota \hat{X}^1_i, \hspace{2mm}  \log(\hat{w}^1_i) = \hat{W}^1_i =- \delta + \iota \hat{Y}^1_i \mbox{ with $\hat{X}^1_i,\hat{Y}^1_i  \in [-\pi, \pi]$ for $i = 1,\dots, N_2 - k$};$$
$$\log (z^1_i) = Z^1_i = \delta + \iota X^1_i, \hspace{2mm}  \log(w^1_i) = W^1_i =- \delta + \iota Y^1_i \mbox{ with $X^1_i,Y^1_i  \in [-\pi, \pi]$ for $i = 1,\dots, N_1 - k$};$$
$$\log(\hat{w}^2_i) = \hat{W}^2_i =-\delta + \iota \hat{Y}^2_i, \hspace{2mm} \log(z^2_i) = Z^2_i = \delta + \iota X^2_i, \hspace{2mm} \log(w^2_i) = W^2_i = \iota Y^2_i $$
$$ \mbox{ with $\hat{Y}^2_i,Y^2_i , X^2_i \in [-\pi, \pi]$ for $i = 1,\dots, k$ }.$$
 Recall from (\ref{DefH3V2}) and (\ref{DefH3V22}) that in these variables we have $A^M_3 =  \prod_{i = 1}^7 A^M_{3,i}$ where
\begin{equation}\label{S6DefH3V2}
\begin{split}
& A^M_{3,1}=  \prod_{i =1}^{N_2-k}e^{ F^M_{2}(\hat{Z}_i^1)}, \hspace{2mm}  A^M_{3,2} = \prod_{i =1}^{N_2-k}e^{- F^M_2(\hat{W}_i^1)},\hspace{2mm} A^M_{3,3} =\prod_{i =1}^{N_1-k} e^{F^M_1({Z}_i^1)}, \\
& A^M_{3,4} = \prod_{i =1}^{N_1-k} e^{-F_1^M(W_i^1)}, A^M_{3,5}= \prod_{i =1}^{k} e^{F_1^M(Z_i^2)},  A^M_{3,6}= \prod_{i =1}^{k} e^{-F_2^M(\hat{W}_i^2)}, A^M_{3,7} = \prod_{i =1}^{k} e^{F_3^M(W_i^2)},
\end{split}
\end{equation}
where $F_1^M, F_2^M, F_3^M$ are given by
\begin{equation}\label{S6DefH3V22}
\begin{split}
&F_1^M(Z) = \exp\left({M S_a(Z) -  M^{1/3}\sigma_a x_1 Z + (M-n_1)R_a(Z) - (Z/2) f''_1 s_1^2 M^{1/3}}\right), \\
& F_2^M(Z) = \exp \left({MS_a(Z) - M^{1/3} \sigma_a x_2 Z} + (M-n_2)R_a(Z) - (Z/2) f''_1 s_2^2 M^{1/3} \right)  \\
& F_3^M(W) = \exp \left(  M^{1/3} \sigma_a (x_1 - x_2)W  + (n_1-n_2)R_a(W) + (W/2)M^{1/3} f''_1 [s_1^2 - s_2^2]  \right).
\end{split}
\end{equation}

By Lemmas \ref{descentLemma2} and \ref{descentLemma3} we have that 
\begin{equation*}
\begin{split}
&|A^M_{3,1}| \leq \exp \left(C_2 M (N_2 -k)a \delta^3  - [C_1 a M  \delta - C_3 a |M-n_2|]  \sum_{i = 1}^{N_2 - k} (\hat{X}^1_i)^2  + C_\alpha [N_2- k] M^{1/3} \right)  \\
& |A^M_{3,2}|  \leq \exp \left(C_2 M (N_2 -k)a \delta^3  - [C_1 a M  \delta - C_3 a |M- n_2|]  \sum_{i = 1}^{N_2 - k} (\hat{Y}^1_i)^2  + C_\alpha [N_2 - k]  M^{1/3} \right), \\
&| A^M_{3,3}| \leq  \exp \left(C_2 M (N_1 -k)a \delta^3  - [C_1 a M  \delta -C_3 a |M-n_1|]  \sum_{i = 1}^{N_1 - k} (X^1_i)^2  + C_\alpha [N_1 - k]  M^{1/3} \right),  \\
&|A^M_{3,4}| \leq\exp \left(C_2 M (N_1 -k)a \delta^3  - [C_1 a M \delta -C_3a |M-n_1|]     \sum_{i = 1}^{N_1 - k} (Y^1_i)^2  + C_\alpha [ N_1 - k ] M^{1/3} \right),  \\
&| A^M_{3,5}| \leq\exp \left(C_2 M k a \delta^3  - [C_1 a M\delta  -C_3a |M-n_1|]    \sum_{i = 1}^{k} (X^2_i)^2  + C_\alpha k  M^{1/3} \right) ,\\
&|A^M_{3,6}| \leq  \exp \left(C_2 M k a \delta^3  - [C_1 a N  \delta - C_3 a |M-n_2| ]  \sum_{i = 1}^{k} (\hat{Y}^2_i)^2  + C_\alpha k   M^{1/3} \right).
\end{split}
\end{equation*}
Furthermore by (\ref{K3}) we know $Re[(R_a(\iota Y^2_i) ] \leq Re[R_a(0)]= 0$ for $ i  =1 ,\dots, k$. This implies that 
$$ |A^M_{3,7}| \leq \exp \left(  C_\alpha M^{1/3} k  \right).$$
Combining all of the above estimates for $A^M_{3,i}$ for $i =1, \dots , 7$  we conclude that
\begin{equation*}
\begin{split}
|A^M_3| \leq &\exp \left( 2C_2 M\Delta  a \delta^3  + 2C_\alpha \Delta M^{1/3}-  [C_1 a M  \delta - C_3 a |M-n_2|]  \sum_{i = 1}^{N_2 - k} [(\hat{X}_i^1)^2 + (\hat{Y}_i^1)^2]  \right) \cdot \\
&\exp\left(- [C_1 a M \delta  - C_3 |M-n_1|] \sum_{i = 1}^{N_1 - k} [(X_i^1)^2 + (Y_i^1)^2] )  \right) \cdot\\
& \exp\left(- [C_1 aM \delta - C_3a|M-n_1|] \sum_{i = 1}^{k} (X_i^2)^2 - [C_2 a M  \delta - C_3 a |M-n_2| ]  \sum_{i = 1}^{k} (\hat{Y}^2_i)^2    \right).
\end{split}
\end{equation*}
The above upper bound for $|A^M_3|$ can be simplified once we use the fact that $M \geq M_2$ and our choice of $M_2$ in Step 1 as well as the definition of $\delta$. Namely, recall from the beginning of the step that $\delta \in [M^{-1/6} (Ra)^{-1/3}, 1]$ and so we have 
$$C_1 a M  \delta - C_3 a  | M-n_i| \geq (C_1/2)a M  \delta + [(C_1/2)a (Ra)^{-1/3} M^{5/6} - C_3 a |M-n_i| ] \geq (C_1/2)a M  \delta,$$
where in the last inequality we used (\ref{M2}). Utilizing the last inequality, together with $M \leq \Delta^2$ and the fact that $\delta^3 M = (Ra)^{-1} \Delta$ we see that 
\begin{equation}\label{C2BoundH3}
\begin{split}
&|A^M_3| \leq \exp \left( 2C_2 R^{-1}\Delta^2    + 2C_\alpha \Delta^{5/3} \right) \cdot \exp \left( -  (C_1/2) a M \delta Q \right), \mbox{ where } \\
&Q = \sum_{i = 1}^{N_2 - k} [(\hat{X}_i^1)^2 + (\hat{Y}_i^1)^2] + \sum_{i = 1}^{N_1 - k} [(X_i^1)^2 + (Y_i^1)^2]   +  \sum_{i = 1}^{k} [ (X_i^2)^2 + (\hat{Y}_i^2)^2] .
\end{split}
\end{equation}

We next consider bounding $|A^M_1|$ and consequently $|A_1^M A_2^M A_3^M A_4^M|$ in the following two cases:
\begin{equation}\label{cases}
\begin{split}
(1) \hspace{2mm}  Q \geq  \frac{4(C_2 + C_5) }{C_1}\cdot \delta^2 \Delta \mbox{ and } (2) \hspace{2mm} Q \leq  \frac{4(C_2 + C_5) }{C_1}\cdot \delta^2 \Delta,
\end{split}
\end{equation}  
where $C_1, C_2$ are as in Lemma \ref{descentLemma2} and $C_5$ is as in Lemma \ref{BoundMixedS}.  \\

The first case in (\ref{cases}) is somewhat trivial and we simply use Hadamard's inequality to get
\begin{equation}\label{C2C1BoundH1}
\begin{split}
&|A^M_1|  \leq \frac{(N_1 +N_2 - k)^{(N_1 + N_2 -k)/2}}{(e^{\delta} - e^{-\delta})^{N_1+ N_2- k}} \leq \delta^{-\Delta}\Delta^{\Delta/2} \leq C^{\Delta}\exp \left( \frac{5 \Delta \log \Delta}{6} \right) ,
\end{split}
\end{equation}
where in the first inequality we used that $e^x - e^{-x} \geq 2x$ for $x > 0$ and the fact that $\Delta \geq N_1 +N_2 - k \geq 2^{-1}\Delta$ as $k \leq \min (N_1, N_2)$. In the second inequality we used the definition of $\delta$ and the fact that $M \leq \Delta^2$. Combining the inequalities (\ref{C2BoundH4}), (\ref{C2BoundH2}), (\ref{C2BoundH3}) and (\ref{C2C1BoundH1}) we conclude that 
\begin{equation}\label{C2C1BoundAll}
\begin{split}
&|A^M_1A^M_2A^M_3A^M_4| \leq C^{\Delta} \exp \left( (tC_5/2 - 2C_5 R^{-1}) \Delta^2  +   \frac{7 \Delta \log \Delta}{6}  + 2C_{\alpha} \Delta^{5/3}  \right),
\end{split}
\end{equation}
where we used that $(C_1/2) a M\delta Q \geq 2(C_2+C_5) R^{-1} \Delta^2 $ in the first case in (\ref{cases}). \\

The second case in (\ref{cases}) is more complicated and to bound $|A_1^M|$ we apply Lemma \ref{LemmaCDet2} with $N = N_1 + N_2 - k$ and $c = \sqrt{\frac{8(C_2 + C_5)}{C_1}}.$ Indeed, from (2) in (\ref{cases}) and the fact that $N_1 + N_2 - k \geq \Delta/2$ as $k \leq \min(N_1, N_2)$ we see that 
$$Q \leq \delta^2 \cdot c^2 \cdot [N_2 +N_1 - k],$$
which in view of Lemma \ref{LemmaCDet2} implies 
\begin{equation}\label{C2C2BoundH1}
\begin{split}
&|A^M_1|  \leq \frac{e^{-f(c) (N_1 +N_2 - k)^2} (N_1 +N_2 - k)^{(N_1 + N_2 -k)/2}}{(e^{\delta} - e^{-\delta})^{N_1+ N_2- k}}\leq C^{\Delta}e^{-f(c) \Delta^2/4}\exp \left( \frac{5 \Delta \log \Delta}{6} \right) ,
\end{split}
\end{equation}
where $f(c)$ is as in Lemma \ref{LemmaCDet2} and the second inequality is derived as in (\ref{C2C1BoundH1}). Combining the inequalities (\ref{C2BoundH4}), (\ref{C2BoundH2}), (\ref{C2BoundH3}) and (\ref{C2C2BoundH1}) we conclude that 
\begin{equation}\label{C2C2BoundAll}
\begin{split}
&|A^M_1A^M_2A^M_3A^M_4| \leq C^{\Delta} \exp \left( (tC_5/2 + 2C_2 R^{-1}- f(c)/4) \Delta^2  +   \frac{7 \Delta \log \Delta}{6}  + 2C_{\alpha} \Delta^{5/3}\right),
\end{split}
\end{equation}
where we used that $Q \geq 0$.

 Recall from (\ref{defR}) and (\ref{Cond3}) that 
$$tC_5/2 + 2C_2 R^{-1}- f(c)/4 \leq R^{-1} (C_5/2 + 2C_2) \leq -f(c)/8 \mbox{ and }tC_5/2- 2C_5 R^{-1} \leq - (3/2) R^{-1} C_5.  $$. 
 So if we combine (\ref{C2C1BoundAll}) and (\ref{C2C2BoundAll}) we conclude that for $\epsilon = \min (f(c)/8, (3/2) R^{-1} C_5)$ we have
\begin{equation}\label{C2BoundAll}
\begin{split}
&|H^N_1H^N_2H^N_3H^N_4| \leq C^{\Delta} \exp \left( -\epsilon \Delta^2  +   \frac{7 \Delta \log \Delta}{6}  + 2C_{\alpha} \Delta^{5/3}\right),
\end{split}
\end{equation}
which together with (\ref{S6INKD}) clearly implies (\ref{BoundEachTerm}). This concludes Step 4.\\

{\bf \raggedleft Step 5.} The purpose of this step is to prove (\ref{BoundEachTerm}) if $\Delta \in [1, M^{1/2}]$. In this step we use the formula (\ref{INK3}) for $I_M(N_1, N_2, k)$, which we recall here for the reader's convenience.
\begin{equation}\label{S5INK3}
\begin{split}
&I_M(N_1, N_2, k) = \frac{(-1)^{N_1 + N_2}}{k! (N_1- k)! (N_2- k)!}  \int_{\gamma_z^{N_1-k}}  \int_{\gamma_z^{N_2-k}} \int_{\gamma_z^{k}} \int_{\gamma_w^{N_1-k}}\int_{\gamma_w^{N_2-k}} \int_{\gamma_w^k}\int_{\gamma_m^{k}}      \\
&H^M_1  H^M_2 H^M_3  H^M_4  \prod_{i = 1}^{k} \frac{{\bf 1}\{ |w_i^2| \leq \pi L\} d{w}^2_i}{2\pi \iota  L} \prod_{i = 1}^{k} \frac{{\bf 1}\{ |w_i^2| \leq \pi \rho_A  L \} d\hat{w}^2_i }{2\pi \iota  L} \prod_{i = 1}^{N_2 - k} \frac{{\bf 1}\{ |\hat{w}_i^1| \leq \pi \rho_A  L\} d\hat{w}^1_i}{2\pi \iota L} \\
& \prod_{i = 1}^{N_1 - k}  \frac{{\bf 1}\{ |w_i^1| \leq \pi \rho_A  L \} dw^1_i}{2\pi \iota  L}  \prod_{i = 1}^{k}   \frac{{\bf 1}\{ |z_i^2| \leq \pi \rho_A  L \} e^{{z}^2_i/ L}dz^2_i}{2\pi \iota  L}  \\
&\prod_{i = 1}^{N_2 - k}\frac{{\bf 1}\{ |\hat{z}_i^1| \leq \pi \rho_A  L \} e^{\hat{z}^1_i/L}d\hat{z}_i}{2\pi \iota L} \prod_{i = 1}^{N_1- k}\frac{{\bf 1}\{ |z_i^1| \leq \pi \rho_A   L \}  e^{{z}^1_i/L}dz^1_i}{2\pi \iota  L} ,
\end{split}
\end{equation}
We recall that $H_1^M$, $H_2^M$, $H_3^M$ and $H_4^M$ are as in (\ref{TildeBlockMatrix}), (\ref{DefH2V2}), (\ref{DefH3V2}) and (\ref{DefH4V2}) respectively. In addition, $\gamma_w, \gamma_z, \gamma_m$ and $A$ are as in \ref{Defcontours}, $L = M^{1/3}$ and $\rho_A = \sqrt{1 + A^2}$. As in Step 3 in Section \ref{Section5.2} we let $E$ denote  the set of points $(\vec{z}^1, \vec{\hat{z}}^1, \vec{z}^2, \vec{w}^1, \vec{\hat{w}}^1, \vec{\hat{w}}^2, \vec{w}^2) \in \mathbb{C}^{2N_1 + 2N_2 - k}$ that satisfy the inequalities in the indicator functions in (\ref{S5INK3}). Then we know from (\ref{BoundH1H4}) that 
 \begin{equation}\label{S5BoundH1H4}
\begin{split}
&|{\bf 1}_{E} H^M_1|  \leq   C^{N_1+N_2} (N_1 +N_2 - k)^{(N_1 + N_2 - k)/2} L^{N_1 +N_2 -k}, \mbox{ and } \\
&|{\bf 1}_{E} H^M_4| \leq C^{N_1 + N_2} L^{N_1 +N_2},
\end{split}
\end{equation}
where as earlier $C$ is a constant that depends on $a,t$ alone. Furthermore, from (\ref{DefH3V2})  and (\ref{BoundH3})
\begin{equation}\label{S5BoundH3}
\begin{split}
&|{\bf 1}_{E} H_3^M| \leq  \exp \left( - \epsilon_1 \sum_{i = 1}^{N_2 - k} \left[ |\hat{z}_i^1|^3 + |\hat{w}_i^1|^3   \right] -\epsilon_1\sum_{i = 1}^{N_1 - k} \left[ |{z}_i^1|^3 + |{w}_i^1|^3 \right] \right) \times  \\
& \exp \left(  - \epsilon_1 \sum_{i = 1}^{k} \left[ |{z}_i^2|^3 + |\hat{w}_i^2|^3   \right] -  \epsilon_1 (s_1 - s_2) \sum_{i = 1}^{k}  |{w}_i^2|^2   \right) \times \exp \left( \tilde{c}(N_1 +N_2 + \tilde{Q}) \right), \mbox{ where }\\
&\tilde{Q} = \sum_{i = 1}^{N_2 - k} \left[   |\hat{z}_i^1|^2 +|\hat{w}_i^1|^2   \right]  + \sum_{i = 1}^{N_1 - k} \left[ |{z}_i^1|^2 +|{w}_i^1|^2   \right]  + \sum_{i = 1}^{k}  \left[  |{z}_i^2|^2 + |\hat{w}_i^2|^2  + |{w}_i^2|  \right],
\end{split}
\end{equation}
where $\tilde{c} > 0$ is a constant that depends on $a,t,x_1, x_2, s_1, s_2$ and $\epsilon_1$ is as in Lemma \ref{LemmaTaylor}. 

Finally, we bound $|{\bf 1}_E H^M_2|$. By Lemma \ref{LBoundQ} we have that 
\begin{equation}\label{C3BoundH2}
\begin{split}
&|{\bf 1}_EH^M_2|  \leq \prod_{i = 1}^{N_1 -k}e^{16t N_2 |{z}^1_i|^2 /L^2} \prod_{i = 1}^{k}e^{16t N_2 |{z}^2_i|^2/L^2}\prod_{i = 1}^{N_1 -k}e^{16t N_2 |{w}^1_i|^2/L^2} \prod_{i = 1}^{N_2 -k}e^{16t N_1 |\hat{z}^1_i|^2/L^2} \\
& \prod_{i = 1}^{N_2 -k}e^{16t N_1 |\hat{w}^1_i|^2/L^2} \cdot \prod_{i = 1}^{k}e^{16t N_2  |{w}^2_i|^2/L^2}\prod_{i = 1}^{k}e^{16t N_2|\hat{w}^2_i|^2/L^2} \leq \\
& \exp\left(16 t \Delta  \tilde{Q}/L^2 + 16 t\Delta \sum_{i = 1}^{k}  |{w}_i^2|^2/L^2 \right)
\end{split}
\end{equation}

Let $H^M$ denote the integrand in (\ref{S5INK3}). Observe that by (\ref{S5BoundH1H4}), (\ref{S5BoundH3}) and (\ref{C3BoundH2}) we have that 
\begin{equation}\label{C3HBoundAll}
\begin{split}
&|H^M|\leq  (N_1 +N_2 - k)^{(N_1 + N_2 - k)/2} \exp \left( - \epsilon_1 \sum_{i = 1}^{N_2 - k} \left[ |\hat{z}_i^1|^3 + |\hat{w}_i^1|^3   \right] -\epsilon_1\sum_{i = 1}^{N_1 - k} \left[ |{z}_i^1|^3 + |{w}_i^1|^3 \right] \right)  \\
& \exp \left(  - \epsilon_1 \sum_{i = 1}^{k} \left[ |{z}_i^2|^3 + |\hat{w}_i^2|^3   \right] -  (\epsilon_1/2) (s_1 - s_2) \sum_{i = 1}^{k}  |{w}_i^2|^2  +\tilde{C}(N_1 +N_2 + \tilde{Q})  \right) ,
\end{split}
\end{equation}
where we used that $\Delta \leq M^{1/2}$, $\tilde{C} = \tilde{c} + 1$, $16t \leq 1$ (this follows from $t \leq t_1 \leq e^{-4\pi}$ -- see (\ref{Cond1})), and $M^{-1/6}  \leq (\epsilon_1/2)(s_1 -s_2)$ (this follows from (\ref{M4})).

 We now note that the upper bound in (\ref{C3HBoundAll}) does not depend on $M$ and is a product of functions in each of the $2N_1 + 2N_2 - k$ variables and each of those functions is integrable on its corresponding contour. For all but the $w^2_i$ variables this is true by the negative cube in the exponential with a positive coefficient $\epsilon_1$, which dominates the quadratic terms in $\tilde{Q}$. Also for the $w^2_i$ variables this is true by the negative square in the exponential with a positive coefficient $(\epsilon_1/2) (s_1 - s_2)$, which dominates the linear terms in $\tilde{Q}$. From (\ref{C3HBoundAll}) and (\ref{S5INK3}) we conclude that
$$ |I_N(N_1, N_2, k)| \leq  \frac{\tilde{C}_1^{\Delta}(N_1 +N_2 - k)^{(N_1 + N_2 - k)/2} }{k! (N_1- k)! (N_2- k)!},$$
for some large enough $\tilde{C}_1$ that depends on $a,t,x_1, x_2, s_1, s_2$.

Finally, we note that 
$$3^{N_1 + N_2 - k} = \sum_{i = 0}^{N_1 +N_2 -k } \sum_{j = 0}^{N_1 + N_2 -k - i} \frac{(N_1 +N_2 - k)!}{i! j! (N_1 +N_2 - k - i -j)!}.$$
Setting $i = k$ and $j = N_1 - k$ above we see that 
$$\frac{3^{N_1 + N_2 - k} }{(N_1 +N_2 - k)!} \geq  \frac{1}{k! (N_1- k)! (N_2- k)!}.$$
Using the latter we conclude that 
 $$ |I_N(N_1, N_2, k)| \leq  \frac{(3\tilde{C}_1)^{\Delta}(N_1 +N_2 - k)^{(N_1 + N_2 - k)/2}}{(N_1 +N_2 - k )!} \leq  $$
$$(3 \cdot e \cdot \tilde{C}_1)^{\Delta} \cdot  \exp\left( - (1/2) (N_1 +N_2 - k ) \log(N_1 +N_2 - k ) \right),$$
where we used $n! \geq  n^{n} e^{-n}$ -see \cite[Equation (1)]{Robbins}. The last inequality now clearly implies (\ref{BoundEachTerm}) once we utilizie that $N_1 +N_2 - k \geq 2^{-1}\Delta$ as $k \leq \min (N_1, N_2)$. This completes Step 5 and hence the proof of the proposition.

%
\section{Lemmas from Sections \ref{Section3} and \ref{Section4}}\label{Section7} In this section we present the proofs of several lemmas from Sections \ref{Section3} and \ref{Section4}.

%
\subsection{Proofs of lemmas from Section 3 }\label{Section7.1} In this section we give the proofs of Lemmas \ref{DetBounds}, \ref{S3BoundOnS}, \ref{S3Analyticity} , \ref{S3Expansion} and \ref{S3BigSumAnal} whose statements are recalled here for the reader's convenience as Lemmas \ref{S7DetBounds}, \ref{S7BoundOnS}, \ref{S7Analyticity} , \ref{Expansion} and \ref{BigSumAnal} respectively.

\begin{lemma}\label{S7DetBounds} Let $N \in \mathbb{N}$.
\begin{enumerate}
\item Hadamard's inequality: If $A$ is an $N \times N$ matrix and $v_1, \dots ,v_N$ denotes the column vectors of $A$ then $|\det A| \leq \prod_{i = 1}^N \|v_i\|$ where $\|x\| = (x_1^2 + \cdots + x_N^2)^{1/2}$ for $x = (x_1, \dots, x_N)$. 
\item Fix $r, R \in (0,\infty)$ with $R > r$. Let $z_i, w_i \in \mathbb{C}$ be such that $|z_i| = R$ and $|w_i|  \leq r$ for $i = 1, \dots, N$. Then
\begin{equation}\label{S7CDetGood}
\left|\det \left[ \frac{1}{z_i - w_j}\right]_{i,j = 1}^N \right| \leq R^{-N} \cdot \frac{N^N \cdot (r/R)^{\binom{N}{2}}}{(1-r/R)^{N^2}}.
\end{equation}
\end{enumerate}
\end{lemma}
\begin{proof} A proof of Hadamard's inequality can be found in \cite[Corollary 33.2.1.1.]{Prasolov}. In the remainder of the proof we focus on establishing (\ref{S7CDetGood}).

By the Cauchy determinant formula, see e.g. \cite[1.3]{Prasolov}, we have that
$$\det \left[ \frac{1}{z_i - w_j}\right]_{i,j = 1}^N  = \frac{1}{R^N} \cdot \frac{\prod_{1 \leq i < j \leq N} (z_i/R - z_j/R) (w_j/R - w_i/R)}{\prod_{i,j = 1}^N (z_i/R - w_j/R)}.$$
By the Vandermonde determinant formula, see e.g. \cite[1.2]{Prasolov}, and Hadamard's inequality we have 
$$\left|\prod_{1 \leq i < j \leq N} (z_i/R - z_j/R) \right| \leq N^{N/2}.$$
In addition, by our assumption that $|w_i| \leq r$ we have that 
$$\left|\frac{1}{\prod_{i,j = 1}^N (z_i/R - w_j/R)} \right| \leq  \frac{1}{(1- r/R)^{N^2}}.$$
Finally, using the Vandermonde determinant formula, Hadamard's inequality and the fact that $|w_i|\leq r$ we see that 
$$\left| \prod_{1 \leq i < j \leq N} (w_j/R - w_i/R) \right| \leq N^{N/2} \cdot \prod_{i =1}^N \frac{r^{i-1}}{R^{i-1}} = N^{N/2} \cdot (r/R)^{\binom{N}{2}}.$$
Combining the last four inequalities gives the second part of the lemma.
\end{proof}

\begin{lemma}\label{S7BoundOnS} Fix $t \in (0,1)$ and compact sets $K_1 \subset (t,1)$ and $K_2 \in \mathbb{C} \setminus [0, \infty)$. Then there exists a constant $M_0 > 0$ depending on $K_1, K_2, t$ such that if $z, w \in \mathbb{C}$ satisfy $|w| = r, |z| = R$ with $R > r > tR > 0$ and $r/R \in K_1$, and $u \in K_2$ then
\begin{equation}
\left| S(w, z; u,t) \right| \leq M_0, 
\end{equation}
where $S(w, z; u,t)$ is as in Definition \ref{DefFunS}.
\end{lemma}
\begin{proof}
Put $A = [\log w - \log z] [\log t]^{-1}$ and observe that by our assumptions there exists $\delta \in (0,1/2)$ depending on $K_1$ and $t$ such that 
$$Re[ A] = \frac{\log R - \log r}{- \log t} \in [\delta, 1 - \delta] .$$
Observe that given $c > 0$ we can find $c' > 0$ such that if $x,y \in \mathbb{R}$ and $d(x, \mathbb{Z}) \geq c$ then 
\begin{equation}\label{S7BoundSine}
\frac{1}{| \sin (\pi x + \iota \pi y)|} \leq c' e^{-\pi |y|}.
\end{equation}
Put $B = - 2\pi [\log t]^{-1}$ and $-u = q e^{\iota \phi}$ with $\phi \in (-\pi, \pi)$ and $q  = |u|$. Observe that there exist $\delta_1 \in (0,1) $ and $R_1 > 1$ (depending on $K_2$) such that $\phi \in [-\pi + \delta_1, \pi - \delta_1]$ and $q \leq R_1$ for all $u \in K_2$. 
Using that $d(Re[A], \mathbb{Z}) \geq \delta$ and (\ref{S7BoundSine}) for $c = \delta$ we see that for any $m \in \mathbb{Z}$ we have
$$\left| \frac{\pi \cdot [ - u ]^{[\log w - \log z] [\log t]^{-1} -  2m \pi  \iota [\log t ]^{-1}}}{\sin(-\pi [[\log w - \log z] [\log t]^{-1} -  2m \pi  \iota [\log t ]^{-1}])} \right| \leq  \pi c' \cdot e^{2  \pi | A|  } R_1^{|A|} e^{ - \delta_1 B |m| } \leq Ce^{ - \delta_1 B |m| },   $$
where the constant $C$ can be taken to be $C = \pi c' \cdot (e^{2\pi} R_1)^{1 + B}.$ Overall, we see that 
$$|S(w, z; u,t)| \leq \sum_{m \in \mathbb{Z}} Ce^{ - \delta_1 B |m| } =: M_0,$$
where $M_0$ is seen to be finite by comparison with the geometric series.
\end{proof}

\begin{lemma}\label{S7Analyticity} Fix $t \in (0,1)$ and $R, r \in (0, \infty)$ such that $R > r > tR$. Denote by $A(r,R) \subset \mathbb{C}$ the annulus of inner radius $r$ and outer radius $R$ that has been centered at the origin. Then the function $S(w, z; u,t)$ from Definition \ref{DefFunS} is well-defined for $(w,z,u) \in Y = \{ (x_1, x_2, x_3) \in  A(r,R) \times A(r,R) \times (\mathbb{C} \setminus [0, \infty)) : |x_1|< |x_2| \}$ and is jointly continuous in those variables (for fixed $t$) over $Y$. If we fix $u \in \mathbb{C} \setminus [0, \infty)$ and $w \in A(r,R)$ then as a function of $z$, $S(w, z; u,t)$ is analytic on $\{ \zeta \in A(r,R)  : |\zeta| > |w|\} $; analogously, if we fix $u \in \mathbb{C} \setminus [0, \infty)$ and $z \in A(r,R)$ then as a function of $w$, $S(w, z; u,t)$ is analytic on $\{ \zeta \in A(r,R)  : |\zeta| < |z|\} $. Finally, if we fix $w,z \in A(r,R)$ with $|z| > |w|$ then $S(w, z; u,t)$ is analytic in $\mathbb{C} \setminus [0, \infty)$ as a function of $u$.
\end{lemma}
\begin{proof} Suppose first that $(w_n, z_n, u_n) \in Y$ converges to $(w,z,u) \in Y$. We first prove that
\begin{equation}\label{LimitSCont}
\lim_{n \rightarrow \infty} S(w_n, z_n; u_n, t) = S(w, z; u, t),
\end{equation}
which implies the joint continuity of $S$. Let $r_n = |w_n|$, $R_n = |z_n|$, $q_n = |u_n|$ and $\theta_n, \psi_n, \phi_n \in (-\pi, \pi]$ be such that 
$$w_n = r_n e^{\iota \theta_n}, \hspace{5mm} z_n = R_n e^{\iota \psi_n}, \hspace{5mm}  u_n = q_n e^{\iota \phi_n }.$$
Analogously, we define $r_0, R_0, q_0, \theta_0, \psi_0, \phi_0$ such that 
$$w = r_0 e^{\iota \theta_0}, \hspace{5mm} z = R_0 e^{\iota \psi_0}, \hspace{5mm}  u_0 = q_0 e^{\iota \phi_0 }.$$
Observe that we have the following convergence statements:
\begin{equation}\label{CS1}
\lim_{n \rightarrow \infty} r_n = r_0, \hspace{5mm} \lim_{n \rightarrow \infty} R_n = R_0, \hspace{5mm} \lim_{n \rightarrow \infty} q_n = q_0, \hspace{5mm} \lim_{n \rightarrow \infty} \phi_n = \phi_0.
\end{equation}
Furthermore by possibly passing to a subsequence we have the following convergence statements
\begin{equation}\label{CS2}
\lim_{n \rightarrow \infty} \theta_n = \theta_0 + \epsilon_1 \cdot 2\pi \iota \mbox{ and }\lim_{n \rightarrow \infty} \phi_n = \phi_0 + \epsilon_2 \cdot 2\pi \iota,
\end{equation}
where $\epsilon_i \in \{0, -1\}$ for $i = 1, 2$. Indeed, if $\theta_0 \in (-\pi, \pi)$ the above convergence is ensured with $\epsilon_1 = 0$. If $\theta_0 = \pi$ then one can pass to a subsequence such that $\theta_n \in (-\pi, 0)$ or $\theta_n \in [0, \pi]$ -- in the former case $\epsilon_1 =1$ and in the latter $\epsilon_1 = 0$. Analogous arguments work for $\phi_n$ and $\phi_0$. We assume that we have already passed to a subsequence satisfying the above sets of convergence statements and prove (\ref{LimitSCont}) in this case -- we continue to use the index $n$. 

 For $m \in \mathbb{Z}$ we define 
$$H_m(w,z,u) =  \frac{\pi \cdot [ - u ]^{[\log w - \log z] [\log t]^{-1} -  2m \pi  \iota [\log t ]^{-1}}}{\sin(-\pi [[\log w - \log z] [\log t]^{-1} -  2m \pi  \iota [\log t ]^{-1}])}.$$
Observe that from (\ref{CS1}) and (\ref{CS2}) we have that 
$$\lim_{ n \rightarrow \infty} H_m(w_n, z_n, u_n) = H_{m+ \epsilon_2  - \epsilon_1} (w,z,u).$$
Furthermore, from the proof of Lemma \ref{S7BoundOnS} we know that if $K \subset Y$ is compact then we can find constants $C, \delta_1 > 0$ depending on $K$ such that if $(w,z,u) \in K$ then
\begin{equation}\label{BoundHm}
|H_m(w,z,u) | \leq Ce^{-\delta_1 B|m|},
\end{equation}
for all $m \in \mathbb{Z}$, where we recall that $B = - 2\pi [\log t]^{-1}$. By the Dominated convergence theorem with dominating function $Ce^{-\delta_1 B|m|}$ we have that 
$$\lim_{n \rightarrow \infty} S(w_n, z_n; u_n, t) = \lim_{n \rightarrow \infty} \sum_{m \in \mathbb{Z}} H_m(w_n,z_n,u_n) = $$
$$\sum_{m \in \mathbb{Z}}  \lim_{n \rightarrow \infty} H_m(w_n,z_n,u_n)  = \sum_{m \in \mathbb{Z}}  H_{m + \epsilon_2  - \epsilon_1}(w,z,u) = S(w, z; u, t).$$
This proves the continuity of $S(w,z;u ,t)$ in $Y$.\\

We next fix $u \in \mathbb{C} \setminus [0, \infty)$ and $z \in A(r,R)$ and show that as a function of $w$, $S(w,z;u,t)$ is analytic in $\{ \zeta \in A(r,R)  : |\zeta| < |z|\} $. Let $\Omega =  \{ \zeta \in A(r,R)  : |\zeta| < |z|\} \setminus (-\infty, 0]$. By the analyticity of $\log w$ in $\Omega$, we know that $H_m(w,z,u)$ is analytic in $w \in \Omega$ for each $m \in \mathbb{Z}$. Furthermore, by (\ref{BoundHm}) we have that over compacts $S(w,z;u,t)$ is the uniform limit of analytic functions and is thus analytic in $\Omega$, cf. \cite[Theorem 2.5.2]{Stein}. From the first part of the proof, we know that $S(w,z;u,t)$ is continuous in $\{ \zeta \in A(r,R)  : |\zeta| < |z|\}$ and so by the Symmetry principle (see \cite[Theorem 2.5.5]{Stein}) we conclude that $S(w,z;u,t)$  is analytic in $\{ \zeta \in A(r,R)  : |\zeta| < |z|\}$. One analogously proves that if we fix $u, w$ then $S(w,z;u,t)$ is analytic as a function of $z$ in $\{ \zeta \in A(r,R)  : |\zeta| > |w|\} $.

Finally, let us fix $z,w \in A(r,R)$ such that $|w| < |z|$. Then it is again clear by the analyticity of $\log [-u]$ in $\mathbb{C} \setminus [0, \infty)$ that $H_m(w,z,u)$ is analytic in $\mathbb{C} \setminus [0, \infty)$ for all $m \in \mathbb{Z}$. Combining this with (\ref{BoundHm}) we see that over compacts $S(w,z;u,t)$ is the uniform limit of analytic functions and is thus analytic in $\mathbb{C} \setminus [0, \infty)$ as a function of $u$, cf. \cite[Theorem 2.5.2]{Stein}. This suffices for the proof.
\end{proof}

Versions of the following two lemmas appear in Section 3.2 of \cite{BorCor}.
\begin{lemma}\label{Expansion} Let $ M ,n \in \mathbb{N}$, $N \in \mathbb{Z}_{\geq 0}$, $t \in (0,1)$ and $R > 0$. Suppose that $a, z\in \mathbb{C}$ saitsfy $0 < |a|=|z|  \leq 1$.  Assume further that $z_i, w_i , x_j, y_k \in \mathbb{C}$ satisfy $|z_i| \geq |z|$, $|w_i| \geq |z|$, $|x_j| \leq R$, $|y_k| \leq 1$ for $i = 1, \dots, N$, $j = 1, \dots, n$ and $k = 1, \dots, M$. Then if $u \in \mathbb{C}$ we have for any $c \in \mathbb{N}$ 
\begin{equation}\label{TermCBound}
\left| \frac{ u^c }{a - z t^{c} }    \prod_{k = 1}^M \frac{1}{1 - z y_k t^{c}}  \prod_{ j = 1}^{n} (1 - x_j z^{-1}t^{-c})  \prod_{i = 1}^{N}  \frac{ (z w_i^{-1} t^{c}; t)_\infty }{(zz_i^{-1} t^{c}; t)_\infty } \right| \leq  \frac{(-t; t)^N_{\infty}(1 + R/|a|)^n |u|^c t^{-cn} }{|a| (1 - t)^{M+1}(t; t)^N_{\infty}} .
\end{equation} 
In particular, if $|u| < t^{n}$ the function
\begin{equation}\label{HF}
H(u) = \sum_{c = 1}^\infty  \frac{ u^c }{a - z t^{c} }    \prod_{k = 1}^M \frac{1}{1 - z y_k t^{c}}  \prod_{ j = 1}^{n} (1 - x_j z^{-1}t^{-c})  \prod_{i = 1}^{N}  \frac{ (z w_i^{-1} t^{c}; t)_\infty }{(zz_i^{-1} t^{c}; t)_\infty } ,
\end{equation} 
is well-defined and finite. If $|u| < t^{n}$ and $u \in \mathbb{C} \setminus [0, \infty)$ we also have that
\begin{equation}\label{gResidues}
H(u) = \frac{1}{2\pi \iota} \int_{C}  \frac{ S(z,w; u,t) }{w(a - w) [-\log t]}    \prod_{k = 1}^M \frac{1}{1 - w y_k }  \prod_{ j = 1}^{n} (1 - x_j w^{-1})  \prod_{i = 1}^{N}  \frac{ (w w_i^{-1} ; t)_\infty }{(wz_i^{-1} ; t)_\infty } dw,
\end{equation}
where $C$ is a positively oriented circle of radius $ r \in (t |z|, |z|)$ that is centered at the origin and $S(z,w;u,t)$ is as in Definition \ref{DefFunS}. 
\end{lemma}
\begin{proof} For simplicity we split the proof into three steps. In the first step we establish (\ref{TermCBound}). In the second step we find a contour integral representation of $H(u)$ from (\ref{HF}) and in the third step we show that the contour integral we found in the second step equals the one in (\ref{gResidues}). \\

{\bf \raggedleft Step 1.}  Using the various inequalities we have in the statement of the lemma we have for any $c \geq 1$
$$\left| \frac{ u^c }{a - z t^{c} } \right| \leq  \frac{|u|^c}{|a| (1-t)}, \hspace{2mm}\left|   \prod_{k = 1}^M \frac{1}{1 - z y_k t^{c}} \right| \leq \frac{1}{(1- t)^M},  \hspace{2mm} \left| \prod_{ j = 1}^{n} (1 - x_j z^{-1}t^{-c}) \right| \leq  (1 + R/|a|)^n  t^{-cn}.$$
Furthermore, if $|\zeta| \leq \alpha < 1$ we have that 
\begin{equation}\label{PochBound}
(\alpha;t)_\infty = \prod_{n = 0}^\infty( 1 - \alpha t^n) \leq \prod_{n = 0}^\infty | 1 - \zeta t^n| = | (\zeta;t)_\infty | \leq \prod_{n = 0}^\infty( 1 + \alpha t^n)  = (-\alpha; t)_\infty.
\end{equation}
Combining the last two inequalities we arrive at (\ref{TermCBound}). Observe that by (\ref{TermCBound}) the series in (\ref{HF}) is absolutely convergent by comparison with the geometric series so that $H(u)$ is indeed well-defined and finite. \\

{\bf \raggedleft Step 2.} In this step we show that 
\begin{equation}\label{gammacontpoles}
\begin{split}
& H(u) = \frac{1}{2\pi \iota} \int_{1/4 - \iota \infty}^{1/4 + \iota \infty}  \Gamma(-s) \Gamma(1+s)  (-u)^s g(s) ds, \mbox{ where $\Gamma(s)$ is the Euler Gamma  } \\
& \mbox{ function, }g(s) =  \frac{ 1 }{a - z t^{s} }\cdot  \prod_{k = 1}^M \frac{1}{1 - z y_k t^{s}}  \prod_{ j = 1}^{n} (1 - x_j z^{-1}t^{-s})  \prod_{i = 1}^{N}  \frac{ (z w_i^{-1} t^{s}; t)_\infty }{(zz_i^{-1} t^{s}; t)_\infty },
\end{split}
\end{equation}
and the integral is along the vertical line through $1/4$, which is oriented to have an increasing imaginary part -- we will refer to this contour in the sequel by $\gamma$. We remark that the integrand in (\ref{gammacontpoles}) is well-defined for each $s \in \gamma$ (here we use in particular that $u \not \in [0, \infty)$). Part of the work we do in this step is to show that the integral in (\ref{gammacontpoles}) is actually well-defined and finite.

 Let $R_L = L + 1/4$ ($L \in \mathbb{N})$ and set $A^1_L= 1/4 - \iota R_L$, $A^2_L =  1/4 + \iota R_L$, $A^3_L = R_L + \iota R_L$ and $A^4_L = R_L - \iota R_L$. Denote by $\gamma^1_L$ the contour, which goes from $A_L^1$ vertically up to $A_L^2$, by $\gamma^2_L$ the contour, which goes from $A_L^2$ horizontally to $A_L^3$, by $\gamma^3_L$ the contour, which goes from $A_L^3$ vertically down to $A_L^4$, and by $\gamma^4_L$ the contour, which goes from $A_L^4$ horizontally to $A_L^1$. Also let $\gamma_L = \cup_i \gamma_L^i$ traversed in order (see Figure \ref{S7_1}). 

\begin{figure}[h]
\centering
\scalebox{0.6}{\includegraphics{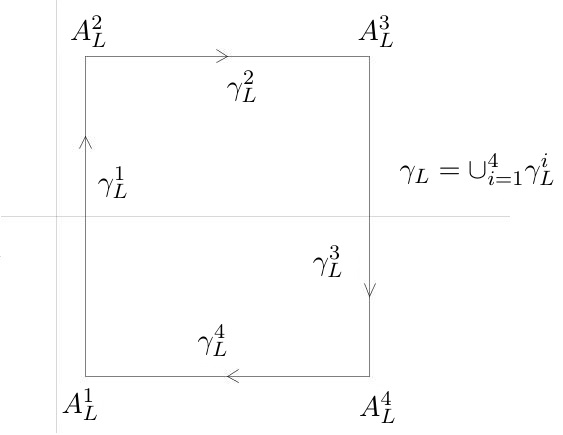}}
\caption{The contours $\gamma_L^i$ for $i = 1,\dots,4$.}
\label{S7_1}
\end{figure}

We make the following observations:
\begin{enumerate}[label = \arabic{enumi}., leftmargin=1.5cm]
\item $\gamma_L$ is negatively oriented.
\item The function $g(s)$ is well-defined and analytic in a neighborhood of the closure of the region enclosed by $\gamma_L$. This follows from $|t^s| < 1$ for $Re(s) > 0$, which prevents any of the poles of $g(s)$ from entering the region $Re(s) > 0$.
\item  If dist$(s, \mathbb{Z}) > c$ for some fixed constant $c > 0$, then $\left| \frac{\pi}{\sin(\pi s)} \right| \leq c'e^{-\pi|Im(s)|}$ for some fixed constant $c'$, depending on $c$. In particular, this estimate holds for all $s \in \gamma_M$ since dist$(\gamma_M, \mathbb{Z}) = 1/4$ for all $L$ by construction.\item If $-u = re^{\iota \theta}$ with $|\theta| < \pi$ and $s = x+ \iota y$ then 
$$(-u)^s = \exp\left( (\log(r) + \iota\theta)(x + \iota y)\right) = \exp\left( \log(r)x - y\theta + \iota(\log(r)y + x\theta)\right),$$
 since we took the principal branch. In particular, $|(-u)^s| = r^xe^{-y\theta}$.
\end{enumerate}
We also recall Euler's Gamma reflection formula
\begin{equation}\label{Euler}
\Gamma(-s)\Gamma(1 + s) = \frac{\pi}{\sin(-\pi s)}.
\end{equation}

We observe for $s = x+ \iota y,$ with $x \geq 1/4$ that
$$\left| \frac{ 1  }{a - z t^{s} }\cdot  \prod_{k = 1}^M \frac{1}{1 - z y_k t^{s}}  \prod_{ j = 1}^{n} (1 - x_j z^{-1}t^{-s})  \prod_{i = 1}^{N}  \frac{ (z w_i^{-1} t^{s}; t)_\infty }{(zz_i^{-1} t^{s}; t)_\infty } \right| \leq  \frac{\prod_{j = 1}^{n} |1 - x_j z^{-1}t^{-s}|}{|a|(1-  t^{1/4})^{M+1}} \cdot \frac{(-t^{1/4};t)_{\infty}^N}{(t^{1/4};t)^N_{\infty}},$$
where we also used (\ref{PochBound}). In addition, we have 
$$\prod_{j = 1}^{n} |1 - x_j z^{-1}t^{-s}| \leq (1 + R |a|^{-1}  t^{-x})^n \leq (1 + R/|a| )^n t^{-x n}.$$

Consequently, we see that 
\begin{equation}\label{BoundIneqB}
|\Gamma(-s) \Gamma(1+s)(-u)^s g(s) | \leq\frac{ c' r^x t^{-xn} e^{-(\pi - |\theta|)|y|}(1 + R/|a| )^n}{|a|(1-  t^{1/4})^{M+1}} \cdot \frac{(-t^{1/4};t)_{\infty}^N}{(t^{1/4};t)^N_{\infty}} \leq Ce^{-(\pi - |\theta|)|y|} ,
\end{equation}
where we used observations 3. and 4. from above together with the fact that $r = |u| < t^{n}$. In particular, the integral in (\ref{gammacontpoles}) is absolutely convergent and we have
$$\lim_{L \rightarrow \infty} \frac{1}{2\pi \iota} \int_{\gamma^1_L}  \Gamma(-s)\Gamma(1 + s) (-u)^sg(s)ds = \frac{1}{2\pi \iota} \int_{\gamma}  \Gamma(-s)\Gamma(1 + s) (-u)^s g(s)ds.$$

From the Residue Theorem we have 
$$\sum_{n = 1}^{L} u^n g_{u}(n)  = \frac{1}{2\pi\iota} \int_{\gamma_L}  \Gamma(-s)\Gamma(1 + s) (-u)^sg(s)ds.$$ 
The last formula used $Res_{s = k}\Gamma(-s)\Gamma(1 + s) = (-1)^{k+1}$ and observations $1.$ and $2.$ above. Using our result from Step 1., we see that what remains to be shown is that
\begin{equation}\label{gammaConv}
\lim_{L \rightarrow \infty} \frac{1}{2\pi \iota} \int_{\gamma_L^i}  \Gamma(-s)\Gamma(1 + s)(-u)^s g(s)ds =0 \mbox{ for $i = 2,3,4$}.
\end{equation}

Suppose that $i = 2$ or $i = 4$. Let $s = x+ \iota y \in \gamma_L^i$, so $|y| = R_L$ and from (\ref{BoundIneqB}) we have
$$\left|  \Gamma(-s)\Gamma(1 + s)(-u)^s g(s) \right| \leq  Ce^{(|\theta| - \pi )R_L},$$
for some constant $C > 0$. Since $|\theta| - \pi < 0$ we see that
$$\left| \frac{1}{2\pi \iota} \int_{\gamma^i_L}  \Gamma(-s)\Gamma(1 + s)(-u)^s g(s) ds \right| \leq C R_L e^{(|\theta| - \pi)R_L} \rightarrow 0 \mbox{ as } L \rightarrow \infty.$$

Finally, let $i = 3$. Let  $s = x+ \iota y \in \gamma_L^3$, so $x= R_L$ and from (\ref{BoundIneqB}) we get
$$\left|  \Gamma(-s)\Gamma(1 + s)(-\zeta)^s g_{w,w'}(t^s) \right| \leq C [rt^{-n}]^{R_L} e^{- ( \pi - |\theta|) |y|}  \leq C [rt^{-n}]^{R_L}.$$
Consequently, we obtain
$$\left| \frac{1}{2\pi \iota} \int_{\gamma^3_L}  \Gamma(-s)\Gamma(1 + s)(-u)^s g(s) ds \right| \leq 2R_L C [rt^{-n}]^{R_L}  \rightarrow 0  \mbox{ as } L \rightarrow \infty,$$
where again we used that $rt^{-n} < 1$. This concludes the proof of (\ref{gammaConv}) and hence (\ref{gammacontpoles}). \\

{\bf \raggedleft Step 3.} Using (\ref{BoundIneqB}) and (\ref{Euler}) we have for any $\phi \in \mathbb{R}$
$$H(u) = \sum_{ m \in \mathbb{Z}} \frac{1}{2\pi \iota} \int_{1/4 + 2\pi \iota m/ [-\log t] + \iota \phi}^{1/4 + 2\pi \iota (m+1) /[-\log t]+ \iota \phi}  \frac{\pi \cdot (-u)^s g(s)}{\sin(-\pi s)} ds = $$
$$\frac{1}{2\pi \iota} \int_{1/4+ \iota \phi}^{1/4 + 2\pi \iota [-\log t] + \iota \phi} g(s) \sum_{m \in \mathbb{Z}}  \frac{\pi \cdot (-u)^{s  + 2\pi \iota m/ [-\log t] } }{\sin(-\pi s - \pi \cdot 2\pi \iota m/ [-\log t]  )} ds.$$
where we used that $t^{2\pi \iota m /[-\log t]} = 1$ and $g(s + 2\pi \iota m / [-\log t]) = g(s)$. 

We proceed to change variables $w = z t^s$ and pick $\phi$ so that $z e^{\iota \phi /\log t} = |a|$. Consequently, we get
\begin{equation*}
\begin{split}
&H(u)  = \frac{1}{2\pi \iota} \int_{C} g( [\log w - \log z]/ \log t ) \cdot \frac{S(w,z; u,t)}{[-\log t] w} dw,
\end{split}
\end{equation*}
which proves (\ref{gResidues}) for the circle $C$ with radius $|a| t^{1/4}$. By Cauchy's theorem we may deform this contour to any circle with radius $r \in (t|z|, |z|)$ without changing the value of the integral. In the last statement we implicitly used the analyticity of $S(w,z;u,t)$ from Lemma \ref{S7Analyticity}. This suffices for the proof.
\end{proof}

\begin{lemma}\label{BigSumAnal}
Let $N, M, n \in \mathbb{N}$ be given with $N \geq n$. Suppose that $a,t \in (0,1)$ are good in the sense of Definition \ref{DefNiceRange} and let $r_1,r_2,r_3,r_4,\rho$ be as in that definition. Assume that $X = (x_1, \dots, x_N)$, $Y = (y_1, \dots, y_M)$ with $x_i, y_j \in \mathbb{C}$, $|x_i| \leq a$ and $|y_j| \leq a$ for $i = 1, \dots, N$ and $j = 1, \dots, M$.
Finally, let $K \subset \mathbb{C} \setminus [0, \infty)$ be a compact set. Then we can find a constant $C$ depending on $r_1, r_2, r_3, r_4,K,a, t, M, N, n$ such that
\begin{equation}\label{S7HFun}
|H(\vec{z}, \vec{w}, \vec{\hat{z}}, \vec{\hat{w}}; N_1, N_2; u_1, u_2)| \leq C^{N_1 + N_2} \cdot \left( \frac{1 + \rho}{2} \right)^{N_1^2 + N_2^2},
\end{equation}
where
\begin{equation}\label{HFun2}
\begin{split}
&H(\vec{z}, \vec{w}, \vec{\hat{z}}, \vec{\hat{w}}; N_1, N_2; u_1, u_2)=   \det \left[\frac{1}{z_i - w_j} \right]_{i,j = 1}^{N_1}  \det \left[ \frac{1}{\hat{z}_i - \hat{w}_j}\right]_{i,j = 1}^{N_2}  \\
&\prod_{i = 1}^{N_1}\frac{ S(w_i, z_i; u_1,t)}{[-\log t] \cdot w_i} \cdot \prod_{i = 1}^{N_2}\frac{ S(\hat{w}_i, \hat{z}_i; u_2,t)}{[-\log t] \cdot \hat{w}_i} \cdot \prod_{i = 1}^{N_1}\prod_{j = 1}^M \frac{1 - y_j z_i}{1 -y_j w_i} 
\prod_{i = 1}^{N_1}\prod_{j = 1}^N \frac{1 - w_i^{-1}x_j}{1 -  z_i^{-1}x_j}\prod_{i = 1}^{N_2} \prod_{j = 1}^M \frac{1 - \hat{z}_i y_j}{1 - \hat{w}_i y_j}  \times \\
&\prod_{i = 1}^{N_2}  \prod_{ j = 1}^{n} \frac{1 - \hat{w}^{-1}_ix_j}{1 - \hat{z}^{-1}_i x_j} \cdot  \prod_{i = 1}^{N_1} \prod_{j = 1}^{N_2} \frac{(\hat{z}_j z_i^{-1}; t)_\infty }{(\hat{w}_j z_i^{-1} ; t)_\infty }\frac{(\hat{w}_j w_i^{-1} ; t)_\infty }{(\hat{z}_j w_i^{-1}; t)_\infty }.
\end{split}
\end{equation}
In (\ref{S7HFun}) we have that $u_1, u_2 \in K$, and $|z_i| = r_1$, $|w_i| = r_2$, $|\hat{z}_j| = r_3$ and $|\hat{w}_j| = r_4$ for $i = 1, \dots, N_1$ and $j = 1, \dots, N_2$. In (\ref{HFun2}) the function $S$ is as in Definition \ref{DefFunS}.

 Moreover, if we fix $u_1 \in \mathbb{C} \setminus [0, \infty)$ then the function
\begin{equation}\label{S7HFunB}
H(u_1,u_2) = \hspace{-4mm}\sum_{N_1, N_2 = 0}^\infty \hspace{0mm} \int_{\gamma_1} \int_{\gamma_2} \int_{\gamma_3} \int_{\gamma_4}\hspace{-2mm} \frac{ H(\vec{z}, \vec{w}, \vec{\hat{z}}, \vec{\hat{w}}; N_1, N_2; u_1, u_2)  }{N_1! N_2!}  \hspace{-0.5mm} \prod_{i =1}^{N_2} \frac{d\hat{w}_i}{2\pi \iota}  \hspace{-0.5mm} \prod_{i = 1}^{N_2} \frac{d \hat{z}_i}{2\pi \iota}\prod_{i = 1}^{N_1}  \frac{dw_i}{2\pi \iota} \prod_{i =1}^{N_1}  \frac{d z_i}{2\pi \iota},
\end{equation}
is well-defined and analytic in $u_2 \in \mathbb{C} \setminus [0, \infty)$. In (\ref{S7HFunB}) we have that $\gamma_i$ is a positively oriented circle of radius $r_i$ for $i =1,2,3,4$ and if $N_1 = N_2 = 0$ then the summand equals $1$ by convention. 
\end{lemma}
\begin{proof}
We first establish (\ref{S7HFun}). By Lemma \ref{S7DetBounds} we have that 
$$\left|\det \left[\frac{1}{z_i - w_j} \right]_{i,j = 1}^{N_1} \right| \leq r_1^{-N_1}  \frac{N_1^{N_1} (r_2/r_1)^{\binom{N_1}{2}}}{(1 - r_2/r_1)^{N_1^2}} \mbox{ and } \left| \det \left[ \frac{1}{\hat{z}_i - \hat{w}_j}\right]_{i,j = 1}^{N_2}   \right| \leq  r_3^{-N_2} \frac{N_2^{N_2} (r_4/r_3)^{\binom{N_2}{2}}}{(1 -r_4/r_3)^{N_2^2}} .$$
Furthermore, by (\ref{PochBound}) we have that 
$$ \left|  \prod_{i = 1}^{N_1}\prod_{j = 1}^{N_2}  \frac{(\hat{z}_j z_i^{-1}; t)_\infty }{(\hat{w}_j z_i^{-1} ; t)_\infty }\frac{(\hat{w}_j w_i^{-1} ; t)_\infty }{(\hat{z}_j w_i^{-1}; t)_\infty } \right| \leq\left(\frac{(-r_3/r_1; t)_{\infty} (-r_4/r_2; t)_{\infty} }{(r_4/r_1; t)_\infty (r_3/r_2; t)_{\infty}} \right)^{N_1N_2}.$$
By Lemma \ref{S7BoundOnS} we can find a constant $C_1 > 0$ depending on $K, t, r_1, r_2, r_3, r_4$ such that 
$$\left|\prod_{i = 1}^{N_1}\frac{ S(w_i, z_i; u_1,t)}{[-\log t] \cdot w_i} \cdot \prod_{i = 1}^{N_2}\frac{ S(\hat{w}_i, \hat{z}_i; u_2,t)}{[-\log t] \cdot \hat{w}_i} \right| \leq  C_1^{N_1 + N_2}.$$
Finally, we can find a constant $C_2 > 0$ depending on $r_1, r_2, r_3, r_4,a,t,N,n,M$ such that 
$$\left|\prod_{i = 1}^{N_1}\prod_{j = 1}^M \frac{1 - y_j z_i}{1 -y_j w_i} 
\prod_{i = 1}^{N_1}\prod_{j = 1}^N \frac{1 - w_i^{-1}x_j}{1 -  z_i^{-1}x_j}\prod_{i = 1}^{N_2} \prod_{j = 1}^M \frac{1 - \hat{z}_i y_j}{1 - \hat{w}_i y_j} \cdot \prod_{i = 1}^{N_2}  \prod_{ j = 1}^{n} \frac{1 - \hat{w}^{-1}_ix_j}{1 - \hat{z}^{-1}_i x_j}  \right| \leq C_2^{N_1 + N_2}.$$
Combining the last five inequalities we conclude that 
$$|H(\vec{z}, \vec{w}, \vec{\hat{z}}, \vec{\hat{w}}; N_1, N_2; u_1, u_2)| \leq (C_1C_2)^{N_1 + N_2} r_1^{-N_1} r_3^{-N_2} \cdot \frac{N_1^{N_1} (r_2/r_1)^{\binom{N_1}{2}}}{(1 - r_2/r_1)^{N_1^2}} \cdot \frac{N_2^{N_2} (r_4/r_3)^{\binom{N_2}{2}}}{(1 -(r_4/r_3))^{N_2^2}} \cdot $$
$$\left(\frac{(-r_3/r_1; t)_{\infty} (-r_4/r_2; t)_{\infty} }{(r_4/r_1; t)_\infty (r_3/r_2; t)_{\infty}} \right)^{N_1N_2} \leq C^{N_1 + N_2} \rho^{N_1^2 +N_2^2} N_1^{N_1}N_2^{N_2} ,$$
where in the last inequality we used the $\rho$ from Definition \ref{DefNiceRange}. The latter clearly implies (\ref{S7HFun}) for some possibly bigger $C$.

Since the contours $\gamma_i$ are compact and by Lemma \ref{S7Analyticity} the function $H(\vec{z}, \vec{w}, \vec{\hat{z}}, \vec{\hat{w}}; N_1, N_2; u_1, u_2)$ analytic in $u_2$ and jointly continuous in $\vec{z}, \vec{w}, \vec{\hat{z}}, \vec{\hat{w}}, u_2$ we conclude by \cite[Theorem 2.5.4]{Stein} that each summand on the right of  (\ref{S7HFunB}) is analytic in $u_2$. As $u_2$ varies over a compact set $K \subset \mathbb{C} \setminus [0, \infty)$ we see from (\ref{S7HFun}) that the sum in (\ref{S7HFunB}) converges absolutely. As a result $H(u_1,u_2)$ is well-defined for each $u_2 \in \mathbb{C} \setminus [0, \infty)$ and \cite[Theorem 2.5.2]{Stein} ensures that $H(u_1,u_2)$ is analytic in $u_2 \in \mathbb{C} \setminus [0, \infty)$ as the uniform over compacts limit of analytic functions. This suffices for the proof.
\end{proof}

%
\subsection{Proofs of results from Section 4}\label{Section7.2} In this section we prove Lemma \ref{ProbLemma} and Proposition \ref{PropTermLimit} from Section \ref{Section4.1} that are recalled here as Lemma \ref{S7ProbLemma} and Proposition \ref{S7PropTermLimit} for the reader's convenience.

\begin{lemma}\label{S7ProbLemma}
Suppose that $f_n: \mathbb{R} \rightarrow [0,1]$ is a sequence of functions. Assume that for each $\delta > 0$ one has on $\mathbb{R}\backslash [-\delta,\delta]$, $f_n \rightarrow {\bf 1}_{\{y < 0\}}$ uniformly. Let $(X_n,Y_n)$ be a sequence of random vectors such that for each $x,y  \in \mathbb{R}$ 
$$\mathbb{E}[f_n(X_n - x)f_n(Y_n - y)] \rightarrow p(x,y),$$
and assume that $p(x,y)$ is a continuous probability distribution function on $\mathbb{R}^2$. Then $(X_n,Y_n)$ converges in distribution to a random vector $(X,Y)$, such that $\mathbb{P}(X < x, Y < y) = p(x,y)$.
\end{lemma}
\begin{proof}
The proof is an immediate generalization of the one for \cite[Lemma 4.1.39]{BorCor}, but we give it here for the sake of completeness. Let $\epsilon_1, \epsilon_2 \in (0,1)$ be given and fix $x,y \in \mathbb{R}$. Let $N_1 \in \mathbb{N}$ be sufficiently large so that for all $n \geq N_1$ and $y \in [\epsilon_2, \infty)$ 
$$1 - \epsilon_1 \leq f_n(-y) \mbox{ and } f_n(y) \leq \epsilon_1.$$
The existence of such an $N_1$ follows by our assumption that $f_n \rightarrow {\bf 1}_{\{y < 0\}}$ uniformly on $\mathbb{R}\backslash [-\epsilon_2,\epsilon_2]$.
 Using the fact that $f_n \in [0,1]$ for each $n \in \mathbb{N}$ and the above inequalities we have for all $n \geq N_1$ that
$$  \mathbb{P}(X_n \leq x , Y_n \leq y )  = \mathbb{E} \left[{\bf 1}_{\{ X_n - x-\epsilon_2 \leq -\epsilon_2 \}} \cdot {\bf 1}_{\{ Y_n - y -\epsilon_2\leq -\epsilon_2\}} \right] \leq  $$
$$(1- \epsilon_1)^{-2} \cdot \mathbb{E}[ f_n(X_n - x-\epsilon_2 )f_n(Y_n - y -\epsilon_2 )]. $$
Similarly, we have for $n \geq N_1$
$$  \mathbb{P}(X_n \leq x , Y_n \leq y )  = \mathbb{E} \left[{\bf 1}_{\{ X_n - x+\epsilon_2 \leq \epsilon_2 \}} \cdot {\bf 1}_{\{ Y_n - y +\epsilon_2\leq \epsilon_2\}} \right] \geq  $$
$$  \mathbb{E}[ f_n(X_n - x+\epsilon_2 )f_n(Y_n - y +\epsilon_2 )] - \epsilon_1. $$
The above inequalities show that 
$$\liminf_{n \rightarrow \infty} \mathbb{P}(X_n \leq x , Y_n \leq y )  \geq p(x - \epsilon_2,y - \epsilon_2) - \epsilon_1$$
and 
$$\limsup_{n \rightarrow \infty} \mathbb{P}(X_n \leq x , Y_n \leq y )  \leq  (1- \epsilon_1)^{-2} \cdot p(x + \epsilon_1 ,y + \epsilon_1 ).$$
Taking the limit $\epsilon_1, \epsilon_2 \rightarrow 0+$ above and using the continuity of $p(x,y)$ we conclude that 
$$p(x,y) \leq \liminf_{n \rightarrow \infty} \mathbb{P}(X_n \leq x , Y_n \leq y ) \leq \limsup_{n \rightarrow \infty} \mathbb{P}(X_n \leq x , Y_n \leq y )  \leq p(x,y),$$
which proves the statement of the lemma.
\end{proof}

\begin{proposition}\label{S7PropTermLimit} Let $x_1, x_2, \tau_1, \tau_2 \in \mathbb{R}$ be given such that $\tau_1 > \tau_2$. Let $\Gamma_1, \Gamma_2, \Gamma_3, \Gamma_4$ be vertical contours in $\mathbb{C}$ that pass through the points $c_1, c_2, c_3, c_4$ respectively with $ c_1,c_3 > 0 > c_2, c_4$, $c_2 + \tau_1  > c_3 + \tau_2$, that are oriented in the direction of increasing imaginary parts. For $N_1, N_2 \in \mathbb{Z}_{\geq 0}$ define 
\begin{equation}\label{S7ST0}
\begin{split}
& K(N_1, N_2) = \frac{(-1)^{N_1 + N_2}}{N_1! N_2!}\int_{\Gamma_1^{N_1}} \int_{\Gamma_2^{N_1}}  \int_{\Gamma_3^{N_2}}\int_{\Gamma_4^{N_2}} \det \hat{D} \\
&    \prod_{i = 1}^{N_1} \frac{\exp (S_1(z_i) - S_1(w_i))}{z_i - w_i} \cdot   \prod_{i = 1}^{N_2} \frac{\exp (S_2(\hat{z}_i) - S_2(\hat{w}_i))}{\hat{z}_i - \hat{w}_i}   \cdot \prod_{i = 1}^{N_2}\frac{d\hat{w}_i}{2\pi \iota}\prod_{i = 1}^{N_2}\frac{d\hat{z}_i}{2\pi \iota}\prod_{i = 1}^{N_1}\frac{dw_i}{2\pi \iota}\prod_{i = 1}^{N_1}\frac{dz_i}{2\pi \iota},
\end{split}
\end{equation}
where
\begin{equation}\label{S7ST1}
S_1(z) = \frac{z^3}{3} - x_1 z \mbox{ and } S_2(z) = \frac{z^3}{3} - x_2 z,
\end{equation}
and $\hat{D}$ is a $(N_1 +N_2) \times (N_1 +N_2)$ matrix that has the block form $\hat{D} = \begin{bmatrix} \hat{D}_{11} & \hat{D}_{12} \\ \hat{D}_{21} & \hat{D}_{22} \end{bmatrix},$ with
\begin{equation}\label{S7HatBlockMatrixD}
\begin{split}
&\hat{D}_{11} = \left[ \frac{1}{z_i- w_j} \right]{\substack{i = 1, \dots, N_1  \\ j = 1, \dots, N_1}}, \hspace{2mm}  \hat{D}_{12} = \left[ \frac{1}{z_i - \hat{w}_j + \tau_1 - \tau_2}\right]{\substack{i = 1, \dots, N_1  \\ j = 1, \dots, N_2}}, \\
&  \hat{D}_{21} = \left[ \frac{1}{\hat{z}_i - w_j - \tau_1 + \tau_2}\right]{\substack{i = 1, \dots, N_2  \\ j = 1, \dots, N_1}}, \hspace{2mm} \hat{D}_{22} = \left[ \frac{1}{\hat{z}_i - \hat{w}_j} \right]{\substack{i = 1, \dots, N_2  \\ j = 1, \dots, N_2}.}
\end{split}
\end{equation}
If $N_1 = N_2 = 0$ we use the convention $K(N_1, N_2)  = 1$. Then the integrand in (\ref{S7ST0}) is absolutely integrable so that $K(N_1, N_2)$ is well-defined and moreover, there exists a constant $C > 0$ depending on $x_1, x_2, \tau_1, \tau_2$ such that for all $N_1, N_2 \geq 0$
\begin{equation}\label{S7ST2}
\begin{split}
&\left| K(N_1, N_2) \right| \leq C^{N_1 + N_2} \cdot (N_1 + N_2)^{-(N_1 +N_2)/2}, 
\end{split}
\end{equation}
with the convention $0^0 = 1$. The series 
\begin{equation}\label{S7ST3}
\begin{split}
Q:= \sum_{N_1, N_2 = 0}^\infty  K(N_1, N_2) 
\end{split}
\end{equation}
is absolutely convergent and satisfies the equality $Q = \det \left(I - fA f \right)_{L^2(\{\tau_1, \tau_2\} \times \mathbb{R})}$, where the latter Fredholm determinant is as in (\ref{FiniteDimAiry}).
\end{proposition}
\begin{proof}
Observe that for $z = x+\iota y$ we have 
\begin{equation}\label{FFC-1}
Re[S_1(z)] = \frac{x^3}{3} - xy^2 - x  x_1 \mbox{ and }Re[S_2(z)] = \frac{x^3}{3} - xy^2 - x  x_2.
\end{equation}
The latter implies that there are constants $C_1, C_2 > 0$ depending on $c_1, c_2, c_3, c_4, x_1, x_2$ such that 
\begin{equation}\label{FFC0}
\begin{split}
&\left| \prod_{i = 1}^{N_1} \exp (S_1(z_i) - S_1(w_i)) \prod_{i = 1}^{N_2} \exp (S_2(\hat{z}_i) - S_2(\hat{w}_i)) \right| \leq C_1^{N_1 + N_2} \times \\
&\prod_{i = 1}^{N_1} \exp ( - c_1 |z_i|^2 +  c_2 |w_i|^2 + C_2(|z_i| + |w_i|)) \cdot \prod_{i = 1}^{N_2} \exp ( - c_3 | \hat{z}_i|^2  + c_4 |\hat{w}_i|^2 + C_2(|\hat{z}_i| + |\hat{w}_i|)).
\end{split}
\end{equation}
Furthermore, by our assumptions on $c_1, c_2, c_3, c_4$ we know that the entries of $\hat{D}$ are all uniformly bounded by $C_3$ where 
$$C_3 = \max \left( \frac{1}{c_1 - c_2}, \frac{1}{c_3 - c_4}, \frac{1}{c_2 + \tau_1 - c_3 - \tau_2 } \right).$$
Then by Hadamard's inequality, see Lemma \ref{S7DetBounds}, we conclude that 
$$ |\det \hat{D}| \leq (N_1 + N_2)^{(N_1 +N_2)/2} \cdot C_3^{N_1 + N_2}.$$
Combining all of the above estimates we conclude that 
\begin{equation}\label{FFC1}
\begin{split}
& \left| \det \hat{D}  \prod_{i = 1}^{N_1} \frac{\exp (S_1(z_i) - S_1(w_i))}{z_i - w_i} \prod_{i = 1}^{N_2} \frac{\exp (S_2(\hat{z}_i) - S_2(\hat{w}_i))}{\hat{z}_i - \hat{w}_i} \right| \leq  C_3^{2N_1 +2N_2} (N_1 + N_2)^{(N_1 +N_2)/2}  \\
& \prod_{i = 1}^{N_1} \exp ( - c |z_i|^2 -c |w_i|^2 + C_2(|z_i| + |w_i|)) \cdot \prod_{i = 1}^{N_2} \exp ( - c| \hat{z}_i|^2  -c |\hat{w}_i|^2 + C_2(|\hat{z}_i| + |\hat{w}_i|)),
\end{split}
\end{equation}
where $c = \min (-c_1, -c_3, c_2, c_4) > 0$. From the quadratic terms in the exponential we conclude that the integral in (\ref{S7ST0}) is absolutely convergent and moreover we obtain the inequality
\begin{equation}\label{FFC2}
 |K(N_1, N_2) | \leq \frac{C_3^{2N_1 +2N_2} (N_1 + N_2)^{(N_1 +N_2)/2}  \cdot C_4^{N_1 +N_2}}{N_1! N_2!},
\end{equation}
where 
$$C_4 = \max_{i = 1,2,3,4} \left( \int_{\Gamma_i} \exp ( - c |z|^2 +C_2 |z|) |dz| \right),$$
where $|dz|$ denotes integration with respect to arc-length. 

Using Stirling's approximation, see e.g. \cite[Equation (1)]{Robbins}, we have that 
$$   (N_1 +N_2)! \geq  e^{-N_1 -N_2} (N_1 + N_2)^{N_1 +N_2} .$$  
Combining the latter with (\ref{FFC2}) we conclude that 
$$ |K(N_1, N_2) | \leq  [C_3^2 \cdot C_4]^{N_1 +N_2} \frac{(N_1 + N_2)^{(N_1 +N_2)/2} }{(N_1 +N_2)!} \cdot \frac{(N_1 +N_2)!}{N_1! N_2!} \leq \frac{[2e\cdot C_3^2 \cdot C_4]^{N_1 +N_2}}{(N_1 + N_2)^{(N_1 +N_2)/2}},$$
where in the second inequality we also used the trivial inequality $\frac{(N_1 +N_2)!}{N_1! N_2!} \leq 2^{N_1 +N_2}$. The latter inequality clearly implies (\ref{S7ST2}) with $C = 2e\cdot C_3^2 \cdot C_4$. \\

We next turn to the last part of the proposition. The series on the right side of (\ref{S7ST2}) is easily seen to be convergent and hence by comparison the series defining $Q$ is absolutely convergent. What remains to be shown is that $Q = \det \left(I - fA f \right)_{L^2(\{\tau_1, \tau_2\} \times \mathbb{R})}$. Recall from (\ref{FDE}) that
\begin{equation}\label{S7Fredholm}
 1+ \sum_{n = 1}^\infty \frac{(-1)^n}{n!} \sum_{i_1, \dots, i_n  \in \{1,2\} }  \int_{\tau_{i_1}}^\infty \cdots \int_{\tau_{i_n}}^\infty \det \left[ A(\tau_{i_k}, y_k; \tau_{i_l}, y_l)  \right]_{k,l= 1}^n dy_n \cdots dy_1,
\end{equation}
where $A(\tau, \xi; \tau', \xi')$ denotes the extended Airy kernel
\begin{equation}\label{S7AiryKernel}
A(\tau, \xi; \tau', \xi') = \begin{cases}\int_0^\infty e^{-\lambda (\tau - \tau') }Ai(\xi + \lambda) Ai(\xi' + \lambda) d\lambda \hspace{5mm} &\mbox{ if $\tau \geq \tau'$},\\ -\int_{-\infty}^0 e^{-\lambda (\tau - \tau') }Ai(\xi + \lambda) Ai(\xi' + \lambda) d\lambda \hspace{5mm} &\mbox{ if $\tau < \tau'$}.\end{cases}
\end{equation}
We recall that the sum in (\ref{S7Fredholm}) converges absolutely as it is the Fredholm series expansion of a trace class operator, see \cite{JDPNG}. From the absolute convergence of the series defining $Q$ and (\ref{S7Fredholm}) we see that to prove that $Q = \det \left(I - fA f \right)_{L^2(\{\tau_1, \tau_2\} \times \mathbb{R})}$ it suffices to show that for each $n \in \mathbb{N}$ we have
\begin{equation}\label{FFC3}
\frac{(-1)^n}{n!}\sum_{i_1, \dots, i_n  \in \{1,2\} }  \int_{x_{i_1}}^\infty \cdots \int_{x_{i_n}}^\infty \det \left[ A(\tau_{i_k}, y_k; \tau_{i_l}, y_l)  \right]_{k,l= 1}^n dy_n \cdots dy_1 = \sum_{k = 0}^{n} K(k,n-k).
\end{equation}
In the rest of the proof we establish (\ref{FFC3}).\\

We recall the fact that for any $x \in \mathbb{R}$ we have
\begin{equation}\label{S7complexExpon} 
\frac{e^{-cx}}{c} = \begin{cases} \int_{x}^\infty e^{-\lambda  \cdot c} d\lambda, &\mbox{ if $c \in \mathbb{C}$ and $Re[c] > 0$}, \\ - \int_{-\infty}^x e^{-\lambda  \cdot c} d\lambda &\mbox{ if $c \in \mathbb{C}$ and $Re[c] < 0$}, \end{cases}
\end{equation}
which can be found in \cite[Lemma 2.3, Chapter 4]{Stein}. Combining (\ref{S7complexExpon}) with (\ref{S7ST0}) we obtain
\begin{equation*}
\begin{split}
& K(N_1, N_2) = \frac{(-1)^{N_1 + N_2}}{N_1! N_2!}\int_{\Gamma_1^{N_1}} \int_{\Gamma_2^{N_1}}  \int_{\Gamma_3^{N_2}}\int_{\Gamma_4^{N_2}} \int_{(x_1, \infty)^{N_1}}\int_{(x_2, \infty)^{N_2}} \det \hat{D} \\
&    \prod_{i = 1}^{N_1} \exp (S_1(z_i) - S_1(w_i)) e^{-(z_i - w_i) (\lambda_i - x_1)} \cdot   \prod_{i = 1}^{N_2}\exp (S_2(\hat{z}_i) - S_2(\hat{w}_i)) e^{-(\hat{z}_i - \hat{w}_i) (\hat{\lambda}_i - x_2)}   \cdot  \\
& \prod_{i =1}^{N_2} d\hat{\lambda}_i \prod_{i =1}^{N_1} d\lambda^1_i \prod_{i = 1}^{N_2}\frac{d\hat{w}_i}{2\pi \iota}\prod_{i = 1}^{N_2}\frac{d\hat{z}_i}{2\pi \iota}\prod_{i = 1}^{N_1}\frac{dw_i}{2\pi \iota}\prod_{i = 1}^{N_1}\frac{dz_i}{2\pi \iota}.
\end{split}
\end{equation*}
In view of (\ref{FFC0}) we know that the above integral is absolutely convergent and so by Fubini's theorem we may rearrange the contours to obtain
\begin{equation*}
\begin{split}
& K(N_1, N_2) = \frac{(-1)^{N_1 + N_2}}{N_1! N_2!}  \int_{(x_1, \infty)^{N_1}}\int_{(x_2, \infty)^{N_2}}  \int_{\Gamma_1^{N_1}} \int_{\Gamma_2^{N_1}}  \int_{\Gamma_3^{N_2}}\int_{\Gamma_4^{N_2}}\det \hat{D} \\
&    \prod_{i = 1}^{N_1} \exp (S_1(z_i) - S_1(w_i)) e^{-(z_i - w_i) (\lambda_i - x_1)} \cdot   \prod_{i = 1}^{N_2}\exp (S_2(\hat{z}_i) - S_2(\hat{w}_i)) e^{-(\hat{z}_i - \hat{w}_i) (\hat{\lambda}_i - x_2)}   \cdot  \\
& \prod_{i = 1}^{N_2}\frac{d\hat{w}_i}{2\pi \iota}\prod_{i = 1}^{N_2}\frac{d\hat{z}_i}{2\pi \iota}\prod_{i = 1}^{N_1}\frac{dw_i}{2\pi \iota}\prod_{i = 1}^{N_1}\frac{dz_i}{2\pi \iota} \prod_{i =1}^{N_2} d\hat{\lambda}_i \prod_{i =1}^{N_1} d\lambda_i .
\end{split}
\end{equation*}
We expand the determinant $\det \hat{D}$ and substitute the definitions of $S_1$ and $S_2$, which gives
\begin{equation}\label{FFC4}
\begin{split}
& K(N_1, N_2) = \frac{(-1)^{N_1 + N_2}}{N_1! N_2!}  \int_{(x_1, \infty)^{N_1}}\int_{(x_2, \infty)^{N_2}} \sum_{\sigma \in S_{N_1 + N_2}} (-1)^{\sigma} \int_{\Gamma_1^{N_1}} \int_{\Gamma_2^{N_1}}  \int_{\Gamma_3^{N_2}}\int_{\Gamma_4^{N_2}}\\
&  \prod_{i =1}^{N_1} \frac{1}{z_i + \tau_1 - w_i^{\sigma} - {\tau}_i^{\sigma}}   \prod_{i =1}^{N_2} \frac{1}{\hat{z}_i + \tau_2 - \hat{w}_i^{\sigma} - \hat{\tau}_i^{\sigma}}  \prod_{i = 1}^{N_1} e^{ z_i^3/3 - z_i \lambda_i }  e^{-w_i^3/3 + w_i \lambda_i } \cdot    \\
&\prod_{i = 1}^{N_2}e^{ \hat{z}_i^3/3 - \hat{z}_i \hat{\lambda}_i }  e^{-\hat{w}_i^3/3 + \hat{w}_i \hat{\lambda}_i}  \cdot  \prod_{i = 1}^{N_2}\frac{d\hat{w}_i}{2\pi \iota}\prod_{i = 1}^{N_2}\frac{d\hat{z}_i}{2\pi \iota}\prod_{i = 1}^{N_1}\frac{dw_i}{2\pi \iota}\prod_{i = 1}^{N_1}\frac{dz_i}{2\pi \iota} \prod_{i =1}^{N_2} d\hat{\lambda}_i \prod_{i =1}^{N_1} d\lambda_i ,
\end{split}
\end{equation}
where $(w_1^{\sigma}, \cdots, w_{N_1}^{\sigma}, \hat{w}_1^{\sigma}, \cdots, \hat{w}_{N_2}^{\sigma})$
is the vector obtained by permuting $(w_1 , \dots, w_{N_1} , \hat{w}_1 , \dots, \hat{w}_{N_2} )$ by $\sigma$ and $(\tau_1^{\sigma}, \cdots, \tau_{N_1}^{\sigma}, \hat{\tau}_1^{\sigma}, \cdots, \hat{\tau}_{N_2}^{\sigma})$ is the vector obtained by permuting 
$$( \underbrace{\tau_1, \dots, \tau_1}_{N_1}, \underbrace{\tau_2, \dots, \tau_2}_{N_2})$$
by $\sigma$. Let $I_{\sigma}$ denote the set of indices in $\llbracket 1, N_2 \rrbracket = \{1, \dots, N_2\}$ such that $\hat{\tau}^{\sigma}_i = \tau_1$. We also write $I^c_\sigma = \llbracket 1, N_2 \rrbracket  \setminus I_{\sigma}$. Then using (\ref{S7complexExpon}) we have that 
\begin{equation*}
\begin{split}
&\prod_{i =1}^{N_1} \frac{1}{z_i + \tau_1 - w_i^{\sigma} - {\tau}_i^{\sigma}}   \prod_{i =1}^{N_2} \frac{1}{\hat{z}_i + \tau_2 - \hat{w}_i^{\sigma} - \hat{\tau}_i^{\sigma}}  = (-1)^{|I_{\sigma}|} \int_{(0, \infty)^{N_1}}  \int_{(0, \infty)^{ |I_{\sigma}^c| }}   \int_{(-\infty, 0)^{|I_{\sigma}| }} \\
&\prod_{i =1}^{N_1} e^{-\mu_i(z_i + \tau_1 - w_i^{\sigma} - {\tau}_i^{\sigma})}  \prod_{i =1}^{N_2} e^{-\hat{\mu}_i(\hat{z}_i + \tau_2 - \hat{w}_i^{\sigma} - \hat{\tau}_i^{\sigma})}  \prod_{i \in I_\sigma} d\hat{\mu}_i \prod_{i \in I^c_\sigma} d\hat{\mu}_i  \prod_{i = 1}^{N_1} d\mu_i.
\end{split}
\end{equation*}
We substitute the last formula in (\ref{FFC4}) and by a similar argument as before we may apply Fubini's theorem to rearrange the integrals. We thus obtain 
\begin{equation*}
\begin{split}
& K(N_1, N_2) = \frac{(-1)^{N_1 + N_2}}{N_1! N_2!}  \int_{(x_1, \infty)^{N_1}}\int_{(x_2, \infty)^{N_2}} \sum_{\sigma \in S_{N_1 + N_2}} (-1)^{\sigma}(-1)^{|I_{\sigma}|} \\
&  \int_{(0, \infty)^{N_1}}  \int_{(0, \infty)^{ |I_{\sigma}^c| }}   \int_{(-\infty, 0)^{|I_{\sigma}| }} \int_{\Gamma_1^{N_1}} \int_{\Gamma_2^{N_1}}  \int_{\Gamma_3^{N_2}}\int_{\Gamma_4^{N_2}}\prod_{i =1}^{N_1} e^{-\mu_i(z_i + \tau_1 - w_i^{\sigma} - {\tau}_i^{\sigma})}    \prod_{i =1}^{N_2} e^{-\hat{\mu}_i(\hat{z}_i + \tau_2 - \hat{w}_i^{\sigma} - \hat{\tau}_i^{\sigma})} \\
&  \prod_{i = 1}^{N_1} e^{ z_i^3/3 - z_i \lambda_i }  e^{-w_i^3/3 + w_i \lambda_i } \cdot  \prod_{i = 1}^{N_2}e^{ \hat{z}_i^3/3 - \hat{z}_i \hat{\lambda}_i }  e^{-\hat{w}_i^3/3 + \hat{w}_i \hat{\lambda}_i }  \cdot \\
&\prod_{i \in I_\sigma} d\hat{\mu}_i \prod_{i \in I^c_\sigma} d\hat{\mu}_i  \prod_{i = 1}^{N_1} d\mu_i \prod_{i = 1}^{N_2}\frac{d\hat{w}_i}{2\pi \iota}\prod_{i = 1}^{N_2}\frac{d\hat{z}_i}{2\pi \iota}\prod_{i = 1}^{N_1}\frac{dw_i}{2\pi \iota}\prod_{i = 1}^{N_1}\frac{dz_i}{2\pi \iota} \prod_{i =1}^{N_2} d\hat{\lambda}_i \prod_{i =1}^{N_1} d\lambda_i.
\end{split}
\end{equation*}

We next recall the contour integral formula for the Airy function, cf. \cite[Appendix A.3]{Stein}, 
\begin{equation}\label{contourAiry}
Ai(s) = \frac{1}{2 \pi \iota} \int_{\delta + \iota \mathbb{R}} e^{z^3/3 - sz} dz = \frac{1}{2 \pi \iota} \int_{-\delta + \iota \mathbb{R}} e^{-z^3/3 + sz} dz,
\end{equation}
for any $\delta > 0$. Using (\ref{contourAiry}) in our last formula for $K(N_1, N_2) $ we obtain
\begin{equation*}
\begin{split}
& K(N_1, N_2) = \frac{(-1)^{N_1 + N_2}}{N_1! N_2!}  \int_{(x_1, \infty)^{N_1}}\int_{(x_2, \infty)^{N_2}} \sum_{\sigma \in S_{N_1 + N_2}} (-1)^{\sigma}(-1)^{|I_{\sigma}|} \\
&  \int_{(0, \infty)^{N_1}}  \int_{(0, \infty)^{ |I_{\sigma}^c| }}   \int_{(-\infty, 0)^{|I_{\sigma}| }}  \prod_{i =1}^{N_1} e^{-\mu_i( \tau_1  - {\tau}_i^{\sigma})} A_i( \lambda_i + \mu_i) A_i(\lambda_i + \mu^{\sigma^{-1}}_i)  \\
&  \prod_{i =1}^{N_2} e^{-\hat{\mu}_i( \tau_2 - \hat{\tau}_i^{\sigma})}  A_i( \hat{\lambda}_i  + \hat{\mu}_i) A_i( \hat{\lambda}_i  + \hat{\mu}^{\sigma^{-1}}_i )\prod_{i \in I_\sigma} d\hat{\mu}_i \prod_{i \in I^c_\sigma} d\hat{\mu}_i  \prod_{i = 1}^{N_1} d\mu_i  \prod_{i =1}^{N_2} d\hat{\lambda}_i \prod_{i =1}^{N_1} d\lambda_i,
\end{split}
\end{equation*}
where $(\mu_1^{\sigma^{-1}}, \cdots, \mu_{N_1}^{\sigma^{-1}}, \hat{\mu}_1^{\sigma^{-1}}, \cdots, \hat{\mu}_{N_2}^{\sigma^{-1}})$
is the vector obtained by permuting the $N_1 + N_2$-dimensional vector $(\mu_1 , \dots, \mu_{N_1} , \hat{\mu}_1 , \dots, \hat{\mu}_{N_2} )$ by $\sigma^{-1}.$ 

We write   $(\lambda_1^{\sigma}, \cdots, \lambda_{N_1}^{\sigma}, \hat{\lambda}_1^{\sigma}, \cdots, \hat{\lambda}_{N_2}^{\sigma})$ for the vector obtained by permuting 
 $(\lambda_1 , \dots, \lambda_{N_1} , \hat{\lambda}_1 , \dots, \hat{\lambda}_{N_2} )$ by $\sigma$. Using this notation we see that 
\begin{equation*}
\begin{split}
& K(N_1, N_2) = \frac{(-1)^{N_1 + N_2}}{N_1! N_2!}  \int_{(x_1, \infty)^{N_1}}\int_{(x_2, \infty)^{N_2}} \sum_{\sigma \in S_{N_1 + N_2}} (-1)^{\sigma}(-1)^{|I_{\sigma}|} \\
&  \int_{(0, \infty)^{N_1}}  \int_{(0, \infty)^{ |I_{\sigma}^c| }}   \int_{(-\infty, 0)^{|I_{\sigma}| }}  \prod_{i =1}^{N_1} e^{-\mu_i( \tau_1  - {\tau}_i^{\sigma})} A_i( \lambda_i  + \mu_i) A_i( \lambda_i^{\sigma}  + \mu_i)  \\
&  \prod_{i =1}^{N_2} e^{-\hat{\mu}_i( \tau_2 - \hat{\tau}_i^{\sigma})}  A_i(\hat{\lambda}_i  + \hat{\mu}_i) A_i( \hat{\lambda}^{\sigma}_i + \hat{\mu}_i )\prod_{i \in I_\sigma} d\hat{\mu}_i \prod_{i \in I^c_\sigma} d\hat{\mu}_i  \prod_{i = 1}^{N_1} d\mu_i  \prod_{i =1}^{N_2} d\hat{\lambda}_i \prod_{i =1}^{N_1} d\lambda_i,
\end{split}
\end{equation*}
Using (\ref{S7AiryKernel}) we have that the above equals
\begin{equation*}
\begin{split}
& K(N_1, N_2) = \frac{(-1)^{N_1 + N_2}}{N_1! N_2!}  \int_{(x_1, \infty)^{N_1}}\int_{(x_2, \infty)^{N_2}} \sum_{\sigma \in S_{N_1 + N_2}} (-1)^{\sigma} \\
&\prod_{i = 1}^{N_1} A(\tau_1,  \lambda_i  ; \tau^{\sigma}_i,  \lambda_i^{\sigma} )   \prod_{i =1}^{N_2} A(\tau_2,  \hat{\lambda}_i ; \hat{\tau}^{\sigma}_i,  \hat{\lambda}_i^{\sigma} )  \prod_{i =1}^{N_2} d\hat{\lambda}_i \prod_{i =1}^{N_1} d\lambda_i = \\
&\frac{(-1)^{N_1 + N_2}}{N_1! N_2!} \int_{x_{j_1}}^{\infty} \cdots\int_{x_{j_n}}^{\infty} \det \left[ A(\tau_{j_k}, y_k; \tau_{j_l}, y_l)  \right]_{k,l= 1}^n dy_n \cdots dy_1,
\end{split}
\end{equation*}
where $n = N_1 +N_2$, $j_r = 1$ for $r = 1, \dots, N_1$ and $j_r = 2$ for $r = N_1 +1, \dots, N_1 + N_2$. 

The last equation shows that to prove (\ref{FFC3}) it suffices to show that 
\begin{equation}\label{FFC3V2}
\begin{split}
&\frac{(-1)^n}{n!}\sum_{i_1, \dots, i_n  \in \{1,2\} }  \int_{x_{i_1}}^\infty \cdots \int_{x_{i_n}}^\infty \det \left[ A(\tau_{i_k}, y_k; \tau_{i_l}, y_l)  \right]_{k,l= 1}^n dy_n \cdots dy_1 = \\
& \sum_{k = 0}^{n} \frac{(-1)^{n}}{k! (n-k)!} \int_{x_{j_1}}^{\infty} \cdots\int_{x_{j_n}}^{\infty} \det \left[ A(\tau_{j_k}, y_k; \tau_{j_l}, y_l)  \right]_{k,l= 1}^n dy_n \cdots dy_1,
\end{split}
\end{equation}
where on the right side $j_r = 1$ for $r = 1, \dots, k$ and $j_r = 2$ for $r = k +1, \dots, n$. To see why (\ref{FFC3V2}) holds notice that on the top the summand corresponding to $(i_1, \dots, i_n)$ does not change upon permuting these indices and if $k$ of the $i_1$'s are equal to $1$ and $n-k$ are equal to $i_2$ there are $\frac{n!}{k!(n-k)!}$ distinct permutations. Consequently, if we split the sum in the first line of (\ref{FFC3V2}) according to how many times $1$ appears in the list $(i_1, \dots, i_n)$, calling this number $k$, we obtain
$$ \frac{(-1)^{n}}{n!}  \sum_{k = 0}^{n} \frac{n!}{k!(n-k)!} \int_{x_{j_1}}^{\infty} \cdots\int_{x_{j_n}}^{\infty} \det \left[ A(\tau_{j_k}, y_k; \tau_{j_l}, y_l)  \right]_{k,l= 1}^n dy_n \cdots dy_1,$$
which clearly matches the second line in (\ref{FFC3V2}). This suffices for the proof.
\end{proof}

%
\section{Lemmas from Sections \ref{Section5} and \ref{Section6}}\label{Section8} In this section we present the proofs of several lemmas from Sections \ref{Section5} and \ref{Section6}.

%
\subsection{Proofs of lemmas from Section \ref{Section5}}\label{Section8.1} In this section we give the proofs of Lemmas \ref{descentLemma}, \ref{LemmaTaylor}, \ref{LemmaSwap1} and \ref{LemmaSwap2} whose statements are recalled here for the reader's convenience as Lemmas  \ref{S8descentLemma}, \ref{S8LemmaTaylor}, \ref{S8LemmaSwap1} and \ref{S8LemmaSwap2} respectively. For $a \in (0,1)$ we recall from (\ref{FunExp}) the functions
\begin{equation}\label{S8FunExp}
\begin{split}
S_a(z) = \log (1 + ae^z) - \log (1 + ae^{-z}) - \frac{2az}{1+a};  \hspace{2mm} R_a(z) = \log (1 + ae^{-z}) - \log (1 + a) + \frac{az}{1 + a}.
\end{split}
\end{equation}

\begin{lemma}\label{S8descentLemma} Let $a \in (0,1)$ be given. Then there exists $A_0 \in (0,1)$ depending on $a$ such that $S_a(z)$ and $R_a(z)$ are analytic in the vertical strip $\{z \in \mathbb{C}: |Re(z)| < \pi A_0\}$ and for any $A \in (0, A_0)$ and $\epsilon \in \{-1, 1\}$ we have that
\begin{equation}\label{S8K1}
\frac{d}{dy}Re [S_{a}(Ay + \epsilon \iota y)] \leq 0 , \hspace{2mm}\mbox{ and } \frac{d}{dy}Re[S_{a}(-Ay + \epsilon \iota y)] \geq 0 \mbox{ for all } y\in \left[0,\pi \right].
\end{equation}
Furthermore, we have
\begin{equation}\label{S8K3}
\frac{d}{dy}Re [R_{a}( \epsilon \iota y)] \leq 0 \mbox{ for $y \in [0, \pi]$ and $\epsilon \in \{-1, 1\}$ }.
\end{equation}
\end{lemma}
\begin{proof} Choose $A_1 \in (0,1)$ sufficiently small so that $ae^{A_1 \pi} < 1$ and note that both $S_a(z)$ and $R_a(z)$ are analytic in the vertical strip $\{z \in \mathbb{C}: |Re(z)| < \pi A_1\}$. A direct computation shows that
$$\frac{d}{dy}Re[S_{a}(Ay + \epsilon \iota y) ] = a\cdot  Re\left[  \left[\frac{e^{Ay + \epsilon \iota y}}{1 +ae^{Ay + \epsilon \iota y}} + \frac{e^{-(Ay + \epsilon \iota y)}}{1 + ae^{-(Ay + \epsilon \iota y)}} - \frac{2}{1 + a}\right](A + \epsilon \iota)\right] = a\cdot (I_1 +  I_2),$$
where
$$I_1 = A\left[\frac{ae^{2Ay} + e^{Ay}\cos(y)}{|1 + ae^{Ay + \epsilon\iota y}|^2} +  \frac{ae^{-2Ay} + e^{-Ay}\cos(y)}{|1 + ae^{-Ay - \epsilon\iota y}|^2} - \frac{2}{1+a}\right] \mbox{ and }$$
$$I_2 = \frac{- e^{Ay}\epsilon\sin(\epsilon y)}{|1 + ae^{Ay + \epsilon\iota y}|^2} + \frac{e^{-Ay} \epsilon\sin(\epsilon y)}{|1 + ae^{-Ay - \epsilon\iota y}|^2}.$$
In the proof of \cite[Lemma 6.6]{ED} it was shown that $I_2 \leq 0$ for all $A > 0$ and that for small enough $A$ we also have that $I_1 \leq 0$. This proves that we can find $A_0$ sufficiently small so that $A_1 \geq A_0$ and for $A \in (0, A_0)$ the first inequality in (\ref{S8K1}) holds. Since $S_a(z)$ is an odd function the second inequality in (\ref{K1}) also holds. 

Finally, we have that 
$$\frac{d}{dy}Re [R_{a}( \epsilon \iota y)]  = - a \cdot \frac{e^{-Ay} \epsilon\sin(\epsilon y)}{|1 + ae^{-(Ay + \epsilon\iota y)}|^2}  \leq 0,$$
for all $y \in [0, \pi]$, which proves (\ref{S8K3}).
\end{proof}

\begin{lemma}\label{S8LemmaTaylor} Let $a, t \in (0,1)$ be given and suppose that $A(a,t)$ is as in Definition \ref{Defcontours}. There exist $\epsilon_0 \in (0,1)$ and $C_0 > 0$ such that for all $|z| \leq \epsilon_0$ we have
\begin{equation}\label{S8taylorE1}
\begin{split}
&\left|S_a(z) -\frac{a(1-a)}{3(1 + a)^3} \cdot z^3 \right| \leq C_0 |z|^5 \mbox{, and } \left| R_a(z) - \frac{a}{2(1+a)^2} \cdot  z^2  \right| \leq C_0 |z|^3.
\end{split}
\end{equation}
Furthermore, there is a constant $\epsilon_1 > 0$ such that 
\begin{equation}\label{S8taylorE2}
\begin{split}
&Re[S_a(z)] \geq \epsilon_1 |z|^3  \mbox{ if } z \in \gamma_W, \hspace{3mm} Re[S_a(z)] \leq - \epsilon_1 |z|^3 \mbox{ if } z \in \gamma_Z \mbox{, }\\
& \mbox{ and } Re[R_a(z)] \leq -\epsilon_1 |z|^2 \mbox{ if $z \in \gamma_{mid}$},
\end{split}
\end{equation}
where $\gamma_Z, \gamma_W, \gamma_{mid}$ are as in Definition \ref{Defcontours}.
\end{lemma}
\begin{proof} We observe by a direct Taylor series expansion that in a neighborhood of $0$ we have
$$ \log(1 + ae^z) =  \sum_{k = 1}^\infty \frac{a^k e^{kz}(-1)^{k-1}}{k} =  \sum_{k = 1}^\infty \sum_{m = 0}^\infty \frac{a^k (-1)^{k-1} (zk)^m}{k \cdot m!} = \sum_{m = 0}^\infty z^m a_m,$$ 
where 
$$a_m =  \frac{1}{m!} \sum_{k = 1}^\infty a^k(-1)^{k-1} k^{m-1}.$$  
The latter shows that 
\begin{equation*}
\begin{split}
S_{a}(z) = \frac{a(1-a)}{3(1 + a)^3} \cdot z^3 + O(|z|^5) \mbox{ and } R_a(z) =  \frac{a}{2(1+a)^2} \cdot  z^2 + O(|z|^3).
\end{split}
\end{equation*}
This proves the existence of $\epsilon_0, C_0$ satisfying (\ref{S8taylorE1}).

Notice that if $z \in \gamma_W$ then
$$Re \left[ \frac{a(1-a)}{3(1 + a)^3} \cdot z^3 \right] = \frac{a(1-a)}{3(1 + a)^3} \cdot |z|^3 \cos (\beta) = c \cdot  |z|^3,$$
where $\beta \in (-\pi/2, 0)$. Here we used the second property in Definition \ref{Defcontours}. The last equation and (\ref{S8taylorE1}) imply that there exists some $\epsilon_2 \in (0, \pi \sqrt{1 + A^2}]$ such that 
$$Re[S_a(z)] \geq (c/2) |z|^3  \mbox{ if } z \in \gamma_W \mbox{ and }|z| \leq \epsilon_2.$$
Let $z \in \gamma_W$ be such that $|z| \geq \epsilon_2$ and put $z_2 = \frac{z}{|z|} \cdot \epsilon_2$. Then by Lemma \ref{S8descentLemma} and the above inequality we conclude that
$$Re[S_a(z)] \geq Re[S_a(z_2)] \geq (c/2) |z_2|^3 \geq (c/2) \cdot \frac{|z|^3}{\pi^3 (1 + A^2)^{3/2}} \cdot \epsilon_2^3,$$
where we used that $|z| \leq \pi \sqrt{ 1 + A^2}$ for all $z \in \gamma_W$. The above two inequalities establish the first inequality in (\ref{S8taylorE2}) with 
$$\epsilon_1 = (c/2) \cdot  \frac{\epsilon_2^3}{\pi^3 (1 + A^2)^{3/2}} .$$
One shows the second inequality in (\ref{S8taylorE2}) in an analogous fashion and as $S_a(z)$ is an odd function the same $\epsilon_1$ can be taken. 

We now turn to the third inequality in (\ref{S8taylorE2}). Observe that if $z \in \gamma_{mid}$ then 
$$Re \left[ \frac{a}{2(1 + a)^2} \cdot z^2 \right] = - \frac{a}{2(1 + a)^2}\cdot |z|^2 .$$
The latter equation and (\ref{S8taylorE1}) imply that there exists some $\epsilon_3 \in (0, \pi]$ such that 
$$Re[R_a(z)] \leq -  \frac{a}{4(1 + a)^2} \cdot |z|^2  \mbox{ if $z \in \gamma_{mid}$ and $|z| \leq \epsilon_3$}.$$
Let $z \in \gamma_{mid}$ be such that $|z| \geq \epsilon_3$ and put $z_3 = \frac{z}{|z|} \cdot \epsilon_3$. Then by Lemma \ref{S8descentLemma} and the above inequality we conclude that
$$Re[R_a(z)] \leq  Re[R_a(z_2)] \leq  -  \frac{a}{4(1 + a)^2} \cdot |z_3|^2 \leq-  \frac{a}{4(1 + a)^2} \cdot \frac{|z|^2}{\pi^2 } \cdot \epsilon_3^2,$$
where we used that $|z| \leq \pi $ for all $z \in \gamma_{mid}$. The above two inequalities establish the third inequality in (\ref{S8taylorE2}) with $\epsilon_1 = \frac{a}{4(1 + a)^2}\cdot  \frac{\epsilon_3^2}{\pi^2 }$. Taking the minimum of the two formulas we obtained we see that all three inequalities in (\ref{S8taylorE2}) are satisfied.
\end{proof}

\begin{lemma}\label{S8LemmaSwap1} Suppose that $N_1, N_2 \in \mathbb{N}$. Suppose that $A(r,R)$ is an annulus with inner radius $r$ and outer radius $R$, which is centered at the origin. Suppose that $C_i$ are positively oriented circles with radius $R_i$ for $i = 1, \dots, 4$ such that $R > R_1 > \cdots > R_4 > r$. Furthermore let $G(z,w), G_1(z,w)$ and $G_2(z,w)$ be functions that are jointly continuous in $A(r,R)$ and for a fixed $w \in A(r,R)$ are analytic in $z$ in $A(r,R)$ and for a fixed $z \in A(r,R)$ are analytic in $w \in A(r,R)$. Moreover, $G(z,w)$ is non-vanishing as $z,w$ vary over $A(r,R)$. Define with the above data
\begin{equation}\label{S8FN1N2}
\begin{split}
& F(N_1, N_2) := \frac{1}{N_1! N_2!}  \oint_{C_{1}^{N_1}} \oint_{C^{N_1}_2}  \oint_{C_3^{N_2}}\oint_{C^{N_2}_{4}} \det D\cdot \prod_{i = 1}^{N_1} \prod_{j = 1}^{N_2}  \frac{G(\hat{z}_j, z_i) G(\hat{w}_j, w_i)}{G(\hat{w}_j, z_i) G(\hat{z}_j, w_i)} \cdot    \\
&  \prod_{i =1}^{N_2} G_1(\hat{z}_i,\hat{w}_i) \prod_{i =1}^{N_1} G_2(z_i,w_i) \cdot \prod_{i = 1}^{N_2}\frac{d\hat{w}_i}{2\pi \iota}\prod_{i = 1}^{N_2}\frac{d\hat{z}_i}{2\pi \iota}\prod_{i = 1}^{N_1}\frac{dw_i}{2\pi \iota}\prod_{i = 1}^{N_1}\frac{dz_i}{2\pi \iota},
\end{split}
\end{equation}
where $D$ is a $(N_1 +N_2) \times (N_1 +N_2)$ matrix that has the block form $D = \begin{bmatrix} {D}_{11} & {D}_{12} \\ {D}_{21} & {D}_{22} \end{bmatrix},$ with
\begin{equation}\label{S8BlockMatrixD}
\begin{split}
&{D}_{11} = \left[ \frac{1}{z_i- w_j} \right]{\substack{i = 1, \dots, N_1  \\ j = 1, \dots, N_1}}, \hspace{2mm}  {D}_{12} = \left[ \frac{1}{z_i - \hat{w}_j }\right]{\substack{i = 1, \dots, N_1  \\ j = 1, \dots, N_2}}, \\
&  {D}_{21} = \left[ \frac{1}{\hat{z}_i - w_j}\right]{\substack{i = 1, \dots, N_2  \\ j = 1, \dots, N_1}}, \hspace{2mm} {D}_{22} = \left[ \frac{1}{\hat{z}_i - \hat{w}_j} \right]{\substack{i = 1, \dots, N_2  \\ j = 1, \dots, N_2}.}
\end{split}
\end{equation}
Then we have $ F(N_1, N_2) = \sum_{k = 0}^{\min(N_1, N_2)} F(N_1, N_2,k)$, where 
\begin{equation}\label{S8FN1P1}
\begin{split}
&F(N_1, N_2,k) = \frac{1}{k! (N_1- k)! (N_2- k)!}  \oint_{C_{1}^{N_1-k}} \oint_{C^{N_1-k}_3}  \oint_{C_2^{N_2-k}}\oint_{C^{N_2-k}_{4}} \oint_{C^{k}_1}\oint_{C^{k}_3} \oint_{C^k_4}\\
& \det B \cdot  F_1 F_2  \prod_{i = 1}^{k}\frac{d\hat{w}^2_i}{2\pi \iota}\prod_{i = 1}^{k}\frac{d{w}^2_i}{2\pi \iota}\prod_{i = 1}^{k}\frac{dz^2_i}{2\pi \iota}\prod_{i = 1}^{N_2 - k}\frac{d\hat{w}^1_i}{2\pi \iota} \prod_{i = 1}^{N_2 - k} \frac{d\hat{z}^1_i}{2\pi \iota} \prod_{i = 1}^{N_1 - k}\frac{dw^1_i}{2\pi \iota}\prod_{i = 1}^{N_1- k}\frac{dz^1_i}{2\pi \iota}.
\end{split}
\end{equation}
In the above foruma we have that $B$ is a $(N_1 +N_2 - k) \times (N_1 + N_2 - k)$ matrix that has the block form $B = \begin{bmatrix} B_{11} & B_{12} & B_{13} \\ B_{21} & B_{22} & B_{23} \\ B_{31} & B_{32} & B_{33} \end{bmatrix}$ with blocks given by
\begin{equation}\label{S8BlockMatrix}
\begin{split}
&B_{11} = \left[ \frac{1}{z_i^1 - w_j^1} \right]{\substack{i = 1, \dots, N_1 - k \\ j = 1, \dots, N_1- k}}, \hspace{2mm} B_{12} = \left[ \frac{1}{z^1_i - \hat{w}_j^1}\right]{\substack{i = 1, \dots, N_1 - k \\ j = 1, \dots, N_2- k}} , \hspace{2mm} B_{13} =\left[ \frac{1}{z^1_i - \hat{w}_j^2}\right]{\substack{i = 1, \dots, N_1 - k \\ j = 1, \dots,k}} \\
&B_{21} = \left[ \frac{1}{\hat{z}^1_i - w_j^1}\right]{\substack{i = 1, \dots, N_2 - k \\ j = 1, \dots, N_1- k}},  \hspace{2mm}  B_{22} = \left[ \frac{1}{\hat{z}^1_i - \hat{w}^1_j} \right]{\substack{i = 1, \dots, N_2 - k \\ j = 1, \dots, N_2- k}}, \hspace{2mm} B_{23} = \left[ \frac{1}{\hat{z}^1_i - \hat{w}^2_j} \right]{\substack{i = 1, \dots, N_2 - k\\ j = 1, \dots, k}} \\
&  B_{31} =  \left[ \frac{1}{z^2_i - w_j^1}\right]{\substack{i = 1, \dots, k \\ j = 1, \dots, N_1- k}},  \hspace{2mm} B_{32} = \left[ \frac{1}{z^2_i - \hat{w}^1_j} \right]{\substack{i = 1, \dots, k\\ j = 1, \dots, N_2 - k}},  \hspace{2mm} B_{33} = \left[ \frac{1}{z_i^2 - \hat{w}_j^2}\right]{\substack{i = 1, \dots, k\\ j = 1, \dots, k}}.
\end{split}
\end{equation}
The functions $F_1, F_2$ are given by
\begin{equation}\label{S8FN1P3}
\begin{split}
F_1 = &\prod_{i= 1}^{N_1-k} \prod_{j= 1}^{N_2-k} \frac{G(\hat{z}^1_j, z^1_i) G(\hat{w}^1_j, w^1_i)}{G(\hat{w}^1_j, z^1_i) G(\hat{z}^1_j, w^1_i)} \cdot \prod_{i= 1}^{k} \prod_{j = 1}^{N_2-k}  \frac{G(\hat{z}^1_j, z^2_i) G(\hat{w}^1_j, w^2_i)}{G(\hat{w}^1_j, z^2_i) G(\hat{z}^1_j, w^2_i)} \cdot \\
&  \prod_{i = 1}^{N_1-k}  \prod_{j = 1}^{k} \frac{G(w^2_j, z^1_i) G(\hat{w}^2_j, w^1_i)}{G(\hat{w}^2_j, z^1_i) G(w^2_j, w^1_i)} \cdot \prod_{i = 1}^{k} \prod_{j= 1}^{k} \frac{G(w^2_j, z^2_i) G(\hat{w}^2_j, w^2_i)}{G(\hat{w}^2_j, z^2_i) G(w^2_j, w^2_i)};
\end{split}
\end{equation}
\begin{equation}\label{S8FN1P4}
\begin{split}
&F_2 = \prod_{i =1}^{N_2-k} G_1(\hat{z}^1_i,\hat{w}^1_i) \cdot \prod_{i =1}^{k} G_1(w^2_i,\hat{w}^2_i)  \cdot \prod_{i =1}^{N_1-k} G_2(z^1_i,w^1_i) \cdot \prod_{i =1}^{k} G_2(z^2_i,w^2_i) .
\end{split}
\end{equation}
\end{lemma}
\begin{proof} By the Cauchy determinant formula, see e.g. \cite[1.3]{Prasolov}, we have that
\begin{equation}\label{CDExpand}
\begin{split}
\det D= &\frac{\prod_{1 \leq i < j \leq N_1} (z_i - z_j)(w_j - w_i)}{\prod_{i,j = 1}^{N_1} (z_i - w_j)} \times  \\
&\frac{\prod_{1 \leq i < j \leq N_2} (\hat{z}_i - \hat{z}_j)(\hat{w}_j - \hat{w}_i)}{\prod_{i,j = 1}^{N_2} (\hat{z}_i - \hat{w}_j)}\prod_{i = 1}^{N_1}\prod_{j = 1}^{N_2} \frac{(z_i - \hat{z}_j)(w_i - \hat{w}_j)}{(z_i - \hat{w}_j)(w_i - \hat{z}_j)}.
\end{split}
\end{equation}
Let $C_{23}$ be the positively oriented circle of radius $\frac{R_2 + R_3}{2}$. By Cauchy's theorem we may deform the $C_2$ contours to $C_{23}$ without affecting the value of the integral since we do not cross any poles in the process of deformation. We may now deform the $C_3$ contours to $C_2$ and observe that in the process of deformation we cross the simple poles where $w_i - \hat{z}_j$ in (\ref{CDExpand}) vanishes for $i = 1, \dots, N_1$ and $j = 1, \dots, N_2$. As we deform the $\hat{z}_j$ contour we thus obtain a contribution coming from the poles $\hat{z}_j = w_{m_j}$ for some $m_j \in \{1, \dots, N_1\}$ and from the integration of $\hat{z}_j$ over $C_2$. Notice that the presence of the Vandermonde determinant $\prod_{1 \leq i < j \leq N_2} (\hat{z}_i - \hat{z}_j)$ in the numerator in (\ref{CDExpand}) implies that we only get a non-trivial contribution from the residues when $m_i \neq m_j$ for $1\leq i \neq j \leq N_2$. Consequently, by the residue theorem we have
\begin{equation}\label{RF1}
\begin{split}
& F(N_1, N_2) = \frac{1}{N_1! N_2!}  \sum_{k = 0}^{\min(N_1, N_2)} \sum_{J \subset \llbracket 1, N_2 \rrbracket: |J| = k} \sum_{I \subset \llbracket 1, N_1 \rrbracket: |I| = k} \sum_{\sigma \in S_k} \oint_{C_{1}^{N_1}} \oint_{C^{N_1}_{23} } \oint_{C^{N_2}_{4}} \oint_{C_2^{N_2 -k}}      \\
& \mathsf{Res}_{J,I, \sigma} \prod_{i  = 1, i \not \in J}^{N_2}\frac{d\hat{z}_i}{2\pi \iota} \prod_{i = 1}^{N_2}\frac{d\hat{w}_i}{2\pi \iota}\prod_{i = 1}^{N_1}\frac{dw_i}{2\pi \iota}\prod_{i = 1}^{N_1}\frac{dz_i}{2\pi \iota}.
\end{split}
\end{equation}
In the above equation $\llbracket 1, N \rrbracket$ represents the set $\{1, \dots, N\}$ and $S_k$ is the permutation group of $k$ elements $\{1, \dots, k\}$, so that the second sum is over subsets $k$-element subsets $J$ of $\llbracket 1, N_2 \rrbracket$, the third sum is over $k$-element subsets $I$ of $\llbracket 1, N_1 \rrbracket$ and the fourth sum is over permutations in $S_k$. If $I = \{ i_1, \dots, i_k\}$ and $J = \{ j_1, \dots, j_k\}$ with $i_1 < i_2 < \cdots < i_k$ and $j_1 < j_2 < \cdots < j_k$ and $\sigma \in S_k$ then the expression $\mathsf{Res}_{J,I, \sigma}$ stands for $(-1)^k$ times the residue of
$$\det D\cdot \prod_{i = 1}^{N_1} \prod_{j = 1}^{N_2}  \frac{G(\hat{z}_j, z_i) G(\hat{w}_j, w_i)}{G(\hat{w}_j, z_i) G(\hat{z}_j, w_i)} \cdot  \prod_{i =1}^{N_2} G_1(\hat{z}_i,\hat{w}_i) \prod_{j =1}^{N_1} G_2(z_j,w_j) $$
at $\hat{z}_{j_r} = w_{i_\sigma(r)}$ for $r = 1, \dots, k$. Note that the $(-1)^k$ comes from the fact that the poles we are crossing are outside of $C_{3}$. The summand corresponding to the quadruple $(k,J,I,\sigma)$ is precisely contribution we obtain in the process of deforming $C_3$ to $C_2$ when $\hat{z}_{j_r}$ picks up the (minus) residue from the simple pole at $w_{i_{\sigma(r)}}$ for $r = 1, \dots, k$, while $\hat{z}_i$ for $i \in \llbracket 1, N_2 \rrbracket \setminus J$ do not pick up any residue and are deformed to $C_2$.

Using (\ref{CDExpand}) and writing $I^c$ for $\llbracket 1, N_1 \rrbracket \setminus I$, $J^c$ for $\llbracket 1, N_2 \rrbracket \setminus J$ and $\tau : J \rightarrow I$ for the map $\tau(j_r) = i_{\sigma(r)}$ for $r = 1, \dots, k$ we have the following formula for $ \mathsf{Res}_{J,I, \sigma}$
\begin{equation*}
\begin{split}
&\mathsf{Res}_{J,I, \sigma} = \tilde{F}_1 \cdot \tilde{F}_2 \cdot \tilde{F}_3 \mbox{, where } 
\end{split}
\end{equation*}
\begin{equation*}
\begin{split}
&\tilde{F}_1 =  \prod_{i \in I^c} \prod_{j \in J^c} \frac{G(\hat{z}_j, z_i) G(\hat{w}_j, w_i)}{G(\hat{w}_j, z_i) G(\hat{z}_j, w_i)} \prod_{i \in I} \prod_{j \in J^c} \frac{G(\hat{z}_j, z_i) G(\hat{w}_j, w_i)}{G(\hat{w}_j, z_i) G(\hat{z}_j, w_i)} \times   \\
& \prod_{i \in I^c} \prod_{j \in J} \frac{G(w_{\tau(j)}, z_i) G(\hat{w}_j, w_i)}{G(\hat{w}_j, z_i) G(w_{\tau(j)}, w_i)}  \prod_{i \in I} \prod_{j \in J} \frac{G(w_{\tau(j)}, z_i) G(\hat{w}_j, w_i)}{G(\hat{w}_j, z_i) G(w_{\tau(j)}, w_i)};
\end{split}
\end{equation*}
\begin{equation*}
\begin{split}
&\tilde{F}_2 =  \prod_{i \in J}G_1(w_{\tau(i)},\hat{w}_i)  \prod_{i \in J^c} G_1(\hat{z}_i,\hat{w}_i) \prod_{i \in I} G_2(z_i,w_i) \prod_{i \in I^c} G_2(z_i,w_i) ;
\end{split}
\end{equation*}
\begin{equation}\label{F3Spec}
\begin{split}
&\tilde{F}_3 = \frac{\prod_{1 \leq i < j \leq N_1} (z_i - z_j)(w_j - w_i)}{\prod_{i,j = 1}^{N_1} (z_i - w_j)} \times \frac{\prod_{1 \leq i < j \leq N_2}(\hat{w}_j - \hat{w}_i)}{\prod_{i \in J} \prod_{j = 1}^{N_2} (w_{\tau(j)} - \hat{w}_j)\prod_{i \in J^c} \prod_{j = 1}^{N_2} (\hat{z}_i - \hat{w}_j)}  \times \\
& \prod_{i,j \in J: i< j} (w_{\tau(i)} - w_{\tau(j)})\prod_{i \in J, j \in J^c: i< j} (w_{\tau(i)} - \hat{z}_j)\prod_{i \in J^c, j \in J: i< j} (\hat{z}_i - w_{\tau(j)})\prod_{i , j\in J^c: i< j} (\hat{z}_i - \hat{z}_j) \times \\
&\prod_{i \in I^c}\prod_{j \in J^c} \frac{(z_i - \hat{z}_j)(w_i - \hat{w}_j)}{(z_i - \hat{w}_j)(w_i - \hat{z}_j)}  \prod_{i \in I}\prod_{j \in J^c} \frac{(z_i - \hat{z}_j)(w_i - \hat{w}_j)}{(z_i - \hat{w}_j)(w_i - \hat{z}_j)} \times \\
& \prod_{i \in I^c}\prod_{j \in J} \frac{(z_i - w_{\tau(j)})(w_i - \hat{w}_j)}{(z_i - \hat{w}_j)(w_i - w_{\tau(j)})}\prod_{i \in I}\prod_{j \in J} \frac{(z_i - w_{\tau(j)})(w_i - \hat{w}_j)}{(z_i - \hat{w}_j)}  \prod_{ i \in I} \prod_{j \in J: \tau(j) \neq i}  \frac{1}{w_i - w_{\tau(j)}}.
\end{split}
\end{equation}
Let $j_1', \dots, j'_{N_2-k}$ be the elements in $J^c$ sorted in increasing order, and $i_1', \dots, i_{N_1 - k}'$ the elements in $I^c$ sorted in increasing order. We relabel the variables as follows:
$$z_{i'_r} = z^1_r \mbox{ and } w_{i'_r} = w^1_r \mbox{ for $r =1 , \dots, N_1 -k$; } \hat{w}_{j'_r} = \hat{w}^1_{r} \mbox{ and } \hat{z}_{j'_r} = \hat{z}^1_r \mbox{ for $r = 1, \dots, N_2- k$;}$$
$$z_{i_r} = z^2_{r} \mbox{, } w_{i_r} = w^2_r\mbox{, }  \mbox{ and } \hat{w}_{j_{\sigma(r)}} = \hat{w}^2_{r} \mbox{ for $r = 1, \dots, k$}.$$
With this relabeling we see that $\tilde{F}_1 = F_1$, $\tilde{F}_2 = F_2$ as in (\ref{S8FN1P3}) and (\ref{S8FN1P4}) respectively. Also by the Cauchy determinant formula we have $F_3 = \det B$ as in (\ref{S8BlockMatrix}). Combining the latter with (\ref{RF1}) we conclude that 
\begin{equation*}
\begin{split}
& F(N_1, N_2) = \frac{1}{N_1! N_2!}\hspace{-5mm} \sum_{k = 0}^{\min(N_1, N_2)} \hspace{-5mm}\sum_{J \subset \llbracket 1, N_2 \rrbracket: |J| = k} \sum_{I \subset \llbracket 1, N_1 \rrbracket: |I| = k} \sum_{\sigma \in S_k} \oint_{C_{1}^{N_1-k}} \oint_{C^{N_1-k}_{23}}  \oint_{C_2^{N_2-k}}\oint_{C^{N_2-k}_{4}} \oint_{C^{k}_1}\oint_{C^{k}_{23}} \oint_{C^k_4}\\
& \det B \cdot  F_1 F_2  \prod_{i = 1}^{k}\frac{d\hat{w}^2_i}{2\pi \iota}\prod_{i = 1}^{k}\frac{d{w}^2_i}{2\pi \iota}\prod_{i = 1}^{k}\frac{dz^2_i}{2\pi \iota}\prod_{i = 1}^{N_2 - k}\frac{d\hat{w}^1_i}{2\pi \iota} \prod_{i = 1}^{N_2 - k} \frac{d\hat{z}^1_i}{2\pi \iota} \prod_{i = 1}^{N_1 - k}\frac{dw^1_i}{2\pi \iota}\prod_{i = 1}^{N_1- k}\frac{dz^1_i}{2\pi \iota}.
\end{split}
\end{equation*}
In the above formula we may deform by Cauchy's theorem all the $C_{23}$ contours to $C_3$ without changing the values of the integral. Also we note that for fixed $k$ the summands over $I,J, \sigma$ are all the same and there are $\binom{N_1}{k} \cdot \binom{N_2}{k} \cdot k!$ of them. Since 
$$\frac{1}{N_1! N_2!} \binom{N_1}{k} \cdot \binom{N_2}{k} \cdot k! = \frac{1}{k! (N_1 -k)! (N_2 - k)!},$$
we see that the last equation implies $ F(N_1, N_2) = \sum_{k = 0}^{\min(N_1, N_2)} F(N_1, N_2,k)$ with $F(N_1, N_2,k)$ as in (\ref{S8FN1P1}). This suffices for the proof.
\end{proof}

\begin{lemma}\label{S8LemmaSwap2} Let $x_1, x_2, \tau_1, \tau_2 \in \mathbb{R}$ be given such that $\tau_1 > \tau_2$. Let $\Gamma_1, \Gamma_2, \Gamma_3, \Gamma_4$ and $K(N_1, N_2)$ be as in Proposition \ref{PropTermLimit}. Then we have $ K(N_1, N_2) = \sum_{k = 0}^{\min(N_1, N_2)} K(N_1, N_2,k)$, where 
\begin{equation}\label{S8KN1P1}
\begin{split}
&K(N_1, N_2,k) = \frac{(-1)^{N_1 + N_2}}{k! (N_1- k)! (N_2- k)!}   \int_{\Gamma_{1}^{N_1-k}} \int_{(\Gamma_3 + \tau_2 - \tau_1)^{N_1-k}}  \int_{(\Gamma_2+\tau_1-\tau_2)^{N_2-k}}    \\
& \int_{\Gamma^{N_2-k}_{4}} \int_{\Gamma^{k}_1}\int_{(\iota \mathbb{R})^k} \int_{\Gamma^k_4}  \det \hat{B}  \prod_{i =1}^{N_2-k} \frac{\exp (S_2(\hat{z}^1_i) - S_2(\hat{w}^1_i))}{ \hat{z}_i^1 - \hat{w}^1_i } \cdot \prod_{i =1}^{k} \frac{\exp (S_1({z}^2_i) - S_2(\hat{w}^2_i) + S_3(w_i^2) )}{ (z_i^2 - w_i^2 + \tau_1)(w_i^2 - \hat{w}_i^2 - \tau_2) }        \\
&\prod_{i =1}^{N_1-k} \frac{\exp (S_1({z}^1_i) - S_1({w}^1_i))}{ z_i^1 - w_i^1 }    \prod_{i = 1}^{k}\frac{d\hat{w}^2_i}{2\pi \iota}\prod_{i = 1}^{k}\frac{d{w}^2_i}{2\pi \iota}\prod_{i = 1}^{k}\frac{dz^2_i}{2\pi \iota}\prod_{i = 1}^{N_2 - k}\frac{d\hat{w}^1_i}{2\pi \iota} \prod_{i = 1}^{N_2 - k} \frac{d\hat{z}^1_i}{2\pi \iota} \prod_{i = 1}^{N_1 - k}\frac{dw^1_i}{2\pi \iota}\prod_{i = 1}^{N_1- k}\frac{dz^1_i}{2\pi \iota},
\end{split}
\end{equation}
where $S_1, S_2$ are as in (\ref{ST1}), $S_3$ is given by
\begin{equation}\label{S8DefS3}
S_3(w) = [\tau_1 - \tau_2] w^2 + [x_1 - x_2 - \tau_1^2 + \tau_2^2]w
\end{equation}
and $\hat{B}$ is a $(N_1 +N_2 - k) \times (N_1 + N_2 - k)$ matrix that has the block form $\hat{B} = \begin{bmatrix} \hat{B}_{11} & \hat{B}_{12} & \hat{B}_{13} \\ \hat{B}_{21} & \hat{B}_{22} & \hat{B}_{23} \\ \hat{B}_{31} & \hat{B}_{32} & \hat{B}_{33} \end{bmatrix}$ with blocks given by
\begin{equation}\label{S8HatBlockMatrix}
\begin{split}
&\hat{B}_{11} = \left[ \frac{1}{z_i^1 - w_j^1} \right]{\substack{i = 1, \dots, N_1 - k \\ j = 1, \dots, N_1- k}}, \hspace{2mm} \hat{B}_{12} = \left[ \frac{1}{z^1_i - \hat{w}_j^1 + \tau_1 - \tau_2}\right]{\substack{i = 1, \dots, N_1 - k \\ j = 1, \dots, N_2- k}} , \hspace{2mm} \\
& \hat{B}_{13} =\left[ \frac{1}{z^1_i - \hat{w}_j^2 + \tau_1 - \tau_2}\right]{\substack{i = 1, \dots, N_1 - k \\ j = 1, \dots,k}}, \hspace{2mm} \hat{B}_{21} = \left[ \frac{1}{\hat{z}^1_i - w_j^1 - \tau_1 + \tau_2}\right]{\substack{i = 1, \dots, N_2 - k \\ j = 1, \dots, N_1- k}},  \hspace{2mm} \\
&  \hat{B}_{22} = \left[ \frac{1}{\hat{z}^1_i - \hat{w}^1_j} \right]{\substack{i = 1, \dots, N_2 - k \\ j = 1, \dots, N_2- k}}, \hspace{2mm} \hat{B}_{23} = \left[ \frac{1}{\hat{z}^1_i - \hat{w}^2_j } \right]{\substack{i = 1, \dots, N_2 - k\\ j = 1, \dots, k}}, \hspace{2mm} \hat{B}_{31} =  \left[ \frac{1}{z^2_i - w_j^1}\right]{\substack{i = 1, \dots, k \\ j = 1, \dots, N_1- k}}, \\
& \hat{B}_{32} = \left[ \frac{1}{z^2_i - \hat{w}^1_j + \tau_1 - \tau_2} \right]{\substack{i = 1, \dots, k\\ j = 1, \dots, N_2 - k}},  \hspace{2mm} \hat{B}_{33} = \left[ \frac{1}{z_i^2 - \hat{w}_j^2 + \tau_1 - \tau_2}\right]{\substack{i = 1, \dots, k\\ j = 1, \dots, k}}.
\end{split}
\end{equation}
\end{lemma}
\begin{proof} We apply the change of variables
$$z_i \rightarrow z_i - \tau_1, w_i \rightarrow w_i - \tau_1 \mbox{ for $i =1, \dots, N_1$ and }\hat{z}_i \rightarrow \hat{z}_i - \tau_2, \hat{w}_i \rightarrow \hat{w}_i - \tau_2 \mbox{ for $i =1, \dots, N_2$}$$
to the formula for $K(N_1, N_2)$ in (\ref{ST0}) and obtain
\begin{equation*}
\begin{split}
& K(N_1, N_2) = \frac{(-1)^{N_1 + N_2}}{N_1! N_2!}\int_{(c_1 + \tau_1 + \iota \mathbb{R})^{N_1}} \int_{ (c_2 + \tau_1 + \iota \mathbb{R}) ^{N_1}}  \int_{(c_3 + \tau_2 + \iota \mathbb{R})^{N_2}}\int_{(c_4 + \tau_2 + \iota \mathbb{R})^{N_2}}       \\
&\prod_{i = 1}^{N_1} \frac{\exp (R_1(z_i) - R_1(w_i))}{z_i - w_i} \prod_{i = 1}^{N_2} \frac{\exp (R_2(\hat{z}_i) - R_2(\hat{w}_i))}{\hat{z}_i - \hat{w}_i}  \cdot  \det D  \cdot  \prod_{i = 1}^{N_2}\frac{d\hat{w}_i}{2\pi \iota}\prod_{i = 1}^{N_2}\frac{d\hat{z}_i}{2\pi \iota}\prod_{i = 1}^{N_1}\frac{dw_i}{2\pi \iota}\prod_{i = 1}^{N_1}\frac{dz_i}{2\pi \iota},
\end{split}
\end{equation*}
where $D$ is as in (\ref{S8BlockMatrixD}) and 
\begin{equation*}
R_1(z) = \frac{(z-\tau_1)^3}{3} - x_1 z \mbox{ and } R_2(z) = \frac{(z-\tau_2)^3}{3} - x_2 z.
\end{equation*}
The above formula is similar to (\ref{S8FN1N2}), the main difference being that the contours are not concentric circles but infinite lines. We will show below that the integrand has sufficient decay near infinity that will allow us to virtually repeat the proof of Lemma \ref{S8LemmaSwap1}.

We now deform the $w_i$ contours to $p_{23} + \iota \mathbb{R}$, where $p_{23} = \frac{c_2 + c_3 + \tau_1 + \tau_2}{2}$. By Cauchy's theorem, as we do not cross any poles, this deformation does not affect the value of the integral. The decay estimate necessary to deform the contours near infinity comes from the fact that for $w = x+\iota y$ and $x \in (-\infty, c_2 + \tau_1]$ we have that 
\begin{equation}\label{DecaySimp}
Re[ -R_1(w)] = - x_1x - \frac{(x -\tau_1)^3}{3} + y^2 (x-\tau_1) \leq c_2 y^2 + O(|x|^3 + 1),
\end{equation}
where the constant in the big $O$ notation depends on $x_2, \tau_1, \tau_2$ alone and we recall that $c_2 < 0$ by assumption. 

We next proceed to deform the $\hat{z}_i$ contours to $c_2 + \tau_1 + \iota \mathbb{R}$. Arguing as in the proof of Lemma \ref{S8LemmaSwap1} we have that in the process of deformation we cross simple poles where $w_i - \hat{z}_j$ vanishes for $i = 1,\dots, N_1$ and $j =1 ,\dots, N_2$ -- these poles are in $\det D$, cf (\ref{CDExpand}). Following the same argument and notation as in the proof of Lemma \ref{S8LemmaSwap1} we obtain the following analogue of (\ref{RF1})
\begin{equation}\label{RK1}
\begin{split}
& K (N_1, N_2) = \frac{(-1)^{N_1 + N_2}}{N_1! N_2!}  \sum_{k = 0}^{\min(N_1, N_2)} \sum_{J \subset \llbracket 1, N_2 \rrbracket: |J| = k} \sum_{I \subset \llbracket 1, N_1 \rrbracket: |I| = k} \sum_{\sigma \in S_k} \\
&\int_{(c_1 + \tau_1 + \iota \mathbb{R})^{N_1}}\int_{ (p_{23} + \iota \mathbb{R})^{N_1}}\int_{(c_4 + \tau_2 + \iota \mathbb{R})^{N_2}}   \int_{(c_2 + \tau_1 + \iota \mathbb{R})^{N_2-k}}    \\
& \mathsf{Res}_{J,I, \sigma} \prod_{i  = 1, i \not \in J}^{N_2}\frac{d\hat{z}_i}{2\pi \iota} \prod_{i = 1}^{N_2}\frac{d\hat{w}_i}{2\pi \iota}\prod_{i = 1}^{N_1}\frac{dw_i}{2\pi \iota}\prod_{i = 1}^{N_1}\frac{dz_i}{2\pi \iota},
\end{split}
\end{equation}
where $\mathsf{Res}_{J,I, \sigma}$ stands for $(-1)^k$ times the residue of
$$\det D \prod_{i = 1}^{N_1} \frac{\exp (R_1(z_i) - R_1(w_i))}{z_i - w_i} \prod_{i = 1}^{N_2} \frac{\exp (R_2(\hat{z}_i) - R_2(\hat{w}_i))}{\hat{z}_i - \hat{w}_i}  $$
at $\hat{z}_{j_r} = w_{i_\sigma(r)}$ for $r = 1, \dots, k$. We remark that the decay estimate necessary to deform the contours near infinity comes from the fact that for $z = x+\iota y$ and $x \in [c_3 + \tau_2, \infty)$ we have that 
\begin{equation}\label{DecaySimp2}
Re[ R_2(z)] =  x_1x + \frac{(x -\tau_2)^3}{3} - y^2 (x-\tau_2) \leq - c_3 y^2 + O(|x|^3 + 1),
\end{equation}
where the constant in the big $O$ notation depends on $x_1, \tau_1, \tau_2$ alone and we recall that $c_3 >  0$ by assumption. Writing $I^c$ for $\llbracket 1, N_1 \rrbracket \setminus I$, $J^c$ for $\llbracket 1, N_2 \rrbracket \setminus J$ and $\tau : J \rightarrow I$ for the map $\tau(j_r) = i_{\sigma(r)}$ for $r = 1, \dots, k$ we have the following formula for $ \mathsf{Res}_{J,I, \sigma}$
\begin{equation*}
\begin{split}
&\mathsf{Res}_{J,I, \sigma} =\tilde{F}_3 \prod_{i \in I}  \frac{\exp (R_1(z_i) - R_1(w_i))}{z_i - w_i}  \prod_{i \in I^c} \frac{\exp (R_1(z_i) - R_1(w_i))}{z_i - w_i}  \\
&\prod_{i \in J} \frac{\exp (R_2(w_{\tau(i)}) - R_2(\hat{w}_i))}{w_{\tau(i)} - \hat{w}_i}\prod_{i \in J^c}\frac{\exp (R_2(\hat{z}_i) - R_2(\hat{w}_i))}{\hat{z}_i - \hat{w}_i},
\end{split}
\end{equation*}
where $\tilde{F}_3$ is as in (\ref{F3Spec}). Performing the same change of variables as in the proof of Lemma \ref{S8LemmaSwap1} we arrive at the formula
\begin{equation*}
\begin{split}
& K (N_1, N_2) = \sum_{k = 0}^{\min(N_1, N_2)} \frac{(-1)^{N_1 + N_2}}{k! (N_1-k)! (N_2-k)!}\int_{(c_1 + \tau_1 + \iota \mathbb{R})^{N_1-k}}\int_{ (p_{23} + \iota \mathbb{R})^{N_1-k}}   \\
&\int_{(c_2 + \tau_1 + \iota \mathbb{R})^{N_2 -k}} \int_{(c_4 + \tau_2 + \iota \mathbb{R})^{N_2 -k}} \int_{(c_1 + \tau_1 + \iota \mathbb{R})^{N_2-k}}\int_{ (p_{23} + \iota \mathbb{R})^{k}} \int_{(c_4 + \tau_2 + \iota \mathbb{R})^{k}} \det B  \\
&    \prod_{i =1}^{N_2-k} \frac{\exp (R_2(\hat{z}^1_i) - R_2(\hat{w}^1_i))}{ \hat{z}_i^1 - \hat{w}^1_i } \cdot \prod_{i =1}^{k} \frac{\exp (R_1({z}^2_i) - R_2(\hat{w}^2_i) + R_2(w_i^2) - R_1(w_i^2) )}{ (z_i^2 - w_i^2)(w_i^2 - \hat{w}_i^2 ) }        \\
&\prod_{i =1}^{N_1-k} \frac{\exp (R_1({z}^1_i) - R_1({w}^1_i))}{ z_i^1 - w_i^1 } \prod_{i = 1}^{k}\frac{d\hat{w}^2_i}{2\pi \iota}\prod_{i = 1}^{k}\frac{d{w}^2_i}{2\pi \iota}\prod_{i = 1}^{k}\frac{dz^2_i}{2\pi \iota}\prod_{i = 1}^{N_2 - k}\frac{d\hat{w}^1_i}{2\pi \iota} \prod_{i = 1}^{N_2 - k} \frac{d\hat{z}^1_i}{2\pi \iota} \prod_{i = 1}^{N_1 - k}\frac{dw^1_i}{2\pi \iota}\prod_{i = 1}^{N_1- k}\frac{dz^1_i}{2\pi \iota}.
\end{split}
\end{equation*}
In the last formula we can deform the $w^1_i$ contours to $c_3 + \tau_2 + \iota \mathbb{R}$ without affecting the value of the integral by Cauchy's theorem, using the decay estimates from (\ref{DecaySimp}). We can also deform the $w^2_i$ contours to $\iota \mathbb{R}$ without crossing any poles, where we used the fact that 
$$R_2(w) - R_1(w) = (\tau_1 - \tau_2)w^2 + w(x_1 - x_2 + \tau_2^2 - \tau_1^2),$$
which implies that $Re[R_2(x + \iota y) - R_1(x + \iota y)] \leq  -(\tau_1 - \tau_2) y^2 + O(|x|^2 + 1)$ so that we can deform these contours near infinity. After doing this and changing variables 
$$z^1_i \rightarrow z^1_i + \tau_1, z^2_i \rightarrow z^2_i + \tau_1, w^1_i \rightarrow w^1_i + \tau_1, \hat{z}^1_i \rightarrow \hat{z}_i + \tau_2, \hat{w}^1_i \rightarrow \hat{w}^1_i + \tau_2, \hat{w}^{2}_i \rightarrow \hat{w}^2_i + \tau_2 $$
we see that the above formula for $K(N_1, N_2)$ becomes precisely the one in (\ref{S8KN1P1}).
\end{proof}

%
\subsection{Proofs of lemmas from Section 6}\label{Section8.2} In this section we give the proofs of Lemmas \ref{LCDetBound} , \ref{S5techies}, \ref{descentLemma2}, \ref{descentLemma3}, \ref{descentLemma4}, \ref{descentLemma5}, \ref{LemmaCDet2}, \ref{BoundMixedS} and \ref{LBoundQ} whose statements are recalled here for the reader's convenience as Lemmas  \ref{S8LCDetBound} , \ref{S8S5techies}, \ref{S8descentLemma2}, \ref{S8descentLemma3}, \ref{S8descentLemma4}, \ref{S8descentLemma5}, \ref{S8LemmaCDet2}, \ref{S8BoundMixedS} and \ref{S8LBoundQ} respectively. 

\begin{lemma}\label{S8LCDetBound} Suppose that $0 < r < R$ are real numbers and $\alpha \in [0, (r/R)^{1/2})$. Then for any $N \in \mathbb{N}$ and $z_i, w_i \in \mathbb{C}$ with $|w_i| = r$ and $|z_i| = R$ for $i =1, \dots, N$ we have
\begin{equation}\label{S8CDetBound}
\left|\det \left[ \frac{1}{z_i - w_j}\right]_{i,j=1}^N \right| \leq \frac{N^{N/2}}{(R\alpha - r\alpha^{-1})^N} \cdot \left(\frac{R \alpha + r\alpha^{-1} }{r + R} \right)^{N^2}.
\end{equation}
\end{lemma}
\begin{proof}
Fix $\alpha \in [0, (r/R)^{1/2})$. We have by the Cauchy determinant formula, see e.g. \cite[1.3]{Prasolov}, that
\begin{equation}\label{S8RatCDet}
\left|\det \left[ \frac{1}{z_i - w_j}\right]_{i,j=1}^N \right| = \left|\det \left[ \frac{1}{\alpha z_i - \alpha^{-1}w_j}\right]_{i,j=1}^N \right|  \cdot  \prod_{i,j = 1}^N \left| \frac{\alpha z_i - \alpha^{-1} w_j}{z_i - w_j} \right|.
\end{equation}
Let $z = Re^{i \psi}$ and $w = r e^{i \phi}$. Then if we set $\theta = \phi - \psi$ we have
$$ \left| \frac{\alpha z - \alpha^{-1} w}{z- w} \right|^2 =\frac{R^2 \alpha^{2} + r^2 \alpha^{-2} - 2rR \cos \theta}{R^2 + r^2  - 2rR \cos \theta}  = f(\theta).$$
We directly compute that 
$$f'(\theta) = \frac{2rR \sin \theta \cdot (R^2 + r^2  - 2rR \cos \theta) - 2rR \sin \theta(R^2 \alpha^{2} + r^2 \alpha^{-2} - 2rR \cos \theta) }{[R^2 + r^2  - 2rR \cos \theta]^2} = $$
$$ \frac{2rR \sin \theta \cdot (R^2 - r^2\alpha^{-2} ) (1 - \alpha^2) }{[R^2 + r^2  - 2rR \cos \theta]^2},$$
from which we see that $\theta =0 $ is a minimum of $f(\theta)$, while $\theta = \pi$ is a maximum. We compute 
$$\left| \frac{\alpha z - \alpha^{-1} w}{z- w} \right|^2 \leq f(\pi) = \left( \frac{R \alpha + r\alpha^{-1} }{r + R} \right)^2$$
On the other hand, by Hadamard's inequality we have 
$$\left|\det \left[ \frac{1}{\alpha z_i - \alpha^{-1}w_j}\right] \right|   \leq  \frac{N^{N/2}}{(R\alpha - r\alpha^{-1})^N}.$$
Combining the last two inequalities with (\ref{S8RatCDet}) we obtain (\ref{S8CDetBound}). 
\end{proof}

\begin{lemma}\label{S8S5techies}
Let $t \in (0, 1)$ and suppose that $U \in \mathbb{R}$ satisfies $0 < U \leq [-\log t]/4$. Suppose further that $z \in \mathbb{C}$ is such that $Re(z) \in [U, [-\log t]/2]$. Then there exists a constant $C_t^1 > 0$, depending on $t$ alone, such that the following holds
\begin{equation}\label{S8S5yellow1}
 \sum_{k \in \mathbb Z} \left| \frac{1}{\sin(-\pi [z + 2\pi\iota k]/ [-\log t])}\right| \leq \frac{C_t^1}{U}.
\end{equation}
\end{lemma}
\begin{proof}
Let $\sigma =  (-\log t)^{-1}$ and $z = x+ \iota y$. Then we have for any $q \in \mathbb{R}$
$$ \left| \frac{1}{\sin(-\pi\sigma (x + \iota y + \iota q ))}\right| = \left| \frac{2}{e^{-\iota \pi\sigma x}e^{\pi\sigma(y + q)} -  e^{\iota \pi\sigma x}e^{-\pi\sigma(y + q)}}\right| $$
If $q \geq -y$ we see 
$$\left| e^{-\iota \pi\sigma x}e^{\pi\sigma(y + q)} -  e^{\iota \pi\sigma x}e^{-\pi\sigma(y + q)} \right| = \left| e^{-2\iota \pi\sigma x}e^{\pi\sigma(y + q)} -  e^{-\pi\sigma(y + q)} \right| \geq e^{\pi\sigma(y+q)}|\sin(2 \pi\sigma x)|.$$
Conversely, if $q < -y$ we see
$$\left| e^{-\iota \pi\sigma x}e^{\pi\sigma(y + q)} -  e^{\iota \pi\sigma x}e^{-\pi\sigma(y + q)} \right| = \left| e^{\pi\sigma(y + q)} - e^{2\iota \pi\sigma x} e^{-\pi\sigma(y + q)} \right| \geq e^{-\pi\sigma(y+q)}|\sin(2 \pi\sigma x)|.$$
We thus conclude that
\begin{equation*}
\left| \frac{1}{\sin(-\pi\sigma (x + \iota y + \iota q ))}\right| \leq e^{-\pi\sigma|y+q|}\frac{2}{|\sin(2 \pi\sigma x)|}.
\end{equation*}
By assumption we know that $2\pi \sigma x \in [ 2\pi \sigma U, \pi/2]$. This implies that
\begin{equation*}
\frac{2}{|\sin(2 \pi\sigma x)|} \leq \frac{1}{ \sigma U},
\end{equation*}
where we used that $\sin x$ satisfies $\pi \sin x \geq x$ on $[0, \pi/2]$. Combining the last inequalities we obtain
$$
\sum_{k \in \mathbb Z} \left|\frac{1}{\sin(-\pi\sigma (x + \iota y +  2\pi \iota k) )}\right| \leq \sum_{k \in \mathbb Z}e^{-\pi\sigma|y+ 2\pi k|}\sigma^{-1}U^{-1} \leq 2\sigma^{-1}U^{-1} \sum_{k \geq 0}e^{-2k\pi^2 \sigma}.$$
This proves (\ref{S8S5yellow1}).
\end{proof}

\begin{lemma}\label{S8descentLemma2} Let $S_a(z)$ be as in (\ref{S8FunExp}). There exist universal constants $C_1, C_2> 0$ such that the following holds. Let $a \in (0,e^{-2\pi}]$ and $\delta \in [0, \pi]$ be given. Then for $y\in [-\pi, \pi]$ we have
\begin{equation}\label{S8SaBoundVert}
\pm Re \left[ S_a(\pm \delta + \iota y  )  \right] \leq - a \cdot \delta \cdot \left[ C_1 \cdot y^2 - C_2 \cdot \delta^2   \right] . 
\end{equation}
\end{lemma}
\begin{proof} It follows from the proof of Lemma \ref{S8descentLemma} that 
$$\frac{d}{dy} Re \left[S_a(iy) \right] = 0,$$
and since $S_a(0)= 0$ we conclude that $Re[S_a(iy)] = 0$ for all $y \in \mathbb{R}$. Fix $y \in [-\pi, \pi]$ and observe that
$$\frac{d}{dx}Re[S_{a}(x + \iota y) ] = a\cdot  Re\left[  \frac{e^{x + \iota y}}{1 +ae^{x + \iota y}} + \frac{e^{-(x + \iota y)}}{1 + ae^{-(x + \iota y)}} - \frac{2}{1 + a} \right] = $$
$$  a\cdot \left[ \frac{e^x \cos y + ae^{2x} }{1 + a^2 e^{2x} + 2a e^x \cos y} + \frac{e^{-x} \cos y + a e^{-2x} }{1 + a^2e^{-2x} + 2ae^{-x} \cos y} - \frac{2}{1+a} \right] = a(1-a)[I_1 +I_2],$$
where 
$$ I_1=  \frac{e^{2x} (\cos y -1) [ (1-a)^2 (e^x + e^{-x}) + 4a(1- \cos y)] }{(1+a)(a^2 e^{2x} + 2ae^x \cos y + 1)(e^{2x} + a^2 + 2ae^x \cos y)},  \mbox{ and }$$
$$I_2 =  \frac{e^{2x} (e^x +e^{-x} - 2) [1 + a(e^x + e^{-x}) + a^2]}{(1+a)(a^2 e^{2x} + 2ae^x \cos y + 1)(e^{2x} + a^2 + 2ae^x \cos y)}.$$
We next note that for $x \in [-\pi, \pi]$ we have that 
$$|e^x + e^{-x} - 2| \leq 2\sum_{k = 2}^\infty \frac{|x|^k}{k!} \leq 2|x|^2 \sum_{k = 2}^\infty \frac{|x|^{k-2}}{k!} \leq 2|x|^2 e^{|x|} \leq 2e^\pi|x|^2 .$$
In addition, we have the trivial inequality for all $y \in [-\pi, \pi]$
$$\cos y - 1 \leq -\frac{y^2}{10}.$$
Combining the last two statements we see that 
$$I_1 \leq  - a \cdot C_1 \cdot y^2,$$
where $C_1$ does not depend on $a, x$ provided that $a \in (0, e^{-2 \pi}]$ and $x \in [-\pi, \pi]$.
On the other hand,
$$|I_2| \leq 2 a(1-a)e^{\pi} |x|^2 \cdot \frac{e^{2\pi}  [1 + e^{-2\pi} (e^\pi + e^{-\pi}) + e^{-4\pi} ]}{ (1 - e^{-\pi})^2(1 - e^{-2\pi})^2} \leq 3 a \cdot C_2 \cdot |x|^2,$$
where $C_2$ does not depend on $a, x$ provided that $a \in (0, e^{-2 \pi}]$ and $x \in [-\pi, \pi]$. From the above two inequalities and the fact that $Re[S_a(iy)] = 0$ for all $y \in \mathbb{R}$ we conclude that 
$$Re[S_{a}(\delta + \iota y) ] = \int_0^\delta \frac{d}{dx}Re[S_{a}(x + \iota y) ] dx  \leq a \cdot C_2 \cdot \delta^3  - a \cdot C_1 \cdot \delta \cdot y^2.$$ 
Similarly, we have 
$$- Re[S_{a}(-\delta + \iota y) ] = \int_{-\delta}^{0} \frac{d}{dx}Re[S_{a}(x + \iota y) ] dx  \leq  a \cdot C_2 \cdot \delta^3  - a \cdot C_1 \cdot \delta \cdot y^2.$$ 
The above two inequalities imply (\ref{S8SaBoundVert}).
\end{proof}

\begin{lemma}\label{S8descentLemma3} Let $R_a(z)$ be as in (\ref{S8FunExp}). Suppose that $a \in (0, e^{-2\pi}]$ and $z = x + \iota y$ with $x, y \in [-\pi, \pi]$. There is a universal constant $C_3 > 0$ such that
\begin{equation}\label{S8RBound2}
-C_3 \cdot a \cdot|z|^2 \leq Re[R_a(z)]  \leq  C_3 \cdot a \cdot  |z|^2.
\end{equation}
\end{lemma}
\begin{proof}
We have that 
$$\log (1 + ae^{-z}) = \sum_{k = 1}^\infty (-1)^{k-1} \frac{a^k e^{-zk}}{k} = \sum_{k = 1}^\infty \frac{(-1)^{k-1}}{k} a^k \sum_{n = 0}^\infty \frac{(-zk)^n}{n!}.$$
The above formula implies that 
$$R_a(z)= \sum_{k = 1}^\infty \frac{(-1)^{k-1}}{k} a^k \sum_{n = 2}^\infty \frac{(-zk)^n}{n!},$$
and so 
$$|R_a(z)| \leq |z|^2 \sum_{k = 1}^\infty  ka^k  \sum_{n = 2}^\infty \frac{|zk|^{n-2}}{n!} \leq |z|^2 \cdot \sum_{k = 1}^\infty k [a e^{|z|}]^k = \frac{|z|^2 a e^{|z|}}{(1 - ae^{|z|})^2} \leq C_3 \cdot a \cdot |z|^2 ,$$ 
where the constant $C_3$ does not depend on $z,a$ provided that $|z| \leq \sqrt{2} \pi$ and $a \in (0, e^{-2\pi}].$ The last equation clearly implies (\ref{S8RBound2}).
\end{proof}

\begin{lemma}\label{S8descentLemma4} There is a universal constant $C_4 > 0$ such that for all $\delta \in [-\pi, \pi]$ we have 
\begin{equation}\label{S8HypSineBound}
\frac{e^{\delta/2} + e^{-\delta/2} }{ e^{\delta} + e^{-\delta} } \leq e^{-C_4 \delta^2}. 
\end{equation}
\end{lemma}
\begin{proof}
By a direct Taylor series expansion near zero we know that there exists $\epsilon > 0$ such that 
$$e^{\delta/2} + e^{-\delta/2} \leq e^{2\delta^2/3} \mbox{ and } e^{\delta} + e^{-\delta} \geq e^{5\delta^2/6} ,$$
for all $\delta \in [-\epsilon, \epsilon]$. This proves that the inequality (\ref{S8HypSineBound}) holds for $\delta \in [-\epsilon, \epsilon]$ for all $C_4 \in (0, 1/6)$. 
Observe that for any $A > 0$ we have
$$\frac{A + A^{-1}}{A^2 + A^{-2}} < 1 \iff A + A^{-1} < A^2 + A^{-2} \iff A^{-2} (A-1)^2 (A^2 + A + 1) > 0,$$
which holds provided that $A \neq 1$. Consequently, we see that there exists $\tilde{\epsilon} > 0$ such that for $\delta \in [-\pi, \pi] \setminus [-\epsilon, \epsilon],$
$$\frac{e^{\delta/2} + e^{-\delta/2} }{ e^{\delta} + e^{-\delta} }  \leq 1- \tilde{\epsilon}.$$
The latter implies that by taking $C_4$ sufficiently small in $(0, 1/6)$ we can ensure that (\ref{S8HypSineBound}) also holds for $\delta \in [-\pi, \pi] \setminus [-\epsilon, \epsilon].$
\end{proof}

\begin{lemma}\label{S8descentLemma5} For any $x \in [0, 1/2]$ we have
\begin{equation}\label{S8RatFracIneq}
\frac{1 + x}{1 - x} \leq e^{3x}.
\end{equation}
\end{lemma}
\begin{proof}
Put $f(x) = 1- x - e^{-2x}$ and note that $f''(x) = - 4e^{-2x} < 0$. This means that $f(x)$ is strictly concave and so it attains its minimum on $[0,1/2]$ at one of its endpoints. A direct computation shows that $f(0) = 0$ and $f(1/2) = \frac{1}{2} - \frac{1}{e} > 0$ and so $f(x) \geq 0$ for $x \in [0,1/2]$. We thus conclude that $1 - x \geq e^{-2x}$ for $x \in [0,1/2]$ and combining the latter with the trivial inequality $1 + x \leq e^x$ we obtain (\ref{S8RatFracIneq}).
\end{proof}

\begin{lemma}\label{S8LemmaCDet2} There exists a function $f: (0,\infty) \rightarrow (0,\infty)$ such that the following holds. Let $\delta \in (0, 1],$ $c > 0$ and $N \in \mathbb{N}$. Suppose that $x_i, y_i \in \mathbb{R}$ for $i = 1, \dots ,N$  are such that 
\begin{equation}\label{S8squaresBound}
\sum_{i = 1}^N (x_i^2 +  y_i^2) \leq N c^2 \delta^2 .
\end{equation}
Then 
\begin{equation}\label{S8CDetSquare}
\left| \det \left[ \frac{1}{e^{\delta} e^{\iota x_i} - e^{-\delta} e^{\iota y_j}} \right]_{i,j = 1}^N\right| \leq  \frac{e^{-f(c) N^2} N^{N/2}}{(e^{\delta/2} - e^{-\delta/2})^N}.
\end{equation}
\end{lemma}
\begin{proof} For clarity we split the proof into two steps.

{\bf \raggedleft Step 1.} Let us set $z_i = e^{\delta} e^{\iota x_i}$ and $w_ i = e^{-\delta} e^{\iota y_i}$ for $i = 1, \dots , N$. Then we have by (\ref{S8RatCDet}) with $R = e^{\delta}$, $r = e^{-\delta}$ and $\alpha = e^{-\delta/2}$ that 
 \begin{equation}\label{S8RatCDet2}
\left|\det \left[ \frac{1}{z_i - w_j}\right]_{i,j=1}^N \right| = \left|\det \left[ \frac{1}{e^{\delta/2} e^{\iota x_i} - e^{-\delta/2} e^{\iota y_j} }\right]_{i,j=1}^N \right|  \cdot  \prod_{i,j = 1}^N \left| \frac{e^{\delta/2} e^{\iota x_i} - e^{-\delta/2} e^{\iota y_j}}{e^{\delta} e^{\iota x_i} - e^{-\delta} e^{\iota y_j}} \right|.
\end{equation}
Let us put $\theta_{i,j} = x_i - y_j$ for $i ,j = 1, \dots, N$. Then we observe that 
$$\left| \frac{e^{\delta/2} e^{\iota x_i} - e^{-\delta/2} e^{\iota y_j}}{e^{\delta} e^{\iota x_i} - e^{-\delta} e^{\iota y_j}} \right|^2 =\frac{e^{\delta} + e^{-\delta} - 2 \cos \theta_{i,j}}{e^{2\delta} + e^{-2\delta}  - 2 \cos \theta_{i,j}}  =: h(\delta, \theta_{i,j}).$$
Notice that since $ e^{\delta} + e^{-\delta} < e^{2\delta} + e^{-2\delta} $ we know that for any $x \in \mathbb{R}$ and $\delta \in (0,1]$ 
\begin{equation}\label{S8hineq}
0 <h(\delta, x) < 1.
\end{equation}
Put $K = \{ (\delta, x): \delta \in (0,1], x \in [-4 c \delta, 4 c \delta]\}$ and let 
$$A = \sup_{(\delta, x) \in K} h(\delta,x).$$
We claim that $A < 1$. We will prove this statement in the second step. For now we assume its validity and conclude the proof of the lemma.

We set $f(c) = -\log A/ 8$. Notice that by (\ref{S8squaresBound}) we know that there exist sets $I, J \subset \{1, \dots, N\}$ such that $|x_i| \leq 2c\delta$ for $i \in I$ and $|y_j| \leq 2c\delta$ for $j \in J$ with $|I|, |J| \geq N/2$. The latter implies that $|\theta_{i,j}| \leq 4c\delta$ for $i \in I$ and $j \in J$. Consequently, we have
$$\left|\det \left[ \frac{1}{z_i - w_j}\right]_{i,j=1}^N \right| \leq \left|\det \left[ \frac{1}{e^{\delta/2} e^{\iota x_i} - e^{-\delta/2} e^{\iota y_j} }\right]_{i,j=1}^N \right|  \cdot  \prod_{i \in I, j\in J} \left| \frac{e^{\delta/2} e^{\iota x_i} - e^{-\delta/2} e^{\iota y_j}}{e^{\delta} e^{\iota x_i} - e^{-\delta} e^{\iota y_j}} \right| \leq $$
$$ \frac{ N^{N/2}}{(e^{\delta/2} - e^{-\delta/2})^N} \cdot [A]^{N^2/8} \leq   \frac{e^{-f(c) N^2} N^{N/2}}{(e^{\delta/2} - e^{-\delta/2})^N}.$$
In the first inequality we used (\ref{S8RatCDet2}) and (\ref{S8hineq}). In the second inequality we used Hadamard's inequality and the fact that $|I|,|J| \geq N/2$ as well as the definition of $A$. In the last inequality we used the definition of $f(c)$. The last tower of inequalities implies (\ref{S8CDetSquare}) and concludes the proof of the lemma.\\

{\bf \raggedleft Step 2.} let us put $u = e^{\delta} + e^{-\delta}$ and observe that
$$h(\delta, x) = \frac{u - 2\cos x}{u^2 - 2 - 2\cos x} = 1 - \frac{u^2 -u - 2 }{u^2 - 2 - 2\cos x}.$$
Suppose now that $\epsilon \in (0,1) $ is sufficiently small so that $4c \epsilon \leq \pi.$ Then for $|x| \leq 4 c \delta$ and $\delta \leq \epsilon$ we have that 
$$h(\delta, x) = \frac{u - 2\cos x}{u^2 - 2 - 2\cos x} \leq  1 - \frac{u^2 -u - 2 }{u^2 - 2 - 2\cos (4c\delta) } =:g(\delta).$$
From (\ref{S8hineq}) we know that $g(\delta) < 1$ for each $\delta > 0$ and $g(\delta)$ is continuous on $(0, \epsilon]$. Furthermore it is easy to see that 
$$\lim_{\delta \rightarrow 0^+} g(\delta) = 1 - \lim_{\delta \rightarrow 0^+}\frac{3\delta^2 }{4 \delta^2 + 8 c^2 \delta^2} = \frac{1 + 8c^2}{4 + 8 c^2} < 1.$$
We conclude that we can find $A^1_{\epsilon} < 1$ such that 
$$h(\delta, x) \leq A^1_{\epsilon}$$
whenever $\delta \in [0, \epsilon]$ and $(\delta ,x) \in K$. 

Put $K_{\epsilon} = \{ (\delta, x) \in K: \delta \geq \epsilon\}$ and note that $h(\delta, x)$ is continuous on $K_{\epsilon}$ and $h(\delta, x) < 1$ for $(\delta,x) \in K_{\epsilon}$. This means that we can find $A^2_{\epsilon} < 1$ such that 
$$\sup_{(x,\delta) \in K_{\epsilon}} h(\delta ,x) \leq A^2_{\epsilon}.$$
Overall, we see that 
$$\sup_{(\delta, x) \in K} h(\delta,x) \leq A,$$
with $A = \max(A^1_{\epsilon}, A^2_{\epsilon}).$ This proves the claim from Step 1. 
\end{proof}

\begin{lemma}\label{S8BoundMixedS} There exists a universal constant $C_5 > 0$ such that the following holds. For any $t \in (0, e^{-2\pi}]$ we have that 
\begin{equation*}
\frac{(-te^2;t)_\infty}{(t e^2;t)_\infty} \leq \exp \left( tC_5 \right).
\end{equation*}
\end{lemma}
\begin{proof}
By definition we have
$$\frac{(-te^2;t)_\infty}{(t e^2;t)_\infty} = \prod_{k =1}^\infty \frac{(1 + t^ke^2)}{(1 - t^ke^2)}.$$
Since $t \leq e^{-2\pi}$ by assumption we know that $t^k e^2 \leq 1/2$ and so by Lemma \ref{S8descentLemma5} we conclude that 
$$\frac{(-te^2;t)_\infty}{(t e^2;t)_\infty}  \leq \exp\left(3 e^2 \sum_{k = 1}^\infty t^k \right) = \exp \left( \frac{t  e^2}{1-t}\right),$$
which implies the statement of the lemma with $C_5 = \frac{e^2}{1 - e^{-2\pi}}.$
\end{proof}

\begin{lemma}\label{S8LBoundQ} Let $z = x+ iy$ with $|x| \leq 2\pi$, $|y| \leq 2\pi$ and $t \in (0, e^{-4\pi}]$. For such a choice of $z$ and $t$ define the function
\begin{equation}\label{S8DefQ}
Q_t(z) = \log (te^z;t)_\infty - \log(t;t)_\infty + z \cdot \sum_{n = 1}^\infty \frac{t^n}{1 - t^n}.
\end{equation}
Then we have
\begin{equation}\label{S8QBound}
-2 t|z|^2 \leq Re[Q_t(z)]  \leq 2t |z|^2.
\end{equation}
\end{lemma}
\begin{proof}
By the assumptions on $z$ and $t$ we have that $|te^z| < 1$ and so
$$\log (t e^z; t)_\infty = \sum_{n = 1}^\infty \log (1 - t^ne^z) =- \sum_{n = 1}^\infty \sum_{k = 1}^\infty \frac{t^{kn}e^{kz}}{k} .$$
We then have that 
$$Q_t(z) = - \sum_{n = 1}^\infty \sum_{k = 1}^\infty \frac{t^{kn}[e^{kz} - kz - 1]}{k}.$$
Observe that
$$\left| e^{kz} - kz - 1\right| \leq \sum_{m = 2}^\infty \frac{k^m |z|^m}{m!} \leq k^2|z|^2\sum_{m = 2}^\infty \frac{k^{m-2} |z|^{m-2}}{m!}  \leq k^2 |z|^2 e^{k|z|}.$$
The last inequality implies that 
$$|Q_t(z)| \leq |z|^2 \hspace{-0.5mm} \sum_{n = 1}^\infty \hspace{-0.5mm} \sum_{k = 1}^\infty k t^{kn} e^{k |z|}\hspace{-0.5mm} = \hspace{-0.5mm}|z|^2 \sum_{k = 1}^\infty \hspace{-0.5mm} \frac{kt^ke^{k|z|}}{1 - t^k} \hspace{-0.5mm} \leq\hspace{-0.5mm} |z|^2 \hspace{-0.5mm} \sum_{k = 1}^\infty \hspace{-0.5mm} \frac{kt^k e^{2\sqrt{2} k \pi  }}{1 - t} = \frac{t |z|^2}{(1-t)(1 - te^{2\sqrt{2}\pi })^2} \hspace{-0.5mm} \leq \hspace{-0.5mm} 2t|z|^2 \hspace{-0.5mm} ,$$
where we used that $|z| \leq 2\sqrt{2} \pi $ and that $(1-t)(1 - te^{2\sqrt{2} \pi })^2 \geq 1/2$ as $t \leq e^{-4\pi}$. The latter equation clearly implies (\ref{S8QBound}).
\end{proof}

%
\section{Comparison to previous works}\label{Section9} In this section we give a formal comparison between the arguments in the present paper and those in \cite{Dot13,PSpohn,ISS13} and \cite{NZ}. The purpose of this section is to help readers familiar with some of these works relate them to this paper, but also to present some of our arguments in an informal and more accessible way.

%
\subsection{Prelimit formulas via difference operators}\label{Section9.1} In Section \ref{Section9.1.1} we informally apply the method of the Macdonald difference operators and derive a prelimit formula given in (\ref{GenFunExpand3}), comparing the result with formula (\ref{GenFunForm}), which was obtained for the KPZ equation in \cite{Dot13,PSpohn,ISS13}. The discussion in this section is more conceptual and we do not discuss issues of convergence. In Section \ref{Section9.1.2} we explain how to rewrite (\ref{GenFunExpand3}) in terms of contour integrals and match our prelimit formula from Section \ref{Section3}.

%
\subsubsection{KPZ equation}\label{Section9.1.1}
The two-point large time distribution of the KPZ equation \ref{KPZEq} started from narrow wedge initial data was investigated independently in \cite{Dot13} and \cite{PSpohn}, and the authors obtained different formulas, which were ultimately reconciled in \cite{ISS13}. In this section we discuss the approach in these papers and compare it to the one in our paper. The exposition below follows \cite{ISS13}.

Let $\mathcal{H}(X,T)$ denote the solution to the KPZ equation \ref{KPZEq} with narrow wedge initial data. The Cole-Hopf transform $\mathcal{Z} = e^{\mathcal{H}}$ becomes a solution to the {\em stochastic heat equation}
\begin{equation}\label{SHEEq}
\partial_T\mathcal{Z}(X,T) = \frac{1}{2} \partial^2_X \mathcal{Z}(X,T)+  \xi (X,T)\mathcal{Z}(X,T),
\end{equation}
with initial condition $\mathcal{Z}(x,0) = \delta(x)$. Let us set $\gamma_T = (T/2)^{1/3}$ and define $\tilde{\mathcal{H}}$ through
$$\mathcal{H}(X,T) = - \frac{T}{24} + \gamma_T \tilde{\mathcal{H}}(X,T).$$
We also introduce the generating function 
\begin{equation}\label{GenFun}
\begin{split}
&G_{T,X_1, X_2}(s_1, s_2) = \mathbb{E} \left[ \exp \left( - e^{\gamma_T \tilde{\mathcal{H}}( X_1,T) - s_1 }- e^{\gamma_T \tilde{\mathcal{H}}( X_2,T) - s_2 }\right)\right] = \\
& \mathbb{E} \left[ \exp \left( - e^{(T/24) - s_1} \mathcal{Z}(X_1,T)- e^{(T/24) - s_2} \mathcal{Z}(X_2, T) \right) \right].
\end{split}
\end{equation}
Provided the limit exists, we have the following asymptotic equality
\begin{equation}\label{S1lim}
\begin{split}
&\lim_{T \rightarrow \infty} G_{T,2\gamma_T^2 X_1, 2\gamma_T^2 X_2}(\gamma_T (s_1 - X_1^2), \gamma_T(s_2 - X_2^2) )= \\
&\lim_{T \rightarrow \infty} \mathbb{P} \left( \tilde{\mathcal{H}}(X_1, T) + X_1^2 \leq s_1, \tilde{\mathcal{H}}(X_2, T) + X_2^2 \leq s_2 \right).
\end{split}
\end{equation}
Equation (\ref{S1lim}) reduces the question of understanding the two-point large time distribution of the KPZ equation to understanding the limit of the generating function $G_{T,X_1, X_2}(s_1, s_2) $. In \cite{Dot13} and \cite{PSpohn} the authors found different ways to rewrite this generating series so that the limit can be studied. 

In the present paper we consider the stochastic six vertex model, which can be thought of as a discrete analogue of the KPZ equation. Indeed, the convergence of the height function $h$ of the six vertex model to the KPZ equation under a weak asymmetric scaling ($b_2 \in (0,1)$-fixed and $b_1 \rightarrow b_2^-$) was recently proved in \cite{CGST} for several, but in particular step, initial conditions. We remark that earlier in \cite[Theorem 12.3]{BO17} the one-point marginals of the height function $h$ were shown to converge to the KPZ equation in the limit $b_1, b_2 \rightarrow 1^-$. For the stochastic six-vertex model $\mathcal{P}(b_1, b_2)$ from Section \ref{Section1.2} we have the following formal correspondence with the above notation:
\begin{equation}\label{Corr}
\begin{split}
&M \leftrightarrow T,\hspace{5mm} M - h(\cdot, M) \leftrightarrow \mathcal{H}(\cdot, T), \hspace{5mm} t^{h(\cdot , M) - M} \leftrightarrow \mathcal{Z}(\cdot, T)\\
&\mathbb{E}\left[ \frac{1}{(u_1t^{h(n_1 + 1 , M) - M};t)_{\infty}} \frac{1}{(u_2t^{h(n_2 + 1 , M) - M};t)_{\infty}} \right] \leftrightarrow G_{T,X_1, X_2}(s_1, s_2) .
\end{split}
\end{equation}
We recall that $t = b_1/b_2$ and the weak asymmetric limit taking the six-vertex model to the KPZ equation corresponds to taking $t \rightarrow 1^-$. In view of the above correspondence, we see that in our paper we are in a sense studying the asymptotics of a discrete analogue of the KPZ equation, by taking the limit of a discrete analogue of the joint Laplace transform.\\

The first step in rewriting $G_{T,X_1, X_2}(s_1, s_2)$ in a form suitable for asymptotics is to form the moment expansion
\begin{equation}\label{GenFunExpand}
\begin{split}
&G_{T,X_1, X_2}(s_1, s_2) = \sum_{L = 0}^\infty \sum_{R =0}^\infty \frac{(-e^{-s_1 +T/24} )^L}{L!} \frac{(-e^{-s_2 + T/24})^R}{R!}  \mathbb{E} \left[ \mathcal{Z}(X_1,T)^L \mathcal{Z}(X_2,T)^R \right].
\end{split}
\end{equation}
Equation (\ref{GenFunExpand}) is already mathematically ill-posed since $\log \mathbb{E} \left[\mathcal{Z}(X,T)^n \right] \sim n^3$, which means that the radius of convergence of the above expansion in the variables $-e^{-s_1 + T/24}$ and  $-e^{-s_2 + T/24}$ is zero. The issue of formulas involving divergent sums is ubiquitous in the physics replica approach and we will ignore any convergence issues from here on when we discuss it.

The analogue of (\ref{GenFunExpand}) in our setting is given by
\begin{equation}\label{GenFunExpand2}
\begin{split}
&\mathbb{E}\left[ \frac{1}{(u_1t^{h(n_1 + 1 , M) - M};t)_{\infty}} \frac{1}{(u_2t^{h(n_2 + 1 , M) - M};t)_{\infty}} \right]= \\
&\sum_{L = 0}^\infty \sum_{R =0}^\infty \frac{u_1^L(1-t)^{-L}}{L_t!}\frac{u_2^R(1-t)^{-R}}{R_t!} \mathbb{E}\left[ t^{L (h(n_1 + 1 , M) - M)} t^{R (h(n_2 + 1 , M) - M)}\right],
\end{split}
\end{equation}
where we used the power series expansion of the $t$-exponential function (\ref{S1tExpon}). Note that equation (\ref{GenFunExpand2}) as a power series in $u_1, u_2$ has a positive radius of convergence, since $h \in [0,M]$ with probability $1$, making the above moments bounded by $t^{-M(R+L)}$. In particular, the moment expansion in our discrete setting is rigorous and one can view the expansion in (\ref{GenFunExpand}) as a formal $t \rightarrow 1^-$ limit of the one in (\ref{GenFunExpand2}).\\

Both in our paper and in \cite{ISS13} the essential goal is to rewrite the $(L,R)$-th summand in (\ref{GenFunExpand2}) and (\ref{GenFunExpand}) respectively, and rearrange the sum in a form that is suitable for asymptotic analysis. The way this is achieved in \cite{ISS13} is by utilizing the Feynman-Kac formula, which implies that the $n$-point correlation function $\mathbb{E} \left[ \mathcal{Z}(X_1, t) \cdots \mathcal{Z}(X_n,T)\right]$ satisfies the {\em imaginary-time Schr{\"o}dinger} equation
$$ -\partial_T \mathbb{E} \left[ \mathcal{Z}(X_1, t) \cdots \mathcal{Z}(X_n,T)\right] = H_n \mathbb{E} \left[ \mathcal{Z}(X_1, T) \cdots \mathcal{Z}(X_n,T)\right],$$
with initial condition $\mathcal{Z}(X_i,0) = \delta(X_i)$, where $H_n$ denotes the Lieb-Liniger quantum Hamiltonian of $n$ particles on the line with attractive $\delta$-interaction \cite{Lieb, McGuire}:
$$H_n = -\frac{1}{2} \sum_{i = 1}^n \partial_{X_i}^2 - \frac{1}{2} \sum_{i \neq j = 1}^n \delta (X_i - X_j).$$
By using the Bethe ansatz \cite{Dot10} derived a set of orthonormal eigenfunctions and corresponding eigenvalues for the Lieb-Liniger quantum Hamiltonian $H_n$. We mention that the orthonormal set found in \cite{Dot10} for the attractive case is not proved therein to be complete, see the last paragraph in \cite[Section B.3]{Dot10}; however, completeness was established in \cite{HO97, Oxford79, PS11}. Once the eigenbasis is obtained, one can express the $(L,R)$-th summand in (\ref{GenFunExpand}) in terms of it and rearrange the resulting sum to obtain
\begin{equation}\label{GenFunForm}
\begin{split}
&G_{T,X_1, X_2}(s_1, s_2) = \sum_{M = 0}^\infty \frac{(-1)^M}{M!} \prod_{\alpha = 1}^M \int_{\mathbb{R}} \frac{dq_\alpha}{2\pi} \sum_{m_\alpha + l_\alpha > 0} (-1)^{m_\alpha + l_\alpha - 1} \binom{m_\alpha + l_\alpha}{m_\alpha} \times \\
&\det \left[ \frac{1}{\frac{1}{2} (m_\alpha + l_\alpha) + \frac{1}{2} (m_\beta + l_\beta) + \iota (q_\alpha - q_\beta)}\right]_{\alpha, \beta = 1}^M \times \\
& e^{(m_\alpha + l_\alpha)^3 T/24 - \frac{1}{2}(m_\alpha + l_\alpha)q_\alpha^2T - m_\alpha s_1 - l_\alpha s_2 + \iota X_1 m_\alpha q_\alpha + \iota X_2 l_\alpha q_\alpha - m_{\alpha} l_\alpha(X_2 - X_1)/2} \times CT_{KPZ},
\end{split}
\end{equation}
where 
\begin{equation}\label{S1CT2}
\begin{split}
&CT_{KPZ} = \prod_{\alpha \neq \beta} \frac{\Gamma\left[1 + \frac{1}{2} (m_\alpha + A_{\alpha, \beta} + l_\beta) +B_{\alpha, \beta}  \right] \Gamma\left[1 + \frac{1}{2} (-m_\alpha + A_{\alpha, \beta} - l_\beta) + B_{\alpha, \beta}\right]}{\Gamma\left[1 + \frac{1}{2} (-m_\alpha + A_{\alpha, \beta} + l_\beta) +B_{\alpha, \beta}  \right] \Gamma\left[1 + \frac{1}{2} (m_\alpha + A_{\alpha, \beta} - l_\beta) + B_{\alpha, \beta}\right]},
\end{split}
\end{equation}
and we have set $A_{\alpha, \beta} = m_\beta + l_\alpha$, $B_{\alpha, \beta} = \iota (q_\alpha - q_\beta)$. 

We will not have much to say about equation (\ref{GenFunForm}) apart from the fact that the eigenbasis for the Lieb-Liniger quantum Hamiltonian $H_n$ is labeled by complex momenta, which have continuous real parts (labeled by $q_\alpha$'s) and discrete imaginary parts (labeled by $n_{\alpha}$'s) and the latter upon re-groupings and symmetrizations become the discrete labels $m_\alpha$'s and $l_\alpha$'s in (\ref{GenFunForm}). Correspondingly, the summations and integrals in (\ref{GenFunForm}) represent a decomposition of (\ref{GenFunExpand}) onto the eigenbasis. \\

The approach in the present paper formally corresponds to rewriting the $(L,R)$-th summand in (\ref{GenFunExpand2}) using the relation
\begin{equation}\label{S1MomentFun}
\begin{split}
\mathbb{E}\left[ t^{L (h(n_1 + 1 , M) - M)} t^{R (h(n_2 + 1 , M) - M)}\right] = \frac{\mathcal{D}^R_{n_2}\mathcal{D}^L_{n_1} \Pi(X;Y)}{ \Pi(X;Y)} \Big{\vert}_{x_i =a = y_j},
\end{split}
\end{equation}
where $a^2 = \frac{1- b_2}{1-b_1}$, $\mathcal{D}_{n}$ is an affine shifted degeneration of the Macdonald difference operator acting on the variables $(x_1, \dots, x_n)$, and $\Pi(X;Y) =\prod_{i = 1}^N \prod_{j = 1}^M \frac{1 - t x_i y_j}{1 - x_iy_j}$. The reason we use the term ``formally'' here is that one needs to assume for example that $u_1, u_2$ are suffuciently small in (\ref{GenFunExpand2}), and then do certain analytic continuation arguments. For the purposes of this section, we will ignore these convergence issues and continue with the discussion.

The right side of (\ref{S1MomentFun}) can be written as a $(L+R)$-fold nested contour integral using Proposition \ref{milestone1}. The result is 
\begin{equation}\label{S1MomentFun2}
\begin{split}
&\mathbb{E}\left[ t^{L (h(n_1 + 1 , M) - M)} t^{R (h(n_2 + 1 , M) - M)}\right] = \frac{1}{(2\pi \iota)^{R+L}} \int_{C^1_{0,1}} \cdots \int_{C^1_{0,L}}  \int_{C^2_{0,1}} \cdots \int_{C^2_{0,R}}\\
&   \prod_{1 \leq \alpha < \beta \leq L}  \frac{\omega_\alpha - \omega_\beta}{\omega_\alpha - \omega_\beta t^{-1}} \prod_{1 \leq \alpha < \beta \leq R}  \frac{\zeta_\alpha - \zeta_\beta}{\zeta_\alpha - \zeta_\beta t^{-1}}  \prod_{i =1}^L \left( \frac{1- a\omega_i^{-1}t^{-1}}{1 - a\omega_i^{-1}} \right)^{n_1}\left(\frac{1 - a\omega_i}{1 - at\omega_i} \right)^M \\
& \prod_{i =1}^R \left( \frac{1- a\zeta_i^{-1}t^{-1}}{1 - a\zeta_i^{-1}} \right)^{n_2} \left( \frac{1 - a\zeta_i}{1 - at\zeta_i} \right)^M \prod_{i = 1}^L\prod_{j =1}^R    \frac{\omega_i - \zeta_j}{\omega_i- \zeta_jt^{-1}} \prod_{i =1}^R\frac{d\zeta_i}{\zeta_i}\prod_{i =1}^L\frac{d\omega_i}{\omega_i}.
\end{split}
\end{equation}
In (\ref{S1MomentFun2}) the contours are nested zero-centered positively oriented circles with $C_{0,r}^i$ containing $t^{-1} \cdot C_{0,{r+1}}^i$ and $C_{0,1}^2$ contained in $C^1_{0,L}$. In addition, one has that $a$ needs to be inside the smallest contour $C^2_{0,R}$ and $a^{-1}$ needs to be outside the biggest contour $C_{0,1}^1$. 

We next proceed to deform the $C_{0,i}^1$ contours to $C_{0,L}^1$, which can be taken to be a circle of radius $1+\varepsilon$ and also we deform the $C_{0,i}^2$ contours to $C_{0,1}^2$, which can be taken to be a circle of radius $1 - \varepsilon$. In plain words we are deforming the inner circles $C_{0,i}^2$ in the outward direction and the outer circles $C_{0,i}^1$ in the inward direction and the two almost meet in the middle. We explain in Section \ref{Section9.1.2} why we make this particular choice of deformation of the contours. As we deform the contours, we pick up poles from $\zeta_i = \zeta_jt^{-1}$ and $\omega_i = \omega_jt^{-1}$ in (\ref{S1MomentFun2}). The result is that the dimension of the contour integrals goes down and we end up integrating over residue subspaces, which for the $\omega$-contours are labelled by partitions $\lambda \vdash L$ and for the $\zeta$-contours by partitions $\mu \vdash R$. This contour deformation procedure is an instance of the contour integral ansatz, see Lemma \ref{nestedContours} in the main text (its origin is \cite[Proposition 7.2]{BBC}). The result of applying this ansatz is
\begin{equation}\label{S1MomentFun3}
\begin{split}
&\mathbb{E}\left[ t^{L (h(n_1 + 1 , M) - M)} t^{R (h(n_2 + 1 , M) - M)}\right] = \sum_{\lambda \vdash L} \sum_{\mu \vdash R} \frac{(t^{-1} - 1)^{L}(1-t)^RR_t! L_t!}{\prod_{i = 1}^\infty m_i(\lambda)! m_i(\mu)!}  \int_{C_{1 + \varepsilon}^{\ell(\lambda)}} \int_{C_{1-\varepsilon}^{\ell(\mu)}} \\
&  \det \left[ \frac{1}{w_it^{-\lambda_i} - w_j}\right]_{i,j = 1}^{\ell(\lambda)} \prod_{i = 1}^{\ell(\lambda)} \left( \frac{1 - aw_it^{- \lambda_i}}{1 - aw_i } \right)^M  \left(\frac{1 - a w_i^{-1} }{1 - a w_i^{-1}  t^{\lambda_i}}\right)^{n_1} \\
&  \det \left[ \frac{1}{\hat{z}_i - \hat{z}_j t^{\mu_j}}\right]_{i,j = 1}^{\ell(\mu)} \prod_{i = 1}^{\ell(\mu)} \left( \frac{1 - a\hat{z}_i }{1 - a \hat{z}_it^{\mu_i}} \right)^M    \left( \frac{1 - a\hat{z}^{-1}_it^{-\mu_i}}{1 - a\hat{z}^{-1}_i } \right)^{n_2} \\
& \prod_{i = 1}^{\ell(\lambda)}  \prod_{j = 1}^{\ell(\mu)}   \frac{(\hat{z}_j w_i^{-1} t^{\lambda_i}; t)_\infty }{(\hat{z}_j w_i^{-1} t^{ \lambda_i + \mu_j}; t)_\infty }\frac{(\hat{z}_j w_i^{-1} t^{\mu_j}; t)_\infty }{(\hat{z}_j w_i^{-1}; t)_\infty } \prod_{i = 1}^{\ell(\mu)} \frac{d \hat{z}_i}{2\pi \iota}\prod_{i = 1}^{\ell(\lambda)} \frac{d w_i}{2\pi \iota},
\end{split}
\end{equation}
where $k_t! = \frac{(1-t)(1-t^2) \cdots (1- t^k)}{(1-t)^k}$ and $m_i(\lambda)$ denotes the multiplicity of $i$ in the partition $\lambda$ (see Section \ref{Section2.1} for more notation about parititons). The contours $C_{1-\varepsilon}$ and $C_{1+ \varepsilon}$ are positively oriented, zero-centered circles of radius $1-\varepsilon$ and $1 + \varepsilon$ respectively. A priori, equation (\ref{S1MomentFun3}) holds only if $a$ is sufficiently close to zero (depending on $L$ and $R$), because only then is equation (\ref{S1MomentFun2}) valid. However, the immediate benefit of performing the contour deformation is that the right side of (\ref{S1MomentFun3}) is analytic in $a$ and then the equality in (\ref{S1MomentFun3}) extends to all $a \in \mathbb{C}$ such that $|a|^{-1} > 1 + \varepsilon$ and $|a| < 1 - \varepsilon$ by analyticity (the pole $w_i = a^{-1}$ needs to be outside the contour $C_{1 +\varepsilon}$ and the pole $\hat{z}_i = a$ needs to be inside the contour $C_{1 - \varepsilon}$). 

We may now symmetrize (\ref{S1MomentFun3}) in the parts $(\lambda_1, \dots, \lambda_{\ell(\lambda)})$ and $(\mu_1, \dots, \mu_{\ell(\mu)})$ -- observe that there are precisely $\frac{\ell(\lambda)! \ell(\mu)!}{\prod_{i = 1}^\infty m_i(\lambda)! m_i(\mu)!}$ ways to do this. Afterwards we perform the summation in (\ref{GenFunExpand2}) and rearrange the resulting sum according to the values $N_1 = \ell(\lambda)$ and $N_2 = \ell(\mu)$. The result is
\begin{equation}\label{GenFunExpand3}
\begin{split}
&\mathbb{E}\left[ \frac{1}{(u_1t^{h(n_1 + 1 , M) - M};t)_{\infty}} \frac{1}{(u_2t^{h(n_2 + 1 , M) - M};t)_{\infty}} \right]=\sum_{N_1 = 0}^\infty \sum_{N_2 = 0}^{\infty} \frac{1}{N_1! N_2!} \sum_{\lambda_1, \dots, \lambda_{N_1} = 1}^\infty  \sum_{\mu_1, \dots, \mu_{N_2} = 1}^\infty\\
&\int_{C_{1 + \varepsilon}^{N_1}} \int_{C_{1 - \varepsilon}^{N_2}}   \det \left[ \frac{1}{w_it^{-\lambda_i} - w_j}\right]_{i,j = 1}^{N_1} \prod_{i = 1}^{N_1} \left( \frac{1 - aw_it^{- \lambda_i}}{1 - aw_i } \right)^M  \left(\frac{1 - a w_i^{-1} }{1 - a w_i^{-1}  t^{\lambda_i}}\right)^{n_1} [u_1t^{-1}]^{\lambda_i} \\
&  \det \left[ \frac{1}{\hat{z}_i- \hat{z}_j t^{\mu_j}}\right]_{i,j = 1}^{N_2} \prod_{i = 1}^{N_2} \left( \frac{1 - a\hat{z}_i }{1 - a \hat{z}_it^{\mu_i}} \right)^M    \left( \frac{1 - a\hat{z}^{-1}_it^{-\mu_i}}{1 - a\hat{z}^{-1}_i } \right)^{n_2}  [u_2]^{\mu_i}\\
& \prod_{i = 1}^{N_1}  \prod_{j = 1}^{N_2}   \frac{(\hat{z}_j w_i^{-1} t^{\lambda_i}; t)_\infty }{(\hat{z}_j w_i^{-1} t^{ \lambda_i + \mu_j}; t)_\infty }\frac{(\hat{z}_j w_i^{-1} t^{\mu_j}; t)_\infty }{(\hat{z}_j w_i^{-1}; t)_\infty } \prod_{i = 1}^{N_2} \frac{d \hat{z}_i}{2\pi \iota}\prod_{i = 1}^{N_1} \frac{d w_i}{2\pi \iota},
\end{split}
\end{equation}
where the $N_1 = N_2 = 0$ term in (\ref{GenFunExpand3}) is just $1$. We remark that the $t^{-1}$ next to $u_1$ on the second line of (\ref{GenFunExpand3}) is not a typo. Later in (\ref{GenFunExpand5}) we will express the sums over $\mu$'s and $\lambda$'s as contour integrals, which will result in a symmetric expression, for which the extra $t^{-1}$ is important.

At this time, we can do a formal comparison between (\ref{GenFunExpand3}) and (\ref{GenFunForm}). Firstly, the integration over $w_i$ and $\hat{z}_i$ roughly corresponds to the integration over $q_\alpha$, while the summation over $\mu_i$ and $\lambda_i$ roughly corresponds to the one over $m_\alpha$ and $l_\alpha$. The cross term that appears in the last line and the two Cauchy determinants in (\ref{GenFunExpand3}) can be rewritten as
\begin{equation}\label{CrossV1}
\begin{split}
&\det \left[ \frac{1}{w_it^{-\lambda_i} - w_j}\right]_{i,j = 1}^{N_1} \prod_{i = 1}^{N_1} \det \left[ \frac{1}{\hat{z}_i - \hat{z}_j t^{\mu_j}}\right]_{i,j = 1}^{N_2} \prod_{j = 1}^{N_2}   \frac{(\hat{z}_j w_i^{-1} t^{\lambda_i}; t)_\infty }{(\hat{z}_j w_i^{-1} t^{ \lambda_i + \mu_j}; t)_\infty }\frac{(\hat{z}_j w_i^{-1} t^{\mu_j}; t)_\infty }{(\hat{z}_j w_i^{-1}; t)_\infty } = \\
& \det \left[ \frac{1}{Z_i - W_j}\right]_{i,j = 1}^{N_1 + N_2} \cdot  \prod_{i = 1}^{N_1} \prod_{j = 1}^{N_2}\frac{ \Gamma_t(1 +\log_t \hat{z}_j - \log_t w_i) \Gamma_t (1 + \log_t \hat{w}_j - \log_t z_i)}{\Gamma_t(1 + \log_t \hat{z}_j - \log_tz_i) \Gamma_t(1 + \log_t \hat{w}_j - \log_t w_i )},
\end{split}
\end{equation}
where $Z = (w_1 t^{-\lambda_1}, \dots, w_{N_1} t^{-\lambda_{N_1}}, \hat{z}_1, \dots, \hat{z}_{N_2})$ and $W = (w_1, \dots, w_{N_1}, \hat{z}_1 t^{\mu_1}, \dots, \hat{z}_{N_2}t^{\mu_{N_2}})$ and $\Gamma_t$ denotes the $t$-Gamma function, see \cite[(10.3.3)]{Andrews},
$$\Gamma_t(x) = \frac{(t;t)_\infty}{(t^x;t)_\infty}(1-t)^{1-x}.$$
In deriving (\ref{CrossV1}) we used the identity $(a;t)_\infty = (1-a) (at;t)_\infty$ and the Cauchy determinant formula, see e.g. \cite[1.3]{Prasolov}. In veiw of (\ref{CrossV1}) we can now recognize in (\ref{GenFunExpand3}) both an analogue of the Cauchy determinant in (\ref{GenFunForm}) as well as an analogue of the cross term $CT_{KPZ}$ with Gamma functions being replaced by their discrete analogue.

As we can see, there are many similarities between formulas (\ref{GenFunExpand3}) and (\ref{GenFunForm}), although we do not have a good conceptual understanding of why that is the case. The starting point for these two formulas (i.e. equations (\ref{GenFunExpand}) and (\ref{GenFunExpand2}) ) are clearly analogues of each other; however, the frameworks that get us from those formulas to (\ref{GenFunForm}) and (\ref{GenFunExpand3}) respectively are different. In simple words, we do not have a good way to translate the imaginary time Schr{\"o}dinger equation + Bethe ansatz approach to the one of Macdonald difference operators + contour integral ansatz. The question is further obscured by the different roles played by the joint moments in the two frameworks. Specifically, in the context of the KPZ equation, the moments of the solutions of the SHE are {\em eigenfunctions} for the  imaginary time Schr{\"o}dinger operator $- \partial_T - H_n$, while the $t$-moments in our setup are {\em eigenvalues} for the Macdonald difference operators (the eigenfunctions being the Macdonald symmetric functions). We believe that searching for deeper connections between the two frameworks is warranted, but we will leave this task for the future.

%
\subsubsection{Contour integral formulas}\label{Section9.1.2}
Equation (\ref{GenFunForm}) can be found as \cite[(3.7)]{ISS13}, and upon a change of variables matches \cite[(25)]{Dot13}. As explained after equation (3.7) in \cite{ISS13}, the authors of \cite{PSpohn} derived an analogous formula (see \cite[(3.8)]{ISS13})  but only under a certain factorization assumption, which corresponds to setting the cross term $CT_{KPZ} = 1$ in (\ref{GenFunForm}). The way the argument proceeds after \cite[(25)]{Dot13} and after \cite[(3.7)]{ISS13} is to rewrite the discrete sums in (\ref{GenFunForm}) as suitable contour integrals and then perform a formal asymptotic analysis demonstrating the emergence of the Fredholm determinant formula for the two-point cdf of the Airy process. We mention here that in rewriting the discrete sums in terms of integrals, \cite{Dot13} ignores the contribution of the infinitely many poles coming from the Gamma functions in the cross term $CT_{KPZ}$. The formal limit shows that the cross term $CT_{KPZ}$ converges to $1$, which {\em a posteriori} gives some credence to why one can ignore the contribution of these poles, and why one can make the factorization assumption in \cite{PSpohn} that corresponds to setting $CT_{KPZ} = 1$ before taking the limit.

In the remainder of this section, we explain what guided our choice of contour deformation in going from (\ref{S1MomentFun2}) to (\ref{S1MomentFun3}) and explain how starting from (\ref{GenFunExpand3}) we can obtain the formula in Theorem \ref{PrelimitT}, which is the one we use in the paper to do asymptotics. In the course of our discussion, we will see that when rewriting (\ref{GenFunExpand3}) as a contour integral we do not encounter the infinite cross-term pole problem from \cite{Dot13}.\\

Starting from (\ref{S1MomentFun2}) there are a few natural ways one can attempt to deform the contours using the contour integral ansatz. The first is to simply deform all $L+R$ contours down to the same contour $C^2_{0,R}$ (the innermost circle). If one does this, the residue subspaces become labeled by partitions $\nu \vdash (L +R)$. The structure that is lost in performing this ansatz is that the resulting integral is no longer (visibly) analytic in the parts $(\nu_1, \dots, \nu_{\ell(\nu)})$. Since eventually, we will want to rewrite any discrete sum over partitions as contour integrals, it is desirable for the resulting expression after the contour integral ansatz to be an analytic function of the parts of these partitions. As one can see from (\ref{S1MomentFun3}) this is clearly the case for the deformation we performed. Another possibility is to deform all contours $C^1_{0,i}$ down to $C^{1}_{0,L}$ and all contours $C^{2}_{0,i}$ down to $C^2_{0,R}$. After doing this contour integral ansatz, one obtains an expression that is very similar to (\ref{S1MomentFun3}). The problem that arises in this case is that the requirement for $a$ to be inside $C^2_{0,R}$ and $a^{-1}$ outside $C^{1}_{0,L}$ makes the analogue of (\ref{S1MomentFun3}) valid only for small enough (depending on $R+L$) values of $a$. In particular, we cannot find contours that work for all the terms in (\ref{GenFunExpand2}) simultaneously unless $a = 0$, which trivializes the model. There are a few other possible ways of deforming the contours, but they all run into one of the above two problems and the only working solution we have is through equation (\ref{S1MomentFun3}).

We finally, discuss how to rewrite the sums over $\mu_i$ and $\lambda_i$ in (\ref{GenFunExpand3}) as contour integrals. The keen reader may have noticed that there is an extra $t^{-1}$ next to $u_1$ in the second line of (\ref{GenFunExpand3}) compared to the third line. This extra term is not a typo and after re-expressing the discrete sums in (\ref{GenFunExpand3}) with integrals below, we will obtain a symmetric expression. The key identity we use is as follows:
$$\sum_{n_1, \dots, n_k = 1}^\infty f(t^{n_1}, \dots, t^{n_k}) y_1^{n_1} \cdots y_k^{n_k} = \frac{1}{(2 \pi \iota)^k} \int_{\ell_\delta^k} \prod_{i = 1}^k \frac{\pi [-y_i ]^{s_i} }{\sin (-\pi s_i)} f(t^{s_1}, \dots, t^{s_k}) ds_1 \cdots ds_k,$$
where $\ell_\delta$ is a vertical contour passing through $\delta \in (0,1)$, with the orientation of increasing imaginary part. The latter equation essentially follows from $Res_{s = n} \frac{\pi}{\sin(-\pi s)} = (-1)^{n+1}$. Fixing a small $\delta > 0$ and applying the above formula to (\ref{GenFunExpand3}) we obtain
\begin{equation}\label{GenFunExpand4}
\begin{split}
&\mathbb{E}\left[ \frac{1}{(u_1t^{h(n_1 + 1 , M) - M};t)_{\infty}} \frac{1}{(u_2t^{h(n_2 + 1 , M) - M};t)_{\infty}} \right]=\sum_{N_1 = 0}^\infty \sum_{N_2 = 0}^{\infty} \frac{1}{N_1! N_2!}\int_{C_{1 + \varepsilon}^{N_1}} \int_{C_{1 - \varepsilon}^{N_2}} \int_{\ell^{N_1}_\delta}\int_{\ell^{N_2}_\delta} \\
&  \det \left[ \frac{1}{w_it^{-s_i} - w_j}\right]_{i,j = 1}^{N_1} \prod_{i = 1}^{N_1} \left( \frac{1 - aw_it^{- s_i}}{1 - aw_i } \right)^M  \left(\frac{1 - a w_i^{-1} }{1 - a w_i^{-1}  t^{s_i}}\right)^{n_1}\frac{\pi [-u_1t^{-1}]^{s_i}}{\sin ( -\pi s_i)} \\
&  \det \left[ \frac{1}{\hat{z}_i - \hat{z}_j t^{v_j}}\right]_{i,j = 1}^{N_2} \prod_{i = 1}^{N_2} \left( \frac{1 - a\hat{z}_i }{1 - a \hat{z}_it^{v_i}} \right)^M    \left( \frac{1 - a\hat{z}^{-1}_it^{-v_i}}{1 - a\hat{z}^{-1}_i } \right)^{n_2} \frac{ [-u_2]^{v_i}}{\sin(-\pi v_i)}\\
& \prod_{i = 1}^{N_1}  \prod_{j = 1}^{N_2}   \frac{(\hat{z}_j w_i^{-1} t^{s_i}; t)_\infty }{(\hat{z}_j w_i^{-1} t^{ s_i + v_j}; t)_\infty }\frac{(\hat{z}_j w_i^{-1} t^{v_j}; t)_\infty }{(\hat{z}_j w_i^{-1}; t)_\infty } \prod_{i = 1}^{N_2} \frac{d v_i}{2\pi \iota}\prod_{i = 1}^{N_1} \frac{d s_i}{2\pi \iota} \prod_{i = 1}^{N_2} \frac{d \hat{z}_i}{2\pi \iota}\prod_{i = 1}^{N_1} \frac{d w_i}{2\pi \iota}.
\end{split}
\end{equation}
We point out that the integrand in (\ref{GenFunExpand4}) has no poles in the right half-plane in either of the $s_i$ or $v_i$ variables. In particular, unlike \cite{Dot13}, our contour representation does not include any poles from the cross term. Finally, we may perform the change of variables $z_i = w_i t^{-s_i}$ for $i = 1, \dots, N_1$ and $\hat{w}_i = \hat{z}_i t^{v_i}$ for $i = 1, \dots, N_2$ to rewrite (\ref{GenFunExpand4}) as 
\begin{equation}\label{GenFunExpand5}
\begin{split}
&\mathbb{E}\left[ \frac{1}{(u_1t^{h(n_1 + 1 , M) - M};t)_{\infty}} \frac{1}{(u_2t^{h(n_2 + 1 , M) - M};t)_{\infty}} \right]=\sum_{N_1 = 0}^\infty \sum_{N_2 = 0}^{\infty} \frac{1}{N_1! N_2!}\int_{C_{1 + \varepsilon}^{N_1}}\int_{C_{1 - \varepsilon}^{N_2}}  \\
& \int_{C_{(1+\varepsilon)t^{-\delta}}^{N_1}}\int_{C_{(1-\varepsilon)t^{\delta}}^{N_2}}  \det \left[ \frac{1}{z_i - w_j}\right]_{i,j = 1}^{N_1} \prod_{i = 1}^{N_1} \left( \frac{1 - az_i}{1 - aw_i } \right)^M  \left(\frac{1 - a w_i^{-1} }{1 - a z_i^{-1}  }\right)^{n_1}\frac{S(w_i, z_i; u_1, t)}{[-\log t] w_i} \\
&  \det \left[ \frac{1}{\hat{z}_i - \hat{w}_j }\right]_{i,j = 1}^{N_2} \prod_{i = 1}^{N_2} \left( \frac{1 - a\hat{z}_i }{1 - a \hat{w}_i} \right)^M    \left( \frac{1 - a\hat{w}^{-1}_i}{1 - a\hat{z}^{-1}_i } \right)^{n_2}\frac{S(\hat{w}_i, \hat{z}_i; u_2, t)}{[-\log t] \hat{w}_i}\\
& \prod_{i = 1}^{N_1}  \prod_{j = 1}^{N_2}   \frac{(\hat{z}_j z_i^{-1} ; t)_\infty }{(\hat{w}_j z_i^{-1} ; t)_\infty }\frac{(\hat{w}_j w_i^{-1} ; t)_\infty }{(\hat{z}_j w_i^{-1}; t)_\infty } \prod_{i = 1}^{N_2} \frac{d \hat{w}_i}{2\pi \iota}\prod_{i = 1}^{N_1} \frac{d z_i}{2\pi \iota} \prod_{i = 1}^{N_2} \frac{d \hat{z}_i}{2\pi \iota}\prod_{i = 1}^{N_1} \frac{d w_i}{2\pi \iota},
\end{split}
\end{equation}
where $C_r$ is a positively oriented circle of radius $r$, centered at the origin, and $S(w,z;u,t)$ is the function from Definition \ref{DefFunS}:
\begin{equation}\label{S1SpiralDef}
S(w, z; u,t) =   \sum_{m \in \mathbb{Z}} \frac{\pi \cdot [ - u ]^{[\log w - \log z] [\log t]^{-1} -  2m \pi  \iota [\log t ]^{-1}}}{\sin(-\pi [[\log w - \log z] [\log t]^{-1} -  2m \pi  \iota [\log t ]^{-1}])}.
\end{equation}
Notice that (\ref{GenFunExpand5}) is now symmetric in the $u_1, u_2$ variables (compared to (\ref{GenFunExpand4})) and matches equation (\ref{PrelimitEq}) from Theorem \ref{PrelimitT} for the case when $x_i = a = y_j$. 

%
\subsection{Convergence and asymptotics}\label{Section9.2}
In Section \ref{Section9.2.1} we explain how our prelimit formulas compare to the ones for the log-gamma polymer in \cite{NZ}. In Section \ref{Section9.2.2} we address the questions of convergence of our formulas, and give an overview of our asymptotic analysis.

%
\subsubsection{Log-gamma polymer}\label{Section9.2.1}\label{Section9.2.1} In this section we summarize the approach of studying the two-point asymptotics of the log-gamma polymer model from \cite{NZ} and compare it to the present paper. The log-gamma polymer depends on two sequences of real parameters $\{ \alpha_i \}_{i \geq 1}$ and $\{\hat{\alpha}_i\}_{i \geq 1}$ such that $\alpha_i + \hat{\alpha_j} > 0$ for all $i, j \in \mathbb{N}$. One takes a countable collection of independent random variables $w_{i,j}$ indexed by $(i,j) \in \mathbb{Z}^2_{\geq 1}$ such that $w_{i,j}$ has the {\em inverse-Gamma distribution with parameter} $\theta_{i,j} = \alpha_i + \hat{\alpha}_i$, meaning that 
$$\mathbb{P}(w_{i,j} \in dx) ={\bf 1}\{ x > 0 \} \cdot \frac{1}{\Gamma(\theta_{i,j})} x^{-\theta_{i,j} -1} \exp (-1/x) dx.$$
From this data one defines the (random) {\em log-gamma partition function} as
\begin{equation}\label{S1LGPP}
Z_{m,n}: = \sum_{\pi \in \Pi_{m,n}} \prod_{(i,j) \in \pi} w_{i,j},
\end{equation}
where $(m,n) \in \mathbb{Z}^2_{\geq 1}$, the sum is over all up-right paths from $(1,1)$ to $(m,n)$ and the product is over the vertices that belong to $\pi$. We remark that in \cite{NZ} the convention is for the numbers to increase in the downward direction so that the paths are down-right. This should cause no confusion.

\cite[Conjecture 4.1]{NZ} states that if $(m_1, n_1) = (N - t_1 N^{2/3}, N + t_1 N^{2/3})$ and $(m_2, n_2) = (N + t_2N^{2/3}, N - t_2N^{2/3})$ with $t_1, t_2 > 0$ then the vectors $(\log Z_{m_1, n_1}, \log Z_{m_2, n_2})$ under suitable shifts and scales should weakly converge to the two-point distribution of the Airy process. The starting point of the analysis is the following formula for the joint Laplace transform for $Z_{m_1, n_1}$ and $Z_{m_2, n_2}$
\begin{equation}\label{S1LGJL}
\begin{split}
&\mathbb{E}\hspace{-1mm} \left[ e^{-u_1 Z_{m_1, n_1} - u_2 Z_{m_2, n_2}}\right] \hspace{-1mm}= \hspace{-1mm}\int_{\ell_{\delta}^{m_1}}\hspace{-1mm}\int_{\ell_{\delta+ \gamma}^{n_2}}  s_{m_1}(\vec{\lambda}) s_{n_2}(\vec{\mu})  F^{LG}(\vec{\lambda},\vec{\mu};u_1, u_2)  \prod_{i = 1}^{m_1} \prod_{j = 1}^{n_2} \frac{\Gamma(\lambda_i + \mu_j)}{\Gamma(\alpha_i + \hat{\alpha}_j)}d\vec{\mu}  d\vec{\lambda},
\end{split}
\end{equation}
which is obtained via the geometric Robinson-Schensted-Knuth correspondence. In (\ref{S1LGJL}) the contour $\ell_x$ is a vertically oriented line passing through $x$ in the complex plane, and $s_n(\vec{\lambda})$ denotes the {\em Sklyanin measure}
$$s_n(\lambda) = \frac{1}{(2\pi \iota)^n n!} \prod_{i \neq j } \Gamma(\lambda_i - \lambda_j)^{-1}.$$
The precise formula (i.e. the definition of $\delta, \gamma$ and $F^{LG}$ in (\ref{S1LGJL})) can be found as \cite[(2.21)]{NZ}. Equation (\ref{S1LGJL}) is not suitable for asymptotic analysis since the number of contours goes to infinity, and in order to re-express it \cite{NZ} uses \cite[Theorem 2]{BCR}, which allows one to rewrite the $d\vec{\mu}$ and $d\vec{\lambda}$ integrals in (\ref{S1LGJL}) as Fredholm determinants. 

By applying \cite[Theorem 2]{BCR} twice to (\ref{S1LGJL}), expanding the corresponding Fredholm determinants and formally rearranging the resulting contours one obtains
\begin{equation}\label{S1LGJLV2}
\begin{split}
&\mathbb{E}\left[ e^{-u_1 Z_{m_1, n_1} - u_2 Z_{m_2, n_2}}\right] = \sum_{N_1 = 0}^{m_1} \sum_{N_2 = 0}^{n_2} \frac{1}{N_1! N_2!} \frac{1}{ (2\pi \iota)^{2N_1 + 2N_2}} \int_{\ell_{\delta_1}^{N_1}} \int_{C_{\delta_1}^{N_1}}\int_{\ell_{\delta_1}^{N_2}} \int_{C_{\delta_1}^{N_2}}   \\
&D( \vec{w}, \vec{z})  {G}^{LG}_{n_1,m_1}(\vec{w},\vec{z} , u_1) \cdot D( \vec{\hat{w}}, \vec{\hat{z}})  {G}^{LG}_{n_2, m_2}( \vec{\hat{w}},\vec{\hat{z}}, u_2)  \cdot CT_{LG}(\vec{w}, \vec{z}; \vec{\hat{w}}, \vec{\hat{z}}) d\vec{\hat{w}} d\vec{\hat{z}}d\vec{w} d\vec{z},
\end{split}
\end{equation}
where $C_{\delta}$ is a positively oriented, zero-centered circle of radius $\delta$, $D(\vec{a}, \vec{b}) = \det \left[\frac{1}{b_i - a_j} \right]_{i,j = 1}^{n}$ is the Cauchy determinant and the cross term $CT_{LG}$ is given by
\begin{equation}\label{S1LGCT}
\begin{split}
&CT_{LG}(\vec{w}, \vec{z}; \vec{\hat{w}}, \vec{\hat{z}}) = \prod_{i = 1}^{N_1} \prod_{j = 1}^{N_2} \frac{\Gamma(\gamma - w_i - \hat{w}_j) \Gamma(\gamma -z_i - \hat{z}_j)}{\Gamma(\gamma - w_i - \hat{z}_j) \Gamma(\gamma -z_i - \hat{w}_j)}
\end{split}
\end{equation}
We refer the interested reader to \cite[(4.13)]{NZ} and the discussion surrounding that equation for more information on why the rearrangement that gives (\ref{S1LGJLV2}) is formal in the case of the log-gamma polymer but it is completely fine for the O'Connell-Yor polymer and its mixed version with the log-gamma polymer. In equation (\ref{S1LGJLV2}) all the $\alpha$ parameters of the model are set to $0$, and all the $\hat{\alpha}$ parameters are set to $\gamma > 0$, and $\delta \in (0, \gamma/2)$, while $\delta_1 \in (0, \min (\delta, 1- \delta))$. We will not write down the formulas for ${G}^{LG}_{n_i,m_i}$ but refer the interested reader to \cite[Proposition 4.5]{NZ}. \\

The analogue of (\ref{S1LGJLV2}) we derive in this paper formulated in terms of the stochastic six-vertex model is as follows
\begin{equation}\label{GenFunExpand6}
\begin{split}
&\mathbb{E}\left[ \frac{1}{(u_1t^{h(n_1 + 1 , M) - M};t)_{\infty}} \frac{1}{(u_2t^{h(n_2 + 1 , M) - M};t)_{\infty}} \right]=\sum_{N_1 = 0}^\infty \sum_{N_2 = 0}^{\infty}  I_{M}(N_1,N_2), \mbox{ where }\\
&I_M(N_1, N_2) = \frac{1}{N_1! N_2!}\int_{\gamma_1^{N_1}} \int_{\gamma_2^{N_1}}  \int_{\gamma_3^{N_2}}\int_{\gamma_4^{N_2}} D(\vec{w}, \vec{z})G(\vec{w}, \vec{z}, n_1, u_1)      \\
&  D(\vec{\hat{w}}, \vec{\hat{z}})G(\vec{\hat{w}}, \vec{\hat{z}},n_2, u_2) \cdot  CT(\vec{w}, \vec{z}; \vec{\hat{w}}, \vec{\hat{z}}) \prod_{i = 1}^{N_2}\frac{d\hat{w}_i}{2\pi \iota}\prod_{i = 1}^{N_2}\frac{d\hat{z}_i}{2\pi \iota}\prod_{i = 1}^{N_1}\frac{dw_i}{2\pi \iota}\prod_{i = 1}^{N_1}\frac{dz_i}{2\pi \iota}.
\end{split}
\end{equation}
In (\ref{GenFunExpand6}) the contours $\gamma_i$ are positively oriented zero-centered circles of radius $r_i \in (a, a^{-1})$ such that $r_1 > r_2 > r_3 > r_4 > tr_1$, where we recall that $t = b_1/b_2$ and $a^2 = \frac{1 - b_1}{1-b_2}$. In addition, we have
\begin{equation}\label{S1MPLTGv2}
\begin{split}
& D(\vec{w}, \vec{z}) =   \det \left[ \frac{1}{z_i - w_j}\right]_{i,j = 1}^{N_1} \hspace{5mm} CT(\vec{w}, \vec{z}; \vec{\hat{w}}, \vec{\hat{z}}) =  \prod_{i = 1}^{N_1}\prod_{j = 1}^{N_2}  \frac{(\hat{z}_j z_i^{-1}; t)_\infty }{(\hat{w}_j z_i^{-1}; t)_\infty }\frac{(\hat{w}_j w_i^{-1} ; t)_\infty }{(\hat{z}_j w_i^{-1}; t)_\infty },  \\
&G(\vec{w}, \vec{z}, n, u) =  \prod_{i = 1}^{N_1} \left( \frac{1 - az_i}{1 - aw_i} \right)^M \cdot \left( \frac{1 - a/ w_i}{1 - a/ z_i}\right)^n \cdot \frac{ S(w_i, z_i; u,t )}{ -\log t \cdot w_i} , \\
\end{split}
\end{equation}
where $S(w,z;u,t)$ is the function from Definition \ref{DefFunS} (alternatively see equation (\ref{S1SpiralDef}). If $N_1 = N_2 = 0$ then we adopt the convention $I_M(0,0) = 1$. Equation (\ref{GenFunExpand6}) is formulated for the ascending Hall-Littlewood process in Theorem \ref{PrelimitT}, and one deduces the analogous result from the distributional equality of the latter and the stochastic six-vertex model, see the end of Section \ref{Section2.2}.

Comparing (\ref{S1LGJLV2}) with (\ref{GenFunExpand6}) we observe striking structural similarities. On the left side of (\ref{GenFunExpand6})  we have a joint $t$-Laplace transform, which is a discrete analogue of the joint Laplace transform in (\ref{S1LGJLV2}). On the right side of both (\ref{S1LGJLV2}) and (\ref{GenFunExpand6})  we have an analogous double sum of $(2N_1 + 2N_2)$-fold contour integrals, where the integrand is a product of similar three types of terms. The third term in both formulas is a cross term, which in the log-gamma case is a double product of Gamma functions, while in our case is a double product of $t$-Gamma functions, cf. (\ref{CrossV1}). We believe that the similarities between (\ref{S1LGJLV2}) and (\ref{GenFunExpand6}) are not coincidental but come from the known connections between the Macdonald processes and the log-gamma polymer, cf. \cite[Section 4]{BCFV}. Specifically, we expect that one can carry out a similar program to the one in the present paper and derive joint $q$-Laplace transform formulas for the $q$-Whittaker process (a certain dual to the Hall-Littlewood process considered in our paper). Since the $q$-Whittaker process converges to the log-gamma polymer model, as shown in \cite[Section 4]{BCFV}, we expect such joint $q$-Laplace transform formulas to asymptotically give the formulas in \cite{NZ}. At this time this approach is purely conjectural, but if correct it would explain the strong similarities we see between our formulas and those of \cite{NZ}.

%
\subsubsection{Handling the cross term}\label{Section9.2.2}
We first discuss the convergence issues in equation (\ref{GenFunExpand6}). There are four types of terms that influence the growth of $I_M(N_1,N_2)$: (1) the functions $G$, (2) the cross term $CT$, (3) the Cauchy determinants $D$ and (4) the factorials $N_1!N_2!$ in the front. Unlike (\ref{S1LGJLV2}) all of our contours are compact and so obtaining point-wise estimates for the integrand essentially suffices for controlling $I_M(N_1,N_2)$. Out of the four terms, the functions $G$ are benign, because they grow as $e^{C(N_1 + N_2)}$ and are thus controlled by the factorials that decay like $e^{-N_1 \log N_1 - N_2 \log N_2}$. In the past, the known way to control the Cauchy determinant is via Hadamard's inequality, which states that 
\begin{equation}\label{S1Had}
|\det A| \leq \prod_{i = 1}^N \|v_i\|
\end{equation}
for an $N \times N$ matrix $A$ with column vectors $v_1, \dots ,v_N$. Consequently, if we ignore the cross term in (\ref{GenFunExpand6}) all other terms can be bounded by
$$\exp\left( O(N_1 + N_2) - \frac{N_1 \log N_1}{2} - \frac{N_2 \log N_2}{2} \right),$$
which would be summable over $(N_1, N_2) \in \mathbb{Z}_{\geq 0}^2.$ The problem we run into is that the cross term $CT$ is pointwise of order $e^{cN_1N_2}$ and so the factorials are no longer enough to control it. 

It is worth pointing out that the prelimit formula (\ref{S1LGJLV2}) from \cite{NZ} does not encounter this cross term problem, because the sum is {\em finite}. The reason \cite{NZ} obtained a finite sum is a consequence of the fact that in applying \cite[Theorem 2]{BCR} the kernels involved are finite rank, and the resulting Fredholm determinant expansion series terminate after finitely many terms. The $G^{LG}$ functions in (\ref{S1LGJLV2}) all decay sufficiently fast near infinity and so even though the contours in that formula are infinite each summand is easily seen to be finite, and so the whole sum is finite. As one takes the $N \rightarrow \infty$ limit in (\ref{S1LGJLV2})  the cross terms will start being problematic, but at least at a finite $N$ level they do not cause problems unlike our setup.

The way we deal with $CT$ in our formula is by utilizing some {\em hidden built-in} decay in the Cauchy determinants $D$, which to our knowledge has not been previously recognized. In particular, we have the following estimate for any $r, R \in (0,\infty)$ with $R > r$
\begin{equation}\label{S1CDetGood}
\left|\det \left[ \frac{1}{z_i - w_j}\right]_{i,j = 1}^N \right| \leq R^{-N} \cdot \frac{N^N \cdot (r/R)^{\binom{N}{2}}}{(1-r/R)^{N^2}},
\end{equation}
where $z_i, w_i \in \mathbb{C}$ are such that $|z_i| = R$ and $|w_i|  \leq r$ for $i = 1, \dots, N$. This results appears as Lemma \ref{DetBounds} in the main text.

The way we use (\ref{S1CDetGood}) to control the summands in (\ref{GenFunExpand6}) is by deforming the $\gamma_i$ contours so that $r_2/r_1$ and $r_4/r_3$ are both very small and then one obtains the following bound on the Cauchy determinants 
\begin{equation}\label{S1Had2}
|D(\vec{w}, \vec{z}) D(\vec{\hat{w}}, \vec{\hat{z}})| \leq  \exp( - c (N_1^2  +N_2^2)),
\end{equation}
where the constant $c$ increases as the ratios $r_2/r_1$ and $r_4/r_3$ become smaller. The restrictions $r_4 > tr_1$ and $r_i \in (a, a^{-1})$ prevent us from freely deforming the contours $\gamma_i$ so that these ratios are arbitrarily small, but at least if $t$ and $a$ are small enough, this is possible and one is able to control the cross term $CT$ using the Cauchy determinants $D(\vec{w}, \vec{z})$, $D(\vec{\hat{w}}, \vec{\hat{z}})$. This is one of the sources of the parameter restriction in Theorem \ref{thmMain}, and we can only prove (\ref{GenFunExpand6}) in Theorem \ref{PrelimitT} for a special range of parameters. See Remark \ref{RemarkRest} for more details.\\

We conclude this section with a discussion about the asymptotic analysis of equations (\ref{S1LGJLV2}) and (\ref{GenFunExpand6}). In both equations  (\ref{S1LGJLV2}) and  (\ref{GenFunExpand6}) it is relatively easy to show that the right sides term-wise converge to a suitable Fredholm determinant expansion for the Airy process (as $N \rightarrow \infty$ in (\ref{S1LGJLV2}) and $M \rightarrow \infty$ in (\ref{GenFunExpand6})). The arguments in both cases involve deforming contours to descent ones for the $G$ functions and applying a careful steepest descent argument. In the present paper, the term-wise convergence is stated as Proposition \ref{PropTermConv} and proved in Section \ref{Section5}.  The identification of the resulting sum with the Fredholm determinant expansion for the Airy process is formulated as Proposition \ref{PropTermLimit} and proved in Section \ref{Section7.2}.

The essential ingredient missing and making the proof in \cite{NZ} conditional is a uniform in $N$ estimate on the growth of the terms (in terms of $N_1$ and $N_2$) in the series (\ref{S1LGJLV2}) that would allow one to exchange the order of the sum and the limit $N \rightarrow \infty$. The difficulty in obtaining such growth estimates comes from the presence of the cross term $CT_{LG}$, whose behavior as $N_1, N_2$ become large is very complicated. Part of the progress made in our paper is the ability to control the cross terms $CT$ in our formulas (at least for some range of parameters) and obtain the necessary bounds that would allow one to exchange the sums and the $M \rightarrow \infty$ limit in (\ref{GenFunExpand6}). The precise bounds we can obtain are given in Proposition \ref{PropTermBound} and its proof can be found in Section \ref{Section6}. 

While the proof of Proposition \ref{PropTermBound} is fairly technical, we try to give a rough account of the ideas contained in it below. The way the proof of Proposition \ref{PropTermBound} goes is by splitting the problem into three cases: (1) $\Delta \gtrsim M$, (2) $M^{1/2} \lesssim \Delta \lesssim M$ and (3) $\Delta \lesssim M^{1/2}$, where $\Delta = N_1 + N_2$. The goal is then to deform the contours $\gamma_1, \dots, \gamma_4$ in a suitable way so that the different terms ($D$, $G$ and $CT$) can be balanced in a favorable way so that a rapid enough decay in $\Delta$ is achieved. The Cauchy determinants $D$ are always helpful (i.e. they are sources of decay), the cross term $CT$ is always harmful (it is a source of growth) and the $G$ terms can be helpful if the contours $\gamma_i$ are descent contours or harmful if they are not. 

In the case $\Delta \gtrsim  M$ the idea is simply to spread apart the $\gamma_i$ so that the ratios of radii $r_2/r_1$ and $r_4/r_3$ are very small. By doing this we obtain
$$ \left| D(\vec{w}, \vec{z}) D(\vec{\hat{w}}, \vec{\hat{z}}) \right| \leq \exp( -c_1 \Delta^2), \left|G(\vec{w}, \vec{z}, n_1, u_1) G(\vec{\hat{w}}, \vec{\hat{z}},n_2, u_2) \right| \leq \exp ( c_2 \Delta M), \left| CT \right|\leq  \exp ( c_3 \Delta^2).$$
Notice that by spreading apart the contours $\gamma_i$, we move them away from the descent contours for the $G$ functions and so they grow; however, if $a$ is sufficiently small and $\Delta \gtrsim M$ the Cauchy determinants are decaying fast enough to offset both the $G$ functions and the cross term $CT$, and one obtains 
$$|I_M(N_1,N_2)| \leq \exp \left( - \epsilon \Delta^2 \right).$$

In the case $\Delta \lesssim M^{1/2}$ the idea is simply to deform $\gamma_i$ to descent contours. In this case, the $G$ functions decay fast enough to offset the cross term, and one can bound the Cauchy determinants by Hadamard's inequality (\ref{S1Had}) and use the factorials $N_1! N_2!$ in $I_M(N_1, N_2)$ to prove
$$|I_M(N_1, N_2)| \leq \exp \left( O(\Delta) - \frac{\Delta \log \Delta}{4}\right).$$

The case $M^{1/2} \lesssim \Delta \lesssim M$ is the more involved. In this case we spread the $\gamma_i$ contours a little bit, so that {\em some} decay is extracted from the Cauchy determinants $D$, but at the same time the growth of the $G$-functions, which are big away from their descent contours, is still manageable. In this case, we need to Taylor expand $\log CT$ upto second order around the critical point of $\log G$, and use the $G$ functions to control $CT$ when enough of the $z$ and $w$ variables are away from the critical point. If many of these variables are close to the critical point, then that implies that they are close to each other and one obtains {\em improved} decay estimates for the Cauchy determinants $D$. This extra bit of decay can be traced to the presence of the Vandermonde determinants in the numerator of the Cauchy determinant formula, which makes the $D$ terms even smaller if many of the $z$, $w$ variables start being close to each other. Ultimately, there is a delicate balance of all three types of terms $D, G$ and $CT$ that is favorably resolved leading to the estimate
$$|I_M(N_1,N_2)| \leq \exp \left( - \epsilon \Delta^2 \right),$$
provided the $t$ and $a$ parameters are sufficiently small. 

The above exposition, aimed to illustrate on a very basic level how the proof of Proposition \ref{PropTermBound} is structured. The interested reader is referred to Section \ref{Section6} for the technical details.

\bibliographystyle{alpha}
\bibliography{PD}

\end{document}